\documentclass[
               a4paper,
               twoside,
               ]{amsart}
\usepackage{amssymb} 
\usepackage{MnSymbol} 
\usepackage{mathtools} 
\usepackage[all]{xy} 
\usepackage{color} 
\usepackage{ifdraft} 
\usepackage{lastpage} 
\usepackage[pdfstartview=FitH]{hyperref} 

\DeclareMathAlphabet{\mathpzc}{OT1}{pzc}{m}{it}

\newtheorem{thm}{Theorem}[section]
\newtheorem{thmA}{Theorem}
\newtheorem{cor}[thm]{Corollary}

\newtheorem{lem}[thm]{Lemma}

\newtheorem{conj}[thm]{Conjecture}

\newtheorem{prop}[thm]{Proposition}
\newtheorem{propA}[thmA]{Proposition}

\theoremstyle{definition}
\newtheorem{defn}[thm]{Definition}

\theoremstyle{remark}
\newtheorem{rem}[thm]{Remark}

\newtheorem{exmpl}[thm]{Example}

\numberwithin{equation}{section}

\providecommand\given{}
\newcommand\SetSymbol[1][]{%
\nonscript\:#1\vert
\allowbreak
\nonscript\:
\mathopen{}}
\DeclarePairedDelimiterX\set[1]\{\}{%
\renewcommand\given{\SetSymbol[\delimsize]}
#1
}

\newcommand*{\cat}[1]{\mathbf{#1}} 
\newcommand*{\mto}{\rightarrow} 
\newcommand*{\cto}{\rightarrowtail} 
\newcommand*{\qto}{\twoheadrightarrow} 
\newcommand*{\wto}{\xrightarrow{\sim}} 
\newcommand*{\otw}{\xleftarrow{\sim}} 
\newcommand*{\pair}[2]{\left\langle #1,#2 \right\rangle} 
\newcommand*{\com}[2]{\left[#1,#2\right]} 
\newcommand*{\del}{\partial} 
\newcommand*{\D}{\mathcal{D}} 
\newcommand*{\comp}{\circ} 
\newcommand*{\Int}{\mathbb{Z}} 
\newcommand*{\Rat}{\mathbb{Q}} 
\newcommand*{\C}{\mathbb{C}} 
\newcommand*{\N}{\mathbb{N}} 
\newcommand*{\Real}{\mathbb{R}} 
\newcommand*{\FF}{\mathbb{F}} 
\newcommand*{\cmplx}[1]{{#1}^\bullet} 
\newcommand*{\tensor}{\otimes} 
\newcommand*{\ctensor}{\hat{\otimes}} 
\newcommand*{\Ltensor}{\tensor^{\mathbb{L}}} 
\newcommand*{\isomorph}{\cong} 
\newcommand*{\op}{\mathrm{op}} 
\newcommand*{\sm}{\mathrm{sm}} 
\newcommand*{\cont}{\mathrm{cont}} 
\newcommand*{\sesi}{\mathrm{ss}} 
\newcommand*{\cycchar}{\kappa} 
\newcommand*{\algc}[1]{\overline{#1}} 
\newcommand*{\sheaf}[1]{\mathpzc{#1}} 

\newcommand{\openideals}{\mathfrak{I}} 
\newcommand{\NS}{\mathfrak{N}} 
\newcommand{\Sect}{\Gamma} 
\newcommand{\Sectc}{\Gamma_{\mathrm{c}}} 
\newcommand{\nr}{\mathrm{nr}} 
\newcommand{\cyc}{\mathrm{cyc}} 
\newcommand{\ncL}{\mathcal{L}} 
\newcommand{\X}{\mathcal{X}} 
\newcommand*{\pdual}[1]{{#1}^{\vee}} 
\newcommand{\Motive}{\mathcal{M}} 
\newcommand{\IntR}{\mathcal{O}} 
\newcommand{\Gm}{{\mathbb{G}_{\mathrm{m}}}} 
\newcommand{\cSK}{\widehat{\SK}} 
\newcommand{\PSp}{\mathbb{P}} 
\newcommand{\dv}{\mathrm{div}} 
\newcommand{\bh}{d} 
\newcommand{\mdual}{\ast} 
\newcommand{\Mdual}{\circledast} 
\newcommand{\ncM}{\mathcal{M}} 
\newcommand{\Tau}{\mathrm{T}} 

\newcommand{\ringtransf}{\Psi} 
\newcommand{\id}{\mathrm{id}} 
\newcommand{\Frob}{\mathfrak{F}} 
\newcommand{\eval}{\Phi} 
\newcommand{\ev}{\mathrm{ev}} 
\newcommand{\aug}{\varphi} 
\newcommand{\cc}{\iota} 
\newcommand{\Div}{\sheaf{Div}} 

\DeclareMathOperator{\Gl}{Gl} 
\DeclareMathOperator{\Cyl}{Cyl} 
\DeclareMathOperator{\CoCyl}{CoCyl} 
\DeclareMathOperator{\Cone}{Cone} 
\DeclareMathOperator{\CoCone}{CoCone} 
\DeclareMathOperator{\Shift}{\Sigma} 
\DeclareMathOperator{\CoShift}{Co\Sigma} 
\DeclareMathOperator{\HF}{H} 
\DeclareMathOperator{\Hom}{Hom} 
\DeclareMathOperator{\Tor}{Tor} 
\DeclareMathOperator{\Jac}{Jac} 
\DeclareMathOperator{\im}{im} 
\DeclareMathOperator{\coker}{coker} 
\DeclareMathOperator{\Spec}{Spec} 
\DeclareMathOperator{\God}{G} 
\DeclareMathOperator{\RDer}{R} 
\DeclareMathOperator{\Gal}{Gal} 
\DeclareMathOperator{\KTh}{K} 
\DeclareMathOperator{\SK}{SK} 
\DeclareMathOperator{\Tate}{T} 
\DeclareMathOperator{\Res}{Res} 
\DeclareMathOperator{\Ind}{Ind} 
\DeclareMathOperator*{\Rlim}{\RDer\varprojlim} 
\DeclareMathOperator{\sheafHom}{\mathpzc{Hom}} 
\let\Re\relax 
\DeclareMathOperator{\Re}{Re} 

\newdir{ >}{{}*!/-5pt/@{>}} 
\newdir{O}{{}*-\cir<2.5pt>{}} 
\xyoption{2cell} 
\UseTwocells %
\xyoption{poly}
\newcommand*{\lbl}[1]{\save+<-1pc,0pc>*\hbox{$\scriptstyle{#1}\quad$}\restore}
\newcommand*{\lbu}[1]{\save+<0pc,1pc>*\hbox{$\scriptstyle{#1}$}\restore} 

\allowdisplaybreaks[1]

\newcommand{\comment}[1]{\ifdraft{\textcolor{red}{#1}}{}}

\begin{document}


\title[$L$-functions for $p$-adic representations]{Non-commutative $L$-functions for $p$-adic representations over totally real fields}%
\author{\sc Malte Witte}%

\address{Malte Witte\newline Ruprecht-Karls-Universit\"at Heidelberg\newline
Mathematisches Institut\newline
Im Neuenheimer Feld 288\newline
D-69120 Heidelberg\newline GERMANY}%
\email{witte@mathi.uni-heidelberg.de}


\date{\today}%

\begin{abstract}
We prove a unicity result for the $L$-functions appearing in the non-commutative Iwasawa main conjecture over totally real fields. We then consider continuous representations $\rho$ of the absolute Galois group of a totally real field $F$ on adic rings in the sense of Fukaya and Kato. Using our unicity result, we show that there exists a unique sensible definition of a non-commutative $L$-function for any such $\rho$ that factors through the Galois group of a possibly infinite totally real extension. We also consider the case of CM-extensions and discuss the relation with the equivariant main conjecture for realisations of abstract $1$-motives of Greither and Popescu.
\end{abstract}


\maketitle

\section{Introduction}

Let $F_\infty/F$ be an  admissible  $p$-adic Lie extension (in the sense of \cite{Kakde2}) of a totally real field $F$ that is unramified over an open dense subscheme $U$ of the spectrum $X$ of the algebraic integers of $F$ and write $G\coloneqq\Gal(F_\infty/F)$ for its Galois group. We further assume that $p$ is invertible on $U$. The non-commutative main conjecture of Iwasawa theory for $F_\infty/F$ predicts the existence of a non-commutative $p$-adic $L$-function $\ncL_{F_\infty/F}(U,\Int_p(1))$ living in the first algebraic $\KTh$-group $\KTh_1(\Int_p[[G]]_S)$ of the localisation at Venjakob's canonical Ore set $S$ of the profinite group ring
\[
\Int_p[[G]]\coloneqq\varprojlim_{\substack{H\triangleleft G\\ \text{open}}}\Int_p[[G/H]].
\]
This $L$-function is supposed to satisfy the following two properties:

\begin{enumerate}
\item[(1)] It is a characteristic element for the total complex $\RDer\Sectc(U,f_!f^*\Int_p(1))$ of \'etale cohomology with proper support with values in the sheaf corresponding to the first Tate twist of the Galois module $\Int_p[[G]]^\sharp$, on which an element $\sigma$ of the absolute Galois group $\Gal_F$ acts by right multiplication with $\sigma^{-1}$. We denote this sheaf by $f_!f^*\Int_p(1)$.
\item[(2)] It interpolates the values of the complex $L$-functions $L_{X-U}(\rho,s)$ for all Artin representations $\rho$ factoring through $G$.
\end{enumerate}
We refer to Theorem~\ref{thm:Kakdes theorem} for a more precise formulation.

Under the assumption that
\begin{enumerate}
\item[(a)] $p\neq 2$,
\item[(b)] the Iwasawa $\mu$-invariant of any totally real field is zero,
\end{enumerate}
the non-commutative main conjecture is now a theorem, first proved by Ritter and Weiss \cite{RW:MainConjecture}. Almost simultaneously, Kakde \cite{Kakde2} published an alternative proof, building upon unpublished work of Kato \cite{Kato:HeisenbergType} and the seminal article \cite{Burns:MCinNCIwasawaTh+RelConj} by Burns. The book \cite{CSSV:Survey} is a comprehensive introduction to Kakde's work. The vanishing of the $\mu$-invariant is still an open conjecture. We refer to \cite{Mihailescu:MuInvariant} for a recent attempt to settle it.

It turns out that properties (1) and (2) are not sufficient to guarantee the uniqueness of $\ncL_{F_\infty/F}(U,\Int_p(1))$. It is only well-determined up to an element of a subgroup
\[
\cSK_1(\Int_p[[G]])\subset\KTh_1(\Int_p[[G]]_S).
\]
The first objective of this article is to eradicate this indeterminacy. Under the assumptions (a) and (b) we show in Theorem~\ref{thm:modification factors Z1} that if one lets $F_\infty$ vary over all admissible extensions of $F$ and requires a natural compatibility property for the elements $\ncL_{F_\infty/F}(U,\Int_p(1))$, there is indeed a unique choice of such a family.

In the course of their formulation of a very general version of the equivariant Tamagawa number conjecture, Fukaya and Kato introduced in \cite{FK:CNCIT} a certain class of coefficient rings which we call adic rings for short. This class includes among others all finite rings, the Iwasawa algebras of $p$-adic Lie groups and power series rings in a finite number of indeterminates. Our second objective concerns continuous representations $\rho$ of the absolute Galois group $\Gal_F$ over some adic $\Int_p$-algebra $\Lambda$. Assume that $\rho$ is \emph{smooth at $\infty$} in the sense that it factors through the Galois group of some (possibly infinite) totally real extension of $F$ unramified over $U$. As a consequence of Theorem~\ref{thm:modification factors Z1}, we show in Theorem~\ref{thm:general modification factors} and Corollary~\ref{cor:transformation properties for Lfunctions} that there exists a unique sensible assignment of a non-commutative $L$-function
\[
\ncL_{F_\infty/F}(U,\rho(1))\in\KTh_1(\Lambda[[G]]_S)
\]
to any such $\rho$. In the sequel \cite{Witte:zetaisos} to the present article, we will use our result to prove the existence of the $\zeta$-isomorphism for $\rho$ as predicted by Fukaya's and Kato's central conjecture \cite[Conj.~2.3.2]{FK:CNCIT}.

In fact, Corollary~\ref{cor:transformation properties for Lfunctions} applies more generally to perfect complexes $\cmplx{\sheaf{F}}$ of $\Lambda$-adic sheaves on $U$ which are smooth at $\infty$. Moreover, we consider $L$-functions $\ncL_{F_\infty/F}(W,\RDer k_*\cmplx{\sheaf{F}})$ for the total derived pushforward $\RDer k_*\cmplx{\sheaf{F}}$, where $k\colon U\mto W$ is the open immersion into another dense open subscheme $W$ of $X$. The extension $F_\infty/F$ may be ramified over $W-U$, but places over $p$ remain excluded from $W$. We also prove the existence of a dual non-commutative $L$-function $\ncL^{\Mdual}_{F_\infty/F}(W,k_!\cmplx{\sheaf{F}})$ such that $\ncL^{\Mdual}_{F_\infty/F}(W,k_!\cmplx{\sheaf{F}})^{-1}$ is a characteristic element for the complex $\RDer\Sect(W,k_!f_!f^*\cmplx{\sheaf{F}})$ and satisfies the appropriate interpolation property. If $\rho$ is a continuous representation as above and $\check{\rho}$ the dual representation over the opposite ring $\Lambda^\op$, then $\ncL^{\Mdual}_{F_\infty/F}(U,\rho)$ is the image of $\ncL_{F_\infty/F}(U,\check{\rho}(1))$ under the canonical isomorphism
\[
 \KTh_1(\Lambda^\op[[G]]_S)\isomorph\KTh_1(\Lambda[[G]]_S).
\]

As we explain in Corollary~\ref{cor:transformation properties of L-functions CM case}, all of this can be easily extended to the case that $F_\infty$ is a CM field. In this form, our main result also includes non-commutative generalisations of the type of main conjectures that are treated in \cite{GreitherPopescu:EquivariantMC} by Greither and Popescu.

We give a short overview on the content of the article. A central input to the proof of Theorem~\ref{thm:modification factors Z1} is Section~\ref{sec:SK}, in which we show that $\cSK_1(\Int_p[[\Gal(F_\infty/F)]])$ vanishes for sufficiently large extensions $F_\infty/F$. The results of this section apply not only to admissible extensions and might be useful in other contexts as well. In Section~\ref{sec:Framework} we recall the essence of the $\KTh$-theoretic machinery behind the formulation of the main conjecture. We also add some new material, which is needed in the later sections. In Section~\ref{sec:S-torsion exchange property} we prove a technical result that will later be used to relate the cohomology of $f_!f^*\cmplx{\sheaf{F}}$ over $U$ to the cohomology of $\cmplx{\sheaf{F}}$ over the integral closure of $U$ in the cyclotomic extension of $F$. Section~\ref{sec:Duality on the level of Kgroups} contains the construction of the isomorphism $\KTh_1(\Lambda^\op[[G]]_S)\isomorph\KTh_1(\Lambda[[G]]_S)$ on the level of Waldhausen categories. First, we explain in full generality how to identify the $\KTh$-groups of a biWaldhausen category with those of its opposite category and show that this identification is compatible with localisation sequences. Then, we specialise to the case of rings and explain how to identify the opposite category of the category of strictly perfect complexes over $\Lambda[[G]]$ with the category of strictly perfect complexes over $\Lambda^\op[[G]]$. Section~\ref{sec:algebraic Lfunctions} contains an investigation of the base change properties of certain splittings of the boundary map
\[
\bh\colon\KTh_1(\Int_p[[G]]_S)\mto\KTh_0(\Int_p[[G]],S),
\]
extending results from \cite{Burns:AlgebraicLfunctions} and \cite{Witte:Splitting}. With the help of these splittings we are able produce characteristic elements with good functorial properties, which we call non-commutative algebraic $L$-functions, following Burns. Section~\ref{sec:perfect complexes of adic sheaves} recalls the definition of a Waldhausen category modeling the derived category of perfect complexes of $\Lambda$-adic sheaves and the explanation of the property of being smooth at $\infty$. In Section~\ref{sec:duality for adic sheaves} we consider local and global duality theorems for smooth $\Lambda$-adic sheaves. In Section~\ref{sec:AdmissibleExt} we recall the notion of admissible extensions and the definition of the complexes $f_!f^*\cmplx{\sheaf{F}}$ induced by the procovering $f\colon U_{F_\infty}\mto U$. We then show in Section~\ref{sec:S-torsion} that our hypothesis (b) on the vanishing of the $\mu$-invariant implies that for any perfect complex of $\Lambda$-adic sheaves $\cmplx{\sheaf{F}}$ smooth at $\infty$ and any admissible extension $F_\infty/F$, the complexes
$
\RDer\Sectc(W,\RDer k_*\cmplx{\sheaf{F}}(1))
$
have $S$-torsion cohomology. We also prove a unconditional local variant thereof. This local variant permits us to introduce the notion of non-commutative Euler factors by producing canonical characteristic elements for the complexes
$
\RDer\Sect(x,i^*\RDer k_*\cmplx{\sheaf{F}}(1))
$
for any closed point $i\colon x\mto W$ of $W$. Comparing Euler factors with the non-commutative algebraic $L$-functions of these complexes, we obtain certain elements in $\KTh_1(\Lambda[[G]])$, which we call local modification factors. In the same way, we also introduce the notion of dual non-commutative Euler factors by producing canonical characteristic elements for the complexes
$
\RDer\Sect(x,i^!k_!\cmplx{\sheaf{F}})
$
and the corresponding dual local modification factors. The investigation of the Euler factors and local modification factors is carried out in Section \ref{sec:euler factors} and Section~\ref{sec:cyclotomic euler factors}, first in general, then in the special case of the cyclotomic extension.  This is followed by a short reminder on $L$-functions of Artin representations in Section~\ref{sec:Artin reps}.

Section~\ref{sec:main results} contains the main results of this article. We use Kakde's non-commutative $L$-functions and the non-commutative algebraic $L$-function of the complex $\RDer\Sectc(U,f_!f^*\Int_p(1))$ to define global modification factors. Changing the open dense subscheme $U$ is reflected by adding or removing local modification factors. This compatibility allows us to pass to field extensions with arbitrarily large ramification. We can then use the results of Section~\ref{sec:SK} to prove the uniqueness of the family of modification factors for all pairs $(U,F_\infty)$ with $F_\infty/F$ admissible and unramified over $U$. The corresponding non-commutative $L$-functions are the product of the global modification factors and the non-commutative algebraic $L$-functions. We then extend in Theorem~\ref{thm:general modification factors} the definition of global modification factors to $\Lambda$-adic sheaves smooth at $\infty$ by requiring a compatibility under twists with certain bimodules. In the same way, we construct global dual modification factors in Theorem~\ref{thm:general dual modification factors}. In Theorem~\ref{thm:change of base field} we show that the global modification factors are also compatible under changes of the base field $F$. The non-commutative $L$-functions for the complexes $\cmplx{\sheaf{F}}$ are then defined as the product of the algebraic $L$-function and the modification factors. Corollary~\ref{cor:transformation properties for Lfunctions} subsumes the transformation properties of the non-commutative $L$-functions.

In Section~\ref{sec:CM-extensions} we extend our results to the case of CM-extensions of $F$. In the following Section~\ref{sec:cohomology} we consider the cohomology of the complexes featuring in our main conjecture. In the final Section~\ref{sec:GreitherPopescu} we establish the link between our conjecture and the formulation of Greither and Popescu. In an appendix we prove a technical result on localisations in polynomial rings with not necessarily commutative coefficients.

The findings of this article were inspired by the author's analogous results \cite{Witte:UnitLfunctions} on geometric main conjectures in the $\ell=p$ case. Parts of it were developed during his stay at the University of Paderborn in the academic year 2014. He thanks the mathematical faculty and especially Torsten Wedhorn for the hospitality. This work is part of a larger project of the research group \emph{Symmetry, Geometry, and Arithmetics} funded by the DFG.

\subsection*{Notational conventions}

All rings in this article will be associative with identity; a module over a ring will always refer to a left unitary module. If $R$ is a ring, $R^\op$ will denote the opposite ring and $R^\times$ its group of units. We will sometimes write $f\lcirclearrowright M$ for an endomorphism $f$ of an object $M$. The symbols $\N$, $\Int$, $\C$ have their usual meanings. For a prime number $p$, $\Int_p$ denotes the ring of $p$-adic integers and $\Rat_p$ its fraction field. We write $\coloneqq$ to denote the definition of a symbol, reserving the symbol $=$ for expressing an identity. Isomorphisms are denoted by $\isomorph$, weak equivalences and quasi-isomorphisms by $\sim$. Cofibrations and quotient maps in Waldhausen categories are denoted by $\cto$ and $\qto$. Graded objects are denoted by $P^{\bullet}$ or $P_{\bullet}$, with $P^n$ and $P_n$ referring to the component in degree $n$, respectively.

\section{On the first special \texorpdfstring{$\KTh$}{K}-group of a profinite group algebra}\label{sec:SK}

Let $p$ be a fixed prime number. In this section only, $p$ is allowed to be even. For any profinite group $G$, we write $\NS(G)$ for its lattice of open normal subgroups and $G_r\subset G$ for subset of $p$-regular elements, i.\,e.\ the union of all $q$-Sylow-subgroups for all primes $q\neq p$. Note that
\[
 G_r=\varprojlim_{U\in\NS(G)} (G/U)_r
\]
is a closed subset of $G$. The group $G$ acts continuously on $G_r$ by conjugation. For any profinite $G$-set $S$ we write $\Int_p[[S]]$ for the compact $G$-module which is freely generated by $S$ as compact $\Int_p$-module.

We want to analyse the completed first special $K$-group
\[
\cSK_1(\Int_p[[G]])\coloneqq\varprojlim_{U\in\NS(G)}\SK_1(\Int_p[G/U])
\]
of the profinite group algebra
\[
 \Int_p[[G]]\coloneqq\varprojlim_{U\in\NS(G)}\Int_p[G/U].
\]
Note that $\cSK_1(\Int_p[[G]])$ is a subgroup of the completed first $K$-group
\[
\widehat{\KTh}_1(\Int_p[[G]])\coloneqq\varprojlim_{U\in\NS(G)}\KTh_1(\Int_p[G/U]).
\]
If $G$ has an open pro-$p$-subgroup which is topologically finitely generated, then
\[
\widehat{\KTh}_1(\Int_p[[G]])=\KTh_1(\Int_p[[G]])
\]
by \cite[Prop.~1.5.3]{FK:CNCIT}. In the case that $G$ is a pro-$p$ $p$-adic Lie group a thorough analysis of $\cSK_1(\Int_p[[G]])$ has been carried out in \cite{SchneiderVenjakob:SK1}. Note in particular that there are examples of torsionfree $p$-adic Lie groups with non-trivial first special $K$-group. Some of the results of \emph{loc.\,cit.} can certainly be extended to the case that $G$ admits elements of order prime to $p$. We will not pursue this further. Instead, we limit ourselves to the following results relevant to our application.

Recall from \cite[Thm.~10.12]{Oliver:WhiteheadGroups} that there is a canonical surjective homomorphism
\[
\HF_2(G,\Int_p[[G_r]])\mto\cSK_1(\Int_p[[G]]).
\]
where
\[
\HF_2(G,\Int_p[[G_r]])\coloneqq\varprojlim_{U\in\NS(G)} \HF_2(G/U,\Int_p[(G/U)_r])
\]
denotes the second continuous homology group of $\Int_p[[G_r]]$. We write $X(G_r)\coloneqq\operatorname{Map}(G_r,\Rat_p/\Int_p)$ for the Pontryagin dual of $\Int_p[[G_r]]$, such that the Pontryagin dual of $\HF_2(G,\Int_p[[G_r]])$ is $\HF^2(G,X(G_r))$.

\begin{lem}\label{lem:finiteness of SK1}
Let $G=H\rtimes\Gamma$ be a semi-direct product of a finite normal subgroup $H\subset G$ and $\Gamma\isomorph\Int_p$. Then $\HF^2(G,X(G_r))$ and $\cSK_1(\Int_p[[G]])$ are finite.
\end{lem}
\begin{proof}
Note that $X(G_r)=X(H_r)$ is of finite corank over $\Int_p$. The Hochschild-Serre spectral sequence induces an exact sequence
\[
 0\mto \HF^1(\Gamma,H^1(H,X(H_r)))\mto \HF^2(G,X(H_r))\mto \HF^0(\Gamma,\HF^2(H,X(H_r)))\mto 0
\]
where both $\HF^1(H,X(H_r))$ and $\HF^2(H,X(H_r))$ are finite $p$-groups. The lemma is an immediate consequence.
\end{proof}

We are interested in the following number theoretic situation. Assume that $K$ is an algebraic number field and $K_\infty$ is a $\Int_p$-extension of $K$. In particular $K_\infty/K$ is unramified in the places (including the archimedean places) of $K$ that do not lie over $p$. Let $L_\infty$ be a finite extension of $K_\infty$ which is Galois over $K$. Write
\begin{align*}
G&\coloneqq\Gal(L_\infty/K),\\
H&\coloneqq\Gal(L_\infty/K_\infty),\\
\Gamma&\coloneqq\Gal(K_\infty/K)
\end{align*}
for the corresponding Galois groups. We fix a splitting $\Gamma\to G$ such that we may write $G$ as the semi-direct product of $H$ and $\Gamma$ and let $L$ be the fixed field of a $p$-Sylow subgroup of $G$ containing $\Gamma$. Write $L^{(p)}$ for the maximal Galois $p$-extension of $L$ inside a fixed algebraic closure $\algc{K}$ of $K$. Note that $L^{(p)}=L_\infty^{(p)}$ is the subfield of $\algc{K}$ fixed by the closed subgroup $\Gal_{L^{(p)}}$ generated by all $q$-Sylow subgroups of the absolute Galois group $\Gal_L$ for all primes $q\neq p$. Hence, $\Gal_{L^{(p)}}\subset\Gal_{L_\infty}$ is a characteristic subgroup and therefore, $L^{(p)}/K$ is a Galois extension. The following is an adaption of the proof of \cite[Prop. 2.3.7]{FK:CNCIT}.

\begin{prop}\label{prop:Vanishing of SK1}
Set $\mathcal{G}\coloneqq\Gal(L^{(p)}/K)$. Then $\HF^2(\mathcal{G},X(\mathcal{G}_r))=\cSK_1(\Int_p[[\mathcal{G}]])=0$.
\end{prop}
\begin{proof}
Note that the projection $\mathcal{G}\to G$ induces a canonical isomorphism $X(\mathcal{G}_r)=X(H_r)$ and that $X(H_r)$ is of finite corank over $\Int_p$. We have
\[
 \HF^i(\Gal(L^{(p)}/L),X(H_r))=\HF^i(\Gal_L,X(H_r))
\]
for all $i$ according to \cite[Cor. 10.4.8]{NSW:CohomNumFields} applied to the class of $p$-groups and the set of all places of $L$. Moreover, $\HF^2(\Gal_L,X(H_r))=0$ as a consequence of the fact that $\HF^2(\Gal_F,\Rat_p/\Int_p)=0$ for any number field $F$ \cite[Prop. 2.3.7, Claim]{FK:CNCIT}.

Since $[L:K]$ is prime to $p$, the restriction map
\[
\HF^2(\mathcal{G},X(H_r))\mto\HF^2(\Gal(L^{(p)}/L),X(H_r))
\]
is split injective. In particular, $\HF^2(\mathcal{G},X(H_r))=0$ as claimed.
\end{proof}

Note that if $L_\infty/K$ is unramified in a real place of $K$ and $p\neq 2$, then $L^{(p)}/K$ is unramified in this real place as well. For the sake of completeness we also deal with the case $p=2$ and consider for a set of real places $\Sigma$ of $K$ such that $L_\infty/K$ is unramified over $\Sigma$ the maximal subfield $L^{(2)}_{\Sigma^c}$ of $L^{(2)}$ which is unramified over $\Sigma$. Note that $L^{(2)}_{\Sigma^c}/K$ is still Galois over $K$.

\begin{prop}\label{prop:Vanishing of SK1 for 2}
Set $\mathcal{G}\coloneqq\Gal(L^{(2)}_{\Sigma^c}/K)$. Then $\HF^2(\mathcal{G},X(\mathcal{G}_r))=\cSK_1(\Int_2[[\mathcal{G}]])=0$.
\end{prop}
\begin{proof}
Let $L'$ be  the subfield fixed by the intersection of the centre of $G$ with $\Gamma$ and let $Y\coloneqq\operatorname{Map}(\Gal(L'/K),X(H_r))$ be the induced module. We obtain a canonical surjection $Y\mto X(H_r)$ with kernel $Z$. For any discrete $\mathcal{G}$-module $A$ we have
\[
 \HF^3(\mathcal{G},A)=\bigoplus_{v\in \Sigma^c_{\Real}}\HF^3(\Gal_{K_v},A)
\]
where $v$ runs through set $\Sigma^c_{\Real}$ of real places of $K$ not in $\Sigma$ and $\Gal_{K_v}=\Int/2\Int$ denotes the Galois group of the corresponding local field $K_v=\Real$ \cite[Prop. 10.6.5]{NSW:CohomNumFields}. By the proof of the $(p=2)$-case in \cite[Prop. 2.3.7, Claim]{FK:CNCIT} we have $\HF^2(\Gal_{K_v},X(H_r))=0$ such that
\[
 \HF^3(\mathcal{G},Z)\mto\HF^3(\mathcal{G},Y)
\]
is injective and hence,
\[
 \HF^2(\Gal(L_{\Sigma^c}^ {(2)}/L'),X(H_r))\isomorph\HF^2(\mathcal{G},Y)\mto \HF^2(\mathcal{G},X(H_r))
\]
is a surjection. Moreover, $\Gal_{L'}$ acts trivially on $X(H_r)$ such that it suffices to show that
\[
 \HF^2(\Gal(L^{(2)}_{\Sigma^c}/L'),\Rat_2/\Int_2)=0.
\]

By the proof of \cite[Thm. 10.6.1]{NSW:CohomNumFields} we obtain an exact sequence
\begin{multline*}
 0\mto\HF^1(\Gal(L^{(2)}_{\Sigma^c}/L'))\mto\HF^1(\Gal(L^{(2)}/L'))\mto\\
\bigoplus_{v\in \Sigma^c_{\Real}(L')}\HF^1(\Gal_{L'_v})\mto\HF^2(\Gal(L^{(2)}_{\Sigma^c}/L'))\mto\HF^2(\Gal(L^{(2)}/L'))
\end{multline*}
where we have omitted the coefficients $\Rat_2/\Int_2$ and $\Sigma^c_{\Real}(L')$ denotes the real places of $L'$ lying over $\Sigma^c_\Real$. But
\[
 \HF^2(\Gal(L^{(2)}/L'))=\HF^2(\Gal_{L'})=0
\]
by \cite[Cor. 10.4.8]{NSW:CohomNumFields} and \cite[Satz 1.(ii)]{Schn:Galoiskohomologiegruppen}. Moreover, $L'$ is dense in the product of its real local fields such that for each real place $v$ of $L'$, we find an element $a$ in $L'$ which is negative with respect to $v$ and positive with respect to all other real places. The element of $\HF^1(\Gal(L^{(2)}/L'))$ corresponding via Kummer theory to a square root of $a$ maps to the non-trivial element of $\HF^1(\Gal_{L'_v})=\Int/2\Int$ and to the trivial element for all other real places. This shows that
\[
\HF^1(\Gal(L^{(2)}_{\Sigma^c}/L'))\mto\bigoplus_{v\in \Sigma^c_{\Real}(L')}\HF^1(\Gal_{L'_v})
\]
must be surjective.
\end{proof}

\begin{cor}\label{cor:vanishing extension}
Let $K_\infty/K$ be a $\Int_p$-extension of a number field $K$ and $L_\infty/K_\infty$ be a finite extension such that $L_\infty/K$ is Galois with Galois group $G$. Assume further that $L_\infty/K$ is unramified in a (possibly empty) set of real places $\Sigma$ of $K$. Then there exists a finite extension $L'_\infty/L_\infty$ such that
\begin{itemize}
 \item[(i)] $[L'_\infty:L_\infty]$ is a power of $p$,
 \item[(ii)] $L'_\infty/K$ is Galois with Galois group $G'$,
 \item[(iii)] $L'_\infty/K$ is unramified over $\Sigma$,
 \item[(iv)] The canonical homomorphism $\cSK_1(\Int_p[[G']])\mto\cSK_1(\Int_p[[G]])$ is the zero map.
\end{itemize}
In particular, $L'_\infty$ may be chosen to be totally real if $L_\infty$ is totally real.
\end{cor}
\begin{proof}
With $L$ as above, set $\mathcal{G}\coloneqq\Gal(L^{(p)}/K)$ if $p\neq 2$ and $\mathcal{G}\coloneqq\Gal(L^{(2)}_{\Sigma^c}/K)$ if $p=2$ and set $\mathcal{H}\coloneqq\ker \mathcal{G}\to\Gal(K_\infty/K)$. According to Lemma~\ref{lem:finiteness of SK1}, $\cSK_1(\Int_p[[G]])$ is finite and so, the image of
\[
 \cSK_1(\Int_p[[\mathcal{G}]])=\varprojlim_{U\in\NS(\mathcal{G})}\cSK_1(\Int_p[[\mathcal{G}/U\cap\mathcal{H}]])\mto\cSK_1(\Int_p[[G]])
\]
will be equal to the image of $\cSK_1(\Int_p[[\mathcal{G}/U_0\cap\mathcal{H}]])$ for some $U_0\in\NS(\mathcal{G})$. We let $L'_\infty$ be the fixed field of $U_0\cap\mathcal{H}$. Then $L'_\infty$ clearly satisfies $(i)$, $(ii)$, and  $(iii)$. Since $\cSK_1(\Int_p[[\mathcal{G}]])=0$ by Prop.~\ref{prop:Vanishing of SK1} and Prop.~\ref{prop:Vanishing of SK1 for 2}, it also satisfies $(iv)$.
\end{proof}

\begin{rem}
The extension $L'_\infty/K$ will be unramified outside a finite set of primes, but we cannot prescribe the ramification locus. However, assume $L_\infty/K$ is unramified outside the set $S$ of places of $K$ and that the Leopoldt conjecture holds for every finite extension $F$ of $K$ inside the maximal $p$-extension $L^{(p)}_S$ which is unramified outside $S$, i.\,e.\ that
\[
 \HF^2(\Gal(L^{(p)}_S/F),\Rat_p/\Int_p)=0.
\]
Then the same method of proof shows that we can additionally chose $L'_\infty$ to lie in $L^{(p)}_S$.
\end{rem}

\section{\texorpdfstring{$\KTh$}{K}-theory of adic rings}\label{sec:Framework}

For the convenience of the reader, we repeat the essentials of the $\KTh$-theoretic framework introduced in \cite{Witte:MCVarFF}.

The formulation of the non-commutative Iwasawa main conjecture involves certain $\KTh$-groups. We are mainly interested in the first $\KTh$-group of a certain class of rings introduced by Fukaya and Kato in \cite{FK:CNCIT}. It consists of those rings $\Lambda$ such that for each $n\geq 1$ the $n$-th power of the Jacobson radical $\Jac(\Lambda)^n$ is of finite index in $\Lambda$ and
$$
\Lambda=\varprojlim_{n\geq 1}\Lambda/\Jac(\Lambda)^n.
$$
By definition, $\Lambda$ carries a natural profinite topology. We will write $\openideals_{\Lambda}$ for the set of open two-sided ideals of $\Lambda$, partially ordered by inclusion. In extension of the definition for commutative rings \cite[Ch.~0, Def.~7.1.9]{EGA1}, these rings should be called \emph{compact adic rings}. We will call these rings \emph{adic rings} for short, as in \cite{Witte:MCVarFF}. We do not intend to insinuate any relation to Huber's more recent concept of adic spaces with this denomination.

Classically, the first $\KTh$-group of $\Lambda$ may be described as the quotient of the group
\[
\Gl_{\infty}(\Lambda)\coloneqq \varinjlim_{d}\Gl_d(\Lambda)
\]
by its commutator subgroup, but for the formulation of the main conjecture, it is more convenient to follow the constructions of higher $\KTh$-theory. Among the many roads to higher $\KTh$-theory, Waldhausen's $S$-construction \cite{Wal:AlgKTheo} turns out to be particularly well-suited for our purposes.

To construct the $\KTh$-groups of $\Lambda$, one can simply apply the $S$-construction to the category of finitely generated, projective modules over $\Lambda$, but the true beauty of Waldhausen's construction is that we can choose among a multitude of different Waldhausen categories that all give rise to the same $\KTh$-groups. Below, we will study a number of different Waldhausen categories whose $\KTh$-theory agrees with that of $\Lambda$.

We recall that for any ring $R$, a complex $\cmplx{M}$ of $R$-modules is called \emph{$DG$-flat} if every module $M^n$ is flat and for every acyclic complex $\cmplx{N}$ of right $R$-modules, the total complex $\cmplx{(N\tensor_R M)}$ is acyclic. In particular, any bounded above complex of flat $R$-modules is $DG$-flat. The notion of $DG$-flatness can be used to define derived tensor products without this boundedness condition. Unbounded complexes will turn up naturally in our constructions. As usual, the complex $\cmplx{M}$ is called \emph{strictly perfect} if $M^n$ is finitely generated and projective for all $n$ and $M^n=0$ for almost all $n$. A complex of $R$-modules is a \emph{perfect} complex if it is quasi-isomorphic to a strictly perfect complex.

\begin{defn}\label{defn:P and SP}
For any ring $R$, we write $\cat{SP}(R)$ for the Waldhausen category of strictly perfect complexes and $\cat{P}(R)$ for the Waldhausen category of perfect complexes of left $R$-modules. In both categories, the weak equivalences are given by quasi-isomorphisms. The cofibrations in $\cat{P}(R)$ are all degree-wise injections, the cofibrations in $\cat{SP}(R)$ are the degree-wise split injections.
\end{defn}

It is a standard fact resulting from the Waldhausen approximation theorem \cite[1.9.1]{ThTr:HAKTS+DC} that the inclusion functor $\cat{SP}(R)\mto\cat{P}(R)$ induces isomorphisms
\[
\KTh_n(\cat{SP}(R))\isomorph\KTh_n(\cat{P}(R))
\]
between the Waldhausen $\KTh$-groups of these categories. Moreover, they agree with the Quillen $\KTh$-groups $\KTh_n(R)$ of $R$ by the Gillet-Waldhausen theorem \cite[Thm. 1.11.2, 1.11.7]{ThTr:HAKTS+DC}.

If $S$ is another ring and $\cmplx{M}$ is a complex of $S$-$R$-bimodules which is strictly perfect as complex of $S$-modules, then the tensor product with $\cmplx{M}$ is a Waldhausen exact functor from $\cat{SP}(R)$ to $\cat{SP}(S)$ and hence, it induces homomorphisms $\KTh_n(R)\mto\KTh_n(S)$. Note, however, that the tensor product with $\cmplx{M}$ does not give a Waldhausen exact functor from $\cat{P}(R)$ to $\cat{P}(S)$, as it does not preserve weak equivalences nor cofibrations. In the context of homological algebra, this problem can be solved by passing to the derived category, but there is no general recipe how to construct the $\KTh$-groups of $R$ on the basis of the derived category alone. As a consequence, in order to view certain homomorphisms between $\KTh$-groups as being induced from a Waldhausen exact functor, one has make a suitable choice of the underlying Waldhausen categories.

If $\Lambda$ is an adic ring, we will mainly work with the following variant taken from \cite{Witte:MCVarFF}. Its main advantage is that it works well with our later definition of adic sheaves in Section~\ref{sec:perfect complexes of adic sheaves} and that it allows a direct construction of most of the relevant Waldhausen exact functors.

\begin{defn}\label{defn:PDG(Lambda)}
Let $\Lambda$ be an adic ring. We denote by
$\cat{PDG}^{\cont}(\Lambda)$ the following Waldhausen category.
The objects of $\cat{PDG}^{\cont}(\Lambda)$ are inverse system
$(\cmplx{P}_I)_{I\in \openideals_{\Lambda}}$ satisfying the
following conditions:
\begin{enumerate}
\item for each $I\in\openideals_{\Lambda}$, $\cmplx{P}_I$ is a
$DG$-flat perfect complex of $\Lambda/I$-modules,

\item for each $I\subset J\in\openideals_{\Lambda}$, the
transition morphism of the system
$$
\varphi_{IJ}:\cmplx{P}_I\mto \cmplx{P}_J
$$
induces an isomorphism of complexes
$$
\Lambda/J\tensor_{\Lambda/I}\cmplx{P}_I\isomorph \cmplx{P}_J.
$$
\end{enumerate}
A morphism of inverse systems $(f_I\colon \cmplx{P}_I\mto
\cmplx{Q}_I)_{I\in\openideals_{\Lambda}}$ in
$\cat{PDG}^{\cont}(\Lambda)$ is a weak equivalence if every $f_I$
is a quasi-isomorphism. It is a cofibration if every $f_I$ is
injective and the system $(\coker f_I)$ is in $\cat{PDG}^{\cont}(\Lambda)$.
\end{defn}


\begin{defn}\label{defn:change of ring functor}
Let $\Lambda'$ be another adic ring and $\cmplx{M}$ a complex of $\Lambda'$-$\Lambda$-bimodules which is strictly perfect as complex of $\Lambda'$-modules. We define $\ringtransf_{\cmplx{M}}$ to be the following Waldhausen exact functor
$$
\ringtransf_{\cmplx{M}}\colon
\cat{PDG}^{\cont}(\Lambda)\mto\cat{PDG}^{\cont}(\Lambda'),\qquad \cmplx{P}\mto(\varprojlim_{J\in\openideals_{\Lambda}}
\Lambda'/I\tensor_{\Lambda'}\cmplx{(M\tensor_{\Lambda}P_{J})})_{I\in\openideals_{\Lambda'}}.
$$
\end{defn}

If $\cmplx{P}$ is a strictly perfect complex of $\Lambda$-modules, we may identify it with the system
\[
(\Lambda/I\tensor_{\Lambda}\cmplx{P})_{I\in\openideals_{\Lambda}}
\]
in $\cat{PDG}^{\cont}(\Lambda)$. By \cite[Prop.~3.7]{Witte:MCVarFF}, the corresponding Waldhausen exact functor
\[
\cat{SP}(\Lambda)\mto\cat{PDG}^{\cont}(\Lambda)
\]
induces isomorphisms
\[
\KTh_n(\cat{SP}(\Lambda))\isomorph\KTh_n(\cat{PDG}^{\cont}(\Lambda))
\]
between the $\KTh$-groups of the Waldhausen categories. Hence, $\KTh_n(\cat{PDG}^{\cont}(\Lambda))$ also coincide with the Quillen $\KTh$-groups of the adic ring $\Lambda$ and the homomorphism
\[
\ringtransf_{\cmplx{M}}\colon \KTh_{n}(\Lambda)\mto \KTh_{n}(\Lambda')
\]
induced by the Waldhausen exact functor $\ringtransf_{\cmplx{M}}$ coincides with the homomorphism induced by
\[
\cat{SP}(\Lambda)\mto\cat{SP}(\Lambda'),\qquad \cmplx{P}\mapsto \cmplx{(M\tensor_{\Lambda}P)}.
\]
The essential point in this observation is that $\openideals_{\Lambda}$ is a countable set and that all the transition maps $\varphi_{IJ}$ are surjective such that passing to the projective limit
\[
 \varprojlim_{I\in\openideals_{\Lambda}}\cmplx{P}_I
\]
is a Waldhausen exact functor from $\cat{PDG}^{\cont}(\Lambda)$ to the Waldhausen category $\cat{P}(\Lambda)$ of perfect complexes of $\Lambda$-modules. We write
\[
 \HF^i((\cmplx{P}_I)_{I\in\openideals_{\Lambda}})\coloneqq\HF^i(\varprojlim_{I\in\openideals_{\Lambda}}\cmplx{P}_I)
\]
for its cohomology groups and note that
\[
 \HF^i((\cmplx{P}_I)_{I\in\openideals_{\Lambda}})=\varprojlim_{I\in\openideals_{\Lambda}}\HF^i(\cmplx{P}_I)
\]
\cite[Prop. 5.2.3]{Witte:PhD}.

We will also need to consider localisations of certain adic rings: Note that for any adic $\Int_p$-algebra $\Lambda$ and any profinite group $G$ such that $G$ has an open pro-$p$-subgroup which is topologically finitely generated, the profinite group algebra $\Lambda[[G]]$ is again an adic ring \cite[Prop. 3.2]{Witte:MCVarFF}. Assume further that $G=H\rtimes\Gamma$ is the semi-direct product of a closed normal subgroup $H$ which is itself topologically finitely generated and a subgroup $\Gamma$ which is isomorphic to $\Int_p$. We set
\begin{multline}\label{eqn:Venjakobs Ore set}
S\coloneqq S_{\Lambda[[G]]}\coloneqq \\
\set{f\in \Lambda[[G]]\given \text{$\Lambda[[G]]\xrightarrow{\cdot f}\Lambda[[G]]$ is perfect as complex of $\Lambda[[H]]$-modules}}
\end{multline}
and call it \emph{Venjakob's canonical Ore set.} We may generalise the results of \cite[\S 2]{CFKSV} as follows.

\begin{lem}\label{lem:S-torsion complexes of length two}
Let
\[
 \cmplx{P}\colon P^{-1}\xrightarrow{\del} P^0
\]
be a complex of length $2$ in $\cat{SP}(\Lambda[[G]])$. Then the following are equivalent:
\begin{enumerate}
\item\label{enum:S-torsion complexes of length two:perfectness} $\cmplx{P}$ is perfect as a complex of $\Lambda[[H]]$-modules.
\item\label{enum:S-torsion complexes of length two:isomorphy and f.g.} $P^{-1}$ and $P^0$ are isomorphic as $\Lambda[[G]]$-modules and $\HF^0(\cmplx{P})$ is finitely generated as $\Lambda[[H]]$-module.
\item\label{enum:S-torsion complexes of length two:injectivity and f.g.} $\HF^{-1}(\cmplx{P})=0$ and $\HF^0(\cmplx{P})$ is finitely generated as $\Lambda[[H]]$-module.
\item\label{enum:S-torsion complexes of length two:injectivity and f.g.p.} $\HF^{-1}(\cmplx{P})=0$ and $\HF^0(\cmplx{P})$ is finitely generated and projective as $\Lambda[[H]]$-module.
\end{enumerate}
\end{lem}
\begin{proof}
Clearly, \eqref{enum:S-torsion complexes of length two:injectivity and f.g.p.} implies \eqref{enum:S-torsion complexes of length two:perfectness}.

We prove that \eqref{enum:S-torsion complexes of length two:perfectness} implies \eqref{enum:S-torsion complexes of length two:isomorphy and f.g.}. Assume that $\cmplx{P}$ is perfect as complex of $\Lambda[[H]]$-modules. Then the class of $\cmplx{P}$ is trivial in $\KTh_0(\Lambda[[G]])$ by \cite[Cor.~3.3]{Witte:Splitting}. As $\Lambda[[G]]$ is compact and semi-local, $\KTh_0(\Lambda[[G]])$ is the free abelian group over the isomorphism classes of indecomposable, projective $\Lambda[[G]]$-submodules of $\Lambda[[G]]$. Hence, $P^{-1}$ and $P^0$ must be isomorphic. Moreover, as $\cmplx{P}$ is quasi-isomorphic to a strictly perfect complex of $\Lambda[[H]]$-modules, the highest non-vanishing cohomology group of $\cmplx{P}$ is a finitely presented $\Lambda[[H]]$-module.

We prove that \eqref{enum:S-torsion complexes of length two:isomorphy and f.g.} implies \eqref{enum:S-torsion complexes of length two:injectivity and f.g.}. It is sufficient to show that
\[
\HF^{-1}(\Lambda/I[[G/U]]\tensor_{\Lambda[[G]]}\cmplx{P})=0
\]
for every open two-sided ideal $I$ of $\Lambda$ and every open subgroup $U$ of $H$ that is normal in $G$. Hence, we may assume that $\Lambda$ and $H$ are finite. Then $\del$ is a homomorphism of the torsion $\Int_p[[\Gamma]]$-modules $P^{-1}$ and $P^0$. As the two modules are isomorphic over $\Int_p[[\Gamma]]$ and $\coker \del$ is finite, $\del$ must be a pseudo-isomorphism. Hence, $\ker \del$ ist finite, as well. But $P^{-1}$ is finitely generated and projective as $\Lambda[[\Gamma]]$-module and therefore, it has no finite $\Lambda[[\Gamma]]$-submodules. We conclude that $\del$ is injective.

We prove that \eqref{enum:S-torsion complexes of length two:injectivity and f.g.} implies \eqref{enum:S-torsion complexes of length two:injectivity and f.g.p.}. Note that
\[
\begin{aligned}
\Lambda/I[[H/U]]\tensor_{\Lambda[[H]]}\HF^0(\cmplx{P})&\isomorph\Lambda/I[[G/U]]\tensor_{\Lambda[[G]]}\HF^0(\cmplx{P})\\
\isomorph\HF^0(\Lambda/I[[G/U]]\tensor_{\Lambda[[G]]}\cmplx{P})&\isomorph\HF^0(\Lambda/I[[H/U]]\tensor_{\Lambda[[H]]}\cmplx{P})
\end{aligned}
\]
for any $I\in\openideals_{\Lambda}$ and any open subgroup $U\subset H$ which is normal in $G$. We conclude that $\HF^0(\cmplx{P})$ is finitely generated and projective as $\Lambda[[H]]$-module if and only if $\HF^0(\Lambda/I[[H/U]]\tensor_{\Lambda[[H]]}\cmplx{P})$ is finitely generated and projective as $\Lambda/I[[H/U]]$-module for every $I$ and $U$. Hence, one may reduce to the case that $\Lambda$ and $H$ are finite. By replacing $G$ by an appropriate open subgroup of $G$ containing $H$, we may assume that $\Gamma$ is central in $G$, such that we may identify $\Lambda[[G]]$ with the power series ring $\Lambda[[H]][[t]]$ over $\Lambda[[H]]$ in one indeterminate $t$. For any finite  right $\Lambda[[H]]$-module $N$, the $\Int_p[[t]]$-module $N\tensor_{\Lambda[[H]]}P^{-1}$ cannot contain non-trivial finite $\Int_p[[t]]$-submodules. Moreover, $P^{-1}$ and $P^0$ are flat $\Lambda[[H]]$-modules such that $\cmplx{P}$ is a flat resolution of $\HF^0(\cmplx{P})$ as $\Lambda[[H]]$-module. Hence, we have
\[
\Tor_i^{\Lambda[[H]]}(N,\HF^0(\cmplx{P}))=0
\]
for $i>1$ and
\[
 \Tor_1^{\Lambda[[H]]}(N,\HF^0(\cmplx{P}))\subset N\tensor_{\Lambda[[H]]}P^0
\]
is a finite $\Int_p[[t]]$-submodule. Therefore,
\[
 \Tor_1^{\Lambda[[H]]}(N,\HF^0(\cmplx{P}))=0
\]
and $\HF^0(\cmplx{P})$ is finite and projective.
\end{proof}

\begin{lem}\label{lem:ore set in the finite case}
If $\Lambda$ and $H$ are finite and $\gamma\in \Gamma$ is a topological generator of $\Gamma$, then
\[
T\coloneqq\set{(\gamma-1)^n\given n\in\N}
\]
is a left and right denominator set in $\Lambda[[G]]$ consisting of left and right non-zero divisors in the sense of \cite[Ch.~10]{GW:NoncommNoethRings} such that  the left and right localisation
\[
\Lambda[[G]]_T.
\]
exists. Moreover, $S$ is equal to the set of elements of $\Lambda[[G]]$ that become units in $\Lambda[[G]]_T$. In particular, $S$ is also a left and right denominator set and
\[
\Lambda[[G]]_S=\Lambda[[G]]_T.
\]
\end{lem}
\begin{proof}
Set $t\coloneqq\gamma-1$. Viewing $\Lambda[[G]]$ as a skew power series ring over $\Lambda[[H]]$ in $t$, it is clear that left and right multiplication with $t^n$ on $\Lambda[[G]]$ is injective with finite cokernel.

According to Lemma~\ref{lem:S-torsion complexes of length two} we have $s\in S$ if and only if $\Lambda[[G]]/\Lambda[[G]]s$ is finite. In particular, we have $T\subset S$. Considering $\Lambda[[G]]/\Lambda[[G]]s$ as a finite $\Int_p[[t]]$-module we see that it is annihilated by a power of $t$. We conclude that there exists an integer $n\geq0 $ such that for any $a\in\Lambda[[G]]$ there exists a $b\in \Lambda[[G]]$ such that
\[
t^na=bs.
\]
Applying this to elements of $T\subset S$, we see that $T$ and $S$ are left denominator set consisting of left and right non-zero divisors such that all elements of $S$ are units in $\Lambda[[G]]_T=\Lambda[[G]]_S$.

Applying the same arguments to $s\in \Lambda[[G]]$ with $\Lambda[[G]]/s\Lambda[[G]]$ finite, we see that $T$ is also a right denominator set.

Assume that $s\in\Lambda[[G]]$ becomes a unit in $\Lambda[[G]]_T$. Then kernel and cokernel of
\[
\Lambda[[G]]\xrightarrow{\cdot s}\Lambda[[G]]
\]
are annihilated by powers of $t$. Considering $\Lambda[[G]]$ as a finitely generated $\Int_p[[t]]$-module annihilated by a power of $p$, we conclude that the cokernel is finite, which implies that $s\in S$. Since $T$ is a right denominator set, the same is then true for
\[
S=\Lambda[[G]]\cap \Lambda[[G]]_T^\times.
\]
\end{proof}

\begin{lem}\label{lem:ore set is right and left}
Assume that $\Lambda[[H]]$ is noetherian. Then:
\begin{enumerate}
\item\label{enum:ore set is right and left:coker f.g. left} \(S=\set{f\in\Lambda[[G]]\given\text{\(\Lambda[[G]]/\Lambda[[G]]f\) is a f.\,g.\ left \(\Lambda[[H]]\)-module}}\).
\item\label{enum:ore set is right and left:coker f.g. right} \(S=\set{f\in\Lambda[[G]]\given\text{\(\Lambda[[G]]/f\Lambda[[G]]\) is a f.\,g.\ right \(\Lambda[[H]]\)-module}}\).
\item\label{enum:ore set is right and left:left and right denominator} $S$ is a left and right denominator set consisting of left and right non-zero divisors.
\item\label{enum:ore set is right and left:S-torsion} A perfect complex of left $\Lambda[[G]]$-modules is perfect as complex of $\Lambda[[H]]$-modules if and only its cohomology groups are $S$-torsion.
\end{enumerate}
\end{lem}
\begin{proof}
Lemma~\ref{lem:S-torsion complexes of length two} implies that the elements of $S$ are right non-zero divisors and that \eqref{enum:ore set is right and left:coker f.g. left} holds. Under the assumption that $\Lambda[[H]]$ is noetherian, we know by \cite[Cor.~2.21]{Witte:Splitting} that $S$ is a left denominator set. Assertion \eqref{enum:ore set is right and left:S-torsion} follows from \cite[Thm.~2.18]{Witte:Splitting}. Write $(\Lambda[[G]])^\op$ and $\Lambda^\op$ for the opposite rings of $\Lambda[[G]]$ and $\Lambda$, respectively. Consider the ring isomorphism
\[
\sharp\colon (\Lambda[[G]])^{\op}\mto \Lambda^{\op}[[G]]
\]
that maps $g\in G$ to $g^{-1}$. To prove the remaining assertions, it is sufficient to show that $\sharp$ maps $S_{\Lambda[[G]]}\subset(\Lambda[[G]])^\op$ to $S_{\Lambda^\op[[G]]}$.

If $\Lambda$ and $H$ are finite and $\gamma\in \Gamma$ is a topological generator, then $\sharp$ maps $t\coloneqq \gamma-1$ to $t'\coloneqq \gamma^{-1}-1$ and hence, it maps $T\coloneqq\set{t^n\given n\in\N}$ to $T'\coloneqq\set{t'^n\given n\in \N}$. Using Lemma~\ref{lem:ore set in the finite case} for $T$ and $T'$, we conclude that $\sharp(S_{\Lambda[[G]]})=S_{\Lambda^\op[[G]]}$.

In the general case, write
\[
\Lambda^\op[[G]]=\varprojlim_{U,I}\Lambda^{\op}/I[[G/H\cap U]]
\]
where the limit runs over all open two-sided ideals $I$ of $\Lambda$ and all open normal subgroups $U$ of $G$ and note that
$\Lambda^{\op}[[G]]\xrightarrow{\cdot s^\sharp}\Lambda^\op[[G]]$ is perfect over $\Lambda^{\op}[[H]]$ if and only if  $(\Lambda/I)^\op[[G/H\cap U]]\xrightarrow{\cdot s^\sharp }(\Lambda/I)^\op[[G/H\cap U]]$ is perfect over the finite ring $(\Lambda/I)^{\op}[[H/H\cap U]]$ for all $I$ and $U$.
\end{proof}

For general $\Lambda$ and $H$, the set $S$ is no longer a left or right denominator set, as the following example shows.

\begin{exmpl}
Assume that either $\Lambda=\FF_p$ is the finite field with $p$ elements and $H$ is the free pro-$p$ group on two topological generators with trivial action of $\Gamma$ or $\Lambda=\FF_p\langle\langle x,y\rangle\rangle$ is the power series ring in two non-commuting indeterminates $x,y$ and $H$ is trivial. In both cases, $\Lambda[[G]]=\FF_p\langle\langle x,y\rangle\rangle[[t]]$ is the power series ring over $\FF_p\langle\langle x,y\rangle\rangle$ with $t$ commuting with $x$ and $y$ and the set $S$ is the set of those power series $f(x,y,t)$ with $f(0,0,t)\neq 0$. Set $s\coloneqq x-t\in S$. If $S$ were a left denominator set, we could find
\[
a\coloneqq\sum_{i=0}^\infty a_it^i\in\FF_p\langle\langle x,y\rangle\rangle[[t]],\quad b\coloneqq\sum_{i=0}^\infty b_it^i\in S
\]
such that $as=by$, i.\,e.\
\[
a_0x=b_0y,\qquad a_{i}x-a_{i-1}=b_{i}y\quad\text{for $i>0$.}
\]
The only solution for this equation is $a=b=0$, which contradicts the assumption $b\in S$.
\end{exmpl}

Nevertheless, using Waldhausen $\KTh$-theory, we can still give a sensible definition of $\KTh_1(\Lambda[[G]]_S)$ even if $\Lambda[[G]]_S$ does not exist.

\begin{defn}
We write $\cat{SP}^{w_H}(\Lambda[[G]])$ for the full
Waldhausen subcategory of $\cat{SP}(\Lambda[[G]])$ of
strictly perfect complexes of $\Lambda[[G]]$-modules which are perfect as complexes of $\Lambda[[H]]$-modules.

We write $w_H\cat{SP}(\Lambda[[G]])$ for the Waldhausen
category with the same objects, morphisms and cofibrations as
$\cat{SP}(\Lambda[[G]])$, but with a new set of weak
equivalences given by those morphisms whose cones are objects of the category
$\cat{SP}^{w_H}(\Lambda[[G]])$.
\end{defn}

The same construction also works for $\cat{PDG}^{\cont}(\Lambda[[G]])$:

\begin{defn}\label{defn:wHPDG}
We write $\cat{PDG}^{\cont,w_H}(\Lambda[[G]])$ for the full
Waldhausen subcategory of $\cat{PDG}^{\cont}(\Lambda[[G]])$ of
objects $(\cmplx{P}_J)_{J\in\openideals_{\Lambda[[G]]}}$ such that
$$
\varprojlim_{J\in\openideals_{\Lambda[[G]]}} \cmplx{P}_J
$$
is a perfect complex of $\Lambda[[H]]$-modules.

We write $w_H\cat{PDG}^{\cont}(\Lambda[[G]])$ for the Waldhausen
category with the same objects, morphisms and cofibrations as
$\cat{PDG}^{\cont}(\Lambda[[G]])$, but with a new set of weak
equivalences given by those morphisms whose cones are objects of the category
$\cat{PDG}^{\cont,w_H}(\Lambda[[G]])$.
\end{defn}

From the approximation theorem \cite[1.9.1]{ThTr:HAKTS+DC} and \cite[Prop.~3.7]{Witte:MCVarFF} we conclude that
\begin{align*}
\KTh_{n}(\cat{SP}^{w_H}(\Lambda[[G]]))&\isomorph\KTh_{n}(\cat{PDG}^{\cont,w_H}(\Lambda[[G]])),\\
\KTh_{n}(w_H\cat{SP}(\Lambda[[G]]))&\isomorph\KTh_{n}(w_H\cat{PDG}^{\cont}(\Lambda[[G]]))
\end{align*}
We may then set for all $n\geq 0$
\begin{align*}
\KTh_{n}(\Lambda[[G]],S)&\coloneqq\KTh_{n}(\cat{PDG}^{\cont,w_H}(\Lambda[[G]])),\\
\KTh_{n+1}(\Lambda[[G]]_S)&\coloneqq\KTh_{n+1}(w_H\cat{PDG}^{\cont}(\Lambda[[G]])).
\end{align*}
If $\Lambda[[H]]$ is noetherian, these groups agree with their usual definition \cite[\S~4]{Witte:MCVarFF}.

A closely related variant of $\cat{SP}^{w_H}(\Lambda[[G]])$ is the following Waldhausen category.

\begin{defn}\label{defn:SP LambdaH G}
Let $\cat{SP}(\Lambda[[H]],G)$ be the Waldhausen category of complexes of $\Lambda[[G]]$-modules which are strictly perfect as complexes of $\Lambda[[H]]$-modules. Cofibrations are the injective morphisms with cokernel in $\cat{SP}(\Lambda[[H]],G)$; the weak equivalences are given by the quasi-isomorphisms.
\end{defn}

In other words,  $\cat{SP}(\Lambda[[H]],G)$ is the Waldhausen category of bounded complexes over the exact category of $\Lambda[[G]]$-modules which are finitely generated and projective as $\Lambda[[H]]$-modules and hence, the groups $\KTh_n(\cat{SP}(\Lambda[[H]],G))$ agree with the Quillen $\KTh$-groups of this exact category. Unfortunately, we cannot prove in general that $\KTh_n(\cat{SP}(\Lambda[[H]],G))$ agrees with $\KTh_n(\Lambda[[G]],S)$. However, we shall see below that we always have a surjection
\[
 \KTh_0(\cat{SP}(\Lambda[[H]],G))\mto\KTh_0(\Lambda[[G]],S).
\]
This is sufficient for our applications.

\begin{lem}\label{lem:resolution by strictly perfect LamdaH-complexes}
Let $\cmplx{P}$ be a complex of projective compact $\Lambda[[G]]$-modules that is bounded above. Assume that there exists a bounded above complex $\cmplx{K}$ of finitely generated, projective $\Lambda[[H]]$-modules that is quasi-isomorphic to $\cmplx{P}$ as complex of $\Lambda[[H]]$-modules. Then there exists in the category of complexes of $\Lambda[[G]]$-modules an injective endomorphism
\[
\psi\colon \Lambda[[G]]\tensor_{\Lambda[[H]]}\cmplx{K}\mto\Lambda[[G]]\tensor_{\Lambda[[H]]}\cmplx{K}
\]
and a quasi-isomorphism
\[
f\colon \cmplx{P}\mto\coker \psi.
\]
such that $\coker \psi$ is a bounded above complex of $\Lambda[[G]]$-modules which are finitely generated and projective as $\Lambda[[H]]$-modules.

In particular, if $\cmplx{P}$ is perfect as complex of $\Lambda[[H]]$-modules,  then $\cmplx{P}$ is perfect as complex of $\Lambda[[G]]$-modules and $\coker \psi$ is in $\cat{SP}(\Lambda[[H]],G)$.
\end{lem}
\begin{proof}
Since $\cmplx{K}$ is a bounded above complex of finitely generated projective $\Lambda[[H]]$-modules, there exists a quasi-isomorphism $\alpha \colon \cmplx{K}\mto \cmplx{P}$ of complexes of $\Lambda[[H]]$-modules, which is automatically continuous for the compact topologies on $\cmplx{K}$ and $\cmplx{P}$. Every projective compact $\Lambda[[G]]$-module is also projective in the category of compact $\Lambda[[H]]$-modules. Hence, there exists a weak equivalence $\beta\colon \cmplx{P}\mto\cmplx{K}$ in the category of complexes of compact $\Lambda[[H]]$-modules such that $\alpha\circ\beta$ and $\beta\circ\alpha$ are homotopic to the identity. Fix a topological generator $\gamma\in\Gamma$ and set
\[
 \begin{aligned}
 g&\colon\cmplx{K}\mto \cmplx{K},\qquad x\mapsto \beta(\gamma\alpha(x)),\\
\psi&\colon \Lambda[[G]]\tensor_{\Lambda[[H]]}\cmplx{K}\mto \Lambda[[G]]\tensor_{\Lambda[[H]]}\cmplx{K},\qquad \lambda\tensor x\mapsto \lambda\tensor x-\lambda\gamma^{-1}\tensor g(x).
 \end{aligned}
\]
Then $\psi$ is a $\Lambda[[G]]$-linear complex morphism. Moreover, $\coker \psi$ is finitely generated over $\Lambda[[H]]$ in each degree. Indeed, if we set $t\coloneqq\gamma-1$ and let $(e_1,\dots, e_m)$ denote a generating system of the $\Lambda[[H]]$-module $K^n$ in degree $n$, then $(t^k\tensor e_i)_{k\in\N_0,i=1,\dots,m}$ is a topological generating system of $\Lambda[[G]]\tensor_{\Lambda[[H]]}K^n$ over $\Lambda[[H]]$. But
\[
 t^k\tensor v=t^{k-1}\tensor(g(v)-v)+\psi(\gamma t^{k-1}\tensor v)
\]
for all $v\in K^n$, such that $\coker \psi$ is already generated by the images of $1\tensor e_1,\dots,1\tensor e_m$.

From Lemma~\ref{lem:S-torsion complexes of length two} we conclude that $\psi$ is injective and that $\coker \psi$ is finitely generated and projective over $\Lambda[[H]]$ in each degree. Set $\cmplx{Q}\coloneqq \coker \psi$. Since $\cmplx{P}$ is a bounded above complex of projective compact $\Lambda[[G]]$-modules, there exists a quasi-isomorphism $f$ completing the homotopy-commutative diagram
\[
\xymatrix@C=1.5pc{
0\ar[r]&\Lambda[[G]]\ctensor_{\Lambda[[H]]}\cmplx{P}\ar[rrr]^{\lambda\ctensor x\mapsto \lambda\ctensor x- \lambda\gamma^{-1}\ctensor \gamma x}\ar[d]^{\id\ctensor \beta}_{\sim}&&& \Lambda[[G]]\ctensor_{\Lambda[[H]]}\cmplx{P}\ar[rr]^{\qquad\lambda\ctensor x\mapsto \lambda x}\ar[d]^{\id\ctensor \beta}_{\sim}&& \cmplx{P}\ar@{.>}[d]^f_{\sim}\ar[r] & 0\\
0\ar[r]&\Lambda[[G]]\tensor_{\Lambda[[H]]}\cmplx{K}\ar[rrr]^{\psi}&&& \Lambda[[G]]\tensor_{\Lambda[[H]]}\cmplx{K}\ar[rr]&& \cmplx{Q}\ar[r]& 0
}
\]
in the category of complexes of compact $\Lambda[[G]]$-modules. Here, $\Lambda[[G]]\ctensor_{\Lambda[[H]]}\cmplx{P}$ denotes the completed tensor product. The exactness of the first row follows from \cite[Prop. 2.4]{Witte:Splitting}. If we can choose $\cmplx{K}$ to be a strictly perfect complex of $\Lambda[[H]]$-modules, then $\cmplx{P}$ is also quasi-isomorphic to the cone of $\psi$, which is strictly perfect as complex of $\Lambda[[G]]$-modules. Moreover, $\coker \psi$ is a bounded complex and hence, an object of $\cat{SP}(\Lambda[[H]],G)$.
\end{proof}

\begin{prop}\label{prop:comparison with strictly perfect LamdaH-complexes}
Let $\gamma\in \Gamma$ be a topological generator. The functor
\[
\begin{gathered}
C_\gamma\colon\cat{SP}(\Lambda[[H]],G)\mto \cat{SP}^{w_H}(\Lambda[[G]]),\\
\cmplx{P}\mapsto \Cone(\Lambda[[G]]\tensor_{\Lambda[[H]]}\cmplx{P}\xrightarrow{\lambda\tensor p\mapsto \lambda\tensor p-\lambda\gamma^{-1}\tensor\gamma p}\Lambda[[G]]\tensor_{\Lambda[[H]]}\cmplx{P})
\end{gathered}
\]
is well defined and Waldhausen exact. It induces a surjection
\[
C_\gamma\colon \KTh_0(\cat{SP}(\Lambda[[H]],G))\mto\KTh_0(\cat{SP}^{w_H}(\Lambda[[G]]))
\]
which is independent of the choice of $\gamma$.
\end{prop}
\begin{proof}
From \cite[Prop.~2.4]{Witte:Splitting} we conclude that
\[
0\mto \Lambda[[G]]\tensor_{\Lambda[[H]]}\cmplx{P}\xrightarrow{\id-(\cdot\gamma^{-1}\tensor \gamma\cdot)}\Lambda[[G]]\tensor_{\Lambda[[H]]}\cmplx{P}\xrightarrow{\lambda\tensor p\mapsto \lambda p}\cmplx{P}\mto 0
\]
is an exact sequence of complexes of $\Lambda[[G]]$-modules for any $\cmplx{P}$ in $\cat{SP}(\Lambda[[H]],G)$. In particular, the strictly perfect complex of $\Lambda[[G]]$-modules $C_\gamma(\cmplx{P})$ is quasi-isomorphic to $\cmplx{P}$ in the category of complexes of $\Lambda[[G]]$-modules and therefore perfect as complex of $\Lambda[[H]]$-modules. Thus, $C_\gamma(\cmplx{P})$ is an object of $\cat{SP}^{w_H}(\Lambda[[G]])$. The Waldhausen exactness of the functor $C_\gamma$ follows easily from the Waldhausen exactness of the cone construction.

Consider the Waldhausen category $\cat{P}^{w_H}(\Lambda[[G]])$ of perfect complexes of $\Lambda[[G]]$-modules which are also perfect as complexes of $\Lambda[[H]]$-modules. The approximation theorem \cite[1.9.1]{ThTr:HAKTS+DC} implies that the inclusion
\[
\iota\colon \cat{SP}^{w_H}(\Lambda[[G]])\mto \cat{P}^{w_H}(\Lambda[[G]])
\]
induces isomorphisms
\[
 \KTh_n(\cat{SP}^{w_H}(\Lambda[[G]]))\isomorph\KTh_n(\cat{P}^{w_H}(\Lambda[[G]]))
\]
for all $n$. The functorial quasi-isomorphism $C_\gamma(\cmplx{P})\wto \cmplx{P}$ in $\cat{P}^{w_H}(\Lambda[[G]])$ implies that the homomorphism of $\KTh$-groups induced by $\iota\comp C_\gamma$ agrees with the homomorphism induced by the inclusion $\iota'\colon \cat{SP}(\Lambda[[H]],G)\mto \cat{P}^{w_H}(\Lambda[[G]])$. Since $\KTh_0(\cat{P}^{w_H}(\Lambda[[G]]))$ is generated by the quasi-isomorphism classes of complexes in $\cat{P}^{w_H}(\Lambda[[G]])$, we deduce from Lemma~\ref{lem:resolution by strictly perfect LamdaH-complexes} that $\iota'$ induces a surjection
\[
\KTh_0(\cat{SP}(\Lambda[[H]],G))\mto\KTh_0(\cat{P}^{w_H}(\Lambda[[G]])).
\]
\end{proof}

In the light of Proposition~\ref{prop:comparison with strictly perfect LamdaH-complexes}, we will write
\[
 [\cmplx{P}]\coloneqq[C_\gamma(\cmplx{P})]\in\KTh_0(\Lambda[[G]],S)
\]
for any $\cmplx{P}$ in $\cat{SP}(\Lambda[[H]],G)$.

\begin{rem}
In order to deduce from the approximation theorem (applied to the opposite categories) that $C_\gamma$ induces isomorphisms
\[
\KTh_n(\cat{SP}(\Lambda[[H]],G))\isomorph\KTh_n(\cat{SP}^{w_H}(\Lambda[[G]]))
\]
for all $n$, it would suffice to verify that for every complex $\cmplx{P}$ in $\cat{SP}(\Lambda[[H]],G)$ and every morphism $f\colon \cmplx{K}\mto \cmplx{P}$ in $\cat{P}^{w_H}(\Lambda[[G]])$, there exists a morphism $f'\colon \cmplx{Q}\mto \cmplx{P}$ in $\cat{SP}(\Lambda[[H]],G)$ and a quasi-isomorphism $w\colon \cmplx{K}\wto\cmplx{Q}$ in $\cat{P}^{w_H}(\Lambda[[G]])$ such that $f=f'\comp w$.
\end{rem}

\comment{To do: Prove that the surjection is an isomorphism.}

Thanks to a result of Muro and Tonks \cite{MT:OnK1WaldCat}, the groups $\KTh_0(\cat{W})$ and $\KTh_1(\cat{W})$ of any Waldhausen category $\cat{W}$ can be described as the cokernel and kernel of a homomorphism
\begin{equation}\label{eqn:1-type}
\del\colon\D_1(\cat{W})\mto\D_0(\cat{W})
\end{equation}
between two nil-$2$-groups (i.\,e.\ $\com{a}{\com{b}{c}}=1$ for any three group elements $a,b,c$) that are given by explicit generators and relations in terms of the structure of the underlying Waldhausen category. As additional structure, there exists a pairing
\[
\D_0(\cat{W})\times\D_0(\cat{W})\mto\D_1(\cat{W}),\qquad (A,B)\mapsto\pair{A}{B}
\]
satisfying
\begin{align*}
\del\pair{A}{B}&=B^{-1}A^{-1}BA,\\
\pair{\del a}{\del b}&=b^{-1}a^{-1}ba,\\
\pair{A}{B}\pair{B}{A}&=1,\\
\pair{A}{BC}&=\pair{A}{B}\pair{A}{C}.
\end{align*}
In other words, $\D_\bullet(\cat{W})$ is a stable quadratic module in the sense of \cite{Bau:CombHom}.
In particular, $X\in\D_0(\cat{W})$ operates from the right on $a\in\D_1(\cat{W})$ via
\[
a^X\coloneqq a\pair{X}{\del a}.
\]
More explicitly, $\D_0(\cat{W})$ is the free nil-$2$-group generated by the objects of $\cat{W}$ different from the zero object, while $\D_1(\cat{W})$ is generated by all weak equivalences and exact sequences in $\cat{W}$ subject to the following list of relations:
\begin{itemize}
    \item[(R1)] $\del[\alpha]=[B]^{-1}[A]$ for a weak equivalence $\alpha\colon A\wto
    B$,
    \item[(R2)] $\del[\Delta]=[B]^{-1}[C][A]$ for an exact sequence $\Delta\colon A\cto B\qto
    C$.
    \item[(R3)] $\pair{[A]}{[B]}=[B\cto
    A\oplus B \qto A]^{-1}[A\cto A\oplus B\qto B]$ for any pair of
    objects $A, B$.
    \item[(R4)] $[0\cto 0\qto 0]=1_{\D_1}$,
    \item[(R5)] $[\beta\alpha]=[\beta][\alpha]$ for weak equivalences $\alpha\colon A\wto
    B$, $\beta\colon B\wto C$,
    \item[(R6)] $[\Delta'][\alpha][\gamma]^{[A]}=[\beta][\Delta]$ for any commutative diagram
    $$
    \xymatrix{
    A\lbl{\Delta\colon}\ar@{ >->}[r]\ar[d]^{\sim}_{\alpha}&
    B\ar@{>>}[r]\ar[d]^{\sim}_{\beta}
    &C\ar[d]^{\sim}_{\gamma}\\
    A'\lbl{\Delta'\colon}\ar@{ >->}[r]&B'\ar@{>>}[r]&C'}
    $$
    \item[(R7)] $[\Gamma_1][\Delta_1]=[\Delta_2][\Gamma_2]^{[A]}$ for any
    commutative diagram
    $$
    \xymatrix{A\lbl{\Delta_1\colon}\ar@{ >->}[r]\ar@{=}[d]
    &B\lbu{\Gamma_1\colon}\ar@{>>}[r]\ar@{ >->}[d]
    &C\lbu{\Gamma_2\colon}\ar@{ >->}[d]\\
    A\lbl{\Delta_2\colon}\ar@{ >->}[r]\ar@{>>}[d]&D\ar@{>>}[r]\ar@{>>}[d]&E\ar@{>>}[d]\\
    {0}\ar@{ >->}[r]&F\ar@{=}[r]&F}
    $$
\end{itemize}
\cite[Def.~1.2]{MT:1TWKTS}, \cite[Def.~A.4]{Witte:MCVarFF}.

In particular, $\KTh_{0}(\Lambda[[G]],S)$ is the abelian group generated by the symbols $[\cmplx{P}]$ with $\cmplx{P}$ in $\cat{PDG}^{\cont,w_H}(\Lambda[[G]])$ modulo the relations
\begin{align*}
[\cmplx{P}]&=[\cmplx{Q}]&&\text{if $\cmplx{P}$ and $\cmplx{Q}$ are weakly equivalent,}\\
[\cmplx{P}_2]&=[\cmplx{P}_1]+[\cmplx{P}_3]&&\text{if $0\mto\cmplx{P}_1\mto\cmplx{P}_2\mto\cmplx{P}_3\mto 0$ is an exact sequence.}
\end{align*}
If $f\colon \cmplx{P}\mto \cmplx{P}$ is an endomorphism which is a weak equivalence in $\cat{PDG}^{\cont}(\Lambda[[G]])$, we can assign to it a class $[f\lcirclearrowright \cmplx{P}]$ in $\KTh_1(\Lambda[[G]])$. By the classical definition of the first $\KTh$-group as factor group of the general linear group it is clear that these classes generate $\KTh_1(\Lambda[[G]])$. It then follows from the splitting of the $\KTh$-theory localisation sequence \cite[Cor.~3.3]{Witte:Splitting} that the classes $[f\lcirclearrowright\cmplx{P}]$ of endomorphisms $f\colon \cmplx{P}\mto \cmplx{P}$ which are weak equivalences in $w_H\cat{PDG}^{\cont}(\Lambda[[G]])$ generate $\KTh_1(\Lambda[[G]]_S)$. The relations that are satisfied by these generators can be read off from the above relations for $\D_1\cat{PDG}^{\cont,w_H}(\Lambda[[G]])$.

\begin{rem}
Some authors prefer the theory of determinant functors and Deligne's category of virtual objects \cite{Del:DC} as an alternative model for the $1$-type of the $\KTh$-theory spectrum. We refer to \cite{MTW:DeterminantFunctors} for the precise connection of the two approaches.
\end{rem}

Let $\Lambda$ and $\Lambda'$ be two adic $\Int_{p}$-algebras and $G=H\rtimes \Gamma$, $G=H'\rtimes \Gamma'$ be profinite groups, such that $H$ and $H'$ contain open, topologically finitely generated pro-$p$ subgroups and $\Gamma\isomorph\Int_p\isomorph \Gamma'$. Suppose that $\cmplx{K}$ is a complex of $\Lambda'[[G']]$-$\Lambda[[G]]$-bimodules, strictly perfect as complex of $\Lambda'[[G']]$-modules and assume that there exists a complex $\cmplx{L}$ of $\Lambda'[[H']]$-$\Lambda[[H]]$-bimodules, strictly perfect as complex of $\Lambda'[[H']]$-modules, and a quasi-isomorphism of complexes of $\Lambda'[[H']]$-$\Lambda[[G]]$-bimodules
$$
\cmplx{L}\ctensor_{\Lambda[[H]]}\Lambda[[G]]\wto \cmplx{K}.
$$
Here,
$$
\cmplx{L}\ctensor_{\Lambda[[H]]}\Lambda[[G]]\coloneqq\varprojlim_{I\in\openideals_{\Lambda'[[G']]}}\varprojlim_{J\in\openideals_{\Lambda[[G]]}}\cmplx{L/IL}\tensor_{\Lambda[[H]]}\Lambda[[G]]/J
$$
denotes the \emph{completed tensor product}.

In the above situation, the Waldhausen exact functor
\begin{equation}\label{eqn:ringtransf}
\ringtransf_{\cmplx{K}}\colon \cat{PDG}^{\cont}(\Lambda[[G]])\mto \cat{PDG}^{\cont}(\Lambda'[[G']])
\end{equation}
takes objects of the category $\cat{PDG}^{\cont,w_H}(\Lambda[[G]])$ to objects of the category $\cat{PDG}^{\cont,w_{H'}}(\Lambda'[[G']])$ and weak equivalences of $w_H\cat{PDG}^{\cont}(\Lambda[[G]])$ to weak equivalences of $w_{H'}\cat{PDG}^{\cont}(\Lambda'[[G']])$ \cite[Prop.~4.6]{Witte:MCVarFF}. Hence, it also induces homomorphisms between the corresponding $\KTh$-groups. In particular, this applies to the following examples:

\begin{exmpl}\label{exmpl:example functors} \cite[Prop.~4.7]{Witte:MCVarFF}
\begin{enumerate}
    \item Assume $G=G'$, $H=H'$. For any complex $\cmplx{P}$ of $\Lambda'$-$\Lambda[[G]]$-bimodules, strictly perfect as complex of $\Lambda'$-modules, let $\cmplx{K}$ be the complex
    \[
    \cmplx{P[[G]]^\delta}\coloneqq\Lambda'[[G]]\tensor_{\Lambda'}\cmplx{P}
    \]
    of $\Lambda'[[G]]$-$\Lambda[[G]]$-bimodules with the right $G$-operation given by the diagonal action on both factors. This applies in particular for any complex $\cmplx{P}$ of $\Lambda'$-$\Lambda$-bimodules, strictly perfect as complex of $\Lambda'$-modules and equipped with the trivial $G$-operation.
    \item Assume $\Lambda=\Lambda'$. Let $\alpha\colon G\mto G'$ be a continuous homomorphism such that $\alpha$ maps $H$ to $H'$ and induces a bijection of $G/H$ and $G'/H'$. Let $\cmplx{K}$ be the $\Lambda[[G']]$-$\Lambda[[G]]$-bimodule $\Lambda[[G']]$.
    \item Assume that $G'$ is an open subgroup of $G$ and set $H'\coloneqq H\cap G'$. Let $\Lambda=\Lambda'$ and let $\cmplx{K}$ be the complex concentrated in degree $0$ given by the $\Lambda[[G']]$-$\Lambda[[G]]$-bimodule $\Lambda[[G]]$.
\end{enumerate}
\end{exmpl}

\begin{exmpl}\label{exmpl:group homomorphisms}
The assumptions in Example \ref{exmpl:example functors}.(2) are in fact stronger than necessary. We may combine it with the following result. Assume that $G$ is an open subgroup of $G'$ such that $H\coloneqq H'\cap G=H'$ and $\Gamma=(\Gamma')^{p^n}$. Let $\Lambda=\Lambda'$ and let $\cmplx{K}$ be the $\Lambda[[G']]$-$\Lambda[[G]]$-bimodule $\Lambda[[G']]$. Fix a topological generator $\gamma'\in\Gamma'$ and let $\cmplx{L}$ be the $\Lambda[[H]]$-$\Lambda[[H]]$-sub-bimodule of $\Lambda[[G']]$ generated as left $\Lambda[[H]]$-module by $1,\gamma',(\gamma')^2,\dots,(\gamma')^{p^n-1}$. Then $\cmplx{L}$ is a strictly perfect complex of $\Lambda[[H]]$-modules concentrated in degree $0$ and the canonical map
$$
\cmplx{L}\ctensor_{\Lambda[[H]]}\Lambda[[G]]\wto \cmplx{K},\qquad \ell\ctensor \lambda\mapsto \ell\lambda
$$
is an isomorphism of $\Lambda'[[H']]$-$\Lambda[[G]]$-bimodules such that \cite[Prop.~4.6]{Witte:MCVarFF} applies. In combination with Example \ref{exmpl:example functors}.(2) this implies that any continuous group homomorphism $\alpha\colon G\mto G'$ such that $\alpha(G)\not\subset H'$ induces Waldhausen exact functors between all three variants of the above Waldhausen categories.
\end{exmpl}

\begin{exmpl}\label{exmpl:def of eval}
As a special case of Example~\ref{exmpl:example functors}.(1), assume that $\Lambda=\Int_p$ and that $\rho$ is some continuous representation of $G$ on a finitely generated and projective $\Lambda'$-module. Let $\rho^\sharp$ be the $\Lambda'$-$\Int_p[[G]]$-bimodule which agrees with $\rho$ as $\Lambda'$-module, but on which $g\in G$ acts from the right by the left operation of $g^{-1}$ on $\rho$. We thus obtain Waldhausen exact functors
\begin{equation}\label{eqn:def of eval}
\eval_\rho\coloneqq\ringtransf_{\Lambda'[[\Gamma]]}\circ\ringtransf_{\rho^\sharp[[G]]^\delta}
\end{equation}
from all three variants of the Waldhausen category $\cat{PDG}^{\cont}(\Int_p[[G]])$ to the corresponding variant of $\cat{PDG}^{\cont}(\Lambda'[[\Gamma]])$. If $\Lambda'$ is a commutative adic $\Int_p$-algebra, then the image of
\[
\left[\Int_p[[G]]\xrightarrow{\cdot g}\Int_p[[G]]\right]\in \KTh_1(\Int_p[[G]]),\quad g\in G,
\]
under the composition of $\eval_\rho$ with
\[
\det\colon \KTh_1(\Lambda'[[\Gamma]])\xrightarrow{\isomorph}\Lambda'[[\Gamma]]^\times
\]
is $\bar{g}\det(\rho (g))^{-1}$, where $\bar{g}$ denotes the image of $g$ under the projection $G\rightarrow \Gamma$. Note that this differs from \cite[(22)]{CFKSV} by a sign. So, our evaluation at $\rho$ corresponds to the evaluation at the representation dual to $\rho$ in terms of the cited article.
\end{exmpl}

\section{A property of \texorpdfstring{$S$}{S}-torsion complexes}\label{sec:S-torsion exchange property}

In this section, we prove Proposition~\ref{prop:S-torsion exchange property}, which is an abstract generalisation of \cite[Prop. 2.1]{Witte:Survey}. We will apply this proposition later in Section~\ref{sec:S-torsion}.

With the notation of the previous section, fix a topological generator $\gamma\in\Gamma$ and set $t\coloneqq \gamma-1$. Assume for the moment that $\Lambda$ is a finite $\Int_p$-algebra and that $H$ is a finite group. By Lemma~\ref{lem:ore set in the finite case}, we have
\[
\Lambda[[G]]_S=\varinjlim_{n\geq 0}\Lambda[[G]]t^{-n}
\]
as $\Lambda[[G]]$-modules.

Assume that $p^{i+1}=0$ in $\Lambda$. Then
\[
\binom{p^{n+i}}{k}=0
\]
in $\Lambda$ whenever $p^n\nmid k$. Hence,
\begin{align*}
\gamma^{p^{n+i}}-1&=(t+1)^{p^{n+i}}-1=t^{p^n}\sum_{k=1}^{p^i}\binom{p^{n+i}}{kp^{n}}t^{p^n(k-1)},\\
t^{p^{n+i}}&=(\gamma-1)^{p^{n+i}}-(1-1)^{p^{n+i}}=\sum_{k=1}^{p^i}\binom{p^{n+i}}{kp^{n}}(\gamma^{kp^{n}}-1)(-1)^{p^{n}(p^{i}-k)}\\
           &=(\gamma^{p^n}-1)\sum_{k=1}^{p^i}\binom{p^{n+i}}{kp^{n}}(-1)^{p^{n}(p^{i}-k)}\sum_{\ell=0}^{k-1}\gamma^{\ell p^{n}}
\end{align*}
and therefore,
\[
\Lambda[[G]]_S=\varinjlim_{n\geq 0}\Lambda[[G]](\gamma^{p^n}-1)^{-1}.
\]

Since $H$ was assumed to be finite, the same is true for the automorphism group of $H$. We conclude that $\gamma^{p^n}$ is a central element of $G$ and $\Gamma^{p^n}\subset G$ a central subgroup for all $n\geq n_0$ and $n_0$ large enough. Set
\[
N_n\coloneqq \sum_{k=0}^{p-1}\gamma^{p^nk}.
\]
The homomorphism
\[
\Lambda[[G]](\gamma^{p^n}-1)^{-1}\mto \Lambda[[G/\Gamma^{p^n}]],\qquad \lambda(\gamma^{p^n}-1)^{-1}\mapsto \lambda+\Lambda[[G]](\gamma^{p^n}-1)
\]
induces an isomorphism $\Lambda[[G]](\gamma^{p^n}-1)^{-1}/\Lambda[[G]]\isomorph \Lambda[[G/\Gamma^{p^n}]]$ such that the diagram
\[
\xymatrix{
\Lambda[[G]](\gamma^{p^n}-1)^{-1}/\Lambda[[G]]\ar[r]^{\subset}\ar[d]^{\isomorph}& \Lambda[[G]](\gamma^{p^{n+1}}-1)^{-1}/\Lambda[[G]]\ar[d]^{\isomorph}\\
\Lambda[[G/\Gamma^{p^n}]]\ar[r]^{\cdot N_n}& \Lambda[[G/\Gamma^{p^{n+1}}]]
}
\]
commutes. Hence, we obtain an isomorphism of (left and right) $\Lambda[[G]]$-modules
\[
\Lambda[[G]]_S/\Lambda[[G]]\isomorph \varinjlim_{n}\Lambda[[G/\Gamma^{p^n}]].
\]
We note that this isomorphism may depend on the choice of the topological generator $\gamma$.

For any strictly perfect complex $\cmplx{P}$ of $\Lambda[[G]]$-modules, we thus obtain an exact sequence
\[
0\mto \cmplx{P}\mto \Lambda[[G]]_S\tensor_{\Lambda[[G]]}\cmplx{P}\mto \varinjlim_{n}\Lambda[[G/\Gamma^{p^n}]]\tensor_{\Lambda[[G]]}\cmplx{P}\mto 0.
\]
If $\cmplx{P}$ is also perfect as a complex of $\Lambda[[H]]$-modules such that the cohomology of $\cmplx{P}$ is $S$-torsion by Lemma~\ref{lem:ore set is right and left}, then we conclude that there exists an isomorphism
\[
\cmplx{P}[1]\isomorph\varinjlim_{n}\Lambda[[G/\Gamma^{p^n}]]\tensor_{\Lambda[[G]]}\cmplx{P}
\]
in the derived category of complexes of $\Lambda[[G]]$-modules. In particular, the righthand complex is perfect as complex of $\Lambda[[G]]$-modules and of $\Lambda[[H]]$-modules. This signifies that its cohomology modules
\begin{equation*}
\HF^k(\varinjlim_{n}\Lambda[[G/\Gamma^{p^n}]]\tensor_{\Lambda[[G]]}\cmplx{P})
\isomorph\varinjlim_{n}\HF^k(\Lambda[[G/\Gamma^{p^n}]]\tensor_{\Lambda[[G]]}\cmplx{P})\isomorph\HF^{k+1}(\cmplx{P})
\end{equation*}
are finite as abelian groups.

We now drop the assumption that $\Lambda$ and $H$ are finite. Let $I\subset J$ be two open ideals of $\Lambda$ and $U\subset V$ be the intersections of two open normal subgroups of $G$ with $H$. Then the diagram
\[
\xymatrix{
0\ar[r]&\Lambda/I[[G/U]]\ar[r]\ar[d]&\Lambda/I[[G/U]]_S\ar[r]\ar[d]& \varinjlim\limits_{n}\Lambda/I[[G/U\Gamma^{p^n}]]\ar[r]\ar[d]& 0\\
0\ar[r]&\Lambda/J[[G/V]]\ar[r]&\Lambda/J[[G/V]]_S\ar[r]& \varinjlim\limits_{n}\Lambda/J[[G/V\Gamma^{p^n}]]\ar[r]& 0
}
\]
commutes and the downward pointing arrows are surjections. Tensoring with $\cmplx{P}$ and passing to the inverse limit we obtain the exact sequence
\[
0\mto \cmplx{P}\mto \varprojlim_{I,U}\Lambda/I[[G/U]]_S\tensor_{\Lambda[[G]]}\cmplx{P}\mto \varprojlim_{I,U}\varinjlim_{n}\Lambda/I[[G/U\Gamma^{p^n}]]\tensor_{\Lambda[[G]]}\cmplx{P}\mto 0.
\]
If $\cmplx{P}$ is also perfect as a complex of $\Lambda[[H]]$-modules, then complex in the middle is acyclic and we obtain again an isomorphism
\[
\cmplx{P}[1]\isomorph\varprojlim_{I,U}\varinjlim_{n}\Lambda/I[[G/U\Gamma^{p^n}]]\tensor_{\Lambda[[G]]}\cmplx{P}
\]
in the derived category of complexes of $\Lambda[[G]]$-modules and hence, isomorphisms of $\Lambda[[G]]$-modules
\[
\HF^{k+1}(\cmplx{P})\isomorph\varprojlim_{I,U}\varinjlim_{n}\HF^{k}(\Lambda/I[[G/U\Gamma^{p^n}]]\tensor_{\Lambda[[G]]}\cmplx{P}).
\]
Here, we use that the modules in the projective system on the righthand side are finite and thus $\varprojlim$-acyclic.

Finally, assume that $(\cmplx{Q}_J)_{J\in\openideals_{\Lambda[[G]]}}$ is a complex in $\cat{PDG}^{\cont,w_H}(\Lambda[[G]])$. Then we can find a strictly perfect complex of $\Lambda[[G]]$-modules $\cmplx{P}$ and a weak equivalence
\[
f\colon (\Lambda[[G]]/J\tensor_{\Lambda[[G]]}\cmplx{P})_{J\in\openideals_{\Lambda[[G]]}}\mto (\cmplx{Q}_J)_{J\in\openideals_{\Lambda[[G]]}}
\]
in  $\cat{PDG}^{\cont,w_H}(\Lambda)$ \cite[Cor. 5.2.6]{Witte:PhD}. Moreover, this complex $\cmplx{P}$ will also be perfect as a complex of $\Lambda[[H]]$-modules. For $I\in\openideals_{\Lambda}$, $U$ the intersection of an open normal subgroup of $G$ with $H$ and a positive integer $n$ such that $\Gamma^{p^n}$ is central in $G/U$ we set
\[
J_{I,U,n}\coloneqq \ker \Lambda[[G]]\mto \Lambda/I[[G/U\Gamma^{p^n}]],
\]
such that the $J_{I,U,n}$ form a final subsystem in $\openideals_{\Lambda[[G]]}$. We conclude:

\begin{prop}\label{prop:S-torsion exchange property}
For $(\cmplx{Q}_J)_{J\in\openideals_{\Lambda[[G]]}}$ in $\cat{PDG}^{\cont,w_H}(\Lambda[[G]])$ there exists an isomorphism
\[
\Rlim_{J\in\openideals_{\Lambda[[G]]}}\cmplx{Q}_J[1]\isomorph \Rlim_{I,U}\varinjlim_{n}\cmplx{Q}_{J_{I,U,n}}
\]
in the derived category of $\Lambda[[G]]$-modules and isomorphisms of $\Lambda[[G]]$-modules
\[
\varprojlim_{J\in\openideals_{\Lambda[[G]]}}\HF^{k+1}(\cmplx{Q}_J)\isomorph
\varprojlim_{I,U}\varinjlim_{n}\HF^{k}(\cmplx{Q}_{J_{I,U,n}}).
\]
\end{prop}

\begin{rem}
For any $(\cmplx{Q}_J)_{J\in\openideals_{\Lambda[[G]]}}$ in $\cat{PDG}^{\cont}(\Lambda[[G]])$ we obtain in the same way a distinguished triangle
\[
\Rlim_{J\in\openideals_{\Lambda[[G]]}}\cmplx{Q}_J\mto \Rlim_{I,U}\left(\Lambda/I[[G/U]]_S\Ltensor_{\Lambda/I[[G/U]]}\Rlim_{n}\cmplx{Q}_{J_{I,U,n}}\right)
\mto\Rlim_{I,U}\varinjlim_{n}\cmplx{Q}_{J_{I,U,n}}
\]
in the derived category of complexes of $\Lambda[[G]]$-modules.
\end{rem}

\section{Duality on the level of \texorpdfstring{$\KTh$}{K}-groups}\label{sec:Duality on the level of Kgroups}

Assume that $\cat{W}$ is a biWaldhausen category in the sense of \cite[Def. 1.2.4]{ThTr:HAKTS+DC}. In particular, the opposite category $\cat{W}^{\op}$ is a biWaldhausen category as well and there are natural isomorphisms
\begin{equation}\label{eqn:identification of W and Wop}
I\colon\KTh_n(\cat{W})\isomorph\KTh_n(\cat{W}^{\op}),
\end{equation}
simply because the topological realisations of the bisimplicial sets $N.w\cat{S}.\cat{W}$ and $N.w\cat{S}.\cat{W}^{\op}$ resulting from Waldhausen's $S$-construction agree \cite[\S 1.5.5]{ThTr:HAKTS+DC}. However, the obvious identifications
\[
N_mw\cat{S}_n\cat{W}\isomorph N_mw\cat{S}_n\cat{W}^{\op}
\]
respect the boundary and degeneracy maps only up to reordering, so that we do not obtain an isomorphism of the bisimplicial sets themselves.

In order to understand the isomorphism \eqref{eqn:identification of W and Wop} in terms of the presentation of $\KTh_1(\cat{W})$ given by \eqref{eqn:1-type}, we will construct a canonical isomorphism
\[
I\colon\D_\bullet(\cat{W})\mto\D_\bullet(\cat{W}^\op).
\]
For any morphism $\alpha\colon A\mto B$ in $\cat{W}$, write $\alpha^{\op}\colon B\mto A$ for the corresponding morphism in the opposite category $\cat{W}^\op$. Further, note that by the definition of biWaldhausen categories, if $A\cto B\qto C$ is an exact sequence in $\cat{W}$, then the dual sequence $C\cto B\qto A$ is exact in $\cat{W}^{\op}$. We then set
\begin{align*}
I([A])&\coloneqq[A]&&\text{for objects $A$ in $\cat{W}$,}\\
I([f\colon A\wto B])&\coloneqq[f^\op\colon B\wto A]^{-1}&&\text{for weak equivalences $f$,}\\
I([A\cto B\qto C])&\coloneqq[C\cto B\qto A]\pair{[A]}{[C]} &&\text{for exact sequences $A\cto B\qto C$.}
\end{align*}

\begin{prop}
For any biWaldhausen category $\cat{W}$, the above assignment defines an isomorphism of stable quadratic modules
\[
I\colon\D_\bullet(\cat{W})\mto\D_\bullet(\cat{W}^{\op}).
\]
\end{prop}
\begin{proof}
It is sufficient to check that $I$ respects the relations (R1)--(R7) in the definition of $\D_\bullet(\cat{W})$. This is a straight-forward, but tedious exercise.
\end{proof}

Next, we investigate in how far $I$ respects the boundary homomorphism of localisation sequences. For this, we consider the same situation as in \cite[Appendix]{Witte:MCVarFF}, but with all Waldhausen categories replaced by biWaldhausen categories. Assume that $\cat{wW}$ is a biWaldhausen category with weak equivalences $\cat{w}$ that is saturated and extensional in the sense of \cite[Def. 1.2.5, 1.2.6]{ThTr:HAKTS+DC}. Let $\cat{vW}$ be a the same category with the same notion of fibrations and cofibrations, but with a coarser notion of weak equivalences $\cat{v}\subset\cat{w}$ and let $\cat{vW^w}$ denote the full biWaldhausen subcategory of $\cat{vW}$ consisting of those objects which are weakly equivalent to the zero object in $\cat{wW}$. We assume that $\Cyl$ is a cylinder functor in the sense of \cite[Def. A.1]{Witte:MCVarFF} for both $\cat{wW}$ and $\cat{vW}$ and that it satisfies the cylinder axiom for $\cat{wW}$. We will write $\Cone$ and $\Shift$ for the associated cone and shift functors, i.\,e.\
\begin{align*}
\Cone(\alpha)&\coloneqq\Cyl(\alpha)/A&&\text{for any morphism $\alpha\colon A\mto B$,}\\
\Shift A&\coloneqq\Cone(A\mto 0)&&\text{for any object $A$.}
\end{align*}
Further, we assume that $\CoCyl$ is a cocylinder functor for both $\cat{wW}$ and $\cat{vW}$ in the sense that the opposite functor $\CoCyl^{\op}$ is a cylinder functor for $\cat{wW}^{\op}$ and $\cat{vW}^{\op}$. Again, we assume that $\CoCyl^{\op}$ satisfies the cylinder axiom for $\cat{wW}^{\op}$. We will write $\CoCone$ and $\CoShift$ for the associated cocone and coshift functors.

Recall from \cite[Thm. A.5]{Witte:MCVarFF} that the assignment
\begin{equation}\label{eqn:def of boundary hom}
\begin{aligned}
\bh(\Delta)&\coloneqq 0 &&\text{for every exact sequence $\Delta$ in
$\cat{wW}$,}\\
\bh(\alpha)&\coloneqq -[\Cone(\alpha)]+[\Cone(\id_A)] &&\text{for every
weak equivalence $\alpha\colon A\wto A'$ in $\cat{wW}$}
\end{aligned}
\end{equation}
defines a homomorphism $\bh\colon \D_1(\cat{wW})\mto
K_0(\cat{vW^w})$ such that the sequence
$$
\KTh_1(\cat{vW})\mto \KTh_1(\cat{wW})\xrightarrow{d}
\KTh_0(\cat{vW^w})\mto \KTh_0(\cat{vW})\mto \KTh_0(\cat{wW})\mto 0
$$
is exact.

\begin{lem}\label{lem:Cone and Cocone construction}
For every weak equivalence $\alpha\colon A\mto B$ in $\cat{wW}$,
\[
\bh(\alpha)=-[\CoCone(\id_{B})]+[\CoCone(\alpha)]
\]
in $\KTh_0(\cat{vW^w})$.
\end{lem}
\begin{proof}
We first assume that $A$ and $B$ are objects of $\cat{vW^w}$. Then
\begin{align*}
B\cto \Cone(\alpha)\qto \Shift A,\\
A\cto \Cone(\id_A)\qto \Shift A,
\end{align*}
are exact sequences in $\cat{vW^W}$. Hence,
\begin{equation}\label{eqn:case of objects in vWw}
\bh(\alpha)=-[B]-[\Shift A]+[\Shift A]+[A]=-[B]+[A]
\end{equation}
in $\KTh_0(\cat{vW^w})$.

For a general weak equivalence $\alpha\colon A\mto B$ in $\cat{wW}$, the natural morphism
\[
\Cone(\alpha)\mto 0
\]
is a weak equivalence in $\cat{wW}$ by the cylinder axiom. The commutative square
\[
\xymatrix{
A\ar@{=}[r]\ar@{=}[d]&A\ar[d]^{\sim}_{\alpha}\\
A\ar[r]^{\sim}_{\alpha}&B}
\]
induces by the functoriality of the cone and shift functor a commutative diagram with exact rows
\[
\xymatrix{
    A\ar@{ >->}[r]\ar[d]^{\sim}_{\alpha}&
    \Cone(\id_A)\ar@{>>}[r]\ar[d]^{\sim}_{\alpha_*}
    &\Shift A\ar@{=}[d]\\
    B\ar@{ >->}[r]&\Cone(\alpha)\ar@{>>}[r]&\Shift A
    }
\]
where all downward pointing arrows are weak equivalences in $\cat{wW}$. Dually, we also obtain a commutative diagram
\[
\xymatrix{
    \CoShift B\ar@{ >->}[r]\ar@{=}[d]&
    \CoCone(\alpha)\ar@{>>}[r]\ar[d]^{\sim}_{\alpha^*}
    &A\ar[d]^{\sim}_{\alpha}\\
    \CoShift B\ar@{ >->}[r]&\CoCone(\id_B)\ar@{>>}[r]&B
    }
\]
where all downward pointing arrows are weak equivalences in $\cat{wW}$.

From (R6) and \eqref{eqn:case of objects in vWw} we conclude
\[
\begin{aligned}
-[\Cone(\alpha)]+[\Cone(\id_A)]=\bh(\alpha_*)&=\bh(\alpha)\\
&=\bh(\alpha^*)=-[\CoCone(\id_B)]+[\CoCone(\alpha)]
\end{aligned}
\]
as desired.
\end{proof}

\begin{rem}
By basically the same argument, one also sees that $d$ is independent of the choice of the particular cylinder functor.
\end{rem}

\begin{prop}
With the notation as above, the diagram
\[
\xymatrix{
\D_1(wW)\ar[r]^{I}\ar[d]^d&\D_1(wW^\op)\ar[d]^d\\
\KTh_0(vW^w)\ar[r]^{I}&\KTh_0((vW^w)^\op)
}
\]
commutes.
\end{prop}
\begin{proof}
This is a direct consequence of the definition of $I$ and Lemma~\ref{lem:Cone and Cocone construction}.
\end{proof}

If $R$ is any ring and $\cmplx{P}$ is a strictly perfect complex of left $R$-modules, then
\[
(\cmplx{P})^{\mdual_R}\coloneqq \Hom_R(P^{-\bullet},R)
\]
is a strictly perfect complex of left modules over the opposite ring $R^{\op}$ and
\[
\cat{SP}(R)^{\op}\mto \cat{SP}(R^\op) \qquad\cmplx{P}\mapsto (\cmplx{P})^{\mdual_R}
\]
is a Waldhausen exact equivalence of categories. We omit the $R$ from $\mdual_R$ if it is clear from the context. By composing with the homomorphisms $I$, we obtain isomorphisms
\[
\KTh_n(R)\isomorph\KTh_n(R^{\op}).
\]
Note that the isomorphism $\KTh_1(R)\mto\KTh_1(R^\op)$ corresponds to the isomorphism induced by the group isomorphism
\[
\Gl_\infty(R)\mto\Gl_\infty(R^\op),\qquad A\mapsto (A^t)^{-1}
\]
that maps a matrix $A$ to the inverse of its transposed matrix.

If $S$ is a second ring and $\cmplx{M}$ a complex of $R$-$S$-bimodules which is strictly perfect as complex of $R$-modules, then $(\cmplx{M})^{\mdual_R}$ is a complex of $R^\op$-$S^{\op}$-bimodules which is strictly perfect as complex of $R^\op$-modules and there exists for any complex $\cmplx{P}$ in $\cat{SP}(S)$ a canonical isomorphism
\begin{equation}\label{eqn:compatibility of dual with ring changes}
\cmplx{(M^{\mdual_R}\tensor_{S^{\op}}P^{\mdual_S})}\isomorph(\cmplx{(M\tensor_S P)})^{\mdual_R}.
\end{equation}
Hence, we obtain a commutative diagram
\[
\xymatrix{
\KTh_n(S)\ar[r]_{\isomorph}\ar[d]_{\cmplx{M}}&\KTh_n(S^\op)\ar[d]^{(\cmplx{M})^{\mdual_R}}\\
\KTh_n(R)\ar[r]_{\isomorph}&\KTh_n(R^\op)
}
\]

We now return to our previous setting: As before, $\Lambda$ is an adic ring and $G=H\rtimes\Gamma$  is a profinite group such that $H$ contains an open, topologically finitely generated pro-$p$ subgroup and $\Gamma\isomorph\Int_p$.

\begin{defn}\label{defn:Mdual}
We define
\[
\sharp\colon \Lambda[[G]]^{\op}\mto\Lambda^{\op}[[G]],\qquad a\mapsto a^\sharp,
\]
to be the ring homomorphism that is the identity on the coefficients and maps $g\in G$ to $g^{-1}$ and write $\Lambda^{\op}[[G]]^\sharp$ for $\Lambda^{\op}[[G]]$ considered as $\Lambda^{\op}[[G]]$-$\Lambda[[G]]^{\op}$-bimodule via $\sharp$.
Further, we let
\[
\Mdual\colon\cat{SP}(\Lambda[[G]])^{\op}\mto\cat{SP}(\Lambda^{\op}[[G]]),\qquad \cmplx{P}\mapsto \Lambda^{\op}[[G]]^{\sharp}\tensor_{\Lambda[[G]]^{\op}}(\cmplx{P})^{\mdual}.
\]
denote the resulting Waldhausen exact equivalence of categories. We also write
\begin{equation}\label{eqn:def of Mdual on K-groups}
\Mdual\colon \KTh_n(\Lambda[[G]])\mto\KTh_n(\Lambda^{\op}[[G]])
\end{equation}
for the homomorphisms obtained by composing $I$ from \eqref{eqn:identification of W and Wop} with the homomorphism
\[
\KTh_n(\cat{SP}(\Lambda[[G]])^{\op})\mto\KTh_n(\cat{SP}(\Lambda^{\op}[[G]]))
\]
induced by the Waldhausen exact functor $\Mdual$.
\end{defn}

\begin{rem}
The author does not know wether it is possible to produce an extension of $\Mdual$ to a Waldhausen exact functor  $\cat{PDG}^{\cont}(\Lambda[[G]])^{\op}\mto\cat{PDG}^{\cont}(\Lambda^{\op}[[G]])$ inducing the same homomorphisms on $\KTh$-theory. This would avoid some technicalities that we need to deal with later on.
\end{rem}

\begin{lem}\label{lem:Mdual and ringtransf}
Assume that $\cmplx{K}$ is in $\cat{SP}(\Lambda[[G]])$.
\begin{enumerate}
    \item Let $\Lambda'$ be another adic $\Int_p$-algebra. For any complex $\cmplx{P}$ of $\Lambda'$-$\Lambda[[G]]$-bimodules, strictly perfect as complex of $\Lambda'$-modules, set
     \[
     (\cmplx{P})^{\mdual,\sharp}\coloneqq(\cmplx{P})^{\mdual_{\Lambda}}\tensor_{\Lambda[[G]]^\op}(\Lambda[[G]]^{\op})^{\sharp}
     \]
     such that $(\cmplx{P})^{\mdual,\sharp}$ is a complex of $\Lambda'^\op$-$\Lambda^\op[[G]]$-bimodules, with $g\in G$ acting by $(g^{-1})^*$. With $\cmplx{P[[G]]^\delta}$ as in Example~\ref{exmpl:example functors},
     \[
      (\ringtransf_{\cmplx{P[[G]]^\delta}}(\cmplx{K}))^{\Mdual}\isomorph
      \ringtransf_{(\cmplx{P})^{\mdual_{\Lambda},\sharp}[[G]]^\delta}((\cmplx{K})^{\Mdual}).
     \]
    \item Let $G'=H'\rtimes\Gamma'$ be another profinite group such that $H'$ contains an open, topologically finitely generated pro-$p$ subgroup and $\Gamma'\isomorph\Int_p$. Let $\alpha\colon G\mto G'$ be a continuous homomorphism such that $\alpha(G)\not\subset H'$. Consider $\Lambda[[G']]$ as a $\Lambda[[G']]$-$\Lambda[[G]]$-bimodule. Then
     \[
      (\ringtransf_{\Lambda[[G']]}(\cmplx{K}))^{\Mdual}\isomorph\ringtransf_{\Lambda^\op[[G']]}((\cmplx{K})^\Mdual).
     \]
    \item Assume that $G'$ is an open subgroup of $G$ and set $H'\coloneqq H\cap G'$. Consider $\Lambda[[G]]$ as a $\Lambda[[G']]$-$\Lambda[[G]]$-bimodule. Then
     \[
      (\ringtransf_{\Lambda[[G]]}(\cmplx{K}))^{\Mdual}\isomorph\ringtransf_{\Lambda^\op[[G]]}((\cmplx{K})^\Mdual).
     \]
\end{enumerate}
\end{lem}
\begin{proof}
Using the canonical isomorphism \eqref{eqn:compatibility of dual with ring changes}, it remains to notice that
\[
 \Lambda'^{\op}[[G]]^\sharp\tensor_{\Lambda'[[G]]^\op}(\cmplx{{P[[G]]^\delta}})^{\mdual_{\Lambda'[[G]]}}\isomorph(\cmplx{P})^{\mdual,\sharp}[[G]]^\delta\tensor_{\Lambda^\op[[G]]}\Lambda^\op[[G]]^\sharp
\]
as complexes of $\Lambda'^{\op}[[G]]$-$\Lambda[[G]]^\op$-bimodules to prove $(1)$. The other two parts are straightforward.
\end{proof}

\begin{prop}\label{prop:main abstract duality}
The functor $\Mdual$ extends to Waldhausen exact equivalences
\begin{align*}
\Mdual\colon (w_H\cat{SP}(\Lambda[[G]]))^{\op}&\mto w_H\cat{SP}(\Lambda^{\op}[[G]]),\\
\Mdual\colon (\cat{SP}^{w_H}(\Lambda[[G]]))^{\op}&\mto \cat{SP}^{w_H}(\Lambda^{\op}[[G]])
\end{align*}
and hence, it induces a commutative diagram
\[
\xymatrix{
0\ar[r]&\KTh_1(\Lambda[[G]])\ar[r]\ar[d]^{\Mdual}_{\isomorph}&\KTh_1(\Lambda[[G]]_S)\ar[d]^{\Mdual}_{\isomorph}\ar[r]^{d}&
\KTh_0(\Lambda[[G]],S)\ar[r]\ar[d]^{\Mdual}_{\isomorph}& 0\\
0\ar[r]&\KTh_1(\Lambda^{\op}[[G]])\ar[r]&\KTh_1(\Lambda^{\op}[[G]]_S)\ar[r]^{d}&\KTh_0(\Lambda^{\op}[[G]],S)\ar[r]& 0
}
\]
with exact rows.
\end{prop}
\begin{proof}
The exactness of the rows follows from \cite[Cor.~3.3]{Witte:Splitting}.
To extend $\Mdual$, it suffices to show that for any strictly perfect complex $\cmplx{P}$ of $\Lambda[[G]]$-modules which is also perfect as complex of $\Lambda[[H]]$-modules, the complex $(\cmplx{P})^\Mdual$ is perfect as complex of $\Lambda^{\op}[[H]]$-modules. By \cite[Prop. 4.8]{Witte:MCVarFF} we may check this after tensoring with $(\Lambda/\Jac(\Lambda))^{\op}[[G/V]]$ with $V\subset G$ a closed normal pro-$p$-subgroup which is open in $H$. Using Lemma~\ref{lem:Mdual and ringtransf}, we may therefore assume that $\Lambda$ and $H$ are finite.

By Lemma~\ref{lem:ore set is right and left}, $S\coloneqq S_{\Lambda[[G]]}\subset\Lambda[[G]]$ is a left and right denominator set and $\sharp$ maps $S_{\Lambda[[G]]}$ to the set $S_{\Lambda^{\op}[[G]]}\subset\Lambda^{\op}[[G]]$. Moreover $(\cmplx{P})^\Mdual$ is perfect as complex of $\Lambda^\op[[H]]$-modules if and only if its cohomology is $S_{\Lambda^{\op}[[G]]}$-torsion.

As $\cmplx{P}$ has $S_{\Lambda[[G]]}$-torsion cohomology and as
\[
 (\Lambda[[G]]_S)^\op\tensor_{\Lambda[[G]]^\op}(\cmplx{P})^{\mdual_{\Lambda[[G]]}}\isomorph (\Lambda[[G]]_S\tensor_{\Lambda[[G]]}\cmplx{P})^{\mdual_{\Lambda[[G]]_S}},
\]
we conclude that $(\cmplx{P})^\Mdual$ is indeed perfect as complex of $\Lambda^\op[[H]]$-modules.
\end{proof}

We may also extend $\Mdual$ to the Waldhausen category $\cat{SP}(\Lambda[[H]],G)$ from Definition~\ref{defn:SP LambdaH G}.

\begin{defn}
For any $\cmplx{P}$ in $\cat{SP}(\Lambda[[H]],G)^{\op}$ we set
\[
 (\cmplx{P})^\Mdual\coloneqq\Lambda^\op[[H]]^\sharp\tensor_{\Lambda[[H]]^\op}(\cmplx{P})^{\mdual_{\Lambda[[H]]}}[-1].
\]
and define a left action of $\gamma\in\Gamma$ on $(\cmplx{P})^\Mdual$ by
\[
 \gamma\colon(\cmplx{P})^\Mdual\mto(\cmplx{P})^\Mdual,\qquad \lambda\tensor f\mapsto \gamma\lambda\gamma^{-1}\tensor \gamma f(\gamma^{-1}\cdot)\gamma^{-1}.
\]
\end{defn}

One checks that $(\cmplx{P})^\Mdual$ is indeed an object of $\cat{SP}(\Lambda^\op[[H]],G)$, such that we obtain a Waldhausen exact equivalence of categories
\[
 \Mdual\colon \cat{SP}(\Lambda[[H]],G)^\op\mto \cat{SP}(\Lambda^\op[[H]],G).
\]

\begin{prop}\label{prop:compatibility of the LambdaH dual}
Let $\gamma\in\Gamma$ be a topological generator. Then for any $\cmplx{P}$ in $\cat{SP}(\Lambda[[H]],G)^{\op}$ we have a commutative diagram
\[
 \xymatrix{
 \Lambda^\op[[G]]\tensor_{\Lambda^\op[[H]]}(\cmplx{P})^\Mdual[1]\ar[d]^{\alpha}_{\isomorph}\ar[rr]^{\id-\cdot\gamma^{-1}\tensor\gamma}&&\Lambda^\op[[G]]\tensor_{\Lambda^\op[[H]]}(\cmplx{P})^\Mdual[1]\ar[d]^{\alpha}_{\isomorph}\\
 (\Lambda[[G]]\tensor_{\Lambda[[H]]}\cmplx{P})^\Mdual\ar[rr]^{\id-(\cdot\gamma\tensor\gamma^{-1}\cdot)^{\Mdual}}&&(\Lambda[[G]]\tensor_{\Lambda[[H]]}\cmplx{P})^\Mdual
 }
\]
in $\cat{SP}(\Lambda^{op}[[G]])$ inducing a canonical isomorphism
\[
 C_{\gamma^{-1}}(\cmplx{P})^{\Mdual}\isomorph C_{\gamma}((\cmplx{P})^{\Mdual}).
\]
In particular, we have
\[
 [\cmplx{P}]^\Mdual=[(\cmplx{P})^\Mdual]
\]
in $\KTh_0(\Lambda^{\op}[[G]],S)$.
\end{prop}
\begin{proof}
For any degree $n$ and any $f\in (P^n)^{\mdual_{\Lambda[[H]]}}$, we write
\[
 \tilde{f}\colon \Lambda[[G]]\tensor_{\Lambda[[H]]}P^n\mto\Lambda[[G]],\qquad \lambda\tensor p\mapsto \lambda f(p).
\]
We then set
\[
 \xymatrix{
 \Lambda^\op[[G]]\tensor_{\Lambda^\op[[H]]}\Lambda^\op[[H]]^\sharp\tensor_{\Lambda[[H]]^\op}(P^n)^{\mdual_{\Lambda[[H]]}}\ar[d]^\isomorph_{\alpha}& \lambda\tensor\mu\tensor f \ar@{|->}[d]\\
 \Lambda^\op[[G]]^\sharp\tensor_{\Lambda[[G]]^\op}(\Lambda[[G]]\tensor_{\Lambda[[H]]}P^n)^{\mdual_{\Lambda[[G]]}} &\lambda\tensor\tilde{f}\mu^{\sharp}\\
 }
\]
It is then straightforward to check that the above diagram commutes. The cone of the first row is $C_{\gamma}((\cmplx{P})^{\Mdual})[1]$, the cone of the second row is the $\Mdual$-dual of the cocone of
\[
 \Lambda[[G]]\tensor_{\Lambda[[H]]}\cmplx{P}\xrightarrow{\id-(\cdot\gamma\tensor\gamma^{-1}\cdot)}\Lambda[[G]]\tensor_{\Lambda[[H]]}\cmplx{P}
\]
in $\cat{SP}^{w_H}(\Lambda[[G]])$, which is the same as $C_{\gamma^{-1}}(\cmplx{P})[-1]$. Finally, recall from Proposition~\ref{prop:comparison with strictly perfect LamdaH-complexes} that the class of $C_{\gamma}(\cmplx{P})$ in $\KTh_0(\Lambda[[G]],S)$ is independent of the choice of the topological generator $\gamma$. Hence,
\[
 [\cmplx{P}]^\Mdual=[C_{\gamma}(\cmplx{P})^{\Mdual}]=[C_{\gamma^{-1}}(\cmplx{P})^{\Mdual}]=[C_{\gamma}((\cmplx{P})^{\Mdual})]=[(\cmplx{P})^\Mdual].
\]
\end{proof}

\section{Non-commutative algebraic \texorpdfstring{$L$}{L}-functions}\label{sec:algebraic Lfunctions}

Let $G=H\rtimes \Gamma$ as before. Recall the split exact sequence
\[
0\mto \KTh_1(\Lambda[[G]])\mto \KTh_1(\Lambda[[G]]_S)\xrightarrow{\bh} \KTh_0(\Lambda[[G]],S)\mto 0.
\]
\cite[Cor.~3.4]{Witte:Splitting}, which is central for the formulation of the non-commutative main conjecture: The map $\KTh_1(\Lambda[[G]])\mto \KTh_1(\Lambda[[G]]_S)$ is the obvious one; the boundary map
\[
\bh\colon\KTh_1(\Lambda[[G]]_S)\mto\KTh_0(\Lambda[[G]],S)
\]
on the class $[f]$ of an endomorphism $f$ which is a weak equivalence in the Waldhausen category $w_H\cat{PDG}^{\cont}(\Lambda[[G]])$ is given by
\[
\bh[f]=-[\cmplx{\Cone(f)}]
\]
where $\cmplx{\Cone(f)}$ denotes the cone of $f$ \cite[Thm.~A.5]{Witte:MCVarFF}. (Note that other authors use $-\bh$ instead.) For a fixed choice of a topological generator $\gamma\in\Gamma$, a splitting $s_\gamma$ of $\bh$ is given by
\begin{equation}\label{eqn:splitting of del}
s_\gamma([\cmplx{P}])\coloneqq[\Lambda[[G]]\ctensor_{\Lambda[[H]]}\cmplx{P}\xrightarrow{x\ctensor y\mapsto x\ctensor y-x\gamma^{-1}\ctensor\gamma y}\Lambda[[G]]\ctensor_{\Lambda[[H]]}\cmplx{P}]^{-1}
\end{equation}
for any $\cmplx{P}$ in $\cat{PDG}^{\cont,w_H}(\Lambda[[G]])$, where the precise definition of $\Lambda[[G]]\ctensor_{\Lambda[[H]]}\cmplx{P}$ as an object of the Waldhausen category $w_H\cat{PDG}^{\cont}(\Lambda[[G]])$ is
\[
\Lambda[[G]]\ctensor_{\Lambda[[H]]}\cmplx{P}=(\varprojlim_{J\in\openideals_{\Lambda[[G]]}}\Lambda[[G]]/I\otimes_{\Lambda[[H]]}\cmplx{P}_J)_{I\in\openideals_{\Lambda[[G]]}}
\]
\cite[Def.~2.12]{Witte:Splitting}. A short inspection of the definition shows that $s_\gamma$ only depends on the image of $\gamma$ in $G/H$. Following \cite{Burns:AlgebraicLfunctions}, we may call $s_\gamma(-A)$ the \emph{non-commutative algebraic $L$-function} of $A\in \KTh_0(\Lambda[[G]],S)$.

\begin{prop}\label{prop:transformation for algebraic L-function}
Consider an element $A\in \KTh_0(\Lambda[[G]],S)$.
\begin{enumerate}
\item Let $\Lambda'$ be another adic $\Int_{p}$-algebra.
For any complex $\cmplx{P}$ of $\Lambda'$-$\Lambda[[G]]$-bimodules which is strictly perfect as complex of $\Lambda'$-modules we have
\[
\ringtransf_{\cmplx{P[[G]]^{\delta}}}(s_\gamma(A))=s_\gamma(\ringtransf_{\cmplx{P[[G]]^{\delta}}}(A))
\]
in $\KTh_1(\Lambda'[[G]]_S)$.

\item Let $G'=H'\rtimes \Gamma'$ such that $H'$ has an open, topologically finitely generated pro-$p$-subgroup and $\Gamma'\isomorph\Int_p$. Assume that $\alpha\colon G\mto G'$ is a continuous homomorphism such that $\alpha(G)\not\subset H'$. Set $r\coloneqq [G': \alpha(G)H']$. Let $\gamma'\in \Gamma'$ be a topological generator such that $\alpha(\gamma)=(\gamma')^r$ in $G'/H'$. Then
$$
\ringtransf_{\Lambda[[G']]}(s_\gamma(A))=s_{\gamma'}(\ringtransf_{\Lambda[[G']]}(A))
$$
in $\KTh_1(\Lambda[[G']]_S)$.

\item Assume that $G'$ is an open subgroup of $G$ and set $H'\coloneqq H\cap G'$, $r\coloneqq [G: G'H]$. Consider $\Lambda[[G]]$ as a $\Lambda[[G']]$-$\Lambda[[G]]$-bimodule. Then $\gamma^r$ generates $G'/H'\subset G/H$ and
   $$
   \ringtransf_{\Lambda[[G]]}(s_\gamma(A))=s_{\gamma^r}(\ringtransf_{\Lambda[[G]]}(A))
  $$
   in $\KTh_1(\Lambda[[G]]_S)$.
\end{enumerate}
\end{prop}
\begin{proof}
For (1), we first note that by applying the Waldhausen additivity theorem \cite[Prop. 1.3.2]{Wal:AlgKTheo} to the short exact sequences resulting from stupid truncation, we have
\[
\ringtransf_{\cmplx{P[[G]]^{\delta}}}=\sum_{i\in\Int}(-1)^i\ringtransf_{P^i[[G]]^{\delta}}
\]
as homomorphisms between the $\KTh$-groups. Hence we may assume that $P=\cmplx{P}$ is concentrated in degree $0$.
We now apply \cite[Prop 2.14.1]{Witte:Splitting} to the $\Lambda'[[G]]$-$\Lambda[[G]]$-bimodule $M\coloneqq P[[G]]^{\delta}$ and its
$\Lambda'[[H]]$-$\Lambda[[H]]$-sub-bimodule
\[
N\coloneqq \Lambda'[[H]]\tensor_{\Lambda'}P
\]
(with the diagonal right action of $H$) and $t_1\coloneqq t_2\coloneqq\gamma-1$, $\gamma_1\coloneqq \gamma_2\coloneqq\gamma$.

For (2), we first assume that $\alpha$ induces an isomorphism $G/H\isomorph G/H'$ and that $\gamma'=\alpha(\gamma)$. We then apply \cite[Prop 2.14.1]{Witte:Splitting} to $M\coloneqq\Lambda[[G']]$, $N\coloneqq\Lambda[[H']]$, and $t_1\coloneqq\gamma-1$, $t_2\coloneqq\alpha(\gamma)-1$, $\gamma_1\coloneqq\gamma$, $\gamma_2\coloneqq\alpha(\gamma)$.

Next, we assume that $G\subset G'$, $H=H'$, and $\gamma=(\gamma')^r$. This case is not covered by \cite[Prop 2.14]{Witte:Splitting} and therefore, we will give more details. Consider the isomorphism of $\Lambda[[G']]$-$\Lambda[[G]]$-bimodules
\begin{equation*}
\begin{split}
 \kappa\colon \Lambda[[G']]\ctensor_{\Lambda[[H]]}\Lambda[[G]]^r&\mto\Lambda[[G']]\ctensor_{\Lambda[[H]]}\Lambda[[G']],\\
 \mu\ctensor\begin{pmatrix} \lambda_0 \\ \vdots \\ \lambda_{r-1} \end{pmatrix}&\mapsto \sum_{i=0}^{r-1}\mu(\gamma')^{-i}\ctensor(\gamma')^i(\lambda_i).
\end{split}
\end{equation*}
Then the map $\mu\ctensor\lambda\mapsto \mu\ctensor\lambda-\mu(\gamma')^{-1}\ctensor\gamma'\lambda$ on the righthand side corresponds to left multiplication with the matrix
\[
 A\coloneqq\begin{pmatrix}
   \id     & 0      & \hdots & 0     &  -(\cdot\gamma^{-1})\ctensor(\gamma\cdot) \\
  -\id     & \id    & \ddots &\vdots & 0  \\
    0      & \ddots & \ddots &  0    & \vdots  \\
    \vdots & \ddots & \ddots & \id   &  0 \\
    0      & \hdots &   0    & -\id  & \id
 \end{pmatrix}
\]
on the left-hand side. Let $\cmplx{P}$ be a complex in $\cat{PDG}^{\cont,w_H}(\Lambda[[G]])$. Then $\kappa$ induces an isomorphism
\[
 \kappa\colon \ringtransf_{\Lambda[[G']]}(\Lambda[[G]]\ctensor_{\Lambda[[H]]}(\cmplx{P})^r)\mto \Lambda[[G']]\ctensor_{\Lambda[[H]]}\ringtransf_{\Lambda[[G']]}(\cmplx{P})
\]
in $w_H\cat{PDG}^{\cont}(\Lambda[[G']])$ while $A\lcirclearrowright \ringtransf_{\Lambda[[G']]}(\Lambda[[G]]\ctensor_{\Lambda[[H]]}(\cmplx{P})^r)$ is a weak equivalence. Hence,
\[
[A]^{-1}=s_{\gamma'}([\ringtransf_{\Lambda[[G']]}(\cmplx{P})])
\]
in $\KTh_1(\Lambda[[G']]_S)$. Moreover,
\[
\begin{pmatrix}
   \id     & 0      & \hdots & \hdots& 0 \\
   \id     & \id    & \ddots &       & \vdots  \\
    \vdots & \ddots & \ddots &\ddots & \vdots  \\
    \vdots &        & \ddots & \id   &  0 \\
   \id     & \hdots & \hdots & \id   & \id
 \end{pmatrix}A=
 \begin{pmatrix}
   \id     & 0      & \hdots & 0     &   -(\cdot\gamma^{-1})\ctensor(\gamma\cdot)\\
    0      & \id    & \ddots &\vdots &  -(\cdot\gamma^{-1})\ctensor(\gamma\cdot) \\
    \vdots & \ddots & \ddots &  0    & \vdots  \\
    \vdots &        & \ddots & \id   &  -(\cdot\gamma^{-1})\ctensor(\gamma\cdot) \\
    0      & \hdots & \hdots & 0  & \id-(\cdot\gamma^{-1})\ctensor(\gamma\cdot)
 \end{pmatrix}.
\]
The relations (R1)--(R7) in the definition of $\D_\bullet(\cat{W})$ imply that the class of a triangular matrix is the product of the classes of its diagonal entries in $\KTh_1(\Lambda[[G']]_S)$. Hence, $[A]^{-1}=\ringtransf_{\Lambda[[G']]}(s_\gamma[\cmplx{P}])$, as desired.

In the general case, we note that the image of $\alpha$ is contained in the  subgroup  $G''$ of $G'$ topologically generated by $(\gamma')^r$ and $H'$ and recall that $s_\gamma$ only depends on the image of $\gamma$ in $G/H$. We are then reduced to the two cases already treated above.

For (3), we first treat the case $r=1$, i.\,e.\ $G'\mto G/H$ is a surjection. Hence, we may assume $\gamma\in G'$. We then apply \cite[Prop 2.14.1]{Witte:Splitting} to $M\coloneqq\Lambda[[G]]$, $N\coloneqq\Lambda[[H]]$, and $t_1\coloneqq t_2\coloneqq\gamma-1$, $\gamma_1\coloneqq \gamma_2\coloneqq\gamma$ as above. If $r>1$ we can thus reduce to the case that $G'$ is topologically generated by $H$ and $\gamma^r$ and apply \cite[Prop 2.14.2]{Witte:Splitting}.

In \cite{Witte:Splitting}, we use a slightly different Waldhausen category for the construction of the $\KTh$-theory of $\Lambda[[G]]$, but the proof of \cite[Prop 2.14]{Witte:Splitting} goes through without changes.
\end{proof}

\begin{exmpl}\label{exmpl:algebraic L-functions}\
\begin{enumerate}
 \item Assume that $M$ is a $\Lambda[[G]]$-module which is finitely generated and projective as a $\Lambda[[H]]$-module. Then the complex
\[
 C_{\gamma}(M)\colon\qquad\underbrace{\Lambda[[G]]\tensor_{\Lambda[[H]]}M}_{\text{degree $-1$}}\xrightarrow{\id-(\cdot\gamma^{-1}\tensor\gamma\cdot)}\underbrace{\Lambda[[G]]\tensor_{\Lambda[[H]]}M}_{\text{degree $0$}}
\]
is an object of $\cat{PDG}^{\cont,w_H}(\Lambda[[G]])$ whose cohomology is $M$ in degree $0$ and zero otherwise. Moreover,
\[
 s_\gamma([M])=s_{\gamma}([C_{\gamma}(M)])=[\id-(\cdot\gamma^{-1}\tensor\gamma\cdot)\lcirclearrowright\Lambda[[G]]\tensor_{\Lambda[[H]]}M]^{-1}
\]
in $\KTh_1(\Lambda[[G]]_S)$. If $\Lambda[[G]]$ is commutative, then the image of the element $s_{\gamma}([M])^{-1}$ under
\[
\det\colon \KTh_1(\Lambda[[G]]_S)\mto\Lambda[[G]]_S^\times
\]
 is precisely the reverse characteristic polynomial
\[
 \det\nolimits_{\Lambda[[H]][t]}(\id-(\cdot t\tensor\gamma\cdot )\lcirclearrowright \Lambda[[H]][t]\tensor_{\Lambda[[H]]} M)
\]
evaluated at $t=\gamma^{-1}\in\Gamma$. In fact, one may extend this to non-commutative $\Lambda[[H]]$ and $G=H\times\Gamma$ as well, using the results of the appendix.

\item If $M=\Lambda[[G]]/\Lambda[[G]]f$ with
\[
 f\coloneqq t^n+\sum_{i=0}^{n-1}\lambda_it^{i}\in\Lambda[[G]]
\]
a polynomial of degree $n$ in $t\coloneqq\gamma-1$ with  $\lambda_i\in\Jac(\Lambda[[H]])$, then $M$ is finitely generated and free as $\Lambda[[H]]$-module. A $\Lambda[[H]]$-basis is given by the residue classes of $1,t,\dots,t^{n-1}\in\Lambda[[G]]$. If we use this basis to identify $\Lambda[[G]]\tensor_{\Lambda[[H]]}M$ with $\Lambda[[G]]^n$, then the $\Lambda[[G]]$-linear endomorphism $\id-(\cdot\gamma^{-1}\tensor\gamma\cdot)$ is given by right multiplication with the matrix
\[
 A\coloneqq\begin{pmatrix}
  \gamma^{-1}t         &  -\gamma^{-1}        & 0      & \hdots                   & 0 \\
  0                    &  \gamma^{-1}t        & \ddots & \ddots                   & \vdots \\
  \vdots               & \ddots               & \ddots & \ddots                   & 0\\
  0                    & \hdots               &  0     & \gamma^{-1}t             & -\gamma^{-1}\\
  \gamma^{-1}\lambda_0 & \gamma^{-1}\lambda_1 & \hdots & \gamma^{-1}\lambda_{n-2} & \gamma^{-1}(t+\lambda_{n-1})
 \end{pmatrix}.
\]
By right multiplication with
\[
E\coloneqq
\begin{pmatrix}
 1      & 0      & \hdots & \hdots & 0\\
 t      & 1      & \ddots &        & \vdots\\
 t^2    & \ddots & \ddots & \ddots & \vdots\\
 \vdots & \ddots & \ddots & \ddots & 0\\
 t^{n-1}& \hdots &  t^2   & t      & 1\\
\end{pmatrix}
\]
 one can transform $A$ into
\[
A'\coloneqq\begin{pmatrix}
  0            &  -\gamma^{-1} & 0       & \hdots  & 0 \\
  \vdots       &  0            & \ddots  & \ddots  & \vdots \\
  \vdots       & \vdots        & \ddots  & \ddots  & 0\\
  0            &  0            & \hdots  &  0       & -\gamma^{-1}\\
  \gamma^{-1}f & \gamma^{-1}(t^{n-1}+\sum_{i=1}^{n-1}\lambda_it^{i-1}) & \hdots & \gamma^{-1}(t^2+\lambda_{n-1}t+\lambda_{n-2}) & \gamma^{-1}(t+\lambda_{n-1})
 \end{pmatrix}.
\]
By left multiplication with
\[
 P\coloneqq
 \begin{pmatrix}
 0      &  0     & \hdots & 0      & 1      \\
 0      &  1     & \ddots & \vdots & 0 \\
 \vdots & \ddots & \ddots & 0      & \vdots \\
 0      & \hdots & 0      & 1      & 0 \\
 1      & 0      & \hdots & 0      & 0
 \end{pmatrix}
\]
one can exchange the first and last row of $A'$ to obtain a triangular matrix. In $\KTh_1(\Lambda[[G]])$, we have
\begin{align*}
[\;\cdot E\lcirclearrowright \Lambda[[G]]^n\;]&=1,\\
[\;\cdot P\lcirclearrowright \Lambda[[G]]^n\;]&=[\;-1\lcirclearrowright \Lambda[[G]]\;]^{n-1}.
\end{align*}
We conclude
\[
\begin{aligned}
s_{\gamma}([M])^{-1}&=[\;\cdot A\lcirclearrowright \Lambda[[G]]^n\;]\\
                    &=[\;\cdot A'\lcirclearrowright \Lambda[[G]]^n\;]\\
                    &=[\;\cdot P\lcirclearrowright \Lambda[[G]]^n\;]^{-1}[\;\cdot (-\gamma^{-1})\lcirclearrowright \Lambda[[G]]\;]^{n-1}[\;\cdot \gamma^{-1}f\lcirclearrowright \Lambda[[G]]\;]\\
                    &=[\;\cdot\gamma^{-n}f \lcirclearrowright \Lambda[[G]]\;].
\end{aligned}
\]
\end{enumerate}
\end{exmpl}

The section $s_\gamma\colon \KTh_0(\Lambda[[G]],S)\mto \KTh_1(\Lambda[[G]]_S)$ also commutes with the homomorphisms $\Mdual\colon \KTh_0(\Lambda[[G]])\mto \KTh_0(\Lambda^\op[[G]])$, $\Mdual\colon \KTh_1(\Lambda[[G]]_S)\mto \KTh_0(\Lambda^\op[[G]]_S)$ from Definition~\ref{defn:Mdual} in the following sense.

\begin{prop}\label{prop:s and Mdual commute}
For any element $A\in \KTh_0(\Lambda[[G]],S)$,
\[
s_{\gamma^{-1}}(A)^\Mdual=s_{\gamma}(A^\Mdual)
\]
in $\KTh_1(\Lambda^\op[[G]]_S)$.
\end{prop}
\begin{proof}
Since $\KTh_0(\cat{SP}(\Lambda[[H]],G))$ surjects onto $\KTh_0(\Lambda[[G]],S)$ by Proposition~\ref{prop:comparison with strictly perfect LamdaH-complexes}, it suffices to prove the formula for $C_\gamma(M)$ with $M$ a $\Lambda[[G]]$-module that is finitely generated and projective over $\Lambda[[H]]$. The equality is then a direct consequence of the diagram in Proposition~\ref{prop:compatibility of the LambdaH dual}.
\end{proof}

\begin{rem}
Note that
\[
 s_\gamma([C_\gamma(M)])=s_{\gamma^{-1}}([C_\gamma(M)])[-(\cdot\gamma^{-1}\tensor\gamma\cdot)\lcirclearrowright \Lambda[[G]]\tensor_{\Lambda[[H]]}M]
\]
for any topological generator $\gamma$ of $\Gamma$ and any $\Lambda[[G]]$-module $M$ that is finitely generated and projective over $\Lambda[[H]]$.
\end{rem}

\section{Perfect complexes of adic sheaves}\label{sec:perfect complexes of adic sheaves}

We will use \'etale cohomology instead of Galois cohomology to formulate the main conjecture. The main advantage is that we have a little bit more flexibility in choosing our coefficient systems. Instead of being restricted to locally constant sheaves corresponding to Galois modules, we can work with constructible sheaves. An alternative would be the use of cohomology for Galois modules with local conditions.

As Waldhausen models for the derived categories of complexes of constructible sheaves, we will use the Waldhausen categories introduced in \cite[\S~5.4--5.5]{Witte:PhD} for separated schemes of finite type over a finite field. The same constructions still work with some minor changes if we consider subschemes of the spectrum of a number ring.

Fix an odd prime $p$. Let $F$ be a number field with ring of integers $\IntR_F$ and assume that $U$ is an open or closed subscheme of $X\coloneqq\Spec \IntR_F$. Recall that for a finite ring $R$, a complex $\cmplx{\sheaf{F}}$ of \'etale sheaves of left $R$-modules on $U$ is called \emph{strictly perfect} if it is strictly bounded and each $\sheaf{F}^n$ is constructible and flat. It is \emph{perfect} if it is quasi-isomorphic to a strictly perfect complex. We call it \emph{$DG$-flat} if for each geometric point of $U$, the complex of stalks is $DG$-flat.

Let $\Lambda$ be an adic $\Int_p$-algebra.

\begin{defn}\label{defn:PDGcont(X,Lambda)}
The \emph{category $\cat{PDG}^{\cont}(U,\Lambda)$ of perfect complexes of adic
sheaves on $U$} is the following
Waldhausen category. The objects of $\cat{PDG}^{\cont}(U,\Lambda)$
are inverse systems $(\cmplx{\sheaf{F}}_I)_{I\in
\openideals_{\Lambda}}$ such that:
\begin{enumerate}
\item for each $I\in\openideals_{\Lambda}$, $\cmplx{\sheaf{F}}_I$
is $DG$-flat perfect complex of \'etale sheaves of $\Lambda/I$-modules on $U$,

\item for each $I\subset J\in\openideals_{\Lambda}$, the
transition morphism
$$
\varphi_{IJ}:\cmplx{\sheaf{F}}_I\mto \cmplx{\sheaf{F}}_J
$$
of the system induces an isomorphism
$$
\Lambda/J\tensor_{\Lambda/I}\cmplx{\sheaf{F}}_I\wto
\cmplx{\sheaf{F}}_J.
$$
\end{enumerate}
Weak equivalences and cofibrations are defined as in Definition~\ref{defn:PDG(Lambda)}.
\end{defn}

\begin{defn}
Any system $\sheaf{F}=(\sheaf{F}_I)_{I\in\openideals_\Lambda}$ in $\cat{PDG}^{\cont}(U,\Lambda)$ consisting of flat, constructible sheaves $\sheaf{F}_I$ of $\Lambda/I$-modules on $U$, regarded as complexes concentrated in degree $0$, will be called a \emph{$\Lambda$-adic sheaf} on $U$. If in addition, the $\sheaf{F}_I$ are locally constant, we call $\sheaf{F}$ a \emph{smooth $\Lambda$-adic sheaf}. We write $\cat{S}(U,\Lambda)$ and $\cat{S}^{\sm}(U,\Lambda)$ for the full Waldhausen categories of $\cat{PDG}^{\cont}(U,\Lambda)$ consisting of $\Lambda$-adic sheaves and smooth $\Lambda$-adic sheaves, respectively.
\end{defn}

\begin{defn}
If $U$ is an open dense subscheme of $\Spec \IntR_F$, we will call a complex $(\cmplx{\sheaf{F}}_I)_{I\in\openideals_{\Lambda}}$ in $\cat{PDG}^{\cont}(U,\Lambda)$ to be \emph{smooth at $\infty$} if for each $I\in\openideals_{\Lambda}$, the stalk of $\cmplx{\sheaf{F}}_I$ in $\Spec \algc{F}$ is quasi-isomorphic to a strictly perfect complex of $\Lambda/I$-modules with trivial action of any complex conjugation $\sigma\in \Gal_F$. The full subcategory of $\cat{PDG}^{\cont}(U,\Lambda)$ of complexes smooth at $\infty$ will be denoted by
\[
\cat{PDG}^{\cont,\infty}(U,\Lambda)
\]
\end{defn}

Since we assume $p\neq 2$, it is immediate that if in an exact sequence
\[
0\mto \cmplx{\sheaf{F}}\mto\cmplx{\sheaf{G}}\mto\cmplx{\sheaf{H}}\mto 0
\]
in $\cat{PDG}^{\cont}(U,\Lambda)$, the complexes $\cmplx{\sheaf{F}}$ and $\cmplx{\sheaf{H}}$ are smooth at $\infty$, then so is $\cmplx{\sheaf{G}}$. It then follows from \cite[Prop. 3.1.1]{Witte:PhD} that $\cat{PDG}^{\cont,\infty}(U,\Lambda)$ is a Waldhausen subcategory of $\cat{PDG}^{\cont}(U,\Lambda)$.

We will write $\Lambda_U$ for the smooth $\Lambda$-adic sheaf on $U$ given by the system of constant sheaves $(\Lambda/I)_{I\in\openideals_{\Lambda}}$ on $U$. Further, if $p$ is invertible on $U$, we will write $\mu_{p^n}$ for the sheaf of $p^n$-th roots of unity on $U$, and
\[
 (\cmplx{\sheaf{F}}_I)_{I\in\openideals_{\Lambda}}(1)\coloneqq(\varprojlim_{n}\mu_{p^n}\otimes_{\Int_p}\cmplx{\sheaf{F}}_I)_{I\in\openideals_{\Lambda}}
\]
for the Tate twist of a complex in $\cat{PDG}^{\cont}(U,\Lambda)$.

We will consider Godement resolutions of the complexes in $\cat{PDG}^{\cont}(U,\Lambda)$. To be explicit,
we will fix an algebraic closure $\algc{F}$ of $F$ and for each place $x$ of $F$ an embedding $\algc{F}\subset \algc{F}_x$ into a fixed algebraic closure of the local field $F_x$ in $x$.
In particular, we also obtain an embedding of the residue field $k(x)$ of $x$ into the algebraically closed residue field $\algc{k(x)}$ of $\algc{F}_x$ for each closed point $x$ of $U$. We write $\hat{x}$ for the corresponding geometric point $\hat{x}\colon \Spec \algc{k(x)}\mto U$ over $x$ and let $U^0$ denote the set of closed points of $U$.

For each \'etale sheaf $\sheaf{F}$ on $U$ we set
\[
(\God_U\sheaf{F})^n\coloneqq\underbrace{\prod_{u\in U^0}\hat{u}_*\hat{u}^*\dots\prod_{u\in U^0}\hat{u}_*\hat{u}^*}_{n+1}\sheaf{F}
\]
and turn $\cmplx{(\God_U\sheaf{F})}$ into a complex by taking as differentials
\[
\del^n\colon(\God_U\sheaf{F})^n\to(\God_U\sheaf{F})^{n+1}
\]
the alternating sums of the maps induced by the natural transformation
\[
\sheaf{F}\to \prod_{u\in U^0}\hat{u}_*\hat{u}^*\sheaf{F}.
\]
The Godement resolution of a complex of \'etale sheaves is given by the total complex of the corresponding double complex as in \cite[Def.~4.2.1]{Witte:PhD}. The Godement resolution of a complex $(\cmplx{\sheaf{F}}_{I})_{I\in\openideals_{\Lambda}}$ in $\cat{PDG}^{\cont}(U,\Lambda)$ is given by
applying the Godement resolution to each of the complexes $\cmplx{\sheaf{F}}_{I}$ individually.

We may define the total derived section functor
\[
\RDer\Sect(U,\cdot)\colon\cat{PDG}^{\cont}(U,\Lambda)\mto \cat{PDG}^{\cont}(\Lambda)
\]
by the formula
\[
\RDer\Sect(U,(\cmplx{\sheaf{F}}_I)_{I\in\openideals_{\Lambda}})\coloneqq(\Sect(U,\God_U \cmplx{\sheaf{F}}_I))_{I\in\openideals_{\Lambda}}.
\]
This agrees with the usual construction if we consider $(\cmplx{\sheaf{F}}_I)_{I\in\openideals_{\Lambda}}$ as an object of the `derived' category of adic sheaves, e.\,g.\ as defined in \cite{KiehlWeissauer:WeilConjectures} for $\Lambda=\Int_p$. In addition, however, we see that $\RDer\Sect(U,\cdot)$ is a Waldhausen exact functor and hence, induces homomorphisms
\[
\RDer\Sect(U,\cdot)\colon \KTh_n(\cat{PDG}^{\cont}(U,\Lambda))\mto\KTh_n(\Lambda)
\]
for all $n$ \cite[Prop. 4.6.6, Def. 5.4.13]{Witte:PhD}. Here, we use the finiteness and the vanishing in large degrees of the \'etale cohomology groups $\HF^n(U,\sheaf{G})$ for constructible sheaves $\sheaf{G}$ of abelian groups in order to assure that $\RDer\Sect(U,(\cmplx{\sheaf{F}}_I)_{I\in\openideals_{\Lambda}})$ is indeed an object of $\cat{PDG}^{\cont}(\Lambda)$. In particular, for each $I\in\openideals_{\Lambda}$, $\RDer\Sect(U,\cmplx{\sheaf{F}}_I)$ is a perfect complex of $\Lambda/I$-modules. Note that we do not need to assume that $p$ is invertible on $U$ (see the remark after \cite[Thm. II.3.1]{Milne:ADT}).

If $j\colon U\mto V$ is an open immersion, we set
\begin{align*}
j_!(\cmplx{\sheaf{F}}_I)_{I\in\openideals_{\Lambda}}&\coloneqq(j_!\cmplx{\sheaf{F}}_I)_{I\in\openideals_{\Lambda}},\\
\RDer j_*(\cmplx{\sheaf{F}}_I)_{I\in\openideals_{\Lambda}}&\coloneqq(j_*\God_U\cmplx{\sheaf{F}}_I)_{I\in\openideals_{\Lambda}}.
\end{align*}
for any $(\cmplx{\sheaf{F}}_I)_{I\in\openideals_{\Lambda}}\in\cat{PDG}^{\cont}(U,\Lambda)$. While the extension by zero $j_!$ always gives us a Waldhausen exact functor
\[
 j_!\colon \cat{PDG}^{\cont}(U,\Lambda)\mto\cat{PDG}^{\cont}(V,\Lambda),
\]
the total direct image
\[
 \RDer j_*\colon \cat{PDG}^{\cont}(U,\Lambda)\mto \cat{PDG}^{\cont}(V,\Lambda)
\]
is only a well-defined Waldhausen exact functor if $p$ is invertible on $V-U$. If $V-U$ contains places above $p$, then $\RDer j_*(\cmplx{\sheaf{F}}_I)_{I\in\openideals_{\Lambda}}$ is still a system of $DG$-flat complexes compatible in the sense of Definition~\ref{defn:PDGcont(X,Lambda)}.(2), but for $I\in\openideals_{\Lambda}$ the cohomology of the complex of stalks of the complexes $\RDer j_*\cmplx{\sheaf{F}}_I$ in the geometric points over places above $p$ is in general not finite, such that  $\RDer j_*\cmplx{\sheaf{F}}_I$ fails to be a perfect complex. In any case, we may consider $\RDer j_*$ as a Waldhausen exact functor from $\cat{PDG}^{\cont}(U,\Lambda)$ to the Waldhausen category of complexes over the abelian category of inverse systems of \'etale sheaves of $\Lambda$-modules, indexed by $\openideals_{\Lambda}$.

The pullback $f^*$ along a morphism of schemes $f$ and the pushforward $f_*$ along a finite morphism of schemes are also defined as Waldhausen exact functors by degreewise application. No Godement resolution is needed, since these functors are exact on all \'etale sheaves.

As a shorthand, we set
\[
 \RDer\Sectc(U,(\cmplx{\sheaf{F}}_I)_{I\in\openideals_{\Lambda}})\coloneqq\RDer(X,j_!(\cmplx{\sheaf{F}}_I)_{I\in\openideals_{\Lambda}})
\]
for $j\colon U\mto X$ the open immersion into $X=\Spec \IntR_F$. Under our assumption that $p\neq 2$, this agrees with the definition of cohomology with proper support in \cite[\S II.2]{Milne:ADT}. If $F$ is a totally real number field and $(\cmplx{\sheaf{F}}_I(-1))_{I\in\openideals_{\Lambda}}$ is smooth at $\infty$, then it also agrees with the definition in \cite[\S 1.6.3]{FK:CNCIT}, but in general, the two definitions differ by a contribution coming from the complex places.

For any closed point $x$ of $X$ and any complex $\cmplx{\sheaf{F}}$ in $\cat{PDG}^{\cont}(x,\Lambda)$, we set
\[
\RDer\Sect(\hat{x},\cmplx{\sheaf{F}})\coloneqq\Sect(\Spec \algc{k(x)},\hat{x}^*\God_x\cmplx{\sheaf{F}})
\]
and let $\Frob_x\in\Gal(\algc{k(x)}/k(x))$ denote the geometric Frobenius of $k(x)$. We obtain an exact sequence
\[
0\mto\RDer\Sect(x,\cmplx{\sheaf{F}})\mto
\RDer\Sect(\hat{x},\cmplx{\sheaf{F}})\xrightarrow{\id-\Frob_x}\RDer\Sect(\hat{x},\cmplx{\sheaf{F}})\mto 0
\]
in $\cat{PDG}^{\cont}(\Lambda)$ \cite[Prop. 6.1.2]{Witte:PhD}. Note that if $\hat{x}'$ is the geometric point corresponding to  another choice of an embedding $\algc{F}\subset\algc{F}_x$ and if $\Frob'_x$ denotes the associated geometric Frobenius, then there is a canonical isomorphism
\[
\sigma\colon\RDer\Sect(\hat{x},\cmplx{\sheaf{F}})\mto\RDer\Sect(\hat{x}',\cmplx{\sheaf{F}})
\]
such that
\begin{equation}\label{eqn:comparison of Frobenii}
\sigma\circ(\id-\Frob_x)=(\id-\Frob'_x)\circ\sigma.
\end{equation}

At some point, we will also make use of the categories $\cat{PDG}^{\cont}(\Spec F_x,\Lambda)$ for the local fields $F_x$ together with the associated total derived section functors. In this case, one can directly appeal to the constructions in \cite[Ch. 5]{Witte:PhD}. We write $F_x^{\nr}$ for the maximal unramified extension field of $F_x$ in $\algc{F}_x$ and note that we have a canonical identification $\Gal(F_x^{\nr}/F_x)\isomorph\Gal(\algc{k(x)}/k(x))$.

\begin{lem}\label{lem:comparison with local cohomology}
Let $j\colon U\mto V$ denote the open immersion of two open dense subschemes of $X$ and assume that $i\colon x\mto V$ is a closed point in the complement of $U$ not lying over $p$. Write $\eta_x\colon \Spec F_x\mto U$ for the map to the generic point of $U$. Then there exists a canonical chain of weak equivalences
\begin{equation}\label{eqn:comparison with local cohomology I}
\RDer\Sect(\hat{x},i^*\RDer j_*\cmplx{\sheaf{F}})\wto\RDer\Sect(\Spec F_x^{\nr},\eta_x^*\God_U\cmplx{\sheaf{F}})\otw\RDer\Sect(\Spec F_x^{\nr},\eta_x^*\cmplx{\sheaf{F}})
\end{equation}
in $\cat{PDG}^{\cont}(\Lambda)$ compatible with the operation of the Frobenius on each complex and hence, a canonical chain of weak equivalences
\begin{equation}\label{eqn:comparison with local cohomology II}
\RDer\Sect(x,i^*\RDer j_*\cmplx{\sheaf{F}})\wto\RDer\Sect(\Spec F_x,\eta_x^*\God_U\cmplx{\sheaf{F}})
\otw\RDer\Sect(\Spec F_x,\eta_x^*\cmplx{\sheaf{F}})
\end{equation}
in $\cat{PDG}^{\cont}(\Lambda)$.
\end{lem}
\begin{proof}
From \cite[Thm. III.1.15]{Milne:EtCohom} we conclude that for each $I\in\openideals_{\Lambda}$, the complex $\eta_x^*\God_U\cmplx{\sheaf{F}_I}$ is a complex of flabby sheaves on $\Spec F_x$ and that
\[
\RDer\Sect(\hat{x},i^*\RDer j_*\cmplx{\sheaf{F}}_I)\mto\Sect(\Spec F_x^{\nr},\eta_x^*\God_U\cmplx{\sheaf{F}_I})
\]
is an isomorphism. Write $\God_{F_x}$ for the Godement resolution on $\Spec F_x$ with respect to $\Spec \algc{F}_x\mto\Spec F_x$. Then
\[
\eta_x^*\God_U\cmplx{\sheaf{F}}_I\mto\God_{F_x}\eta_x^*\God_U\cmplx{\sheaf{F}}_I
\leftarrow\God_{F_x}\eta_x^*\cmplx{\sheaf{F}}_I
\]
are quasi-isomorphisms of complexes of flabby sheaves on $\Spec F_x$. Hence, they remain quasi-isomorphisms if we apply the section functor $\Sect(\Spec F_x^{\nr},-)$ in each degree. Since the Frobenius acts compatibly on $F_x^{\nr}$ and $\algc{k(x)}$, the induced operation on the complexes is also compatible. The canonical exact sequence
\[
0\mto\Sect(\Spec F_x,-)\mto\Sect(\Spec F_x^{\nr},-)\xrightarrow{\id-\Frob_x}\Sect(\Spec F_x^{\nr},-)\mto 0
\]
on flabby sheaves on $\Spec F_x$ implies that the morphisms in the chain \eqref{eqn:comparison with local cohomology II} are also quasi-isomorphisms.
\end{proof}

\begin{rem}
Note that for $x$ lying over $p$, the proof of the lemma remains still valid, except that the complexes in the chain \eqref{eqn:comparison with local cohomology I} do not lie in $\cat{PDG}^{\cont}(\Lambda)$.
\end{rem}

It will be useful to introduce an explicit strictly perfect complex weakly equivalent to $\RDer\Sect(\Spec F_x,\eta_x^*\sheaf{F})$ in the case that $\sheaf{F}$ is a $\Lambda$-adic sheaf on $U$. Assume that $x\in X$ does not lie over $p$. Let $N$ be the compact $\Gal(\algc{F}_x/F_x)$-module corresponding to $\eta_x^*\sheaf{F}$ and write $F_x^{\nr,(p)}$ for the maximal pro-$p$ extension of $F_x^{\nr}$ inside $\algc{F}_x$, such that $\Gal(F_x^{\nr,(p)}/F_x^{\nr})\isomorph \Int_{p}$.

We set $N'\coloneqq N^{\Gal(\algc{F}_x/F_x^{\nr,(p)})}$. Note that $N'$ is a direct summand of the finitely generated, projective $\Lambda$-module $N$, because the $p$-Sylow subgroups of the Galois group $\Gal(\algc{F}_x/F_x^{\nr,(p)})$ are trivial by our assumption that $p$ is different from the characteristic of $k(x)$. In particular, $N'$ is itself finitely generated and projective over $\Lambda$. 

Fix a topological generator $\tau$ of $\Gal(F_x^{\nr,(p)}/F_x^{\nr})$ and a lift $\varphi\in \Gal(F_x^{\nr,(p)}/F_x)$  of the geometric Frobenius $\Frob_x$. Then $\tau$ and $\varphi$ are topological generators of the profinite group $\Gal(F_x^{\nr,(p)}/F_x)$ and
\[
\varphi\tau\varphi^{-1}=\tau^{q^{-1}}
\]
with $q=q_x$ the number of elements of $k(x)$ \cite[Thm.~7.5.3]{NSW:CohomNumFields}.

\begin{defn}
We define a strictly perfect complex $\cmplx{D}_{\hat{x}}(\sheaf{F})$ of $\Lambda$-modules with an action of $\Frob_x$ as follows: For $k\neq 0,1$ we set $D^k_{\hat{x}}(\sheaf{F})\coloneqq 0$. As $\Lambda$-modules we have $D^0_{\hat{x}}(\sheaf{F})=D^1_{\hat{x}}(\sheaf{F})=N'$ and the differential is given by $\id-\tau$. The geometric Frobenius $\Frob_x$ acts on $D^0_{\hat{x}}(\sheaf{F})$ via $\varphi$ and on $D^1_{\hat{x}}(\sheaf{F})$ via
\[
\varphi\left(\frac{\tau^q-1}{\tau-1}\right)\in \Lambda[[\Gal(F_x^{\nr, (p)}/F_x)]]^\times.
\]
\end{defn}

\begin{lem}\label{lem:strictly perfect resolution of local cohomology}
There exists a weak equivalence
\[
\cmplx{D}_{\hat{x}}(\sheaf{F})\wto\RDer\Sect(\Spec F_x^{\nr},\eta_x^*\sheaf{F})
\]
in $\cat{PDG}^{\cont}(\Lambda)$ that is compatible with the operation of the geometric Frobenius $\Frob_x$ on both sides.
\end{lem}
\begin{proof}
Clearly, we have
\[
\Lambda/I\tensor_{\Lambda} \cmplx{D}_{\hat{x}}(\sheaf{F})\isomorph \cmplx{D}_{\hat{x}}(\sheaf{F}_I)
\]
for all $I\in\openideals_{\Lambda}$. We may therefore reduce to the case that $\Lambda$ is a finite $\Int_p$-algebra.

By construction, the perfect complex of $\Lambda$-modules $\RDer\Sect(F_x^{\nr},\eta_x^*\sheaf{F})$ may be canonically identified with the homogenous cochain complex
\[
\cmplx{X}(\Gal(\algc{F}_x/F_x),N)^{\Gal(\algc{F}_x/F_x^{\nr})}
\]
(in the notation of \cite[Ch. I, \S 2]{NSW:CohomNumFields}) of the finite $\Gal(\algc{F}_x/F_x)$-module $N$. Recall that the elements of $X^n(\Gal(\algc{F}_x/F_x),N)$ are continuous maps
\[
f\colon \Gal(\algc{F}_x/F_x)^{n+1}\mto N
\]
and the operation of $\sigma\in \Gal(\algc{F}_x/F_x)$ on $f\in X^n(\Gal(\algc{F}_x/F_x),N)$ is defined by
\[
\sigma f\colon \Gal(\algc{F}_x/F_x)^{n+1}\mto N,\qquad (\sigma_0,\dots,\sigma_n)\mapsto \sigma f(\sigma^{-1} \sigma_0,\dots,\sigma^{-1} \sigma_n).
\]
The inflation map provides a quasi-isomorphism
\[
\cmplx{X}(\Gal(F_x^{\nr, (p)}/F_x),N')^{\Gal(F_x^{\nr, (p)}/F_x^{\nr})}\wto\cmplx{X}(\Gal(\algc{F}_x/F_x),N)^{\Gal(\algc{F}_x/F_x^{\nr})},
\]
which is compatible with the operation of $\Frob_x$ by a lift to $\Gal(\algc{F}_x/F_x)$ on both sides.

We define a quasi-isomorphism
\[
\alpha\colon\cmplx{D}_{\hat{x}}(\sheaf{F})\wto\cmplx{X}(\Gal(F_x^{\nr, (p)}/F_x),N')^{\Gal(F_x^{\nr, (p)}/F_x^{\nr})}
\]
compatible with the $\Frob_x$-operation by
\begin{align*}
\alpha(n)&\colon \Gal(F_x^{\nr, (p)}/F_x)\mto N',\qquad \tau^a\varphi^b\mapsto \tau^a n&\text{for $n\in D^0_{\hat{x}}(\sheaf{F})$,}\\
\alpha(n)&\colon \Gal(F_x^{\nr, (p)}/F_x)^2\mto N',\qquad (\tau^a\varphi^b,\tau^c\varphi^d)\mapsto \frac{\tau^c-\tau^a}{1-\tau}n&\text{for $n\in D^1_{\hat{x}}(\sheaf{F})$,}
\end{align*}
with $a,c\in\Int_p$, $b,d\in\hat{\Int}$.
Note that
\[
\frac{\tau^c-\tau^a}{1-\tau}=\tau^c\sum_{n=1}^\infty \binom{a-c}{n}(\tau-1)^{n-1}
\]
is a well-defined element of $\Lambda[[\Gal(F_x^{\nr, (p)}/F_x)]]$ for any $a,c\in\Int_p$.
\end{proof}

Assume again that $U\subset X$ is an open or closed subscheme. If $\Lambda'$ is another adic $\Int_p$-algebra and $\cmplx{M}$ a complex of $\Lambda'$-$\Lambda$-bimodules which is strictly perfect as complex of $\Lambda'$-modules, we may extend $\ringtransf_{\cmplx{M}}$ to a Waldhausen exact functor
\begin{gather*}
\ringtransf_{\cmplx{M}}\colon \cat{PDG}^{\cont}(U,\Lambda)\mto \cat{PDG}^{\cont}(U,\Lambda'),\\
(\cmplx{\sheaf{P}}_J)_{J\in\openideals_{\Lambda}}\mapsto (\varprojlim_{J\in\openideals_{\Lambda}} \Lambda'/I\tensor_{\Lambda}\cmplx{M}\tensor_{\Lambda}\cmplx{\sheaf{P}_J})_{I\in\openideals_{\Lambda'}}
\end{gather*}
such that
\begin{equation*}
\ringtransf_{\cmplx{M}}\RDer\Sect(U,\cmplx{\sheaf{P}})\mto\RDer\Sect(U,\ringtransf_{\cmplx{M}}(\cmplx{\sheaf{P}}))
\end{equation*}
is a weak equivalence in $\cat{PDG}^{\cont}(\Lambda')$ \cite[Prop.~5.5.7]{Witte:PhD}.

If $i\colon \Sigma \mto V$ is the embedding of a closed subscheme $\Sigma$ of $X$ into an open subscheme $V$ of $X$ with complement $j\colon U\mto V$ and $\sheaf{F}$ is an \'etale sheaf of abelian groups on $V$, then we may consider the sheaf
\[
i^!\sheaf{F}\coloneqq\ker(i^*\sheaf{F}\mto i^*j_*j^*\sheaf{F})
\]
on $\Sigma$. Its global sections $i^!\sheaf{F}(\Sigma)$ are the global sections of $\sheaf{F}$ on $V$ with support on $\Sigma$. The right derived functor $\RDer i^!$ can also be defined via Godement resolution:

\begin{lem}\label{lem:exactness of Ri upper shriek}
Assume that $p$ is invertible on $\Sigma$.
\[
 \RDer i^!\colon\cat{PDG}^{\cont}(V,\Lambda)\mto \cat{PDG}^{\cont}(\Sigma,\Lambda),\qquad (\cmplx{\sheaf{F}}_{I})_{I\in\openideals_{\Lambda}}\mapsto (i^!\God_V(\cmplx{\sheaf{F}}_I))_{I\in\openideals_{\Lambda}}
\]
is a Waldhausen exact functor and for every $\cmplx{\sheaf{F}}$ in $\cat{PDG}^{\cont}(V,\Lambda)$ there is an exact sequence
\[
0\mto i_*\RDer i^!\cmplx{\sheaf{F}}\mto \God_V(\cmplx{\sheaf{F}})\mto \RDer j_*j^*\cmplx{\sheaf{F}}\mto 0
\]
in $\cat{PDG}^{\cont}(V,\Lambda)$. In particular, if $i^*\cmplx{\sheaf{F}}$ is weakly equivalent to $0$, then there exists a chain of weak equivalences
\[
 i^*\RDer j_*j^*\cmplx{\sheaf{F}}\sim \RDer i^!\cmplx{\sheaf{F}}[1].
\]
\end{lem}
\begin{proof}
Note that for any abelian \'etale sheaf $\sheaf{F}$ on $V$, we habe $j^*\God_V(\sheaf{F})=\God_U(\sheaf{F})$. Moreover, by \cite[Exp. XVII, Prop. 4.2.3]{SGA4-3}, $\God_U(\sheaf{F})$ is a complex of flasque sheaves in the sense of \cite[Exp. V, Def. 4.1]{SGA4-3}. In particular, $\God_V(\sheaf{F})\mto j_*j^*\God_V(\sheaf{F})$ is surjective in the category of presheaves by \cite[Exp. V, Prop. 4.7]{SGA4-2}. If $\cmplx{\sheaf{F}}$ is a complex of abelian sheaves, $\God_V(\cmplx{\sheaf{F}})$ is constructed as the total complex of the double complex obtained by taking the Godement resolution of each individual sheaf. In particular, $\God_V(\cmplx{\sheaf{F}})$ is a complex of possibly infinite sums of flasque sheaves. Note that infinite sums of flasque sheaves are not necessarily flasque. Still, as \'etale cohomology of noetherian schemes commutes with filtered direct limits, $\God_V(\cmplx{\sheaf{F}})\mto j_*j^*\God_V(\cmplx{\sheaf{F}})$ is always surjective in the category of presheaves. This proves the exactness of the above sequence. Moreover, it implies that $\God_V(\cmplx{\sheaf{F}})$ is a $i^!$-acyclic resolution of $\cmplx{\sheaf{F}}$ such that $i^!\God_V$ preserves quasi-isomorphisms and injections. If $\cmplx{\sheaf{F}}$ is a perfect complex of sheaves of $\Lambda$-modules on $V$ for any finite ring $\Lambda$, then $i^!\God_V(\cmplx{\sheaf{F}})$ is perfect since this is true for $i^*\God_V(\cmplx{\sheaf{F}})$ and $i^*j_*j^*\God_V(\cmplx{\sheaf{F}})$. Similarly, we see that $i^!\God_V$ commutes with tensor products with finitely generated right $\Lambda$-modules. In particular, $\RDer i^!$ does indeed take values in $\cat{PDG}^{\cont}(\Sigma,\Lambda)$ for any adic ring $\Lambda$. Finally, if $i^*\cmplx{\sheaf{F}}$ is weakly equivalent to $0$, then we obtain the chain of weak equivalences
\[
 i^*\RDer j_*j^*\cmplx{\sheaf{F}}\otw\Cone\left(\RDer i^!\cmplx{\sheaf{F}}\mto i^*\God_V(\cmplx{\sheaf{F}})\right)\wto\RDer i^!\cmplx{\sheaf{F}}[1].
\]
\end{proof}

\section{Duality for adic sheaves}\label{sec:duality for adic sheaves}

For any scheme $Z$, any ring $R$ and any two \'etale sheaves of $R$-modules $\sheaf{F}$, $\sheaf{G}$ on $Z$, let
\[
 \sheafHom_{R,Z}(\sheaf{F},\sheaf{G})
\]
denote the sheaf of $R$-linear morphisms $\sheaf{F}\mto\sheaf{G}$ on $Z$. As before, we fix an adic $\Int_p$-algebra $\Lambda$. Let $U\subset X=\Spec \IntR_F$ be an open or closed subscheme. Unfortunately, we cannot present a construction of a Waldhausen exact functor
\[
\mdual\colon \cat{PDG}^{\cont}(U,\Lambda)^{\op}\mto\cat{PDG}^{\cont}(U,\Lambda^{\op})
\]
that would give rise to the usual total derived $\sheafHom$-functor $\sheaf{F}\mapsto \RDer\sheafHom_{\Lambda,U}(\sheaf{F},\Lambda_U)$ on the `derived' category of $\Lambda$-adic sheaves. Instead, we will construct a Waldhausen exact duality functor on the Waldhausen subcategory $\cat{S}^{\sm}(U,\Lambda)$ of smooth $\Lambda$-adic sheaves.

For any smooth $\Lambda$-adic sheaf $\sheaf{F}$,
\[
\begin{aligned}
 \sheaf{F}^{\mdual_{\Lambda}}&\coloneqq(\sheafHom_{\Lambda/I,U}(\sheaf{F}_I,(\Lambda/I)_U))_{I\in\openideals_{\Lambda}}\\
&=(\sheafHom_{\Int,U}(\sheafHom_{\Int,U}((\Lambda/I)_U,(\Rat_p/\Int_p)_U)\tensor_{\Lambda/I}\sheaf{F}_I,(\Rat_p/\Int_p)_U))_{I\in\openideals_{\Lambda}}
\end{aligned}
\]
is a smooth $\Lambda^{\op}$-adic sheaf on $U$. In this way, we obtain a Waldhausen exact equivalence
\[
 \mdual\colon \cat{S}^{\sm}(U,\Lambda)^{\op}\mto \cat{S}^{\sm}(U,\Lambda^{\op})
\]
and, by composing with $I\colon \KTh_n(\cat{S}^{\sm}(U,\Lambda))\xrightarrow{\isomorph}\KTh_n(\cat{S}^{\sm}(U,\Lambda)^{\op})$, isomorphisms
\[
 \mdual\colon \KTh_n(\cat{S}^{\sm}(U,\Lambda))\mto\KTh_n(\cat{S}^{\sm}(U,\Lambda^{\op}))
\]
for each $n\geq 0$.

Assume that $U$ is an open subscheme of $X$ such that $p$ is invertible on $U$. If $\sheaf{F}$ is a smooth $\Lambda$-adic sheaf on $U$, we can find a strictly perfect complex of $\Lambda$-modules $\cmplx{P}$ together with a weak equivalence
\[
 \cmplx{P}\wto \RDer\Sectc(U,\sheaf{F})
\]
in $\cat{PDG}^{\cont}(\Lambda)$. As a consequence of Artin-Verdier duality \cite[Thm. II.3.1]{Milne:ADT}, we then also have a weak equivalence
\begin{equation}\label{eqn:Artin-Verdier duality}
  (\cmplx{P})^{\mdual}\wto\RDer\Sect(U,\sheaf{F}^{\mdual}(1))[-3].
\end{equation}
in $\cat{PDG}^{\cont}(\Lambda^\op)$.

We could proceed in the same way for local duality and duality over finite fields, but instead, we prove the following finer results.

\begin{lem}\label{lem:explicit local duality}
Assume that $U$ is an open subscheme of $X$ such that $p$ is invertible on $U$ and that $i\colon x\mto X$ is a closed point not lying over $p$. For any smooth $\Lambda$-adic sheaf $\sheaf{F}$ on $U$, there exists a weak equivalence
\[
\cmplx{D}_{\hat{x}}(\sheaf{F})^{\mdual}\wto\RDer\Sect(\Spec F_x^{\nr},\eta_x\sheaf{F}^{\mdual}(1))[-1]
\]
in $\cat{PDG}^{\cont}(\Lambda^\op)$, compatible with the operation of $\Frob_x^{\mdual}$ on the left and of $\Frob_x^{-1}$ on the right.
\end{lem}
\begin{proof}
As in the proof of Lemma~\ref{lem:strictly perfect resolution of local cohomology}, we can replace $\RDer\Sect(\Spec F_x^{\nr},\eta_x\sheaf{F}^{\mdual}(1))$ by the homogenous cochain complex $\cmplx{X}(\Gal(F_x^{\nr,(p)}/F_x),(N')^{\mdual}(1))^{\Gal(F_x^{\nr,(p)}/F_x^{\nr})}$. By choosing a basis of the free $\Int_p$-module $\Int_p(1)$, i.\,e.\ a compatible system of $p^{n}$-th roots of unity, we may identify the underlying $\Lambda$-modules of $(N')^{\mdual}$ and $(N')^{\mdual}(1)$. The operation of $\sigma\in\Gal(F_x^{\nr,(p)}/F_x^{\nr})$ on $f\in(N')^{\mdual}$ is given by
\[
\sigma f\coloneqq f\comp (\sigma^\mdual)^{-1}.
\]
The operation of $\Frob_x^\mdual$ on $f\in D^1_{\hat{x}}(\sheaf{F})^{\mdual}=(N')^{\mdual}$ is then given by
\[
\Frob_x^\mdual(f)\coloneqq\left(\frac{\tau^{-q}-1}{\tau^{-1}-1}\right)\varphi^{-1}f
\]
and on $g\in D^0_{\hat{x}}(\sheaf{F})^{\mdual}=(N')^{\mdual}$ by
\[
\Frob_x^\mdual(g)\coloneqq\varphi^{-1}g,
\]
with $\varphi,\tau\in\Gal(F_x^{\nr,(p)}/F_x)$ denoting our fixed topological generators and $q\in\Lambda^\times$ denoting the order of the residue field $k(x)$. Set for $b\in\hat{\Int}$
\[
s(b)\coloneqq q^{-b}\left(\frac{\tau^{-1}-1}{\tau^{-q^{-b}}-1}\right)\in \Lambda[[\Gal(F_x^{\nr,(p)}/F_x^{\nr})]]^{\times}
\]
and note that $s$ satisfies the cocycle relation
\[
s(b+1)=q^{-1}\varphi s(b)\left(\frac{\tau^{-q}-1}{\tau^{-1}-1}\right)\varphi^{-1}=q^{-1}\varphi s(b)\varphi^{-1}s(1).
\]

We define a weak equivalence
\[
\beta\colon \cmplx{D}_{\hat{x}}(\sheaf{F})^{\mdual}\wto\cmplx{X}(\Gal(F_x^{\nr,(p)}/F_x),(N')^{\mdual}(1))^{\Gal(F_x^{\nr,(p)}/F_x^{\nr})}[-1]
\]
by
\[
\beta(f)\colon \Gal(F_x^{\nr, (p)}/F_x)\mto (N')^{\mdual}(1),\qquad \tau^a\varphi^b\mapsto \tau^a s(b)f
\]
for $f\in D^1_{\hat{x}}(\sheaf{F})^{\mdual}$ and by
\[
\beta(g)\colon \Gal(F_x^{\nr, (p)}/F_x)^2\mto (N')^{\mdual}(1),\qquad
 (\tau^a\varphi^b,\tau^c\varphi^d)\mapsto \left(\frac{\tau^a s(b)-\tau^c s(d)}{1-\tau^{-1}}\right)g
\]
for $g\in D^0_{\hat{x}}(\sheaf{F})^{\mdual}$, $a,c\in\Int_p$, $b,d\in\hat{\Int}$.

Using the cocycle relation for $s$, it is easily checked that
\[
\beta\comp\Frob_x^*=\Frob_x^{-1}\comp \beta,
\]
as claimed.
\end{proof}

In particular, if $\cmplx{Q}$ denotes the cocone of
\[
\cmplx{D}_{\hat{x}}(\sheaf{F})\xrightarrow{\id-\Frob_x}\cmplx{D}_{\hat{x}}(\sheaf{F}),
\]
then $\cmplx{Q}$ is a strictly perfect complex of $\Lambda$-modules and there exist weak equivalences
\begin{equation}\label{eqn:local duality}
\begin{aligned}
\cmplx{Q}&\wto \RDer\Sect(\Spec F_x,\eta_x^*\sheaf{F}),\\
(\cmplx{Q})^{\mdual}&\wto \RDer\Sect(\Spec F_x,\eta_x^*\sheaf{F}^\mdual(1))[-2].
\end{aligned}
\end{equation}
in $\cat{PDG}^{\cont}(\Lambda^{\op})$.

Let now $\sheaf{G}$ be a complex in $\cat{S}(x,\Lambda)=\cat{S}^{\sm}(x,\Lambda)$ and let
\[
 \sheaf{G}_{\hat{x}}\coloneqq\varprojlim_{I\in\openideals_{\Lambda}} (\sheaf{G}_I)_{\hat{x}}
\]
be the stalk of $\sheaf{G}$ in the geometric point $\hat{x}$ over $x$. Then $\sheaf{G}_{\hat{x}}$ is a finitely generated, projective $\Lambda$-module, equipped with a natural operation of $\Frob_x$. Clearly, the natural morphism
\begin{equation}\label{eqn:stalk}
 \sheaf{G}_{\hat{x}}\wto\RDer\Sect(\hat{x},\sheaf{G})
\end{equation}
is a weak equivalence in $\cat{PDG}^{\cont}(\Lambda)$ that is compatible with the operation of $\Frob_x$ on both sides. In particular, the cocone $\cmplx{C}$ of
\[
\sheaf{G}_{\hat{x}}\xrightarrow{\id-\Frob_x} \sheaf{G}_{\hat{x}}
\]
is weakly equivalent to $\RDer\Sect(x,\sheaf{G})$.

\begin{lem}\label{lem:explicit finite duality}
With $\sheaf{G}$ as above, there exists an isomorphism
\[
 (\sheaf{G}_{\hat{x}})^{\mdual}\xrightarrow{\isomorph}(\sheaf{G}^{\mdual})_{\hat{x}}
\]
of finitely generated, projective $\Lambda$-modules, compatible with the operation of $\Frob_x^{\mdual}$ on the left and of $\Frob_x^{-1}$ on the right.
\end{lem}
\begin{proof}
Let $R$ be any finite ring. Under the equivalence between the categories of \'etale sheaves of $R$-modules on $x$ and of discrete $R[[\Gal(\algc{k(x)}/k(x))]]$-modules, given by $\sheaf{F}\mapsto\sheaf{F}_{\hat{x}}$, the dual sheaf $\sheaf{F}^{\mdual_R}$ corresponds to the $R^{\op}$-module $(\sheaf{F}_{\hat{x}})^{\mdual_R}$ with $\sigma\in\Gal(\algc{k(x)}/k(x))$ acting on $f\colon \sheaf{F}_{\hat{x}}\mto R$ by $f\comp (\sigma^{\mdual_R})^{-1}$.
\end{proof}

Consequently, we obtain a weak equivalence
\begin{equation}\label{eqn:finite duality}
(\cmplx{C})^{\mdual}\wto \RDer\Sect(x,\sheaf{G}^{\mdual})[-1]
\end{equation}
in $\cat{PDG}^{\cont}(\Lambda^{\op})$. If $\sheaf{G}=i^*\sheaf{F}$ with $\sheaf{F}$ a smooth $\Lambda$-adic sheaf on $U$ as above, then by the exchange formula \cite[Thm. 8.4.7]{LeiFu:EtaleCohomology}, there exists a chain of weak equivalences
\begin{equation}\label{eqn:exchange formula}
\begin{aligned}
 (i^*\sheaf{F})^{\mdual}&=(\RDer\sheafHom_{\Lambda/I,x}(i^*\sheaf{F}_I,(\Lambda/I)_x))_{I\in\openideals_{\Lambda}}\\
                        &\sim(\RDer\sheafHom_{\Lambda/I,x}(i^*\sheaf{F}_I,\RDer i^!(\Lambda/I)_U(1)[-2]))_{I\in\openideals_{\Lambda}}\\
                        &\sim(\RDer i^!\sheafHom_{\Lambda/I,U}(\sheaf{F}_I,(\Lambda/I)_U(1))[-2])_{I\in\openideals_{\Lambda}}\\
                        &=\RDer i^!\sheaf{F}^{\mdual}(1)[-2]
\end{aligned}
\end{equation}
in $\cat{PDG}^{\cont}(x,\Lambda^\op)$.

\section{Admissible extensions}\label{sec:AdmissibleExt}

As before, we fix an odd prime $p$ and a number field $F$. Assume that $F_\infty/F$ is a possibly infinite Galois extension unramified over an open or closed subscheme $U=U_F$ of $X=\Spec \IntR_F$. Let $G\coloneqq\Gal(F_\infty/F)$ be its Galois group. We also assume that $G$ has a topologically finitely generated, open pro-$p$-subgroup, such that for any adic $\Int_p$-algebra $\Lambda$, the profinite group ring $\Lambda[[G]]$ is again an adic ring \cite[Prop.~3.2]{Witte:MCVarFF}. For any intermediate number field $K$ of $F_\infty/F$, we will write $U_K$ for the base change with $X_K\coloneqq\Spec \IntR_K$ and $f_K\colon U_K\mto U$ for the corresponding Galois covering of $U$, such that we obtain a system of Galois coverings $(f_K \colon U_K\mto U)_{F\subset K\subset {F_\infty}}$, which we denote by
\[
 f\colon U_{F_\infty}\mto U.
\]
As in \cite[Def. 6.1]{Witte:MCVarFF} we make the following construction.

\begin{defn}
Let $\Lambda$ be any adic $\Int_p$-algebra.
For $\cmplx{\sheaf{F}}\in \cat{PDG}^{\cont}(U,\Lambda)$ we set
\[
f_!f^*\cmplx{\sheaf{F}}\coloneqq(\varprojlim_{I\in\openideals_{\Lambda}}\varprojlim_{F\subset K\subset F_\infty}
\Lambda[[G]]/J\tensor_{\Lambda[[G]]}{f_K}_!f_K^*\cmplx{\sheaf{F}_I})_{J\in\openideals_{\Lambda[[G]]}}
\]
\end{defn}

\begin{rem}
The functor $f_!f^*$ corresponds to the composition of the restriction and compact induction functors on the level of compact Galois modules. On finite level, the extension by zero ${f_K}_!$ agrees with the direct image ${f_K}_*$. We use the notation of the extension by zero to emphasise its role as a left adjoint of $f_K^*$.
\end{rem}

As in \cite[Prop. 6.2]{Witte:MCVarFF} one verifies that we thus obtain a Waldhausen exact functor
\[
 f_!f^*\colon \cat{PDG}^{\cont}(U,\Lambda)\mto \cat{PDG}^{\cont}(U,\Lambda[[G]]).
\]
In particular, if $U$ is open and dense in $X$ and if $k\colon U\mto W$ denotes the open immersion into another open dense subscheme $W$ of $X$, we obtain for each complex $\cmplx{\sheaf{F}}$ in $\cat{PDG}^{\cont}(U,\Lambda)$ a complex
\[
 \RDer\Sectc(W,\RDer k_* f_!f^*\cmplx{\sheaf{F}})
\]
in $\cat{PDG}^{\cont}(\Lambda[[G]])$.

\begin{rem}\label{rem:exchanging W and V}
For $U$ open and dense in $X$, set $V\coloneqq U\cup (X-W)$ and let $j\colon U\mto V$ denote the corresponding open immersion. Write $j'\colon V\mto X$ and $k'\colon W\mto X$ for the open immersions into $X$. For any \'etale sheaf $\sheaf{G}$ on $U$, the canonical morphism
\[
k'_!k_*\God_U\sheaf{G}\isomorph j'_*j_!\God_U\sheaf{G}\mto
j'_*\God_V j_!\sheaf{G}
\]
is seen to be a quasi-isomorphism by checking on the stalks. Hence, for any $\cmplx{\sheaf{F}}$ in $\cat{PDG}^{\cont}(U,\Lambda)$,  there is a weak equivalence
\[
\RDer\Sectc(W,\RDer k_*f_!f^*\cmplx{\sheaf{F}})\wto\RDer\Sect(V,j_!f_!f^*\cmplx{\sheaf{F}}).
\]
We recall that the righthand complex is always in $\cat{PDG}^{\cont}(\Lambda[[G]])$. Hence, the same is true for the left-hand complex without any condition on $U$ and $W$, even if $\RDer k_*f_!f^*\cmplx{\sheaf{F}}$ fails to be a perfect complex.
In particular, we may use the two complexes interchangeably in our results.
\end{rem}

We recall how the functor $f_!f^*$ transforms under the change of the extension $F_\infty/F$ and under changes of the coefficient ring $\Lambda$.

\begin{prop}\label{prop:ring change for coverings}
Let $f\colon U_{F_\infty}\mto U$ be the system of Galois coverings of the open or closed subscheme $U$ of $X$ associated to the extension $F_\infty/F$ with Galois group $G$ which is unramified over $U$. Let further $\Lambda$ be an adic $\Int_p$-algebra and $\cmplx{\sheaf{F}}$ be a complex in $\cat{PDG}^{\cont}(U,\Lambda)$.
\begin{enumerate}
\item Let $\Lambda'$ be another adic $\Int_p$-algebra and let $\cmplx{P}$ be a complex of $\Lambda'$-$\Lambda[[G]]$-bimodules, strictly perfect as complex of $\Lambda'$-modules. Then there exists a natural isomorphism
    $$
    \ringtransf_{\cmplx{P[[G]]^\delta}}f_!f^*\cmplx{\sheaf{F}}\isomorph f_!f^*\ringtransf_{\cmplx{P}}f_!f^*\cmplx{\sheaf{F}}
    $$
\item Let $F'_\infty\subset F_\infty$ be a subfield such that $F'_\infty/F$ is a Galois extension with Galois group $G'$ and let $f'\colon U_{F'_\infty}\mto U$ denote the corresponding system of Galois coverings. Then there
exists a natural isomorphism
$$
\ringtransf_{\Lambda[[G']]}f_!f^*\cmplx{\sheaf{F}}\isomorph (f')_!(f')^*\cmplx{\sheaf{F}}
$$
in $\cat{PDG}^{\cont}(U,\Lambda[[G']])$.
\item Let $F'/F$ be a finite extension inside $F_\infty/F$, let $f_{F'}\colon U_{F'}\mto U$ denote the associated \'etale covering of $U$ and let $g\colon U_{F_\infty}\mto U_F'$ be the restriction of the system of coverings $f$ to $U_{F'}$. Write $G'\subset G$ for the corresponding open subgroup and view $\Lambda[[G]]$ as a
$\Lambda[[G']]$-$\Lambda[[G]]$-bimodule. Then there exists a
natural isomorphism
$$
\ringtransf_{\Lambda[[G]]}f_!f^*\cmplx{\sheaf{F}}\isomorph f_{F'*}\left(g_!g^*\right)f_{F'}^*\cmplx{\sheaf{F}}
$$
in
$\cat{PDG}^{\cont}(U,\Lambda[[G']])$.
\item With the notation of $(3)$, let $\cmplx{\sheaf{G}}$ be a complex in $\cat{PDG}^{\cont}(U_{F'},\Lambda)$ and view $\Lambda[[G]]$ as a $\Lambda[[G]]$-$\Lambda[[G']]$-bimodule. Then there exists a natural isomorphism
$$
\ringtransf_{\Lambda[[G]]}{f_{F'}}_*g_!g^*\cmplx{\sheaf{G}}\isomorph f_!f^*({f_{F'}}_*\cmplx{\sheaf{G}})
$$
in $\cat{PDG}^{\cont}(U,\Lambda[[G]])$.
\end{enumerate}
\end{prop}
\begin{proof}
Part $(1)-(3)$ are proved in \cite[Prop. 6.5, 6.7]{Witte:MCVarFF}. We prove $(4)$. First, note that for any finite Galois extension $F''/F$ with $F'\subset F''\subset F_\infty$ and any $I\in\openideals_{\Lambda}$ the canonical map
\[
 {g_{F''}}_!g_{F''}^*(\Lambda/I)_{U_{F'}}\mto f_{F'}^*{f_{F''}}_!f_{F''}^*(\Lambda/I)_U
\]
induces an isomorphism
\[
\Lambda/I[\Gal(F''/F)]\tensor_{\Lambda/I[\Gal(F''/F')]}{g_{F''}}_!g_{F''}^*(\Lambda/I)_{U_{F'}}\isomorph f_{F'}^*{f_{F''}}_!f_{F''}^*(\Lambda/I)_U.
\]
Hence,
\[
 \ringtransf_{\Lambda[[G]]}(g_!g^*\Lambda_{U_{F'}})\isomorph f_{F'}^*f_!f^*\Lambda_U
\]
in $\cat{PDG}^{\cont}(U_{F'},\Lambda[[G]])$. We further recall that in the notation of \cite[Prop. 6.3]{Witte:MCVarFF}, there exists an isomorphism
\[
 f_!f^*{f_{F'}}_*\cmplx{\sheaf{G}}\isomorph\ringtransf_{f_!f^*\Lambda_U}{f_{F'}}_*\cmplx{\sheaf{G}}.
\]
The projection formula then implies
\begin{equation*}
\begin{split}
 \ringtransf_{f_!f^*\Lambda_U}{f_{F'}}_*\cmplx{\sheaf{G}}&\isomorph{f_{F'}}_*(\ringtransf_{f_{F'}^*f_!f^*\Lambda_U}(\cmplx{\sheaf{G}}))\\
                                                       &\isomorph{f_{F'}}_*(\ringtransf_{\ringtransf_{\Lambda[[G]]}(g_!g^*\Lambda_{U_{F'}})}(\cmplx{\sheaf{G}}))\\
                                                       &\isomorph{f_{F'}}_*(\ringtransf_{\Lambda[[G]]}(g_!g^*\cmplx{\sheaf{G}}))
\end{split}
\end{equation*}
as desired.
\end{proof}

To understand Part $(1)$ of this proposition, note that if $\rho$ is a representation of $G$ on a finitely generated and projective $\Lambda$-module and $\rho^\sharp$ is the corresponding $\Lambda$-$\Int_p[[G]]$-bimodule as in Example~\ref{exmpl:def of eval}, then
\begin{equation}\label{eqn:def of rep sheaf}
\sheaf{M}(\rho)\coloneqq\ringtransf_{\rho^\sharp}f_!f^*(\Int_p)_U
\end{equation}
is simply the smooth $\Lambda$-adic sheaf on $U$ associated to $\rho$ \cite[Prop. 6.8]{Witte:MCVarFF}. In general,
\begin{equation}\label{eqn:def of twisted sheaf}
\ringtransf_{\cmplx{\tilde{P}}}\cmplx{\sheaf{F}}\coloneqq\ringtransf_{\cmplx{P}}f_!f^*\cmplx{\sheaf{F}}
\end{equation}
should be understood as the derived tensor product over $\Lambda$ of the complex of sheaves associated to $\cmplx{P}$ and the complex $\cmplx{\sheaf{F}}$.

Assume that $\sheaf{F}$ is a smooth $\Lambda$-adic sheaf on $U$. As before, we write
\[
 \sheaf{F}^{\mdual_\Lambda}\coloneqq(\sheafHom_{\Lambda/I,U}(\sheaf{F}_I,\Lambda/I))_{I\in\openideals_{\Lambda/I}}\in\cat{PDG}^{\cont}(U,\Lambda^{\op})
\]
for the $\Lambda$-dual of $\sheaf{F}$ and $\Lambda^{\op}[[G]]^\sharp$ for the $\Lambda^{\op}[[G]]$-$\Lambda[[G]]^\op$-bimodule with $g\in G$ acting by $g^{-1}$ from the right. We then have a natural isomorphism
\begin{equation}\label{eqn:duality for coverings}
 f_!f^*\sheaf{F}^{\mdual_\Lambda}\isomorph\ringtransf_{\Lambda^{\op}[[G]]^\sharp}(f_!f^*\sheaf{F})^{\mdual_{\Lambda[[G]]}}.
\end{equation}
This can then be combined with the duality assertions \eqref{eqn:Artin-Verdier duality}, \eqref{eqn:local duality}, and \eqref{eqn:finite duality}. For example, we may find a strictly perfect complex of $\Lambda^{\op}[[G]]$-modules $\cmplx{P}$ and weak equivalences
\begin{equation}
\begin{aligned}\label{eqn:Artin-Verdier duality for coverings}
\cmplx{P}&\wto\RDer\Sectc(U,f_!f^*\sheaf{F}^{\mdual_\Lambda}(1)),\\
(\cmplx{P})^{\Mdual}&\wto\RDer\Sect(U,f_!f^*\sheaf{F})[-3]
\end{aligned}
\end{equation}
if $\ell$ is invertible on $U$.

Let $F_{\cyc}$ denote the cyclotomic $\Int_p$-extension of $F$ and let $M$ be the maximal abelian $p$-extension of $F_{\cyc}$ unramified outside the places over $p$. Assume that $F$ is a totally real field. By the validity of the weak Leopoldt conjecture for $F_{\cyc}$, the Galois group $\Gal(M/F_{\cyc})$ is a finitely generated torsion module of projective dimension less or equal $1$ over the classical Iwasawa algebra $\Int_p[[\Gal(F_{\cyc}/F)]]$ \cite[Thm. 11.3.2]{NSW:CohomNumFields}. Like in \cite{Kakde2}, we will assume the vanishing of its Iwasawa $\mu$-invariant in the following sense:

\begin{conj}\label{conj:vanishing of mu}
For every totally real field $F$, the Galois group over $F_{\cyc}$ of the maximal abelian $p$-extension of $F_{\cyc}$ unramified outside the places over $p$ is a finitely generated $\Int_p$-module.
\end{conj}

In particular, for any totally real field $F$ and any finite set $\Sigma$ of places of $F$  containing the places over $p$, the Galois group over $F_{\cyc}$ of the maximal abelian $p$-extension of $F_{\cyc}$ unramified outside $\Sigma$ is also a finitely generated $\Int_p$-module, noting that no finite place is completely decomposed in $F_{\cyc}/F$ \cite[Cor. 11.3.6]{NSW:CohomNumFields}. We also observe that the Galois group $\Gal(F_\Sigma^{(p)}/F_{\cyc})$ of the maximal $p$-extension of $F$ unramified outside $\Sigma$ is then a free pro-$p$-group topologically generated by finitely many elements \cite[Thm. 11.3.7]{NSW:CohomNumFields}.

\begin{defn}
Let $F$ be a number field. An extension $F_\infty/F$ inside $\algc{F}$ is called \emph{admissible} if
\begin{enumerate}
 \item[(1)] $F_\infty/F$ is Galois and unramified outside a finite set of places,
 \item[(2)] $F_\infty$ contains the cyclotomic $\Int_{p}$-extension $F_{\cyc}$,
 \item[(3)] $\Gal(F_\infty/F_{\cyc})$ contains an open pro-$p$ subgroup that is topologically finitely generated.
\end{enumerate}
An admissible extension $F_\infty/F$ is called \emph{really admissible} if $F_\infty$ and $F$ are totally real.
\end{defn}

The notion of really admissible extensions is slightly weaker than the notion of admissible extension used in \cite[Def.~2.1]{Kakde2}: We do not need to require $\Gal(F_\infty/F)$ to be a $p$-adic Lie group. For example, as a result of the preceding discussion, we see that we could choose $F_\infty=F_\Sigma^{(p)}$ for some finite set of places $\Sigma$ of $F$ containing the places above $p$, provided that Conjecture~\ref{conj:vanishing of mu} is valid.

If $F_\infty/F$ is an admissible extension, we let $G\coloneqq\Gal(F_\infty/F)$ denote its Galois group and set $H\coloneqq\Gal(F_\infty/F_{\cyc})$, $\Gamma\coloneqq\Gal(F_{\cyc}/F)$. We may then choose a continuous splitting $\Gamma\mto G$ to identify $G$ with the corresponding semi-direct product $G=H\rtimes\Gamma$.

If a really admissible extension $F_\infty/F$ is unramified over the open dense subscheme $U=W$ of $X$, $\Lambda=\Int_p$ and $\cmplx{\sheaf{F}}=(\Int_p)_U(1)$, then
\[
 \varprojlim_{I\in\openideals_{\Int_p[[G]]}}\RDer\Sectc(U,f_!f^*(\Int_p)_U(1))[-3]
\]
is by Artin-Verdier duality and comparison of \'etale and Galois cohomology quasi-isomorphic to the complex $C(F_\infty/F)$ featuring in the main conjecture \cite[Thm.~2.11]{Kakde2}. In particular,
\[
 \RDer\Sectc(U,f_!f^*(\Int_p)_U(1))
\]
is in fact an object of $\cat{PDG}^{\cont,w_H}(\Int_p[[G]])$ under Conjecture~\ref{conj:vanishing of mu}. We will generalise this statement in the next section.

\section{The \texorpdfstring{$S$}{S}-torsion property}\label{sec:S-torsion}

Assume that $F_\infty/F$ is an admissible extension that is unramified over the open dense subscheme $U$ of $X=\Spec \IntR_F$ and that $k\colon U\mto W$ is the open immersion into another open dense subscheme of $X$. Note that $p$ must be invertible on $U$, because the cyclotomic extension $F_\cyc/F$ is ramified in all places over $p$.  We also fix an adic $\Int_p$-algebra $\Lambda$. Our purpose is to prove:

\begin{thm}\label{thm: S torsion global}
Assume that $F_\infty/F$ is really admissible and that $p$ is invertible on $W$. Let $\cmplx{\sheaf{F}}\in \cat{PDG}^{\cont,\infty}(U,\Lambda)$ be a complex of $\Lambda$-adic sheaves smooth at $\infty$. If Conjecture~\ref{conj:vanishing of mu} is valid, then
the complexes
\[
 \RDer\Sectc(W,\RDer k_*f_!f^*\cmplx{\sheaf{F}}(1)),\quad \RDer\Sect(W,k_!f_!f^*\cmplx{\sheaf{F}})
\]
are in $\cat{PDG}^{\cont,w_H}(\Lambda[[G]])$.
\end{thm}

In the course of the proof, we will also need to consider the following local variant, whose validity is independent of Conjecture~\ref{conj:vanishing of mu}.

\begin{thm}\label{thm: S torsion local}
Assume that $F_\infty/F$ is an admissible extension with $k\colon U\mto W$ as above. Let $i\colon \Sigma\mto W$ denote a closed subscheme of $W$ and assume that $p$ is invertible on $\Sigma$. For any complex of $\Lambda$-adic sheaves $\cmplx{\sheaf{F}}$ in $\cat{PDG}^{\cont}(U,\Lambda)$, the complexes
\[
\RDer\Sect(\Sigma,i^*\RDer k_*f_!f^*\cmplx{\sheaf{F}}), \RDer\Sect(\Sigma,\RDer i^!k_!f_!f^*\cmplx{\sheaf{F}})
\]
are in $\cat{PDG}^{\cont,w_H}(\Lambda[[G]])$.
\end{thm}

Using \cite[Prop. 4.8]{Witte:MCVarFF} we may at once reduce to the case that $\Lambda$ is a finite semi-simple $\Int_p$-algebra and that $F_\infty/F_{\cyc}$ is a finite extension. It then suffices to show that the complexes appearing in the above theorems have finite cohomology groups. We may then replace $\cmplx{\sheaf{F}}$ by a quasi-isomorphic strictly perfect complex. Using stupid truncation and induction on the length of the strictly perfect complex we may assume that $\sheaf{F}$ is in fact a flat and constructible sheaf (unramified in $\infty$). Note further that the cohomology groups
\begin{align*}
 \HF^n_c(W,\RDer k_*f_!f^*\sheaf{F}(1))&=\varprojlim_{F\subset K\subset F_{\infty}}\HF^n_c(W_K,\RDer k_*f^*_K\sheaf{F}(1)),\\
 \HF^n(W, k_!f_!f^*\sheaf{F})&=\varprojlim_{F\subset K\subset F_{\infty}}\HF^n(W_K, k_!f^*_K\sheaf{F}),\\
 \HF^n(\Sigma,i^*\RDer k_*f_!f^*\sheaf{F})&=\varprojlim_{F\subset K\subset F_{\infty}}\HF^n(\Sigma_K,i^*\RDer k_*f^*_K\sheaf{F})\\
\HF^n(\Sigma,\RDer i^! k_!f_!f^*\sheaf{F})&=\varprojlim_{F\subset K\subset F_{\infty}}\HF^n(\Sigma_K,\RDer i^! k_!f^*_K\sheaf{F})
\end{align*}
do not change if we replace $F$ by a finite extension of $F$ inside $F_{\infty}$. So, we may assume that $F_\infty=F_{\cyc}$ and that no place in $\Sigma$ splits in $F_{\infty}/F$. Further, we may reduce to the case that $\Sigma$ consists of a single place $x$. In particular, $x$ does not split or ramify in $F_\infty/F$ and $x$ does not lie above $p$.

We consider Theorem~\ref{thm: S torsion local} in the case that $x\in U$ and write $i'\colon x\mto U$ for the inclusion map. Under the above assumptions on $x$, there exists a chain of weak equivalences
\[
 \RDer\Sect(x,i^*\RDer k_*f_!f^*\sheaf{F})\otw\RDer\Sect(x,{i'}^*f_!f^*\sheaf{F})\wto\RDer\Sect(x,g_!g^*{i'}^*\sheaf{F})
\]
where $g\colon x_\infty\mto x$ is the unique $\Int_p$-extension of $x$. We can now refer directly to \cite[Thm. 8.1]{Witte:MCVarFF} or identify
\[
 \HF^n(x,g_!g^*{i'}^*\sheaf{F})=\HF^n(\Gal_{k(x)},\FF_p[[\Gamma]]^{\sharp}\tensor_{\FF_p}M)
\]
with $\Gal_{k(x)}$ the absolute Galois group of the residue field $k(x)$ of $x$, $M$ the stalk of $\sheaf{F}$ in a geometric point over $x$ and $\FF_p[[\Gamma]]^{\sharp}$ being the $\Gal_{k(x)}$-module $\FF_p[[\Gamma]]$ with $\sigma\in \Gal_{k(x)}$ acting by right multiplication with the image of $\sigma^{-1}$ in $\Gamma$. It is then clear that the only non-vanishing cohomology group is $H^1(\Gal_{k(x)},\FF_p[[\Gamma]]^{\sharp}\tensor_{\FF_p}M)$, of order bounded by the order of $M$.

Write $U'\coloneqq U-\set{x}$ and let $\ell\colon U'\mto U$ denote the inclusion morphism. Then there is an exact sequence
\[
 0\mto\RDer i^!k_!\ell_!\ell^*f_!f^*\sheaf{F}\mto\RDer i^!k_!f_!f^*\sheaf{F}\mto\RDer i^!i_*i^*k_!\sheaf{F}\mto 0.
\]
Moreover, there exists a chain of weak equivalences
\[
 {i'}^*f_!f^*\sheaf{F}\xrightarrow{\isomorph}i^*k_!f_!f^*\sheaf{F}\wto\RDer i^!i_*i^*k_!f_!f^*\sheaf{F}.
\]
Since we already know that the groups $\HF^n(x,{i'}^*f_!f^*\sheaf{F})$ are finite, it is sufficient to prove that $\HF^n(x,\RDer i^!k_!f_!f^*\sheaf{F})$ is finite in the case that $x\in W-U$.

Now we prove Theorem~\ref{thm: S torsion local} in the case that $x\in W-U$. First, note that the complex $\RDer i^!\RDer k_*f_!f^*\sheaf{F}$ is quasi-isomorphic to $0$. Hence, there is a chain of weak equivalences
\[
i^*\RDer k_*f_!f^*\sheaf{F}\sim\RDer i^!k_!f_!f^*\sheaf{F}[1]
\]
by Lemma~\ref{lem:exactness of Ri upper shriek}. So, it suffices to consider the left-hand complex. By Lemma~\ref{lem:comparison with local cohomology} and the smooth base change theorem there exists a chain of weak equivalences
\[
 \RDer\Sect(x,i^*\RDer k_*f_!f^*\sheaf{F})\sim\RDer\Sect(\Spec F_x,h_!h^*\eta_x^*\sheaf{F}),
\]
where $F_x$ is the local field in $x$ with valuation ring $\IntR_{F_x}$, $\eta_x\colon \Spec F_x\mto U$ is the map to the generic point of $U$, and $h\colon \Spec (F_x)_\cyc\mto \Spec F_x$ is the unique $\Int_p$-extension of $F_x$ inside $\algc{F}_x$. We may now identify
\[
\HF^n(x,i^*\RDer k_*f_!f^*\sheaf{F})=\HF^n(\Gal_{F_x},\FF_p[[\Gamma]]^{\sharp}\tensor_{\FF_p}M)
\]
with $\Gal_{F_x}$ the absolute Galois group of the local field $F_x$ in $x$, $M$ the finite $\Gal_{F_x}$-module corresponding to $\eta_x^*\sheaf{F}$ and $\FF_p[[\Gamma]]^{\sharp}$ being the $\Gal_{F_x}$-module $\FF_p[[\Gamma]]$ with $\sigma\in \Gal_{F_x}$ acting by right multiplication with the image of $\sigma^{-1}$ in $\Gamma$. The finiteness of the cohomology group on the righthand side is well-known: We can use local duality to identify it with the Pontryagin dual of
\[
\HF^{2-n}(\Gal_{(F_x)_{\cyc}},\pdual{M}(1))
\]
where $\pdual{M}(1)$ is the first Tate twist of the Pontryagin dual of $M$.

Finally, we prove Theorem~\ref{thm: S torsion global}. Assume that $F_\infty/F$ is really admissible, that $\sheaf{F}$ is smooth at $\infty$, and that $p$ is invertible on $W$. We begin with the case of \'etale cohomology with proper support. Letting $i\colon \Sigma\mto W$ denote the complement of $U$ in $W$, we have the exact excision sequence
\begin{multline*}
0\mto \RDer\Sectc(W,k_!k^*\RDer k_*f_!f^*\sheaf{F}(1))\mto\RDer\Sectc(W,\RDer k_*f_!f^*\sheaf{F}(1))\\
\mto\RDer\Sectc(W,i_*i^*\RDer k_*f_!f^*\sheaf{F}(1))\mto 0
\end{multline*}
and chains of weak equivalences
\begin{align*}
\RDer\Sectc(W,k_!k^*\RDer k_*f_!f^*\sheaf{F}(1))&\sim\RDer\Sectc(U,f_!f^*\sheaf{F}(1)),\\
\RDer\Sectc(W,i_*i^*\RDer k_*f_!f^*\sheaf{F}(1))&\sim\RDer\Sect(\Sigma,i^*\RDer k_*f_!f^*\sheaf{F}(1)).
\end{align*}
By Theorem~\ref{thm: S torsion local}, we may thus reduce to the case $W=U$. Furthermore, we may shrink $U$ ad libitum. Hence, we may assume that $\sheaf{F}$ is locally constant on $U$ and smooth at $\infty$. Consequently, there exists a finite Galois extension $F'/F$ such that $F'$ is totally real, $g_{F'}\colon U_{F'}\mto U$ is \'etale and $g_{F'}^*\sheaf{F}$ is constant. Then $F'_\cyc/F$ is an admissible extension and
\[
\rho\coloneqq g_{F'}^*\sheaf{F}(U_{F'})
\]
may be viewed as a continuous representation of $G=\Gal(F'_\cyc/F)$ on a finitely generated, projective $\Lambda$-module. Write $g\colon U_{F'_\cyc}\mto U$ for the corresponding system of coverings of $U$ and observe that there exists a weak equivalence
\[
\eval_{\rho}(\RDer\Sectc(U,g_!g^*(\Int_p)_U(1)))\wto\RDer\Sectc(U,f_!f^*\sheaf{F}(1))
\]
with $\eval_{\rho}$ being defined by \eqref{eqn:def of eval} \cite[Prop. 5.9, 6.3, 6.5, 6.7]{Witte:MCVarFF}. Since $\eval_{\rho}$ takes complexes in $\cat{PDG}^{\cont,w_H}(\Int_p[[G]])$ to complexes in $\cat{PDG}^{\cont,w_H}(\Lambda[[\Gamma]])$, it remains to show that the cohomology groups $\HF^n_c(U,g_!g^*(\Int_p)_U(1))$ are finitely generated as $\Int_p$-modules. Now
\[
\HF^n_c(U,g_!g^*(\Int_p)_U(1))=
\begin{cases}
0 & \text{if $n\neq 2,3$,}\\
\Gal(M/F'_{\cyc}) & \text{if $n=2$,}\\
\Int_p & \text{if $n=3$,}
\end{cases}
\]
with $M$ denoting the maximal abelian $p$-extension of $F'_{\cyc}$ unramified over $U$ \cite[p. 548]{Kakde2}. At this point, we make use of Conjecture~\ref{conj:vanishing of mu} on the vanishing of the $\mu$-invariant to finish the proof for the first complex.

We now turn to the second complex.
We still assume that $\Lambda$ is a finite ring. Write $\Sigma\coloneqq W-U$, $V\coloneqq U\cup(X-W)$ and $j\colon U\mto V$, $\ell\colon V\mto X$, $i\colon \Sigma\mto X$ for the natural immersions. As mentioned in Remark~\ref{rem:exchanging W and V}, the exists a chain of weak equivalences
\[
 \RDer\Sectc(V,\RDer j_*f_!f^*\sheaf{F})\sim\RDer\Sect(W,k_!f_!f^*\sheaf{F}).
\]
Moreover, there is an exact sequence
\[
 0\mto\ell_!\RDer j_*f_!f^*\sheaf{F}\mto \RDer (\ell\comp j)_*f_!f^*\sheaf{F}\mto i_*i^*\RDer j_*f_!f^*\sheaf{F}\mto 0
\]
Using Theorem~\ref{thm: S torsion local} we may thus reduce to the case that $V=X$, $W=U$ and $\sheaf{F}$ locally constant on $U$ and smooth at $\infty$.
Let $\cmplx{P}$ be a strictly perfect complex of $\Lambda^{\op}[[G]]$-modules quasi-isomorphic to $\RDer\Sectc(U,f_!f^*\sheaf{F}^{\mdual_\Lambda}(1))$. By what we have proved above, $\cmplx{P}$ is also perfect as complex of $\Lambda^{\op}[[H]]$-modules. By \eqref{eqn:Artin-Verdier duality for coverings}  we obtain a weak equivalence
\[
 (\cmplx{P})^{\Mdual}\wto \RDer\Sect(U,f_!f^*\sheaf{F}).
\]
From Prop.~\ref{prop:main abstract duality} we conclude that $\RDer\Sect(U,f_!f^*\sheaf{F})$ is in $\cat{PDG}^{\cont,w_H}(U,\Lambda[[G]])$, as claimed. This finishes the proof of Theorem~\ref{thm: S torsion global}.

\section{Non-commutative Euler factors}\label{sec:euler factors}

Assume as before that $F_\infty/F$ is an admissible extension of a number field $F$ which is unramified over a dense open subscheme $U$ of $X$ and write $f\colon U_{F_\infty}\mto U$ for the system of Galois coverings of $U$ corresponding to $F_\infty/F$. Let $W$ be another dense open subscheme of $X$ containing $U$, but no place over $p$ and let $k\colon U\mto W$ denote the corresponding open immersion. We consider a complex $\cmplx{\sheaf{F}}$  in $\cat{PDG}^{\cont}(U,\Lambda)$. As the complexes
\[
\RDer\Sect(x,i^*\RDer k_*f_!f^*\cmplx{\sheaf{F}})
\]
are in $\cat{PDG}^{\cont,w_H}(\Lambda[[G]])$ for $i\colon x\mto W$ a closed point, we conclude that the endomorphism
\[
\RDer\Sect(\hat{x},i^*\RDer k_*f_!f^*\cmplx{\sheaf{F}})\xrightarrow{\id-\Frob_x}\RDer\Sect(\hat{x},i^*\RDer k_*f_!f^*\cmplx{\sheaf{F}})
\]
is in fact a weak equivalence in $w_H\cat{PDG}^{\cont}(\Lambda[[G]])$. Hence, it gives rise to an element in $\KTh_1(\Lambda[[G]]_S)$.

\begin{defn}
The non-commutative Euler factor $\ncL_{F_\infty/F}(x,\RDer k_*\cmplx{\sheaf{F}})$ of $\RDer k_*\cmplx{\sheaf{F}}$ at $x$
is the inverse of the class of the above weak equivalence in $\KTh_1(\Lambda[[G]]_S)$:
\[
\ncL_{F_\infty/F}(x,\RDer k_*\cmplx{\sheaf{F}})\coloneqq[\id-\Frob_x\lcirclearrowright\RDer\Sect(\hat{x},i^*\RDer k_*f_!f^*\cmplx{\sheaf{F}})]^{-1}
\]
\end{defn}

Note that $\ncL_{F_\infty/F}(x,\RDer k_*\cmplx{\sheaf{F}})$ is independent of our specific choice of a geometric point above $x$. Indeed, by \eqref{eqn:comparison of Frobenii} and relation (R5) in the definition of $\D_\bullet(\cat{W})$, we conclude that the classes $[\id-\Frob_x]$ and $[\id-\Frob'_x]$ agree in $\KTh_1(\Lambda[[G]]_S)$. Moreover, $\ncL_{F_\infty/F}(x,\RDer k_*\cmplx{\sheaf{F}})$ does not change if we enlarge $W$ by adding points not lying over $p$ or shrink $U$ by removing a finite set of points different from $x$.

\begin{prop}\label{prop:Euler factors are char elements}
The non-commutative Euler factor is a characteristic element for $\RDer\Sect(x,i^*\RDer k_*f_!f^*\cmplx{\sheaf{F}})$:
\[
\bh \ncL_{F_\infty/F}(x,\RDer k_*\cmplx{\sheaf{F}})=-[\RDer\Sect(x,i^*\RDer k_*f_!f^*\cmplx{\sheaf{F}})]
\]
in $\KTh_0(\Lambda[[G]],S)$.
\end{prop}
\begin{proof}
The complex $\RDer\Sect(x,i^*\RDer k_*f_!f^*\cmplx{\sheaf{F}})$ is weakly equivalent to the cone of the endomorphism
\[
\RDer\Sect(\hat{x},i^*\RDer k_*f_!f^*\cmplx{\sheaf{F}})\xrightarrow{\id-\Frob_x}\RDer\Sect(\hat{x},i^*\RDer k_*f_!f^*\cmplx{\sheaf{F}})
\]
shifted by one. Hence, the result follows from the explicit description of $\bh$ given in \eqref{eqn:def of boundary hom}.
\end{proof}

\begin{defn}
For a topological generator $\gamma\in\Gamma$, we define the local modification factor at $x$ to be the element
\[
\ncM_{F_\infty/F,\gamma}(x,\RDer k_*\cmplx{\sheaf{F}})\coloneqq
\ncL_{F_\infty/F}(x,\RDer k_*\cmplx{\sheaf{F}})
s_\gamma([\RDer\Sect(x,i^*\RDer k_*f_!f^*\cmplx{\sheaf{F}})]).
\]
in $\KTh_1(\Lambda[[G]])$.
\end{defn}

We obtain the following transformation properties.

\begin{prop}\label{prop:transformation of Euler factors}
With $k\colon U\mto W$ as above, let $\Lambda$ be any adic $\Int_{p}$-algebra and let $\cmplx{\sheaf{F}}$ be a complex in $\cat{PDG}^{\cont}(U,\Lambda)$.
\begin{enumerate}
\item Let $\Lambda'$ be another adic $\Int_{p}$-algebra.
For any complex $\cmplx{P}$ of $\Lambda'$-$\Lambda[[G]]$-bimodules which is strictly perfect as complex of $\Lambda'$-modules we have
\[
\ringtransf_{\cmplx{P[[G]]^{\delta}}}(\ncL_{F_\infty/F}(x,\RDer k_*\cmplx{\sheaf{F}}))=\ncL_{F_\infty/F}(x,\RDer k_*\ringtransf_{\cmplx{\tilde{P}}}(\cmplx{\sheaf{F}}))
\]
in $\KTh_1(\Lambda'[[G]]_S)$ and
\[
\ringtransf_{\cmplx{P[[G]]^{\delta}}}(\ncM_{F_\infty/F,\gamma}(x,\RDer k_*\cmplx{\sheaf{F}}))=\ncM_{F_\infty/F,\gamma}(x,\RDer k_*\ringtransf_{\cmplx{\tilde{P}}}(\cmplx{\sheaf{F}}))
\]
in $\KTh_1(\Lambda'[[G]])$.

\item Let $F'_\infty/F$ be an admissible subextension of $F_\infty/F$ with Galois group $G'$. Then
\[
\ringtransf_{\Lambda[[G']]}(\ncL_{F_\infty/F}(x,\RDer k_*\cmplx{\sheaf{F}}))=\ncL_{F'_\infty/F}(x,\RDer k_*\cmplx{\sheaf{F}})
\]
in $\KTh_1(\Lambda[[G']]_S)$ and
\[
\ringtransf_{\Lambda[[G']]}(\ncM_{F_\infty/F,\gamma}(x,\RDer k_*\cmplx{\sheaf{F}}))=\ncM_{F'_\infty/F,\gamma}(x,\RDer k_*\cmplx{\sheaf{F}})
\]
in $\KTh_1(\Lambda[[G']])$.

\item Let $F'/F$ be a finite extension inside $F_\infty/F$. Set $r\coloneqq[F'\cap F_\cyc: F]$. Write $f_{F'}\colon U_{F'}\mto U$ for the corresponding \'etale covering and $x_{F'}$ for the fibre in $\Spec \IntR_{F'}$ above $x$. Let $G'\subset G$ be the Galois group of the admissible extension $F_\infty/F'$ and consider $\Lambda[[G]]$ as a $\Lambda[[G']]$-$\Lambda[[G]]$-bimodule. Then
\[
\ringtransf_{\Lambda[[G]]}\big(\ncL_{F_\infty/F}(x,\RDer k_*\cmplx{\sheaf{F}})\big)=\prod_{y\in x_{F'}}\ncL_{F_\infty/F'}(y,\RDer k_* f_{F'}^*\cmplx{\sheaf{F}})
\]
in $\KTh_1(\Lambda[[G']]_S)$ and
\[
\ringtransf_{\Lambda[[G]]}\big(\ncM_{F_\infty/F,\gamma}(x,\RDer k_*\cmplx{\sheaf{F}})\big)=\prod_{y\in x_{F'}}\ncM_{F_\infty/F',\gamma^r}(y,\RDer k_* f_{F'}^*\cmplx{\sheaf{F}})
\]
in $\KTh_1(\Lambda[[G']])$.

\item With the notation of $(3)$, assume that $\cmplx{\sheaf{G}}$ is a complex in $\cat{PDG}^{\cont}(U_{F'},\Lambda)$ and consider $\Lambda[[G]]$ as a $\Lambda[[G]]$-$\Lambda[[G']]$-bimodule. Then
\[
\prod_{y\in x_{F'}}\ringtransf_{\Lambda[[G]]}\big(\ncL_{F_\infty/F'}(y,\RDer k_*\cmplx{\sheaf{G}})\big)=\ncL_{F_\infty/F}(x,\RDer k_* {f_{F'}}_*\cmplx{\sheaf{G}})
\]
in $\KTh_1(\Lambda[[G]]_S)$ and
\[
\prod_{y\in x_{F'}}\ringtransf_{\Lambda[[G]]}\big(\ncM_{F_\infty/F',\gamma^r}(y,\RDer k_*\cmplx{\sheaf{G}})\big)=\ncM_{F_\infty/F',\gamma}(x,\RDer k_* {f_{F'}}_*\cmplx{\sheaf{G}})
\]
in $\KTh_1(\Lambda[[G]])$.
\end{enumerate}
\end{prop}
\begin{proof}
Note that the functor $\ringtransf$ commutes up to weak equivalences with $\RDer\Sect$, $i^*$, and $\RDer k_*$ \cite[5.5.7]{Witte:PhD} and apply Proposition~\ref{prop:ring change for coverings} and Proposition~\ref{prop:transformation for algebraic L-function}. Part $(1)$ and $(2)$ are direct consequences.

For Part $(3)$, we additionally need the same reasoning as in the proof of \cite[Thm.~8.4.(3)]{Witte:MCVarFF} to verify that for any $\cmplx{\sheaf{G}}$ in $\cat{PDG}^{\cont}(U_{F'},\Lambda)$
\begin{equation}\label{eqn:comparison of Frobx and Froby}
 [\id-\Frob_x\lcirclearrowright\RDer\Sect(y\times_{x}\hat{x},\RDer k_*g_!g^*\cmplx{\sheaf{G}})]=
 [\id-\Frob_y\lcirclearrowright\RDer\Sect(\hat{y},\RDer k_*g_!g^*f_{F'}^*\cmplx{\sheaf{G}})]
\end{equation}
in $\KTh_1(\Lambda[[G']]_S)$. Here, $g\colon U_{F_\infty}\mto U_{F'}$ denotes the system of coverings induced by $f$. This implies
the formula for $\ringtransf_{\Lambda[[G]]}\big(\ncL_{F_\infty/F}(x,\RDer k_*\cmplx{\sheaf{F}})\big)$. Moreover, we have a weak equivalence
\[
\ringtransf_{\Lambda[[G]]}\RDer\Sect(x,\RDer k_*f_!f^*\cmplx{\sheaf{F}})\wto
\RDer\Sect(x_{F'},\RDer k_*g_!g^*f_{F'}^*\cmplx{\sheaf{F}})
\]
in $\cat{PDG}^{\cont}(\Lambda[[G']])$. In particular,
\[
s_{\gamma^r}([\ringtransf_{\Lambda[[G]]}\RDer\Sect(x,\RDer k_*f_!f^*\cmplx{\sheaf{F}})])
=\prod_{y\in x_{F'}}s_{\gamma^r}([\RDer\Sect(y,\RDer k_*g_!g^*f_{F'}^*\cmplx{\sheaf{F}})])
\]
from which the formula for $\ringtransf_{\Lambda[[G]]}\big(\ncM_{F_\infty/F,\gamma}(x,\RDer k_*\cmplx{\sheaf{F}})\big)$ follows.

For Part $(4)$ we use \eqref{eqn:comparison of Frobx and Froby} to show
\begin{equation*}
\begin{split}
\prod_{y\in x_{F'}}\ringtransf_{\Lambda[[G]]}\big(\ncL_{F_\infty/F'}&(y,\RDer k_*\cmplx{\sheaf{G}})\big)=\\
&=\ringtransf_{\Lambda[[G]]}([\id-\Frob_x\lcirclearrowright\RDer\Sect(x_{F'}\times_{x}\hat{x},\RDer k_*g_!g^*\cmplx{\sheaf{G}})]^{-1})\\
&=[\id-\Frob_x\lcirclearrowright\RDer\Sect(\hat{x},\RDer k_*f_!f^*{f_{F'}}_*\cmplx{\sheaf{G}})]^{-1}\\
&=\ncL_{F_\infty/F}(x,\RDer k_* {f_{F'}}_*\cmplx{\sheaf{G}}).
\end{split}
\end{equation*}
On the other hand, we also have a weak equivalence
\[
 \ringtransf_{\Lambda[[G]]}\RDer\Sect(x_{F'},\RDer k_*g_!g^*\cmplx{\sheaf{G}})\wto \RDer\Sect(x,\RDer k_*f_!f^*{f_{F'}}_*\cmplx{\sheaf{G}}),
\]
thence the formula for the local modification factors.
\end{proof}

If $\sheaf{G}$ is a smooth $\Lambda$-adic sheaf on $U$ and $x$ is a point in $U$, it makes sense to consider the element
\[
\ncL_{F_\infty/F}(x,\RDer k_*\sheaf{G}^{\mdual_{\Lambda}}(1))^{\Mdual}\in\KTh_1(\Lambda[[G]],S)
\]
as an alternative Euler factor, which does not agree with $\ncL_{F_\infty/F}(x,\RDer k_*\sheaf{G})$ in general. We shall show below that
\[
\ncL_{F_\infty/F}(x,\RDer k_*\sheaf{G}^{\mdual_{\Lambda}}(1))^{\Mdual}=[\id-\Frob_x^{-1}\lcirclearrowright\RDer\Sect(\hat{x},\RDer i^! k_!f_!f^*\sheaf{G})]
\]
and take this as a definition for arbitrary complexes $\cmplx{\sheaf{F}}$ in $\cat{PDG}^{\cont}(U,\Lambda)$.

\begin{defn}
The dual non-commutative Euler factor of $k_!\cmplx{\sheaf{F}}$ at $x\in W$
is the element
\[
\ncL^{\Mdual}_{F_\infty/F}(x,k_!\cmplx{\sheaf{F}})\coloneqq[\id-\Frob_x^{-1}\lcirclearrowright\RDer\Sect(\hat{x},\RDer i^! k_!f_!f^*\cmplx{\sheaf{F}})]
\]
in $\KTh_1(\Lambda[[G]]_S)$.
\end{defn}

\begin{prop}\label{prop:dual Euler factors are char elements}
The inverse of the dual non-commutative Euler factor is a characteristic element for $\RDer\Sect(x,\RDer i^!k_!f_!f^*\cmplx{\sheaf{F}})$:
\[
\bh \ncL^{\Mdual}_{F_\infty/F}(x,k_!\cmplx{\sheaf{F}})=[\RDer\Sect(x,\RDer i^! k_!f_!f^*\cmplx{\sheaf{F}})]
\]
in $\KTh_0(\Lambda[[G]],S)$.
\end{prop}
\begin{proof}
The complex $\RDer\Sect(x,i^*\RDer k_*f_!f^*\cmplx{\sheaf{F}})$ is weakly equivalent to the cone of the endomorphism
\[
\RDer\Sect(\hat{x},\RDer i^! k_!f_!f^*\cmplx{\sheaf{F}})\xrightarrow{\id-\Frob_x^{-1}}\RDer\Sect(\hat{x},\RDer i^! k_!f_!f^*\cmplx{\sheaf{F}})
\]
shifted by one. Hence, the result follows from the explicit description of $\bh$ given in \cite[Thm. A.5]{Witte:MCVarFF}.
\end{proof}

\begin{defn}
For a topological generator $\gamma\in\Gamma$, the dual local modification factor $k_!\cmplx{\sheaf{F}}$ at $x$
is the element
\[
\ncM^{\Mdual}_{F_\infty/F,\gamma}(x,k_!\cmplx{\sheaf{F}})\coloneqq
\ncL^{\Mdual}_{F_\infty/F}(x, k_!\cmplx{\sheaf{F}})
s_{\gamma^{-1}}([\RDer\Sect(x,\RDer i^!k_!f_!f^*\cmplx{\sheaf{F}})])^{-1}.
\]
\end{defn}

We obtain the following transformation properties.

\begin{prop}\label{prop:transformation of dual Euler factors}
With $k\colon U\mto W$ as above, let $\Lambda$ be any adic $\Int_{p}$-algebra and let $\cmplx{\sheaf{F}}$ be a complex in $\cat{PDG}^{\cont}(U,\Lambda)$.
\begin{enumerate}
\item Let $\Lambda'$ be another adic $\Int_{p}$-algebra.
For any complex $\cmplx{P}$ of $\Lambda'$-$\Lambda[[G]]$-bimodules which is strictly perfect as complex of $\Lambda'$-modules we have
\[
\ringtransf_{\cmplx{P[[G]]^{\delta}}}(\ncL_{F_\infty/F}^{\Mdual}(x, k_!\cmplx{\sheaf{F}}))=\ncL_{F_\infty/F}^{\Mdual}(x,k_!\ringtransf_{\cmplx{\tilde{P}}}(\cmplx{\sheaf{F}}))
\]
in $\KTh_1(\Lambda'[[G]]_S)$ and
\[
\ringtransf_{\cmplx{P[[G]]^{\delta}}}(\ncM_{F_\infty/F,\gamma}^{\Mdual}(x,k_!\cmplx{\sheaf{F}}))=\ncM_{F_\infty/F,\gamma}^{\Mdual}(x,k_!\ringtransf_{\cmplx{\tilde{P}}}(\cmplx{\sheaf{F}}))
\]
in $\KTh_1(\Lambda'[[G]])$.

\item Let $F'_\infty/F$ be an admissible subextension of $F_\infty/F$ with Galois group $G'$. Then
\[
\ringtransf_{\Lambda[[G']]}(\ncL_{F_\infty/F}^{\Mdual}(x,k_!\cmplx{\sheaf{F}}))=\ncL_{F'_\infty/F}^{\Mdual}(x, k_!\cmplx{\sheaf{F}})
\]
in $\KTh_1(\Lambda[[G']]_S)$ and
\[
\ringtransf_{\Lambda[[G']]}(\ncM_{F_\infty/F,\gamma}^{\Mdual}(x, k_!\cmplx{\sheaf{F}}))=\ncM_{F'_\infty/F,\gamma}^{\Mdual}(x,k_!\cmplx{\sheaf{F}})
\]
in $\KTh_1(\Lambda[[G']])$.

\item Let $F'/F$ be a finite extension inside $F_\infty/F$. Set $r\coloneqq[F'\cap F_\cyc: F]$. Write $f_{F'}\colon U_{F'}\mto U$ for the corresponding \'etale covering and $x_{F'}$ for the fibre in $\Spec \IntR_{F'}$ above $x$. Let $G'\subset G$ be the Galois group of the admissible extension $F_\infty/F'$ and consider $\Lambda[[G]]$ as a $\Lambda[[G']]$-$\Lambda[[G]]$-bimodule. Then
\[
\ringtransf_{\Lambda[[G]]}\big(\ncL_{F_\infty/F}^{\Mdual}(x, k_!\cmplx{\sheaf{F}})\big)=\prod_{y\in x_{F'}}\ncL_{F_\infty/F'}^{\Mdual}(y,k_! f_{F'}^*\cmplx{\sheaf{F}})
\]
in $\KTh_1(\Lambda[[G']]_S)$ and
\[
\ringtransf_{\Lambda[[G]]}\big(\ncM_{F_\infty/F,\gamma}^{\Mdual}(x,k_!\cmplx{\sheaf{F}})\big)=\prod_{y\in x_{F'}}\ncM_{F_\infty/F',\gamma^r}^{\Mdual}(y,k_! f_{F'}^*\cmplx{\sheaf{F}})
\]
in $\KTh_1(\Lambda[[G']])$.

\item With the notation of $(3)$, assume that $\cmplx{\sheaf{G}}$ is a complex in $\cat{PDG}^{\cont}(U_{F'},\Lambda)$ and consider $\Lambda[[G]]$ as a $\Lambda[[G]]$-$\Lambda[[G']]$-bimodule. Then
\[
\prod_{y\in x_{F'}}\ringtransf_{\Lambda[[G]]}\big(\ncL_{F_\infty/F'}^{\Mdual}(y,k_!\cmplx{\sheaf{G}})\big)=\ncL_{F_\infty/F}^{\Mdual}(x, k_! {f_{F'}}_*\cmplx{\sheaf{G}})
\]
in $\KTh_1(\Lambda[[G]]_S)$ and
\[
\prod_{y\in x_{F'}}\ringtransf_{\Lambda[[G]]}\big(\ncM_{F_\infty/F',\gamma^r}^{\Mdual}(y, k_!\cmplx{\sheaf{G}})\big)=\ncM_{F_\infty/F',\gamma}^{\Mdual}(x,k_! {f_{F'}}_*\cmplx{\sheaf{G}})
\]
in $\KTh_1(\Lambda[[G]])$.
\end{enumerate}
\end{prop}
\begin{proof}
The arguments are the same as in the proof of Prop.~\ref{prop:transformation of Euler factors}.
\end{proof}

\begin{prop}\label{prop:Euler factors under duality}\
\begin{enumerate}
\item\label{enum:EF comparison lc case} Let $\sheaf{G}$ be a smooth $\Lambda$-adic sheaf on $U$. Then
\begin{multline*}
(\ncL_{F_\infty/F}(x,\RDer k_*\sheaf{G}^{\mdual_\Lambda}(1)))^{\Mdual}=\ncL_{F_\infty/F}^{\Mdual}(x,k_!\sheaf{G})=\\
=
\begin{cases}
[-\Frob_x\lcirclearrowright \RDer\Sect(\hat{x},i^*f_!f^*\sheaf{G}(-1))]^{-1}\ncL_{F_\infty/F}(x,\sheaf{G}(-1))^{-1}
&\text{if $x\in U$}\\
[-\Frob_x\lcirclearrowright \RDer\Sect(\hat{x},i^*\RDer k_*f_!f^*\sheaf{G})]\ncL_{F_\infty/F}(x,\RDer k_*\sheaf{G})
&\text{if $x\in W-U$}
\end{cases}
\end{multline*}
in $\KTh_1(\Lambda[[G]],S)$ and
\[
(\ncM_{F_\infty/F,\gamma}(x,\RDer k_*\sheaf{G}^{\mdual_\Lambda}(1)))^{\Mdual}=\ncM_{F_\infty/F,\gamma}^{\Mdual}(x,k_!\sheaf{G})
\]
in $\KTh_1(\Lambda[[G]])$.
\item\label{enum:EF comparison point case} Let $\sheaf{G}$ be a $\Lambda$-adic sheaf on $x\in U$. Then
\[
\begin{aligned}
(\ncL_{F_\infty/F}(x,i_*\sheaf{G}^{\mdual_\Lambda}))^{\Mdual}&=\ncL_{F_\infty/F}^{\Mdual}(x,i_*\sheaf{G})\\
                                                             &=[-\Frob_x\lcirclearrowright \RDer\Sect(\hat{x},i^*f_!f^*i_*\sheaf{G})]^{-1}\ncL_{F_\infty/F}(x,i_*\sheaf{G})^{-1}
\end{aligned}
\]
in $\KTh_1(\Lambda[[G]],S)$ and
\[
(\ncM_{F_\infty/F,\gamma}(x,i_*\sheaf{G}^{\mdual_\Lambda}))^{\Mdual}=\ncM_{F_\infty/F,\gamma}^{\Mdual}(x,i_*\sheaf{G})
\]
in $\KTh_1(\Lambda[[G]])$.
\end{enumerate}
\end{prop}
\begin{proof}
We only need to prove the formulas for the non-commutative Euler factors, the formulas for the local modification factors then follow from Proposition~\ref{prop:s and Mdual commute}.

We begin by proving \eqref{enum:EF comparison lc case} in the case that $x\in W-U$. By \eqref{eqn:duality for coverings}, combined with Lemma~\ref{lem:strictly perfect resolution of local cohomology} and Lemma~\ref{lem:comparison with local cohomology}, we have
\begin{multline*}
\ncL_{F_\infty/F}(x,i_*\RDer k_*\sheaf{G}^{\mdual_\Lambda}(1))=\\
\begin{aligned}
&=\ringtransf_{\Lambda^{\op}[[G]]^\sharp}([\id-\Frob_x\lcirclearrowright \RDer\Sect(\hat{x},(i^*\RDer k_*f_!f^*\sheaf{G})^{\mdual_{\Lambda[[G]]}}(1))]^{-1})\\
&=\ringtransf_{\Lambda^{\op}[[G]]^\sharp}([\id-\Frob_x\lcirclearrowright \cmplx{D}_{\hat{x}}((\RDer k_*f_!f^*\sheaf{G})^{\mdual_{\Lambda[[G]]}}(1))]^{-1}).
\end{aligned}
\end{multline*}
From the Definition~\ref{defn:Mdual} of $\Mdual$, Lemma~\ref{lem:explicit local duality}, and again Lemma~\ref{lem:comparison with local cohomology}, we conclude
\begin{multline*}
\left(\ringtransf_{\Lambda^{\op}[[G]]^\sharp}([\id-\Frob_x\lcirclearrowright \cmplx{D}_{\hat{x}}((\RDer k_*f_!f^*\sheaf{G})^{\mdual_{\Lambda[[G]]}}(1))]^{-1})\right)^{\Mdual}=\\
\begin{aligned}
&=[\id-\Frob_x^{\mdual_{\Lambda[[G]]^{\op}}}\lcirclearrowright
\cmplx{D}_{\hat{x}}((\RDer k_*f_!f^*\sheaf{G})^{\mdual_{\Lambda[[G]]}}(1))^{\mdual_{\Lambda[[G]]^\op}}]\\
&=[\id-\Frob_x^{-1}\lcirclearrowright\RDer\Sect(\hat{x},i^*\RDer k_*f_!f^*\sheaf{G})]^{-1}\\
&=[-\Frob_x^{-1}\lcirclearrowright \RDer\Sect(\hat{x},i^*\RDer k_*f_!f^*\sheaf{G})]^{-1}\ncL_{F_\infty/F}(x,\RDer k_*\sheaf{G}).
\end{aligned}
\end{multline*}
Finally,
\[
[\id-\Frob_x^{-1}\lcirclearrowright\RDer\Sect(\hat{x},i^*\RDer k_*f_!f^*\sheaf{G})]^{-1}=\ncL_{F_\infty/F}^{\Mdual}(x,k_!\sheaf{G})
\]
by Lemma~\ref{lem:exactness of Ri upper shriek}.

The validity of the first equality in \eqref{enum:EF comparison point case} follows similarly from Lemma~\ref{lem:explicit finite duality} and the exchange formula \eqref{eqn:exchange formula}:
\[
 \begin{aligned}
  (\ncL_{F_\infty/F}(x,i_*\sheaf{G}^{\mdual_\Lambda}))^{\Mdual}&=([\id-\Frob_x\lcirclearrowright (i^*f_!f^*i_*\sheaf{G}^{\mdual_\Lambda})_{\hat{x}}]^{-1})^{\Mdual}\\
                                                               &=\ringtransf_{\Lambda[[G]]^{\sharp}}([\id-\Frob_x^{\mdual_{\Lambda[[G]]^\op}}\lcirclearrowright((i^*f_!f^*i_*\sheaf{G}^{\mdual_\Lambda})_{\hat{x}})^{\mdual_{\Lambda[[G]]^\op}}])\\
                                                               &=\ringtransf_{\Lambda[[G]]^{\sharp}}([\id-\Frob_x^{-1}\lcirclearrowright \RDer\Sect(\hat{x},(i^*f_!f^*i_*\sheaf{G}^{\mdual_\Lambda}))^{\mdual_{\Lambda[[G]]^\op}}])\\
                                                               &=[\id-\Frob_x^{-1}\lcirclearrowright\RDer\Sect(\hat{x},\RDer i^!f_!f^*(i_*\sheaf{G}^{\mdual_{\Lambda}}))^{\mdual_{\Lambda^\op}}]\\
                                                               &=\ncL_{F_\infty/F}^{\Mdual}(x,i_*\sheaf{G})
 \end{aligned}
\]
Further, write $\ell\colon U'\mto U$ for the complement of $x$ in $U$. Then
\[
 \RDer \ell_*\ell^*f_!f^*i_*\sheaf{G}\isomorph\RDer \ell_* f_!f^*\ell^*i_*\sheaf{G}=0
\]
and hence,
\[
 \RDer i^!f_!f^*i_*\sheaf{G}\isomorph i^*f_!f^*i_*\sheaf{G},
\]
from which the second equality in \eqref{enum:EF comparison point case} follows.

For the proof of \eqref{enum:EF comparison lc case} in the case that $x\in U$, we observe that
\[
\begin{aligned}
 (\ncL_{F_\infty/F}(x,\sheaf{G}^{\mdual_\Lambda}(1)))^{\Mdual}&=(\ncL_{F_\infty/F}(x,\RDer \ell_*\ell^*\sheaf{G}^{\mdual_\Lambda}(1)))^{\Mdual}(\ncL_{F_\infty/F}(x,i_*\RDer i^!\sheaf{G}^{\mdual_\Lambda}(1)))^{\Mdual}\\
                                                              &=(\ncL_{F_\infty/F}(x,\RDer \ell_*\ell^*\sheaf{G}^{\mdual_\Lambda}(1)))^{\Mdual}(\ncL_{F_\infty/F}(x,i_*(i^*\sheaf{G})^{\mdual_\Lambda}))^{\Mdual}\\
                                                              &=\ncL_{F_\infty/F}^{\Mdual}(x,\ell_!\ell^*\sheaf{G})\ncL_{F_\infty/F}^{\Mdual}(x,i^*\sheaf{G})\\
                                                              &=\ncL_{F_\infty/F}^{\Mdual}(x,\sheaf{G})
\end{aligned}
\]
by what we have proved above. For the second equality, we use that by absolute purity \cite[Ch.~II, Cor.~1.6]{Milne:ADT}, there exists chain of weak equivalences
\[
 i^*f_!f^*\sheaf{G}(-1)\sim \RDer i^!f_!f^*\sheaf{G}[2].
\]
\end{proof}

\section{Euler factors for the cyclotomic extension}\label{sec:cyclotomic euler factors}

In the case $F_\infty=F_{\cyc}$, we can give a different description of $\ncL_{F_\infty/F}(x,\RDer k_*\cmplx{\sheaf{F}})$. We will undergo the effort to allow arbitrary adic $\Int_p$-algebras $\Lambda$ as coefficient rings, but in the end, we will use the results only in the case that $\Lambda$ is the valuation ring in a finite extension of $\Rat_p$. If one restricts to this case, some of the technical constructions that follow may be skipped.

Let $\Lambda[t]$ be the polynomial ring over $\Lambda$ in the indeterminate $t$ that is assumed to commute with the elements of $\Lambda$. In the appendix we define a Waldhausen category $w_t\cat{P}(\Lambda[t])$: The objects are perfect complexes of $\Lambda[t]$-modules and cofibrations are injective morphism of complexes such that the cokernel is again perfect. A weak equivalence is a morphism $f\colon \cmplx{P}\mto\cmplx{Q}$ of perfect complexes of $\Lambda[t]$-modules such that $\Lambda\Ltensor_{\Lambda[t]}f$ is a quasi-isomorphism of complexes of $\Lambda$-modules. Here, $\Lambda$ is considered as a $\Lambda$-$\Lambda[t]$-bimodule via the augmentation map and $\Lambda\Ltensor_{\Lambda[t]}\cdot$ denotes the total derived tensor product as functor between the derived categories.

If $\Lambda$ is noetherian, then the subset
\[
S_t\coloneqq\set{f(t)\in \Lambda[t] \given f(0)\in \Lambda^{\times}}\subset\Lambda[t]
\]
is a left and right denominator set, the localisation $\Lambda[t]_{S_t}$ is semi-local and $\Lambda[t]\mto\Lambda[t]_{S_t}$ induces an isomorphism
\[
\KTh_1(w_t\cat{P}(\Lambda[t]))\isomorph\KTh_1(\Lambda[t]_{S_t})
\]
(Proposition~\ref{prop:S_t denominator set}). For commutative adic rings, which are always noetherian \cite[Cor. 36.35]{Warner:TopRings}, we may further identify
\[
\KTh_1(\Lambda[t]_{S_t})\isomorph\Lambda[t]_{S_t}^\times
\]
via the determinant map. In general, $S_t$ is not a left or right denominator set. We then take
\[
\KTh_1(\Lambda[t]_{S_t})\coloneqq\KTh_1(w_t\cat{P}(\Lambda[t]))
\]
as a definition.

For any adic $\Int_p$-algebra $\Lambda$ and any $\gamma\in\Gamma\isomorph\Int_p$, the ring homomorphism
\[
\ev_{\gamma}\colon \Lambda[t]\mapsto \Lambda[[\Gamma]],\quad f(t)\mapsto f(\gamma).
\]
induces a homomorphism
\[
\ev_{\gamma}\colon \KTh_1(\Lambda[t]_{S_t})\mto\KTh_1(\Lambda[[\Gamma]]_S)
\]
(Proposition~\ref{prop:extension of evgamma}). In the noetherian case, the proof boils down to a verification that $\ev_{\gamma}(S_t)\subset S$.

\begin{defn}
For $\cmplx{\sheaf{F}}=(\cmplx{\sheaf{F}}_I)_{I\in\openideals_{\Lambda}}\in \cat{PDG}^{\cont}(U,\Lambda)$ we define
\begin{align*}
 L(x,\RDer k_*\cmplx{\sheaf{F}},t)&\coloneqq[\id-t\Frob_x\lcirclearrowright\cmplx{P}]^{-1}\in \KTh_1(\Lambda[t]_{S_t})\\
 L^\Mdual(x,k_!\cmplx{\sheaf{F}},t)&\coloneqq[\id-t\Frob_x^{-1}\lcirclearrowright \cmplx{Q}]\in\KTh_1(\Lambda[t]_{S_t})
\end{align*}
where
\begin{align*}
 \cmplx{P}&\coloneqq\Lambda[t]\tensor_{\Lambda}\varprojlim_{I\in\openideals_{\Lambda}}\RDer\Sect(\hat{x},i^*\RDer k_*\cmplx{\sheaf{F}}),\\
 \cmplx{Q}&\coloneqq\Lambda[t]\tensor_{\Lambda}\varprojlim_{I\in\openideals_{\Lambda}}\RDer\Sect(\hat{x},\RDer i^!k_!\cmplx{\sheaf{F}})
\end{align*}
For any $1\neq\gamma\in\Gamma$, we write $L(x,\RDer k_* \cmplx{\sheaf{F}},\gamma)$ and $L^{\Mdual}(x,k_! \cmplx{\sheaf{F}},\gamma)$ for the images of $L(x,\RDer k_*\cmplx{\sheaf{F}},t)$ and $L^\Mdual(x,k_!\cmplx{\sheaf{F}},t)$ under
\[
\KTh_1(\Lambda[t]_{S_t})\xrightarrow{\ev_\gamma}\KTh_1(\Lambda[[\Gamma]]_S).
\]
\end{defn}

Since the endomorphism $\id-t\Frob_x$ is canonical, it follows easily from the relations in the definition of $\D_\bullet(\cat{W})$ that $L(x,\cmplx{\sheaf{F}},t)$ does only depend on the weak equivalence class of $\cmplx{\sheaf{F}}$ and is multiplicative on exact sequences. So, it defines a homomorphism
\[
 L(x,\RDer k_*(-),t)\colon\KTh_0(\cat{PDG}^{\cont}(U,\Lambda))\mto \KTh_1(\Lambda[t]_{S_t}).
\]
The same is also true for $L^{\Mdual}(x,k_!\cmplx{\sheaf{F}},t)$.

\begin{prop}
Let $\gamma_x\in\Gamma$ be the image of $\Frob_x$ in $\Gamma$. Then
\begin{align*}
\ncL_{F_{\cyc}/F}(x,\RDer k_*\cmplx{\sheaf{F}})&=L(x,\RDer k_*\cmplx{\sheaf{F}},\gamma_x^{-1}),\\
\ncL^{\Mdual}_{F_{\cyc}/F}(x,k_!\cmplx{\sheaf{F}})&=L^{\Mdual}(x,k_!\cmplx{\sheaf{F}},\gamma_x).
\end{align*}
\end{prop}
\begin{proof}
Since $p$ is invertible on $W$, the extension $F_{\cyc}/F$ is unramified over $W$. By the smooth base change theorem applied to the \'etale morphism $f_K\colon W_K\mto W$ for each finite subextension $K/F$ of $F_{\cyc}/F$ and the quasi-compact morphism $k\colon U\mto W$ there exists a weak equivalence
\[
f_!f^*\RDer k_*\cmplx{\sheaf{F}}\wto \RDer k_*f_!f^*\cmplx{\sheaf{F}}
\]
in $\cat{PDG}^{\cont}(W,\Lambda)$. By the proper base change theorem, there exists also an isomorphism
\[
 f_!f^*k_!\cmplx{\sheaf{F}}\isomorph k_!f_!f^*\cmplx{\sheaf{F}}
\]
Hence, we may assume $x\in U=W$.

For any finite subextension $K/F$ in $F_{\cyc}/F$ write $x_K$ for the set of places of $K$ lying over $x$ and $g\colon x_{F_{\cyc}}\mto x$ for the corresponding system of Galois covers. (We note that this system might not be admissible in the sense of \cite[Def. 2.6]{Witte:MCVarFF} for any base field $\FF\subset k(x)$: for example if $F=\Rat$ and $x=(\ell)$ with $\ell\neq p$ splitting in the cyclotomic $\Int_p$-extension of $\Rat$.) By the proper base change theorem there exists an isomorphism
\[
i^*f_!f^*\cmplx{\sheaf{F}}\isomorph g_!g^*i^*\cmplx{\sheaf{F}}.
\]
From Lemma~\ref{lem:exactness of Ri upper shriek} we can then also infer the existence of a weak equivalence
\[
 \RDer i^!f_!f^*\cmplx{\sheaf{F}}\wto g_!g^*\RDer i^!\cmplx{\sheaf{F}}.
\]

We will now concentrate on the proof of the equality
\[
\ncL_{F_{\cyc}/F}(x,\cmplx{\sheaf{F}})=L(x,\cmplx{\sheaf{F}},\gamma_x^{-1}).
\]
The proof of the equality for the dual Euler factors follows along the same lines, with $\Frob_x$ replaced by $\Frob_x^{-1}$ and $\gamma_x^{-1}$ replaced by $\gamma_x$.

By our choice of the embedding $\algc{F}\subset \algc{F}_x$, we have a compatible system of morphisms $\Spec \algc{k(x)}\mto x_K$ for each $K\subset F_{\cyc}$ and hence, distinguished isomorphisms
\[
\alpha\colon \Int[\Gal(K/F)]\tensor_{\Int}\sheaf{M}_{\hat{x}}\mto ({g_K}_!g_K^*\sheaf{M})_{\hat{x}}
\]
for the stalk $\sheaf{M}_{\hat{x}}$ in $\hat{x}$ of any \'etale sheaf $\sheaf{M}$ on $x$. The action of the Frobenius $\Frob_x$ on the righthand side corresponds to the operation of $\cdot\gamma_x^{-1}\tensor\Frob_x$ on the left-hand side. By compatibility, we may extend $\alpha$ to an isomorphism
\[
\alpha\colon\ringtransf_{\Lambda[[\Gamma]]}\RDer\Sect(\hat{x},i^*\cmplx{\sheaf{F}})\isomorph \RDer\Sect(\hat{x},g_!g^*i^*\cmplx{\sheaf{F}})
\]
in $\cat{PDG}^{\cont}(\Lambda[[\Gamma]])$. Hence,
\[
\ncL_{F_{\cyc}/F}(x,\cmplx{\sheaf{F}})=[\id-\gamma_x^{-1}\tensor \Frob_x \lcirclearrowright \ringtransf_{\Lambda[[\Gamma]]}\RDer\Sect(\hat{x},i^*\cmplx{\sheaf{F}})]^{-1}
\]
in $\KTh_1(\Lambda[[\Gamma]]_S)$. Furthermore, we may choose a strictly perfect complex of $\Lambda$-modules $\cmplx{P}$ with an endomorphism $f$ and a quasi-isomorphism
\[
\beta\colon\cmplx{P}\mto\varprojlim_{I\in\openideals_{\Lambda}}\RDer\Sect(\hat{x},\cmplx{\sheaf{G}})
\]
under which $f$ and $\Frob_x$ are compatible up to chain homotopy \cite[Lemma~3.3.2]{Witte:PhD}. The endomorphism
\[
 \id-tf\colon\Lambda[t]\tensor_{\Lambda}\cmplx{P}\mto\Lambda[t]\tensor_{\Lambda}\cmplx{P}
\]
is clearly a weak equivalence in $w_t\cat{P}(\Lambda[t])$. By \cite[Lemma 3.1.6]{Witte:PhD}, homotopic weak auto-equivalences have the same class in the first $\KTh$-group. Hence, we may conclude
\[
[\id-tf\lcirclearrowright \Lambda[t]\tensor_{\Lambda}\cmplx{P}]^{-1}=L(x,\RDer k_*\cmplx{\sheaf{F}},t)
\]
in $\KTh_1(\Lambda[t]_{S_t})$ and
\[
L(x,\RDer k_*\cmplx{\sheaf{F}},\gamma_x^{-1})=\ncL_{F_{\cyc}/F}(x,\RDer k_*\cmplx{\sheaf{F}})
\]
in $\KTh_1(\Lambda[[\Gamma]]_S)$.
\end{proof}

We will make this construction a little more explicit in the case that $\sheaf{F}$ is a $\Lambda$-adic sheaf on $U$. If $x\in U$, recall from \eqref{eqn:stalk} that there is a weak equivalence
\[
\sheaf{F}_{\hat{x}}\wto\RDer\Sect(\hat{x},i^*\RDer k_* \sheaf{F})
\]
in $\cat{PDG}^{\cont}(\Lambda)$ compatible with the operation of the Frobenius $\Frob_x$ on both sides.

Hence, we have
\begin{equation}\label{eqn:Euler factor}
L(x,\RDer k_*\sheaf{F},t)=
[\Lambda[t]\tensor_{\Lambda}\sheaf{F}_{\hat{x}}\xrightarrow{\id-t\Frob_x}\Lambda[t]\tensor_{\Lambda}\sheaf{F}_{\hat{x}}]^{-1}
\end{equation}
in $\KTh_1(\Lambda[t]_{S_t})$ and
\[
\ncL_{F_{\cyc}/F}(x,\RDer k_*\sheaf{F})=
[\Lambda[[\Gamma]]\tensor_{\Lambda}\sheaf{F}_{\hat{x}}\xrightarrow{\id-\gamma_x^{-1}\tensor\Frob_x}\Lambda[[\Gamma]]\tensor_{\Lambda}\sheaf{F}_{\hat{x}}]^{-1}
\]
in $\KTh_1(\Lambda[[\Gamma]]_S)$. In particular, if $\Lambda$ is commutative, then the isomorphism
\[
\KTh_1(\Lambda[t]_{S_t})\xrightarrow{\det}\Lambda[t]_{S_t}^\times
\]
sends $L(x,\RDer k_*\sheaf{F},t)$ to the inverse of the reverse characteristic polynomial of the geometric Frobenius operation on $\sheaf{F}_{\hat{x}}$.

If $\sheaf{F}$ is smooth in $x\in U$, then by absolute purity \cite[Ch.~II, Cor.~1.6]{Milne:ADT}, there exists chain of weak equivalences
\[
 \sheaf{F}(-1)_{\hat{x}}\sim \RDer\Sect(\hat{x},\RDer i^!k_!\sheaf{F})[2].
\]
Hence,
\begin{equation}\label{eqn:dual Euler factor smooth case}
L^\Mdual(x,k_!\sheaf{F},t)=
[\Lambda[t]\tensor_{\Lambda}\sheaf{F}(-1)_{\hat{x}}\xrightarrow{\id-t\Frob_x^{-1}}\Lambda[t]\tensor_{\Lambda}\sheaf{F}(-1)_{\hat{x}}]
\end{equation}
in $\KTh_1(\Lambda[t]_{S_t})$ and
\[
\ncL^\Mdual_{F_{\cyc}/F}(x,k_!\sheaf{F})=
[\Lambda[[\Gamma]]\tensor_{\Lambda}\sheaf{F}(-1)_{\hat{x}}\xrightarrow{\id-\gamma_x\tensor\Frob_x^{-1}}\Lambda[[\Gamma]]\tensor_{\Lambda}\sheaf{F}(-1)_{\hat{x}}]
\]
in $\KTh_1(\Lambda[[\Gamma]]_S)$. If $\sheaf{F}=i_*\sheaf{G}$ for some $\sheaf{G}$ in $\cat{PDG}^{\cont}(x,\Lambda)$, then there exists a weak equivalence
\[
 \sheaf{G}_{\hat{x}}\wto \RDer\Sect(\hat{x},\RDer i^!k_!\sheaf{F}).
\]
Hence,
\begin{equation}\label{eqn:dual Euler factor skyscraper case}
L^\Mdual(x,k_!\sheaf{F},t)=
[\Lambda[t]\tensor_{\Lambda}\sheaf{G}_{\hat{x}}\xrightarrow{\id-t\Frob_x^{-1}}\Lambda[t]\tensor_{\Lambda}\sheaf{G}_{\hat{x}}]
\end{equation}
in $\KTh_1(\Lambda[t]_{S_t})$ and
\[
\ncL^\Mdual_{F_{\cyc}/F}(x,k_!\sheaf{F})=
[\Lambda[[\Gamma]]\tensor_{\Lambda}\sheaf{G}_{\hat{x}}\xrightarrow{\id-\gamma_x\tensor\Frob_x^{-1}}\Lambda[[\Gamma]]\tensor_{\Lambda}\sheaf{G}_{\hat{x}}]
\]
in $\KTh_1(\Lambda[[\Gamma]]_S)$.

If $x\in W-U$, there exists by Lemma~\ref{lem:comparison with local cohomology}, Lemma~\ref{lem:strictly perfect resolution of local cohomology} and Lemma~\ref{lem:exactness of Ri upper shriek} a chain of weak equivalences
\[
\cmplx{D}_{\hat{x}}(\sheaf{F})\sim\RDer\Sect(\hat{x},i^*\RDer k_*f_!f^*\sheaf{F})\sim \RDer\Sect(\hat{x},\RDer i^!k_!f_!f^*\sheaf{F})[1]
\]
compatible with the Frobenius operation.

We conclude that for $x\in W-U$,
\begin{equation*}
\begin{aligned}
L(x,\RDer k_*\sheaf{F},t)&=[\id-t\Frob_x\lcirclearrowright \Lambda[t]\tensor_{\Lambda}D^0_{\hat{x}}(\sheaf{F})]^{-1}[\id-t\Frob_x\lcirclearrowright\Lambda[t]\tensor_{\Lambda}D^1_{\hat{x}}(\sheaf{F})],\\
L^{\Mdual}(x,k_!\sheaf{F},t)&=[\id-t\Frob_x^{-1}\lcirclearrowright \Lambda[t]\tensor_{\Lambda}D^0_{\hat{x}}(\sheaf{F})]^{-1}[\id-t\Frob_x^{-1}\lcirclearrowright\Lambda[t]\tensor_{\Lambda}D^1_{\hat{x}}(\sheaf{F})]
\end{aligned}
\end{equation*}
in $\KTh_1(\Lambda[t]_{S_t})$ and
\begin{equation*}
\begin{aligned}
\ncL_{F_{\cyc}/F}(x,\RDer k_*\sheaf{F})=
&[\id-\gamma_x^{-1}\tensor\Frob_x\lcirclearrowright\Lambda[[\Gamma]]\tensor_{\Lambda}D^0_{\hat{x}}(\sheaf{F})]^{-1}\\
&[\id-\gamma_x^{-1}\tensor\Frob_x\lcirclearrowright\Lambda[[\Gamma]]\tensor_{\Lambda}D^1_{\hat{x}}(\sheaf{F})],\\
\ncL_{F_{\cyc}/F}^{\Mdual}(x,k_!\sheaf{F})=
&[\id-\gamma_x\tensor\Frob_x^{-1}\lcirclearrowright\Lambda[[\Gamma]]\tensor_{\Lambda}D^0_{\hat{x}}(\sheaf{F})]^{-1}\\
&[\id-\gamma_x\tensor\Frob_x^{-1}\lcirclearrowright\Lambda[[\Gamma]]\tensor_{\Lambda}D^1_{\hat{x}}(\sheaf{F})].
\end{aligned}
\end{equation*}

Let $N$ be the stalk of $\sheaf{F}$ in the geometric point $\Spec \algc{F}$, viewed as $\Gal(\algc{F}_x/F_x)$-module. If the image of $\Gal(\algc{F}_x/F_x^{\nr})$ in the automorphism group of $N$ has trivial $p$-Sylow subgroups, then $N^{\Gal(\algc{F}_x/F_x^{\nr})}=D^{0}_{\hat{x}}(\sheaf{F})$ and the differential of $\cmplx{D}_{\hat{x}}(\sheaf{F})$ is trivial. Our formula then simplifies to
\begin{equation*}
\begin{aligned}
L(x,\RDer k_*\sheaf{F},t)&=
[\id-t\Frob_x\lcirclearrowright\Lambda[t]\tensor_{\Lambda}N^{\Gal(\algc{F}_x/F_x^{\nr})}]^{-1}\\
&\phantom{=}[\id-tq_x\Frob_x\lcirclearrowright\Lambda[t]\tensor_{\Lambda}N^{\Gal(\algc{F}_x/F_x^{\nr})}],\\
L^{\Mdual}(x,k_!\sheaf{F},t)&=[\id-t\Frob_x^{-1}\lcirclearrowright\Lambda[t]\tensor_{\Lambda}N^{\Gal(\algc{F}_x/F_x^{\nr})}]^{-1}\\
&\phantom{=}[\id-tq_x^{-1}\Frob_x^{-1}\lcirclearrowright\Lambda[t]\tensor_{\Lambda}N^{\Gal(\algc{F}_x/F_x^{\nr})}],
\end{aligned}
\end{equation*}
where $q_x$ is the order of the residue field $k(x)$. In particular, this is the case if $N$ is unramified in $x$.

Conversely, assume that the differential of $\cmplx{D}_{\hat{x}}(\sheaf{F})$ is an isomorphism. Since the operation of $\Frob_x$ commutes with the differential, we then have
\[
L(x,\RDer k_*\sheaf{F},t)=L^{\Mdual}(x,k_!\sheaf{F},t)=1.
\]

If $\Lambda=\IntR_C$ is the valuation ring of a finite field extension $C$ of $\Rat_p$, then we may replace $C\tensor_{\IntR_C}N$ by its semi-simplification $(C\tensor_{\IntR_C}N)^{\sesi}$ as a $\Gal(\algc{F}_x/F_x)$-module and obtain a corresponding decomposition $(C\tensor_{\IntR_C}\cmplx{D}_{\hat{x}}(\sheaf{F}))^{\sesi}$ of the complex $\cmplx{D}_{\hat{x}}(\sheaf{F})$. Note that
\[
(C\tensor_{\IntR_C}N)^{\sesi}=((C\tensor_{\IntR_C}N)^{\sesi})^{\Gal(\algc{F}_x/F_x^{\nr})}\oplus V
\]
with $V^{\Gal(\algc{F}_x/F_x^{\nr})}=0$. In particular, on each simple part of $(C\tensor_{\IntR_C}\cmplx{D}_{\hat{x}}(\sheaf{F}))^{\sesi}$, the differential is either trivial or an isomorphism. We conclude
\begin{equation}\label{eqn:modified Euler factor}
\begin{aligned}
\det L(x,\RDer k_*\sheaf{F},t)&=
\det(\id-t\Frob_x\lcirclearrowright ((C\tensor_{\IntR_C}N)^{\sesi})^{\Gal(\algc{F}_x/F_x^{\nr})})^{-1}\\
&\phantom{=}\det(\id-tq_x\Frob_x\lcirclearrowright((C\tensor_{\IntR_C}N)^{\sesi})^{\Gal(\algc{F}_x/F_x^{\nr})}),\\
\det L^{\Mdual}(x,k_!\sheaf{F},t)&=
\det(\id-t\Frob_x^{-1}\lcirclearrowright((C\tensor_{\IntR_C}N)^{\sesi})^{\Gal(\algc{F}_x/F_x^{\nr})})^{-1}\\
&\phantom{=}\det(\id-tq_x^{-1}\Frob_x^{-1}\lcirclearrowright((C\tensor_{\IntR_C}N)^{\sesi})^{\Gal(\algc{F}_x/F_x^{\nr})}),
\end{aligned}
\end{equation}
in the units $(C[t]_{(t)})^\times$ of the localisation of $C[t]$ at the prime ideal $(t)$.

We can also give more explicit formulas for the local modification factors. We will not make use of the following calculations in any other part of the text.

Let $M$ be one of the finitely generated and projective $\Lambda$-modules $\sheaf{F}_{\hat{x}}$, $\sheaf{F}(-1)_{\hat{x}}$, $D^0_{\hat{x}}(\sheaf{F})$, $D_{\hat{x}}^1(\sheaf{F})$. Then $M$ comes equipped with a continuous action of the Galois group $\Gal(\algc{k(x)}/k(x))$. Let $k(x)_\cyc$ denote the unique $\Int_p$-extension of $k(x)$ and let $r$ be the index of the image of $\Gamma'\coloneqq\Gal(k(x)_\cyc/k(x))$ in $\Gamma=\Gal(F_\cyc/F)$. Fix a topological generator $\gamma\in\Gamma$.

Clearly,
\[
 [\id-\gamma_x^{-1}\tensor\Frob_x\lcirclearrowright \Lambda[[\Gamma]]\tensor_{\Lambda}M]=
 \ringtransf_{\Lambda[[\Gamma]]}([\id-\gamma_x^{-1}\tensor\Frob_x\lcirclearrowright \Lambda[[\Gamma']]\tensor_{\Lambda}M])
\]
in $\KTh_1(\Lambda[[\Gamma]]_S)$, while
\begin{multline*}
\ringtransf_{\Lambda[[\Gamma]]}(s_{\gamma^r}([\coker(\id-\gamma_x^{-1}\tensor \Frob_x\lcirclearrowright\Lambda[[\Gamma']]\tensor_{\Lambda}M)]))=\\
s_\gamma(\ringtransf_{\Lambda[[\Gamma]]}([\coker(\id-\gamma_x^{-1}\tensor \Frob_x\lcirclearrowright\Lambda[[\Gamma']]\tensor_{\Lambda}M)]))\\
=s_\gamma([\coker(\id-\gamma_x^{-1}\tensor \Frob_x\lcirclearrowright\Lambda[[\Gamma]]\tensor_{\Lambda}M)])
\end{multline*}
by Proposition~\ref{prop:transformation for algebraic L-function}. Hence, it suffices to consider the case that $x$ does not split in $F_\cyc/F$. So, we assume from now on that $r=1$.

The image of $\Gal(\algc{k(x)}/k(x)_\cyc)$ in the automorphism group of $M$ is a finite commutative group $\Delta$ of order $d$ prime to $p$. Write
\[
e_\Delta\coloneqq\frac{1}{d}\sum_{\delta\in\Delta}\delta
\]
for the corresponding idempotent in the endomorphism ring. We thus obtain a canonical decomposition
\[
M\isomorph M'\oplus M''
\]
of $M$ with $M'\coloneqq e_\Delta M$ and $M''\coloneqq (\id-e_\Delta)M$.

Since
\[
(\id-\delta)(\id-e_{\Delta})m=(\id-\delta)m
\]
for every $\delta\in\Delta$ and every $m\in M$, the action of $\Delta$ on $M''$ is faithful. Since $d$ is prime to $p$, the action of $\Delta$ on $M''/\Jac(\Lambda)M''$ is still faithful. Indeed, the kernel $K$ of $\id-\delta\lcirclearrowright M''$ is a direct summand of $M''$ with $K\neq M''$. The Nakayama lemma then implies $K/\Jac(\Lambda) K\neq M''/\Jac(\Lambda)M''$ such that $\delta\lcirclearrowright M''/\Jac(\Lambda)M''$ cannot be the identity.

Note that $K$ is trivial if $\delta$ is a generator of $\Delta$. In this case, $\id-\delta\lcirclearrowright  M''/\Jac(\Lambda)M''$ must be an automorphism. We may apply this to a suitable $p^n$-th power of $\Frob_x$ to infer that $\id-\Frob_x\lcirclearrowright  M''/\Jac(\Lambda)M''$ is an automorphism. By the Nakayama lemma for $\Lambda[[\Gamma]]$ we conclude that the endomorphism $\id-\gamma_x^{-1}\tensor\Frob_x$ of $\Lambda[[\Gamma]]\tensor_{\Lambda}M''$ is also an automorphism.

The $\Lambda$-module $M'$ can be seen as $\Lambda[[\Gamma]]$-module with $\gamma_x$ acting as $\Frob_x$ and by Example~\ref{exmpl:algebraic L-functions} we have
\[
[\id-\gamma_x^{-1}\tensor\Frob_x\lcirclearrowright \Lambda[[\Gamma]]\tensor_{\Lambda}M'']=
s_{\gamma_x}([\coker(\id-\gamma_x^{-1}\tensor\Frob_x\lcirclearrowright \Lambda[[\Gamma]]\tensor_{\Lambda}M')])^{-1}
\]

We conclude that
\begin{multline*}
[\id-\gamma_x^{-1}\tensor\Frob_x\lcirclearrowright \Lambda[[\Gamma]]\tensor_{\Lambda}M]s_\gamma([\coker(\id-\gamma_x^{-1}\tensor\Frob_x\lcirclearrowright \Lambda[[\Gamma]]\tensor_{\Lambda}M)])=\\
[\id-\gamma_x^{-1}\tensor\Frob_x\lcirclearrowright \Lambda[[\Gamma]]\tensor_{\Lambda}M'']\frac{s_\gamma}{s_{\gamma_x}}([\coker(\id-\gamma_x^{-1}\tensor\Frob_x\lcirclearrowright \Lambda[[\Gamma]]\tensor_{\Lambda}M')]).
\end{multline*}
In particular, if $x\in U$ and $\HF^0(\Spec k(x)_\cyc,i^*\sheaf{F})=0$, then
\[
\ncM_{F_\cyc/F,\gamma}(x,\RDer k_*\sheaf{F})=\ncL_{F_\cyc/F}(x,\RDer k_*\sheaf{F}).
\]
If $x\in U$, $\HF^0(\Spec k(x)_\cyc,i^*\sheaf{F})=\sheaf{F}_{\hat{x}}$ and $\gamma=\gamma_x$, then
\[
\ncM_{F_\cyc/F,\gamma}(x,\RDer k_*\sheaf{F})=1.
\]
The same considerations apply to the dual local modification factors.

\begin{rem}
Let $V^0$ denote the set of closed points of $V$. Note that the infinite product
\[
\prod_{x\in V^0}\ncM_{F_\cyc/F,\gamma}(x,\RDer k_*\sheaf{F})
\]
does not converge in the compact topology of $\KTh_1(\Lambda[[\Gamma]])$. Indeed, by the Chebotarev density theorem,  we may find for each non-trivial finite subextensions $F'/F$ of $F_\cyc/F$ an infinite subset $S\subset U$ of closed points such that the elements
\[
\ncM_{F_\cyc/F,\gamma}(x,\RDer k_*\sheaf{F})\in\KTh_1(\Lambda[[\Gamma]])
\]
for $x\in S$ have a common non-trivial image in $\KTh_1(\Lambda/\Jac(\Lambda)[[\Gal(F'/F)]])$.
\end{rem}

\section{Artin representations}\label{sec:Artin reps}

Let $\IntR_C$ be the valuation ring of a finite extension field $C$ of $\Rat_p$ inside the fixed algebraic closure $\algc{\Rat}_p$ and assume as before that $\Gamma\isomorph\Int_p$. The augmentation map $\aug\colon \IntR_C[[\Gamma]]\mto \IntR_C$ extends to a map
\[
\aug\colon \KTh_1(\IntR_C[[\Gamma]]_S)\xrightarrow[\isomorph]{\det}\IntR_C[[\Gamma]]_S^\times\mto \PSp^1(C).
\]
Indeed, let $\frac{a}{s}\in\IntR_C[[\Gamma]]_S^\times$. Since $\IntR_C[[\Gamma]]$ is a unique factorisation domain and the augmentation ideal is a principal prime ideal, we may assume that not both $a$ and $s$ are contained in the augmentation ideal. Hence, we obtain a well-defined element
\[
\aug\left(\frac{a}{s}\right)\coloneqq[\aug(a): \aug(s)]\in \PSp^1(C)=C\cup\set{\infty},
\]
with $[x:y]$ denoting the standard projective coordinates of $\PSp^1(C)$. Note that this map agrees with $\aug'$ in \cite[\S 2.4]{Kakde2}. We further note that $\det a^{\Mdual}=((\det a)^\sharp)^{-1}$ for any $a\in \KTh_1(\IntR_C[[\Gamma]]_S)$. Since $\sharp\colon \IntR_C[[\Gamma]]_S\mto \IntR_C[[\Gamma]]_S$ maps $\gamma\in\Gamma$ to $\gamma^{-1}$ and is given by the identity on $\IntR_C$, we conclude that
\begin{equation}\label{eqn:evaluation and MDual}
\aug(a^\Mdual)=\aug(a)^{-1}.
\end{equation}
Finally, note that the diagram
\[
\xymatrix{
 \KTh_1(\IntR_C[t]_{S_t})\ar[r]^{\det}\ar[d]^{t\mapsto \gamma}& (C[t]_{(t)})^\times \ar[d]^{\frac{f}{g}\mapsto [f(1):g(1)]}\\
 \KTh_1(\IntR_C[[\Gamma]]_S)\ar[r]^{\aug}& \PSp^1(C)
}
\]
commutes for any choice of $\gamma\in\Gamma$ with $\gamma\neq 1$. On the right downward pointing map, $\frac{f}{g}$ denotes a reduced fraction.

From now on, we let $F$ denote a totally real number field. Consider an Artin representation $\rho\colon \Gal_F\mto \Gl_d(\IntR_C)$ (i.\,e.\ with open kernel) over $\IntR_C$. We will write
\[
 \check{\rho}\colon \Gal_F\mto \Gl_d(\IntR_C),\qquad g\mapsto \rho(g^{-1})^t
\]
for the dual representation. Assume for simplicity that $\rho$ is unramified over $U$: For each $x\in U$,  the restriction $\rho\restriction_{\Gal(\algc{F}_x/F_x^\nr)}$ of $\rho$ to $\Gal(\algc{F}_x/F_x^\nr)$ is trivial. Then $\rho$ corresponds to the smooth $\IntR_C$-adic sheaf $\sheaf{M}(\rho)$ on $U\subset W$ defined by \eqref{eqn:def of rep sheaf} and therefore, to an object in $\cat{PDG}^{\cont}(U,\IntR_C)$. Analogously, $\check{\rho}$ corresponds to $\sheaf{M}(\check{\rho})=\sheaf{M}(\rho)^{\mdual_{\IntR_C}}$.

We set $\Sigma\coloneqq\Spec(\IntR_F)-W$, $\Tau\coloneqq W-U$. Since the image of $\Gal(\algc{F}_x/F_x)$ in $\Gl_d(\IntR_C)$ is finite, the base change of $\rho\restriction_{\Gal(\algc{F}_x/F_x)}$ to $C$ is automatically a semi-simple representation of $\Gal(\algc{F}_x/F_x)$ for any $x\in X$. We let $\rho_x$ denote the representation of $\Gal(F_x^\nr/F_x)$ obtained from $\rho\restriction_{\Gal(\algc{F}_x/F_x)}$ by taking invariants under $\Gal(\algc{F}_x/F_x^\nr)$.

From \eqref{eqn:Euler factor}, \eqref{eqn:dual Euler factor smooth case}, and \eqref{eqn:modified Euler factor} we conclude
\begin{equation}\label{eqn:evaluation of Euler factors}
\begin{aligned}
\aug(\ncL_{F_\cyc/F}(x,&\RDer k_*\sheaf{M}(\rho)(n)))\\
&=\begin{cases}
[1 : \det(1-\rho_x(\Frob_x)q_x^{-n})] & \text{if $x\in U$,}\\
[\det(1-\rho_{x}(\Frob_x)q_x^{1-n}) : \det(1-\rho_{x}(\Frob_x)q_x^{-n})] & \text{if $x\in \Tau$},
\end{cases}\\
\aug(\ncL^{\Mdual}_{F_\cyc/F}(x,&k_!\sheaf{M}(\rho)(n)))\\
&=\begin{cases}
[\det(1-\check{\rho}_x(\Frob_x)q_x^{n-1}):1] & \text{if $x\in U$,}\\
[\det(1-\check{\rho}_{x}(\Frob_x)q_x^{n-1}) : \det(1-\check{\rho}_{x}(\Frob_x)q_x^{n})] & \text{if $x\in \Tau$},
\end{cases}
\end{aligned}
\end{equation}
where $q_x$ denotes the number of elements of the residue field $k(x)$. Classically, one often demands that $\rho$ is also unramified over $T$. Therefore, we stress that the above considerations  hold even if this is not the case. Note that $\det(1-\rho_x(\Frob_x)q_x^{-n})=0$ if and only if $n=0$ and $\rho_x$ contains the trivial representation as a subrepresentation.

For any open dense subscheme $V$ of $X$, we write $V^0$ for the set of closed points of $V$. Let $\alpha\colon \algc{\Rat}_p\mto \C$ be an embedding of $\algc{\Rat}_p$ into the complex numbers. We can then associate to the complex Artin representation $\alpha\comp\rho$ the classical $\Sigma$-truncated $\Tau$-modified Artin $L$-function with the product formula
\[
L_{\Sigma,\Tau}(\alpha\comp\rho,s)\coloneqq\prod_{x\in U^0}\det(1-\alpha\comp\rho_x(\Frob_x)q_x^{-s})^{-1}\prod_{x\in \Tau}\frac{\det(1-\alpha\comp\rho_{x}(\Frob_x)q_x^{1-s})}{\det(1-\alpha\comp\rho_{x}(\Frob_x)q_x^{-s})}
\]
for $\Re s>1$. Note that we follow the geometric convention of using the geometric Frobenius in the definition of the Artin $L$-function as in \cite{CoatesLichtenbaum:lAdicZetaFunctions}. With this convention, we have
\[
\begin{split}
L_{\Sigma,\Tau}(\alpha\comp\rho,n)&=\prod_{x\in W^0}\alpha(\aug(\ncL_{F_\cyc/F}(x,\RDer k_*\sheaf{M}(\rho)(n))))\\
                               &=\prod_{x\in W^0}\alpha(\aug(\ncL^\Mdual_{F_\cyc/F}(x,k_!\sheaf{M}(\check{\rho})(1-n))))^{-1}
\end{split}
\]
for all $n\in\Int$, $n>1$.

By \cite[Cor 1.4]{CoatesLichtenbaum:lAdicZetaFunctions} there exists for each $n\in\Int$, $n<0$ a well defined number $L_{\Sigma,\Tau}(\rho,n)\in C$ such that
\[
\alpha(L_{\Sigma,\Tau}(\rho,n))=L_{\Sigma,\Tau}(\alpha\comp\rho,n)\in \C
\]
Consequently,
\[
L_{\Sigma',\Tau'}(\rho,n)=L_{\Sigma,\Tau}(\rho,n)\prod_{x\in \Sigma'\cup \Tau'-\Sigma\cup \Tau}\aug(\ncL_{F_\cyc/F}(x,\RDer k_*\sheaf{M}(\rho)(n)))^{-1}
\]
if $\Sigma\subset\Sigma'$ and $\Tau\subset \Tau'$ with disjoint subsets $\Sigma'$ and $\Tau'$ of $X=\Spec \IntR_F$ such that $\rho$ is unramified over $X-\Sigma'-\Tau'$ and all the primes over $p$ are contained in $\Sigma'$.

Let $\cycchar_F\colon \Gal_F\mto \Int_p^\times$ denote the cyclotomic character such that
\[
\sigma(\zeta)=\zeta^{\cycchar_F(\sigma)}
\]
for every $\sigma\in\Gal_F$ and $\zeta\in\mu_{p^\infty}$. Further, we write $\omega_F\colon \Gal_F\mto \mu_{p-1}$ for the Teichm\"uller character, i.\,e.\ the composition of $\cycchar_F$ with the projection $\Int_p^\times\mto\mu_{p-1}$. Finally, we set
$\epsilon_F\coloneqq\cycchar_F\omega^{-1}$ and note that $\epsilon_F$ factors through $\Gamma=\Gal(F_\cyc/F)$.

Assume that $\rho$ factors through the Galois group of a totally real field extension of $F$. Then $\sheaf{M}(\rho)$ is smooth at $\infty$. Assume further that $U=W$, such that $\Tau$ is empty. Under Conjecture~\ref{conj:vanishing of mu} it follows from \cite{Greenberg:ArtinLfunctions} and from the validity of the classical main conjecture that for every integer $n$ there exist unique elements
\[
\ncL_{F_\cyc/F}(U,\sheaf{M}(\rho\omega^n_F)(1-n))\in \KTh_1(\IntR_C[[\Gamma]]_S)
\]
such that
\begin{equation}\label{eqn:values of padic Lfunctions}
\begin{aligned}
\eval_{\epsilon_F^{n}}(\ncL_{F_\cyc/F}(U,\sheaf{M}(\rho\omega^n_F)(1-n)))&=\ncL_{F_\cyc/F}(U,\sheaf{M}(\rho)(1)),\\
\aug(\ncL_{F_\cyc/F}(U,\sheaf{M}(\rho\omega^n_F)(1-n)))&=L_{\Sigma,\emptyset}(\rho\omega^n_F,1-n)\quad\text{if $n>1$}.
\end{aligned}
\end{equation}
with $\eval_{\epsilon_F^{n}}$ as defined in Example~\ref{exmpl:def of eval}. Beware that Greenberg uses the arithmetic convention for $L$-functions.

\begin{defn}\label{defn:global modification factors}
Let $\gamma\in \Gamma$ be a topological generator. We define the \emph{global modification factor} for $\sheaf{M}(\rho\omega^n_F)(1)$ and $f\colon U_{F_\cyc}\mto U$ to be the element
\begin{equation*}
\ncM_{F_\cyc/F,\gamma}(U,\sheaf{M}(\rho)(1))\coloneqq
\ncL_{F_\cyc/F}(U,\sheaf{M}(\rho)(1))s_\gamma([\RDer\Sectc(U,f_!f^*\sheaf{M}(\rho)(1))])
\end{equation*}
in $\KTh_1(\IntR_C[[\Gamma]])$.
\end{defn}

\section{Non-commutative \texorpdfstring{$L$}{L}-functions for \texorpdfstring{$\Lambda$-}{}adic sheaves}\label{sec:main results}

Throughout this section, we assume the validity of Conjecture~\ref{conj:vanishing of mu}. We recall the main theorem of \cite{Kakde2}.

\begin{thm}\label{thm:Kakdes theorem}
Let $U\subset X$ be a dense open subscheme with complement $\Sigma$ and assume that $p$ is invertible on $U$. Assume that $F_\infty/F$ is a really admissible extension which is unramified over $U$ and that $G\coloneqq\Gal(F_\infty/F)$ is a $p$-adic Lie group. Then there exists unique elements $\tilde{\ncL}_{F_\infty/F}(U,(\Int_p)_U(1))\in\KTh_1(\Int_p[[G]]_S)/\cSK_1(\Int_p[[G]])$ such that
\begin{enumerate}
\item \[\bh \tilde{\ncL}_{F_\infty/F}(U,(\Int_p)_U(1))=-[\RDer\Sectc(U,f_!f^*(\Int_p)_U(1))],\]
\item For any Artin representation $\rho$ factoring through $G$
    \[
    \eval_{\rho}(\tilde{\ncL}_{F_\infty/F}(U,(\Int_p)_U(1)))=\ncL_{F_\cyc/F}(U,\sheaf{M}(\rho)(1))
    \]
\end{enumerate}
\end{thm}
\begin{proof}
This is \cite[Thm.~2.11]{Kakde2} translated into our notations. Recall that our $\eval_{\rho\cycchar_F^{-n}}$ corresponds to $\eval_{\bar{\rho}\cycchar_F^n}$ in the notation of the cited article. Moreover, Kakde uses the arithmetic convention in the definition of $L$-values. Further, note that the $p$-adic $L$-function $\ncL_{F_\cyc/F}(U,\sheaf{M}(\rho)(1))$ is uniquely determined by the values $\aug(\eval_{\epsilon_F^n}(\ncL_{F_\cyc/F}(U,\sheaf{M}(\rho)(1))))$ for $n<0$ and $n\equiv 0 \mod p-1$.
Finally, Kakde's complex $\mathcal{C}(F_{\infty}/F)$ corresponds to $\RDer\Sectc(U,f_!f^*(\Int_p)_U(1))$ shifted by $3$ and therefore, the images of the two complexes under $\bh$ differ by a sign, but at the same time, his definition of $\bh$ differs by a sign from ours.
\end{proof}

We will improve this theorem as follows. Let $\Xi\coloneqq\Xi_F$ be the set of pairs $(U,F_\infty)$ such that $U\subset X$ is a dense open subscheme with $p$ invertible on $U$ and $F_\infty/F$ is a really admissible extension unramified over $U$.

\begin{thm}\label{thm:modification factors Z1}
Let $\gamma\in \Gamma\coloneqq\Gal(F_{\cyc}/F)$ be a topological generator. There exists a unique family of elements
\[
\left(\ncM_{F_\infty/F,\gamma}(U,(\Int_p)_U(1))\right)_{(U,F_\infty)\in \Xi}
\]
such that
\begin{enumerate}
\item $\ncM_{F_\infty/F,\gamma}(U,(\Int_p)_U(1))\in \KTh_1(\Int_p[[\Gal(F_\infty/F)]])$,
\item if $U\subset U'$ with complement $\Sigma$ and $(U,F_\infty), (U',F_\infty)\in \Xi$, then
\[
\ncM_{F_\infty/F,\gamma}(U',(\Int_p)_{U'}(1))=\ncM_{F_\infty/F,\gamma}(U,(\Int_p)_U(1))\prod_{x\in\Sigma}\ncM_{F_\infty/F,\gamma}(x,(\Int_p)_{U'}(1)),
\]
\item if $(U,F_\infty),(U,F'_\infty)\in \Xi$ such that $F'_\infty\subset F_\infty$ is a subfield, then
\[
\ringtransf_{\Int_p[[\Gal(F'_\infty/F)]]}(\ncM_{F_\infty/F,\gamma}(U,(\Int_p)_U(1)))=\ncM_{F'_\infty/F,\gamma}(U,(\Int_p)_U(1)),
\]
\item if $(U,F_\infty)\in \Xi$ and $\rho\colon \Gal(F_\infty/F)\mto \Gl_{n}(\IntR_C)$ is an Artin representation, then
\[
\eval_{\rho}(\ncM_{F_\infty/F,\gamma}(U,(\Int_p)_U(1)))=\ncM_{F_{\cyc}/F,\gamma}(U,\sheaf{M}(\rho)(1))
\]
with $\ncM_{F_{\cyc}/F,\gamma}(U,\sheaf{M}(\rho)(1))$ as in Definition~\ref{defn:global modification factors}.
\end{enumerate}
\end{thm}
\begin{proof}
\emph{Uniqueness:} Assume that $m_k(U,F_\infty)$, $k=1,2$ are two families with the listed properties. Then
\[
d(F_\infty)\coloneqq m_2(U,F_\infty)^{-1} m_1(U,F_\infty)
\]
does not depend on $U$ as a consequence of $(2)$.

Let $(U,F_\infty)\in\Xi$ be any pair such that $F_\infty/F_{\cyc}$ is finite and write $f\colon U_{F_\infty}\mto U$ for the system of coverings of $U$ associated to $F_\infty/F$. Then $(4)$ implies that the elements
\[
m_i(U,F_\infty)s_\gamma(-[\RDer\Sectc(U,f_!f^*(\Int_p)_U(1))])
\]
both agree with $\tilde{\ncL}_{F_\infty/F}(U,(\Int_p)_U(1))$ modulo $\cSK_1(\Int_p[[\Gal(F_\infty/F)]])$. Hence,
\[
d(F_\infty)\in \cSK_1(\Int_p[[\Gal(F_\infty/F)]]).
\]
By Corollary~\ref{cor:vanishing extension}, we may find a pair $(U',F'_\infty)\in \Xi$ such that $F'_\infty/F_\infty$ is finite, $U'\subset U$, and
\begin{equation}\label{eqn:right choice of a larger field to trivialize Sk}
\ringtransf_{\Int_p[[\Gal(F_\infty/F)]]}\colon \cSK_1(\Int_p[[\Gal(F'_\infty/F)]])\mto\cSK_1(\Int_p[[\Gal(F_\infty/F)]])
\end{equation}
is the zero map. We conclude from $(3)$ that $d(F_\infty)=1$ for all $(U,F_\infty)$ with $F_\infty/F_\cyc$ finite. Now for any really admissible extension $F_\infty/F$,
\[
\KTh_1(\Int_p[[\Gal(F_\infty/F)]])=\varprojlim_{F'_\infty}\KTh_1(\Int_p[[\Gal(F'_\infty/F)]])
\]
where $F'_\infty$ runs through the really admissible subextensions of $F_\infty/F$ with $F'_\infty/F_\cyc$ finite \cite[Prop.~1.5.3]{FK:CNCIT}. We conclude that $d(F_\infty)=1$ in general.

\emph{Existence:} It suffices to construct the elements for $(U,F_\infty)\in\Xi$ with $F_\infty/F_\cyc$ finite. Choose $(U',F'_\infty)$ as above such that the map in \eqref{eqn:right choice of a larger field to trivialize Sk} becomes trivial. Pick any $m\in\KTh_1(\Int_p[[\Gal(F'_\infty/F)]])$ such that
\begin{multline*}
m s_\gamma(-[\RDer\Sectc(U,f_!f^*(\Int_p)_U(1))])\equiv\\
 \tilde{\ncL}_{F'_\infty/F}(U',(\Int_p)_{U'}(1)) \mod \cSK_1(\Int_p[[\Gal(F'_\infty/F)]])
\end{multline*}
Define
\[
\ncM_{F_\infty/F,\gamma}(U,(\Int_p)_U(1))\coloneqq\ringtransf_{\Int_p[[\Gal(F_\infty/F)]]}(m)\prod_{x\in U-U'}\ncM_{F_\infty/F,\gamma}(x,(\Int_p)_U(1)).
\]
By Proposition~\ref{prop:transformation for algebraic L-function} and Proposition~\ref{prop:transformation of Euler factors} we conclude that
\[
\begin{split}
 \ncM_{F_\infty/F,\gamma}(U,(\Int_p)_U(1))s_\gamma(-[\RDer\Sectc(U,f_!f^*(\Int_p)_U(1))]) &\equiv \tilde{\ncL}_{F_\infty/F}(U,(\Int_p)_{U}(1)) \\
 \mod \cSK_1(\Int_p[[\Gal(F_\infty/F)]])
\end{split}
\]
and that $\ncM_{F_\infty/F,\gamma}(U,(\Int_p)_U(1))$ satisfies
\[
\eval_{\rho}(\ncM_{F_\infty/F,\gamma}(U,(\Int_p)_U(1)))=\ncM_{F_{\cyc}/F,\gamma}(U,\sheaf{M}(\rho)(1))
\]
for any Artin representation $\rho\colon \Gal(F_\infty/F)\mto \Gl_{n}(\IntR_C)$. In particular, the system
\[
 (\ncM_{F_\infty/F,\gamma}(U,(\Int_p)_U(1)))_{(F_\infty,U)\in\Xi}
\]
satisfies $(1)$ and $(4)$. By construction and again by Proposition~\ref{prop:transformation of Euler factors}, it is independent of the choices of $(U',F'_\infty)$ and $m$ and satisfies $(2)$ and $(3)$.
\end{proof}

\begin{cor}\label{cor:existence of equivariant zeta functions}
There exists a unique family of elements
\[
\left(\ncL_{F_\infty/F}(U,(\Int_p)_U(1))\right)_{(U,F_\infty)\in \Xi}
\]
such that
\begin{enumerate}
\item $\ncL_{F_\infty/F}(U,(\Int_p)_U(1))\in \KTh_1(\Int_p[[\Gal(F_\infty/F)]]_S)$,
\item if $(U,F_\infty)\in \Xi$ and $f\colon U_{F_\infty}\mto U$ denotes the associated system of coverings, then
\[
\bh\ncL_{F_\infty/F}(U,(\Int_p)_U(1))=-[\RDer\Sectc(U,f_!f^*(\Int_p)_U(1))]
\]
\item if $U'\subset U$ with complement $\Sigma$ and $(U',F_\infty), (U,F_\infty)\in \Xi$, then
\[
\ncL_{F_\infty/F}(U,(\Int_p)_U(1))=\ncL_{F_\infty/F}(U',(\Int_p)_{U'}(1))\prod_{x\in\Sigma}\ncL_{F_\infty/F}(x,(\Int_p)_U(1)),
\]
\item if $(U,F_\infty),(U,F'_\infty)\in \Xi$ such that $F'_\infty\subset F_\infty$ is a subfield, then
\[
\ringtransf_{\Int_p[[\Gal(F'_\infty/F)]]}(\ncL_{F_\infty/F}(U,(\Int_p)_U(1)))=\ncL_{F'_\infty/F}(U,(\Int_p)_U(1)),
\]
\item if $(U,F_\infty)\in \Xi$ and $\rho\colon \Gal(F_\infty/F)\mto \Gl_{d}(\IntR_C)$ is an Artin representation, then
\[
\eval_{\rho}(\ncL_{F_\infty/F}(U,(\Int_p)_U(1)))=\ncL_{F_{\cyc}/F}(U,\sheaf{M}(\rho)(1)).
\]
\end{enumerate}
\end{cor}
\begin{proof}
Fix a topological generator $\gamma\in\Gamma$ and set
\[
\ncL_{F_\infty/F}(U,(\Int_p)_U(1))\coloneqq \ncM_{F_\infty/F,\gamma}(U,(\Int_p)_U(1))s_{\gamma}(-[\RDer\Sectc(U,f_!f^*(\Int_p)_U(1))]).
\]
If $(\ell(U,F_\infty))_{(U,F_\infty)\in\Xi}$ is a second family with the listed properties, then
\[
\ell(U,F_\infty)s_{\gamma}([\RDer\Sectc(U,f_!f^*(\Int_p)_U(1))])=\ncM_{F_\infty/F,\gamma}(U,(\Int_p)_U(1))
\]
by the uniqueness of $\ncM_{F_\infty/F,\gamma}(U,(\Int_p)_U(1))$.
\end{proof}


Let $\Theta\coloneqq\Theta_F$ be the set of triples $(U,F_\infty,\Lambda)$ such that $U\subset X$ is a dense open subscheme with $p$ invertible on $U$, $F_\infty/F$ is a really admissible extension unramified over $U$ and $\Lambda$ is an adic $\Int_p$-algebra.

\begin{thm}\label{thm:general modification factors}
Let $\gamma\in \Gamma\coloneqq\Gal(F_{\cyc}/F)$ be a topological generator. There exists a unique family of homomorphisms
\[
\left(\ncM_{F_\infty/F,\gamma}(U,(-)(1))\colon \KTh_0(\cat{PDG}^{\cont,\infty}(U,\Lambda))\mto\KTh_1(\Lambda[[\Gal(F_\infty/F)]])\right)_{(U,F_\infty,\Lambda)\in \Theta}
\]
such that
\begin{enumerate}
\item for any $(U,F_\infty,\Int_p)\in\Theta$, $\ncM_{F_\infty/F,\gamma}(U,(\Int_p)_U(1))$ is the element constructed in Theorem~\ref{thm:modification factors Z1},
\item if $j\colon U'\mto U$ is an open immersion and $(U',F_\infty,\Lambda), (U,F_\infty,\Lambda)\in \Theta$, then
\[
\ncM_{F_\infty/F,\gamma}(U,\cmplx{\sheaf{F}}(1))=\ncM_{F_\infty/F,\gamma}(U',j^*\cmplx{\sheaf{F}}(1))
\prod_{x\in U-U'}\ncM_{F_\infty/F,\gamma}(x,\cmplx{\sheaf{F}}(1)),
\]
for any $\cmplx{\sheaf{F}}$ in $\cat{PDG}^{\cont,\infty}(U,\Lambda)$.
\item if $(U,F_\infty,\Lambda),(U,F'_\infty,\Lambda)\in \Theta$ such that $F'_\infty\subset F_\infty$ is a subfield, then
\[
\ringtransf_{\Lambda[[\Gal(F'_\infty/F)]]}(\ncM_{F_\infty/F,\gamma}(U,\cmplx{\sheaf{F}}(1)))=\ncM_{F'_\infty/F,\gamma}(U,\cmplx{\sheaf{F}}(1)),
\]
for any $\cmplx{\sheaf{F}}$ in $\cat{PDG}^{\cont,\infty}(U,\Lambda)$.
\item if $(U,F_\infty,\Lambda),(U,F_\infty,\Lambda')\in \Theta$ and $\cmplx{P}$ is a complex of $\Lambda'$-$\Lambda[[\Gal(F_\infty/F)]]$-bimodules, strictly perfect as complex of $\Lambda'$-modules, then
\[
\ringtransf_{\cmplx{{P[[\Gal(F_\infty/F)]]^{\delta}}}}(\ncM_{F_\infty/F,\gamma}(U,\cmplx{\sheaf{F}}(1)))=
\ncM_{F_\infty/F,\gamma}(U,\ringtransf_{\cmplx{\tilde{P}}}(\cmplx{\sheaf{F}})(1))
\]
for any $\cmplx{\sheaf{F}}$ in $\cat{PDG}^{\cont,\infty}(U,\Lambda)$.
\end{enumerate}
\end{thm}
\begin{proof}
Applying $(4)$ to the $\Lambda/I$-$\Lambda[[G]]$-bimodule $\Lambda/I[[G]]$ for any open two-sided ideal $I$ of $\Lambda$ and using
\[
\KTh_1(\Lambda[[G]])=\varprojlim_{I\in\openideals_{\Lambda}}\KTh_1(\Lambda/I[[G]]),
\]
we conclude that it is sufficient to consider triples $(U,F_\infty,\Lambda)\in\Theta$ with $\Lambda$ a finite ring. So, let $\Lambda$ be finite. Since $\ncM_{F_\infty/F,\gamma}(U,\cmplx{\sheaf{F}}(1))$ depends only on the class of $\cmplx{\sheaf{F}}$ in $\KTh_0(\cat{PDG}^{\cont,\infty}(U,\Lambda))$, we may assume that $\cmplx{\sheaf{F}}$ is a bounded complex of flat constructible \'etale sheaves of $\Lambda$-modules. Using $(2)$ we may shrink $U$ until $\cmplx{\sheaf{F}}$ is a complex of locally constant \'etale sheaves. Hence, there exists a $(U,F'_\infty,\Lambda)\in \Theta$ such that $F_\infty/F$ is a subextension of $F'_\infty/F$ and such that the restriction of $\cmplx{\sheaf{F}}$ to $U_K$ for some finite subextension $K/F$ of $F'_\infty/F$ is a complex of constant sheaves. By $(3)$, we may replace $F_\infty$ by $F'_\infty$. We may then find a complex of $\Lambda$-$\Int_p[[\Gal(F_\infty/F)]]$-bimodules $\cmplx{P}$, strictly perfect as complex of $\Lambda$ modules and a weak equivalence
\[
\ringtransf_{\cmplx{P}}f_!f^*(\Int_p)_U(1)\wto \cmplx{\sheaf{F}}(1)
\]
\cite[Prop. 6.8]{Witte:MCVarFF}. By $(4)$, the only possible definition of $\ncM_{F_\infty/F,\gamma}(U,\cmplx{\sheaf{F}}(1))$ is
\[
\ncM_{F_\infty/F,\gamma}(U,\cmplx{\sheaf{F}}(1))\coloneqq
\ringtransf_{\cmplx{{P[[\Gal(F_\infty/F)]]^{\delta}}}}(\ncM_{F_\infty/F,\gamma}(U,(\Int_p)_U(1))).
\]
It is then clear that this construction satisfies the given properties.
\end{proof}

\begin{thm}\label{thm:general dual modification factors}
Let $\gamma\in \Gamma\coloneqq\Gal(F_{\cyc}/F)$ be a topological generator. There exists a unique family of homomorphisms
\[
\left(\ncM^{\Mdual}_{F_\infty/F,\gamma}(U,-)\colon \KTh_0(\cat{PDG}^{\cont,\infty}(U,\Lambda))\mto\KTh_1(\Lambda[[\Gal(F_\infty/F)]])\right)_{(U,F_\infty,\Lambda)\in \Theta}
\]
such that
\begin{enumerate}
\item for any $(U,F_\infty,\Int_p)\in\Theta$,
\[
\ncM^\Mdual_{F_\infty/F,\gamma}(U,(\Int_p)_U)=(\ncM_{F_\infty/F,\gamma}(U,(\Int_p)_U(1)))^{\Mdual}
\]
\item if $j\colon U'\mto U$ is an open immersion and $(U',F_\infty,\Lambda), (U,F_\infty,\Lambda)\in \Theta$, then
\[
\ncM^\Mdual_{F_\infty/F,\gamma}(U,\cmplx{\sheaf{F}})=\ncM^{\Mdual}_{F_\infty/F,\gamma}(U',j^*\cmplx{\sheaf{F}})
\prod_{x\in U-U'}\ncM^\Mdual_{F_\infty/F,\gamma}(x,\cmplx{\sheaf{F}}),
\]
for any $\cmplx{\sheaf{F}}$ in $\cat{PDG}^{\cont,\infty}(U,\Lambda)$.
\item if $(U,F_\infty,\Lambda),(U,F'_\infty,\Lambda)\in \Theta$ such that $F'_\infty\subset F_\infty$ is a subfield, then
\[
\ringtransf_{\Lambda[[\Gal(F'_\infty/F)]]}(\ncM^{\Mdual}_{F_\infty/F,\gamma}(U,\cmplx{\sheaf{F}}(1)))=\ncM^{\Mdual}_{F'_\infty/F,\gamma}(U,\cmplx{\sheaf{F}}),
\]
for any $\cmplx{\sheaf{F}}$ in $\cat{PDG}^{\cont,\infty}(U,\Lambda)$.
\item if $(U,F_\infty,\Lambda),(U,F_\infty,\Lambda')\in \Theta$ and $\cmplx{P}$ is a complex of $\Lambda'$-$\Lambda[[\Gal(F_\infty/F)]]$-bimodules, strictly perfect as complex of $\Lambda'$-modules, then
\[
\ringtransf_{\cmplx{{P[[\Gal(F_\infty/F)]]^{\delta}}}}(\ncM^{\Mdual}_{F_\infty/F,\gamma}(U,\cmplx{\sheaf{F}}))=
\ncM^{\Mdual}_{F_\infty/F,\gamma}(U,\ringtransf_{\cmplx{\tilde{P}}}(\cmplx{\sheaf{F}}))
\]
for any $\cmplx{\sheaf{F}}$ in $\cat{PDG}^{\cont,\infty}(U,\Lambda)$.
\end{enumerate}
Moreover, for any $(U,F_\infty,\Lambda)\in\Theta$ and any smooth $\Lambda$-adic sheaf $\sheaf{F}$ on $U$, we have
\[
 \ncM^\Mdual_{F_\infty/F,\gamma}(U,\sheaf{F})=(\ncM_{F_\infty/F,\gamma}(U,\sheaf{F}^{\mdual_\Lambda}(1)))^{\Mdual}.
\]
\end{thm}
\begin{proof}
We proceed as in Theorem~\ref{thm:general modification factors} and use Lemma~\ref{lem:Mdual and ringtransf}.
\end{proof}

\begin{prop}\label{prop:changing the generator}
Assume that $\gamma,\gamma'$ are two topological generators of $\Gamma$. Then
\begin{align*}
\frac{\ncM_{F_\infty/F,\gamma}}{\ncM_{F_\infty/F,\gamma'}}(U,\cmplx{\sheaf{F}}(1))
&=\frac{s_{\gamma}}{s_{\gamma'}}([\RDer\Sectc(U,f_!f^*\cmplx{\sheaf{F}}(1))])\\
\frac{\ncM^{\Mdual}_{F_\infty/F,\gamma}}{\ncM^{\Mdual}_{F_\infty/F,\gamma'}}(U,\cmplx{\sheaf{F}})
&=\frac{s_{(\gamma')^{-1}}}{s_{\gamma^{-1}}}([\RDer\Sect(U,f_!f^*\cmplx{\sheaf{F}})])
\end{align*}
for any $(U,F_\infty,\Lambda)\in\Theta$ and any $\cmplx{\sheaf{F}}$ in $\cat{PDG}^{\cont,\infty}(U,\Lambda)$.
\end{prop}
\begin{proof}
By definition, these identities hold for the local modification factors and by Corollary~\ref{cor:existence of equivariant zeta functions} and Proposition~\ref{prop:s and Mdual commute} they hold for $\cmplx{\sheaf{F}}=(\Int_p)_U$. Hence,
\begin{align*}
\ncM_{F_\infty/F,\gamma}(U,\cmplx{\sheaf{F}}(1))&=
\frac{s_{\gamma}}{s_{\gamma'}}([\RDer\Sectc(U,f_!f^*\cmplx{\sheaf{F}}(1))])
\ncM_{F_\infty/F,\gamma'}(U,\cmplx{\sheaf{F}}(1))\\
\ncM^\Mdual_{F_\infty/F,\gamma}(U,\cmplx{\sheaf{F}})&=
\frac{s_{(\gamma')^{-1}}}{s_{\gamma^{-1}}}([\RDer\Sect(U,f_!f^*\cmplx{\sheaf{F}})])
\ncM^\Mdual_{F_\infty/F,\gamma'}(U,\cmplx{\sheaf{F}})
\end{align*}
by the uniqueness assertion in Theorem~\ref{thm:general modification factors} and Theorem~\ref{thm:general dual modification factors}.
\end{proof}

\begin{thm}\label{thm:change of base field}
Let $F'/F$ be a finite extension of totally real fields. Set $r\coloneqq [F'\cap F_\cyc:F]$ and let $\gamma\in\Gal(F_{\cyc}/F)$ be a topological generator. Assume that $(U,F_\infty,\Lambda)\in\Theta_F$ with $F'\subset F_\infty$ and write $f_{F'}\colon U_{F'}\mto U$ for the associated covering. Then
\begin{enumerate}
 \item for every $\cmplx{\sheaf{F}}$ in $\cat{PDG}^{\cont,\infty}(U,\Lambda)$,
\begin{align*}
\ncM_{F_\infty/F',\gamma^r}(U_{F'},f_{F'}^*\cmplx{\sheaf{F}}(1))
&=\ringtransf_{\Lambda[[\Gal(F_\infty/F)]]}\ncM_{F_\infty/F,\gamma}(U,\cmplx{\sheaf{F}}(1)),\\
\ncM^{\Mdual}_{F_\infty/F',\gamma^r}(U_{F'},f_{F'}^*\cmplx{\sheaf{F}})
&=\ringtransf_{\Lambda[[\Gal(F_\infty/F)]]}\ncM^{\Mdual}_{F_\infty/F,\gamma}(U,\cmplx{\sheaf{F}});
\end{align*}
\item for every $\cmplx{\sheaf{G}}$ in $\cat{PDG}^{\cont,\infty}(U_{F'},\Lambda)$,
\begin{align*}
\ncM_{F_\infty/F,\gamma}(U,{f_{F'}}_*\cmplx{\sheaf{G}}(1))
&=\ringtransf_{\Lambda[[\Gal(F_\infty/F)]]}\ncM_{F_\infty/F',\gamma^r}(U_{F'},\cmplx{\sheaf{G}}(1)),\\
\ncM^{\Mdual}_{F_\infty/F,\gamma}(U,{f_{F'}}_*\cmplx{\sheaf{G}})
&=\ringtransf_{\Lambda[[\Gal(F_\infty/F)]]}\ncM^{\Mdual}_{F_\infty/F',\gamma^r}(U_{F'},\cmplx{\sheaf{G}}).
\end{align*}
\end{enumerate}
\end{thm}
\begin{proof}
We prove the identities for the global modification factors; the proof for the global dual modification factors is the same.

We first note that for any complex $\cmplx{P}$ of $\Lambda'$-$\Lambda[[\Gal(F_\infty/F)]]$-bimodules, strictly perfect as complex of $\Lambda'$-modules, there exists an obvious isomorphism of complexes of $\Lambda'[[\Gal(F_\infty/F')]]$-$\Lambda[[\Gal(F_\infty/F)]]$-bimodules
\begin{multline*}
 \Lambda'[[\Gal(F_\infty/F)]]\tensor_{\Lambda'[[\Gal(F_\infty/F)]]}\cmplx{{P[[\Gal(F_\infty/F)]]^\delta}}\isomorph\\
 \cmplx{{P[[\Gal(F_\infty/F')]]^\delta}}\tensor_{\Lambda[[\Gal(F_\infty/F')]]}\Lambda[[\Gal(F_\infty/F)]].
\end{multline*}
Hence,
\begin{equation}\label{eqn:ringtransf under change of fields I}
 \ringtransf_{\Lambda'[[\Gal(F_\infty/F)]]}\comp\ringtransf_{\cmplx{{P[[\Gal(F_\infty/F)]]^\delta}}}=\ringtransf_{\cmplx{{P[[\Gal(F_\infty/F')]]^\delta}}}\comp\ringtransf_{\Lambda[[\Gal(F_\infty/F)]]}
\end{equation}
as homomorphisms from $\KTh_1(\Lambda[[\Gal(F_\infty/F)]])$ to $\KTh_1(\Lambda'[[\Gal(F_\infty/F')]])$. Likewise, for a complex $\cmplx{Q}$ of $\Lambda'$-$\Lambda[[\Gal(F_\infty/F')]]$-bimodules, strictly perfect as complex of $\Lambda'$-modules, we have an equality
\begin{equation}\label{eqn:ringtransf under change of fields II}
 \ringtransf_{\Lambda'[[\Gal(F_\infty/F)]]}\comp\ringtransf_{\cmplx{{Q[[\Gal(F_\infty/F')]]^\delta}}}=\ringtransf_{\cmplx{{Q[[\Gal(F_\infty/F)]]^\delta}}}\comp\ringtransf_{\Lambda[[\Gal(F_\infty/F)]]}.
\end{equation}
in $\Hom(\KTh_1(\Lambda[[\Gal(F_\infty/F')]]),\KTh_1(\Lambda'[[\Gal(F_\infty/F)]]))$.

In particular, we may reduce to the case of finite $\Int_p$-algebras $\Lambda$ by choosing $\cmplx{P}=\Lambda=\cmplx{Q}$ with the trivial action of $\Gal(F_\infty/F)$ and $\Gal(F_\infty/F')$, respectively. By Proposition~\ref{prop:transformation of Euler factors}.(4) we may then shrink $U$ until $\cmplx{\sheaf{F}}$ and $\cmplx{\sheaf{G}}$ may be assumed to be strictly perfect complexes of locally constant \'etale sheaves. Using the identities \eqref{eqn:ringtransf under change of fields I} and \eqref{eqn:ringtransf under change of fields II} again, we may reduce to the case $\Lambda=\Int_p$ and $\cmplx{\sheaf{F}}=(\Int_p)_U$, $\cmplx{\sheaf{G}}=(\Int_p)_{U_{F'}}$. We may then further reduce to the case that $F_\infty/F'_\cyc$ is a finite extension.

Setting
\begin{align*}
q&\coloneqq\frac{\ncM_{F_\infty/F',\gamma^r}(U_{F'},f_{F'}^*(\Int_p)_{U_{F'}}(1))}{\ringtransf_{\Lambda[[\Gal(F_\infty/F)]]}\ncM_{F_\infty/F,\gamma}(U,(\Int_p)_U(1))}\in \KTh_1(\Lambda[[\Gal(F_\infty/F')]]),\\
q'&\coloneqq\frac{\ncM_{F_\infty/F,\gamma}(U,{f_{F'}}_*(\Int_p)_{U_{F'}}(1))}{\ringtransf_{\Lambda[[\Gal(F_\infty/F)]]}\ncM_{F_\infty/F',\gamma^r}(U_{F'},(\Int_p)_U(1))}\in\KTh_1(\Lambda[[\Gal(F_\infty/F)]]),
\end{align*}
it suffices to show that $q=1$ and $q'=1$.

Let $g\colon U_{F_\infty}\mto U_{F'}$ denote the restriction of $f\colon U_{F_\infty}\mto U$. Write
\[
 M\coloneqq\Int_p[\Gal(F_\infty/F')\backslash \Gal(F_\infty/F)]
\]
for the $\Int_p$-$\Int_p[[\Gal(F_\infty/F)]]$-bimodule freely generated as $\Int_p$-module by the right cosets $\Gal(F_\infty/F)\sigma$ for $\sigma\in \Gal(F_\infty/F)$ and on which $\tau\in\Gal(F_\infty/F)$ operates by right multiplication. From Proposition~\ref{prop:ring change for coverings} we conclude
\begin{gather*}
\begin{split}
\ringtransf_{\Int_p[[\Gal(F_\infty/F)]]}\RDer\Sectc(U,f_!f^*(\Int_p)_U(1))&\sim\RDer\Sectc(U,{f_{F'}}_*g_!g^*f_{F'}^*(\Int_p)_U(1))\\
                                                                     &\sim\RDer\Sectc(U_{F'},g_!g^*(\Int_p)_{U_{F'}}(1)),
\end{split}\\
\begin{split}
\ringtransf_{\Int_p[[\Gal(F_\infty/F)]]}\RDer\Sectc(U_{F'},g_!g^*&(\Int_p)_{U_{F'}}(1))\\
                                                                          &\sim\ringtransf_{\Int_p[[\Gal(F_\infty/F)]]}\RDer\Sectc(U,{f_{F'}}_*g_!g^*(\Int_p)_{U_{F'}}(1))\\
                                                                          &\sim\RDer\Sectc(U,f_!f^*{f_{F'}}_*(\Int_p)_{U_{F'}}(1))\\
                                                                          &\sim\ringtransf_{M[[\Gal(F_\infty/F)]]^\delta}\RDer\Sectc(U,f_!f^*(\Int_p)_U(1)).\\
\end{split}
\end{gather*}
Additionally, we note that
\[
 \ncM_{F_\infty/F,\gamma}(U,{f_{F'}}_*(\Int_p)_U(1))=
\ringtransf_{M[[\Gal(F_\infty/F)]]^\delta}\ncM_{F_\infty/F,\gamma}(U,(\Int_p)_U(1))
\]
by Theorem~\ref{thm:general modification factors}.

From this and from Proposition~\ref{prop:transformation for algebraic L-function}, we conclude
\begin{align*}
q&=\frac{\ncL_{F_\infty/F'}(U_{F'},(\Int_p)_{U_{F'}}(1))}{\ringtransf_{\Lambda[[\Gal(F_\infty/F)]]}\ncL_{F_\infty/F}(U,(\Int_p)_U(1))},\\
q'&=\frac{\ringtransf_{M[[\Gal(F_\infty/F)]]^\delta}\ncL_{F_\infty/F}(U,(\Int_p)_U(1))}{\ringtransf_{\Lambda[[\Gal(F_\infty/F)]]}\ncL_{F_\infty/F'}(U_{F'},(\Int_p)_U(1))}.
\end{align*}

Let $C/\Rat_p$ be a finite field extension and
\begin{align*}
\rho'&\colon \Gal(F_\infty/F')\mto\Gl_d(\IntR_C)\\
\rho &\colon \Gal(F_\infty/F)\mto\Gl_d(\IntR_C)
\end{align*}
be Artin representations. Write
\begin{align*}
\aug_{F}&\colon \IntR_C[[\Gal(F_{\cyc}/F)]]\mto \IntR_C\\
\aug_{F'}&\colon \IntR_C[[\Gal(F'_{\cyc}/F')]]\mto \IntR_C
\end{align*}
for the augmentation maps. We denote by $\Ind_{F'}^F\rho'$ and $\Res_{F}^{F'}\rho$ the induced an restricted representations, respectively.

Then for every $n\in\Int$
\[
 \aug_{F'}\comp\eval_{\rho'\epsilon_{F'}^n}\comp \ringtransf_{\Int_p[[\Gal(F_\infty/F)]]}=\aug_{F}\comp\eval_{\Ind_{F'}^F\rho'\epsilon_{F'}^n}=\aug_{F}\comp\eval_{\epsilon_F^n\Ind_{F'}^F\rho'}
\]
as maps from $\KTh_1(\Int_p[[\Gal(F_\infty/F)]]_S)$ to $\PSp^1(C)$ and
\[
 \aug_{F}\comp\eval_{\rho\epsilon_{F}^n}\comp \ringtransf_{\Int_p[[\Gal(F_\infty/F)]]}=\aug_{F'}\comp\eval_{\Res_{F}^{F'}\rho\epsilon_F^n}=\aug_{F'}\comp\eval_{\epsilon_{F'}^n\Res_{F}^{F'}\rho}
\]
as maps from $\KTh_1(\Int_p[[\Gal(F_\infty/F')]]_S)$ to $\PSp^1(C)$. From \eqref{eqn:values of padic Lfunctions} and the transformation properties of the complex Artin $L$-functions with respect to inflation and restriction we conclude that for $n<-1$ and $\Sigma\coloneqq X-U$
\begin{align*}
\aug_{F'}\comp\eval_{\rho'\epsilon_{F'}^n}&(\ringtransf_{\Lambda[[\Gal(F_\infty/F)]]}\ncL_{F_\infty/F}(U,(\Int_p)_U(1)))=\\
                                                                                    &=L_{\Sigma,\emptyset}(\omega_{F}^{-n}\Ind_{F'}^F\rho',1+n)\\
                                                                                    &=L_{\Sigma_{F'},\emptyset}(\rho'\omega_{F'}^{-n},1+n)\\
                                                                                     &=\aug_{F'}\comp\eval_{\rho'\epsilon_{F'}^n}(\ncL_{F_\infty/F'}(U_{F'},(\Int_p)_{U_{F'}}(1))),\\
\aug_F\comp \eval_{\rho\epsilon_{F}^n}&(\ringtransf_{\Lambda[[\Gal(F_\infty/F)]]}\ncL_{F_\infty/F'}(U_{F'},(\Int_p)_{U_{F'}}(1))=\\
                                                                                                   &=L_{\Sigma_{F'},\emptyset}(\omega_{F'}^{-n}\Res_{F}^{F'}\rho,1+n)\\
                                                                                                   &= L_{\Sigma,\emptyset}(\omega_{F}^{-n}\Ind_{F}^{F'}\Res_{F}^{F'}\rho,1+n)\\
                                                                                         &=\aug_{F}\comp \eval_{\epsilon_{F}^n\Ind_{F}^{F'}\Res_{F}^{F'} \rho}(\ncL_{F_\infty/F}(U,(\Int_p)_U(1))))\\
                                                                                       &=\aug_{F}\comp \eval_{\rho\epsilon_{F}^n}(\ringtransf_{M[[\Gal(F_\infty/F)]]^\delta}\ncL_{F_\infty/F}(U,(\Int_p)_U(1))).
\end{align*}
From \cite[Lemma~3.4]{Burns:MCinNCIwasawaTh+RelConj} we conclude that $\eval_{\rho'}(q)=1$ in $\KTh_1(\IntR_C[[\Gamma]])$ and thus $\aug_{F'}(\eval_{\rho'}(q))=1$ in $C$ for every Artin representation $\rho'$ of $\Gal(F_\infty/F')$.
In particular, with $K$ running through the finite Galois extension fields of $F$ in $F_\infty$, the images of $q$ in the groups $\KTh_1(\Rat_p[\Gal(K/F)])$ are trivial. This implies
\[
q\in\cSK_1(\Int_p[[\Gal(F_\infty/F')]]).
\]
 Using Corollary~\ref{cor:vanishing extension} we find a suitable admissible extension $L_\infty/F$ unramified over $U'\subset U$ such that
\[
\ringtransf_{\Int_p[[\Gal(F_\infty/F')]]}\colon\cSK_1(\Int_p[[\Gal(L_\infty/F')]])\mto\cSK_1(\Int_p[[\Gal(F_\infty/F')]])
\]
is the zero map. As
\[
 q=\ringtransf_{\Int_p[[\Gal(F_\infty/F')]]}\left(\frac{\ncM_{L_\infty/F',\gamma^r}(U'_{F'},f_{F'}^*(\Int_p)_{U'_{F'}}(1))}{\ringtransf_{\Lambda[[\Gal(L_\infty/F)]]}\ncM_{L_\infty/F,\gamma}(U',(\Int_p)_{U'}(1))}\right),
\]
we conclude $q=1$. The proof that $q'=1$ follows the same pattern.
\end{proof}

\begin{defn}
Let $F$ be a totally real field, $k\colon U\mto W$ be an open immersion of open dense subschemes of $X=\Spec \IntR_F$ such that $p$ is invertible on $W$, and $\Lambda$ be an adic $\Int_p$-algebra. Fix a topological generator $\gamma\in\Gal(F_{\cyc}/F)$. For any $\cmplx{\sheaf{F}}$ in $\cat{PDG}^{\cont,\infty}(U,\Lambda)$, and any really admissible extension $F_\infty/F$ unramified over $U$, we set
\begin{align*}
\ncM_{F_\infty/F,\gamma}(W,\RDer k_*\cmplx{\sheaf{F}}(1))&\coloneqq \ncM_{F_\infty/F,\gamma}(U,\cmplx{\sheaf{F}}(1))\mkern-10mu\prod_{x\in W-U}\mkern-10mu\ncM_{F_\infty/F,\gamma}(x,\RDer k_*f_!f^*\cmplx{\sheaf{F}}(1)),\\
\ncM^{\Mdual}_{F_\infty/F,\gamma}(W,\RDer k_*\cmplx{\sheaf{F}})&\coloneqq \ncM^{\Mdual}_{F_\infty/F,\gamma}(U,\cmplx{\sheaf{F}})\mkern-10mu\prod_{x\in W-U}\mkern-10mu\ncM^{\Mdual}_{F_\infty/F,\gamma}(x,k_!f_!f^*\cmplx{\sheaf{F}})
\end{align*}
in $\KTh_1(\Lambda[[\Gal(F_\infty/F)]])$ and
\begin{align*}
\ncL_{F_\infty/F}(W,\RDer k_*\cmplx{\sheaf{F}}(1))&\coloneqq \ncM_{F_\infty/F,\gamma}(W,\RDer k_*\cmplx{\sheaf{F}}(1))s_\gamma(-[\RDer\Sectc(W,\RDer k_*f_!f^*\cmplx{\sheaf{F}}(1))]),\\
\ncL^{\Mdual}_{F_\infty/F}(W,k_!\cmplx{\sheaf{F}})&\coloneqq \ncM^{\Mdual}_{F_\infty/F,\gamma}(W,k_!\cmplx{\sheaf{F}})s_{\gamma^{-1}}([\RDer\Sect(W,k_!f_!f^*\cmplx{\sheaf{F}})])
\end{align*}
in $\KTh_1(\Lambda[[\Gal(F_\infty/F)]]_S)$.
\end{defn}

Note that we do not assume that $F_\infty/F$ is unramified over $W$. If it is unramified over $W$, then
\begin{align*}
\RDer\Sectc(W,\RDer k_*f_!f^*\cmplx{\sheaf{F}}(1))&=\RDer\Sectc(W,f_!f^*\RDer k_*\cmplx{\sheaf{F}}(1)),\\
\RDer\Sect(W,k_!f_!f^*\cmplx{\sheaf{F}})&=\RDer\Sect(W,f_!f^*k_!\cmplx{\sheaf{F}})
\end{align*}
and the two possible definitions of $\ncM_{F_\infty/F,\gamma}(W,\RDer k_*\cmplx{\sheaf{F}}(1))$ and $\ncM^{\Mdual}_{F_\infty/F,\gamma}(W,k_!\cmplx{\sheaf{F}})$ agree. Moreover, by Proposition~\ref{prop:changing the generator}, $\ncL_{F_\infty/F}(W,\RDer k_*\cmplx{\sheaf{F}}(1))$ and $\ncL^{\Mdual}_{F_\infty/F}(W,k_!\cmplx{\sheaf{F}})$ do not depend on the choice of $\gamma$.

In the following corollary, we compile a list of the transformation properties of $\ncL_{F_\infty/F}(W,\RDer k_*\cmplx{\sheaf{F}}(1))$ and $\ncL^{\Mdual}_{F_\infty/F}(W,k_!\cmplx{\sheaf{F}})$.

\begin{cor}\label{cor:transformation properties for Lfunctions}
Let $F$ be a totally real field, $k\colon U\mto W$ be an open immersion of open dense subschemes of $X=\Spec \IntR_F$ such that $p$ is invertible on $W$, and $\Lambda$ be an adic $\Int_p$-algebra. Fix a  $\cmplx{\sheaf{F}}$ in $\cat{PDG}^{\cont,\infty}(U,\Lambda)$, and a really admissible extension $F_\infty/F$ unramified over $U$.
\begin{enumerate}
\item Write $f\colon U_{F_\infty}\mto U$ for the system of coverings associated to $F_\infty/F$. Then
\begin{align*}
\bh\ncL_{F_\infty/F}(W,\RDer k_*\cmplx{\sheaf{F}}(1))&=-[\RDer\Sectc(W,\RDer k_* f_!f^*\cmplx{\sheaf{F}}(1))],\\
\bh\ncL^{\Mdual}_{F_\infty/F}(W,k_!\cmplx{\sheaf{F}})&=[\RDer\Sect(W,k_! f_!f^*\cmplx{\sheaf{F}})]
\end{align*}
\item If $\cmplx{\sheaf{G}}$ and $\cmplx{\sheaf{F}}$ are weakly equivalent in $\cat{PDG}^{\cont,\infty}(U,\Lambda)$, then
\begin{align*}
 \ncL_{F_\infty/F}(W,\RDer k_*\cmplx{\sheaf{F}}(1))&=\ncL_{F_\infty/F}(W,\RDer k_*\cmplx{\sheaf{G}}(1)),\\
 \ncL^{\Mdual}_{F_\infty/F}(W,k_!\cmplx{\sheaf{F}})&=\ncL^{\Mdual}_{F_\infty/F}(W,k_!\cmplx{\sheaf{G}}).
\end{align*}
\item If $0\mto\cmplx{\sheaf{F}'}\mto \cmplx{\sheaf{F}}\mto\cmplx{\sheaf{F}''}\mto 0$ is an exact sequence in $\cat{PDG}^{\cont,\infty}(U,\Lambda)$, then
\begin{align*}
 \ncL_{F_\infty/F}(W,\RDer k_*\cmplx{\sheaf{F}}(1))&=\ncL_{F_\infty/F}(W,\RDer k_*\cmplx{\sheaf{F}'}(1))\ncL_{F_\infty/F}(W,\RDer k_*\cmplx{\sheaf{F}''}(1)),\\
 \ncL^{\Mdual}_{F_\infty/F}(W,k_!\cmplx{\sheaf{F}})&=\ncL^{\Mdual}_{F_\infty/F}(W,k_!\cmplx{\sheaf{F}'})\ncL^{\Mdual}_{F_\infty/F}(W,k_!\cmplx{\sheaf{F}''}).
\end{align*}
\item If $W'$ is an open dense subscheme of $X$ on which $p$ is invertible and $k'\colon W\mto W'$ is an open immersion, then
\begin{align*}
\ncL_{F_\infty/F}(W',\RDer (k'k)_*\cmplx{\sheaf{F}}(1))=&\ncL_{F_\infty/F}(W,\RDer k_*\cmplx{\sheaf{F}}(1))\\
&\prod_{x\in W'-W}\ncL_{F_\infty/F}(x,\RDer (k'k)_*\cmplx{\sheaf{F}}(1)),\\
\ncL^{\Mdual}_{F_\infty/F}(W',(k'k)_!\cmplx{\sheaf{F}})=&\ncL^{\Mdual}_{F_\infty/F}(W,k_!\cmplx{\sheaf{F}})\\
&\prod_{x\in W'-W}\ncL^{\Mdual}_{F_\infty/F}(x,(k'k)_!\cmplx{\sheaf{F}}).
\end{align*}
\item If $i\colon x\mto U$ is a closed point, then
\begin{align*}
\ncL_{F_\infty/F}(W,\RDer k_*i_*i^*\cmplx{\sheaf{F}}(1))&=\ncL_{F_\infty/F}(x,\cmplx{\sheaf{F}}(1)),\\
\ncL^{\Mdual}_{F_\infty/F}(W,k_!i_*\RDer i^!\cmplx{\sheaf{F}})&=\ncL^{\Mdual}_{F_\infty/F}(x,\cmplx{\sheaf{F}}).
\end{align*}
\item If $F'_\infty/F$ is a really admissible subextension of $F_\infty/F$, then
\begin{align*}
\ringtransf_{\Lambda[[\Gal(F'_\infty/F)]]}(\ncL_{F_\infty/F}(W,\RDer k_*\cmplx{\sheaf{F}}(1)))&=\ncL_{F'_\infty/F}(W,\RDer k_*\cmplx{\sheaf{F}}(1)),\\
\ringtransf_{\Lambda[[\Gal(F'_\infty/F)]]}(\ncL^{\Mdual}_{F_\infty/F}(W,k_!\cmplx{\sheaf{F}}))&=\ncL^{\Mdual}_{F'_\infty/F}(W,k_!\cmplx{\sheaf{F}}).
\end{align*}
\item If $\Lambda'$ is another adic $\Int_p$-algebra and $\cmplx{P}$ is a complex of $\Lambda'$-$\Lambda[[\Gal(F_\infty/F)]]$-bimodules, strictly perfect as complex of $\Lambda'$-modules, then
\begin{align*}
\ringtransf_{\cmplx{{P[[\Gal(F_\infty/F)]]^{\delta}}}}(\ncL_{F_\infty/F}(W,\RDer k_*\cmplx{\sheaf{F}}(1)))&=
\ncL_{F_\infty/F}(W,\RDer k_*\ringtransf_{\cmplx{\tilde{P}}}(\cmplx{\sheaf{F}})(1)),\\
\ringtransf_{\cmplx{{P[[\Gal(F_\infty/F)]]^{\delta}}}}(\ncL^{\Mdual}_{F_\infty/F}(W,k_!\cmplx{\sheaf{F}}))&=
\ncL^{\Mdual}_{F_\infty/F}(W,k_!\ringtransf_{\cmplx{\tilde{P}}}(\cmplx{\sheaf{F}})).
\end{align*}
\item If $F'/F$ is a finite extension inside $F_\infty$ and $f_{F'}\colon U_{F'}\mto U$ the associated covering, then
\begin{align*}
\ringtransf_{\Lambda[[\Gal(F_\infty/F)]]}(\ncL_{F_\infty/F}(W,\RDer k_*\cmplx{\sheaf{F}}(1)))&=
\ncL_{F_\infty/F'}(W_{F'},\RDer k_*f_{F'}^*\cmplx{\sheaf{F}}(1)),\\
\ringtransf_{\Lambda[[\Gal(F_\infty/F)]]}(\ncL^{\Mdual}_{F_\infty/F}(W,k_!\cmplx{\sheaf{F}}))&=
\ncL^{\Mdual}_{F_\infty/F'}(W_{F'},k_!f_{F'}^*\cmplx{\sheaf{F}}).
\end{align*}
\item With the notation of $(8)$, if $\cmplx{\sheaf{G}}$ is in $\cat{PDG}^{\cont,\infty}(U_{F'},\Lambda)$, then
\begin{align*}
\ringtransf_{\Lambda[[\Gal(F_\infty/F)]]}(\ncL_{F_\infty/F'}(W_{F'},\RDer k_*\cmplx{\sheaf{G}}(1)))&=
\ncL_{F_\infty/F}(W,\RDer k_*{f_{F'}}_*\cmplx{\sheaf{G}}(1)),\\
\ringtransf_{\Lambda[[\Gal(F_\infty/F)]]}(\ncL^{\Mdual}_{F_\infty/F'}(W_{F'},k_!\cmplx{\sheaf{G}}))&=
\ncL^{\Mdual}_{F_\infty/F}(W,k_!{f_{F'}}_*\cmplx{\sheaf{G}}).
\end{align*}
\item If $\sheaf{F}$ is a smooth $\Lambda$-adic sheaf on $U$ which is smooth at $\infty$, then
\[
\ncL^{\Mdual}_{F_\infty/F}(W,k_!\sheaf{F})=(\ncL_{F_\infty/F}(W,\RDer k_*\sheaf{F}^{\mdual_{\Lambda}}(1)))^{\Mdual}.
\]
\item If $C/\Rat_p$ is a finite field extension and $\rho\colon \Gal(F_\infty/F)\mto \Gl_d(\IntR_C)$ is an Artin representation, then
\begin{align*}
\eval_{\rho\epsilon_F^{-n}}(\ncL_{F_\infty/F}(W,\RDer k_*(\Int_p)_U(1)))&=\ncL_{F_\cyc/F}(W,\RDer k_*\sheaf{M}(\rho\omega_F^n)(1-n)),\\
\eval_{\rho\epsilon_F^{n}}(\ncL^{\Mdual}_{F_\infty/F}(W,\RDer k_*(\Int_p)_U))&=\ncL^{\Mdual}_{F_\cyc/F}(W,\RDer k_*\sheaf{M}(\rho\omega_F^{-n})(n)),
\end{align*}
for any integer $n$.
\item If $C/\Rat_p$ is a finite field extension and $\rho\colon \Gal_F\mto \Gl_d(\IntR_C)$ is an Artin representation which factors through a totally real field and which is unramified over $U$, then
\begin{align*}
 \aug(\ncL_{F_\cyc/F}(W,\RDer k_*\sheaf{M}(\rho\omega_F^n)(1-n)))&=L_{\Sigma,\Tau}(\rho\omega_F^n,1-n),\\
 \aug(\ncL^{\Mdual}_{F_\cyc/F}(W,k_!\sheaf{M}(\rho\omega^{-n}_F)(n)))&=L_{\Sigma,\Tau}(\check{\rho}\omega^n_F,1-n)^{-1}
\end{align*}
with $\Sigma\coloneqq X-W$, $\Tau\coloneqq W-U$ and any integer $n>1$.
\end{enumerate}
\end{cor}
\begin{proof}
Properties (1)--(4) are clear by definition. For Property (5) we notice that for $y\in W$
\[
\ncL_{F_\infty/F}(y,\RDer k_*i_*i^*\cmplx{\sheaf{F}}(1))=\begin{cases}
                                                  \ncL_{F_\infty/F}(x,\cmplx{\sheaf{F}}(1)) & \text{if $y=x$,}\\
                                                  1 & \text{else.}
                                                 \end{cases}
\]
Hence,
\[
\ncL_{F_\infty/F}(W,\RDer k_*i_*i^*\cmplx{\sheaf{F}}(1))=\ncL_{F_\infty/F}(U,i_*i^*\sheaf{F}(1))=\ncL_{F_\infty/F}(x,\cmplx{\sheaf{F}}(1))
\]
by (4) and by Theorem~\ref{thm:general modification factors}.(2). The proof for the dual $L$-function is analogous.

Properties (6) and (7) follow from Theorem~\ref{thm:general modification factors} or Theorem~\ref{thm:general dual modification factors} combined with Proposition~\ref{prop:transformation for algebraic L-function} and either Proposition~\ref{prop:transformation of Euler factors} or Proposition~\ref{prop:transformation of dual Euler factors}. For Properties (8) and (9) one applies Theorem~\ref{thm:change of base field}. Property (10) follows from the last part of Theorem~\ref{thm:general dual modification factors} combined with Proposition~\ref{prop:s and Mdual commute} and Proposition~\ref{prop:Euler factors under duality}. Property (11) is just a special case of (7) in a different notation.

It remains to prove (12). The first identity is simply the combination of \eqref{eqn:values of padic Lfunctions} and \eqref{eqn:evaluation of Euler factors}. The second identity follows from \eqref{eqn:evaluation and MDual}, Property (10) and the first identity.
\end{proof}

\section{CM-admissible extensions}\label{sec:CM-extensions}

\begin{defn}
Let $F$ be a totally real number field and $F_\infty/F$ an admissible extension. We call $F_\infty/F$ \emph{CM-admissible} if $F_\infty$ is a CM-field, i.\,e.\ it is totally imaginary and there exists a (unique) involution $\cc\in\Gal(F_\infty/F)$ such that the fixed field $F_\infty^{+}$ of $\cc$ is totally real.
\end{defn}

Note that for a CM-admissible extension $F_\infty/F$ with Galois group
\[
G\coloneqq\Gal(F_\infty/F),
\]
the automorphism $\cc$ is uniquely determined and commutes with every other field automorphism of $F_\infty$. As usual, we write
\[
e_{-}\coloneqq\frac{1-\cc}{2},\quad e_{+}\coloneqq\frac{1+\cc}{2}\in\Lambda[[G]].
\]
for the corresponding central idempotents.

The extension $F^{+}_\infty/F$ is Galois and hence, a really admissible extension. We set $G^+\coloneqq\Gal(F_\infty^{+}/F)$. Moreover, we fix as before an immersion $k\colon U\mto W$ of open dense subschemes of $X=\Spec F$ such that $F_\infty/F$ is unramified over $U$ and $p\neq 2$ is invertible on $W$. Let
\[
 f^+\colon U_{F_\infty^+}\mto U
\]
denote the restriction of the family of coverings $f\colon U_{F_\infty}\mto U$ to $U_{F_\infty^+}$.

If $F_\infty$ contains the $p$-th roots of unity and hence, the $p^n$-th roots of unity for all $n\geq 1$, the cyclotomic character
\[
{\cycchar}_F\colon \Gal_F\mto \Int_p^\times,\quad g\zeta=\zeta^{\cycchar_F(g)},\quad g\in\Gal_F, \zeta\in\mu_{p^{\infty}}
\]
factors through $G\coloneqq\Gal(F_\infty/F)$.  We then obtain for every odd $n\in\Int$ a ring isomorphism
\[
 \Lambda[[G]]\mto \Lambda[[G_{+}]]\times \Lambda[[G_{+}]],\qquad G\ni g\mapsto (g^+,\cycchar_F(g)^n g^+),
\]
where $g^+$ denotes the image of $g\in G$ in $G^+$. The projections onto the two components corresponds to the decomposition of $\Lambda[[G]]$ with respect to $e_+$ and $e_-$.

We will construct the corresponding decomposition of $\cat{A}(\Lambda[[G]])$, where
\[
\cat{A}\in\set{\cat{PDG}^\cont, w_H\cat{PDG}^\cont, \cat{PDG}^{\cont,w_H}}.
\]
 Write $\Lambda(\cycchar_F^n)^\sharp$ for the $\Lambda$-$\Lambda[[G]]$-bimodule $\Lambda$ with $g\in G$ acting by $\cycchar_F^{n}(g^{-1})$ from the right and $\Lambda(\cycchar_F^n)^\sharp[[G]]^\delta$ for the $\Lambda[[G]]$-$\Lambda[[G]]$-bimodule $\Lambda[[G]]\tensor_{\Lambda}\Lambda(\cycchar_F^n)^\sharp$ with the diagonal right action of $G$. According to Example~\ref{exmpl:example functors}, we obtain Waldhausen exact functors
\[
\ringtransf_{\Lambda(\cycchar_F^n)^\sharp[[G]]^\delta}\colon\cat{A}(\Lambda[[G]])\mto\cat{A}(\Lambda[[G]]).
\]
 Moreover, considering $\Lambda[[G^+]]$ as a $\Lambda[[G^+]]$-$\Lambda[[G]]$-bimodule or as a $\Lambda[[G]]$-$\Lambda[[G^+]]$-bimodule, we obtain Waldhausen exact functors
\[
\ringtransf_{\Lambda[[G^+]]}\colon \cat{A}(\Lambda[[G]])\mto\cat{A}(\Lambda[[G^+]]),\qquad \ringtransf_{\Lambda[[G^+]]}\colon \cat{A}(\Lambda[[G^+]])\mto\cat{A}(\Lambda[[G]]).
\]
Note there exists isomorphisms of $\Lambda[[G]]$-$\Lambda[[G]]$-bimodules
\begin{equation*}
\begin{aligned}
e_+\Lambda[[G]]&\isomorph \Lambda[[G^+]]\tensor_{\Lambda[[G^+]]}\Lambda[[G^+]]\\
e_-\Lambda[[G]]&\isomorph \Lambda(\cycchar_F^n)^\sharp[[G]]^\delta\tensor_{\Lambda[[G]]}e_+\Lambda[[G]]\tensor_{\Lambda[[G]]}\Lambda(\cycchar_F^{-n})^\sharp[[G]]^\delta
\end{aligned}
\end{equation*}
for every odd $n\in\Int$ such that the composition
\[
 \ringtransf_{\Lambda[[G^+]]}\comp \ringtransf_{\Lambda[[G^+]]}\colon \cat{A}(\Lambda[[G]])\mto\cat{A}(\Lambda[[G]])
\]
is just the projection onto the $e_+$-component, whereas the projection onto the $e_-$-component may be written as
\[
 \ringtransf_{\Lambda(\cycchar_F^{n})^\sharp[[G]]^\delta\tensor_{\Lambda[[G]]}\Lambda[[G^+]]}\comp\ringtransf_{\Lambda[[G^+]]\tensor_{\Lambda[[G]]}\Lambda(\cycchar_F^{-n})^\sharp[[G]]^\delta}\colon \cat{A}(\Lambda[[G]])\mto\cat{A}(\Lambda[[G]]).
\]
We further note that
\[
 \ringtransf_{\Lambda(\cycchar_F^{n})^\sharp[[G]]^\delta}(f_!f^*\cmplx{\sheaf{F}})\isomorph f_!f^*\cmplx{\sheaf{F}}(n).
\]

If $\Lambda'$ is another adic $\Int_p$-algebra and $\cmplx{P}$ is a complex of $\Lambda'$-$\Lambda[[G]]$-bimodules, strictly perfect as complex of $\Lambda'$-modules, we set
\begin{equation}\label{eqn:decomposition in odd and even part}
 \cmplx{P}_+\coloneqq\cmplx{P}e_+,\qquad \cmplx{P}_-\coloneqq\cmplx{P}e_-
\end{equation}
such that $\cc$ acts trivially on $\cmplx{P}_+$ and by $-1$ on $\cmplx{P}_-$. Both are again complex of $\Lambda'$-$\Lambda[[G]]$-bimodules and strictly perfect as complex of $\Lambda'$-modules. In particular, we have an isomorphism of complexes of $\Lambda'[[G]]$-$\Lambda[[G]]$-bimodules
\[
 \cmplx{P[[G]]^\delta}\isomorph\cmplx{P_+[[G]]^\delta}\oplus\cmplx{P_-[[G]]^\delta}.
\]
Beware that $\cmplx{P_+[[G]]^\delta}$ differs from $\cmplx{P[[G]]^\delta} e_+$. The element $\cc$ acts as $\cc\tensor\id$ on the first complex and trivially on the second. In fact, we have
\begin{align*}
 \cmplx{P_+[[G]]^\delta}e_+&= e_+\cmplx{P_+[[G]]^\delta},&\cmplx{P_+[[G]]^\delta}e_-&= e_-\cmplx{P_+[[G]]^\delta},\\
 \cmplx{P_-[[G]]^\delta}e_+&= e_-\cmplx{P_-[[G]]^\delta},&\cmplx{P_-[[G]]^\delta}e_-&= e_+\cmplx{P_-[[G]]^\delta}.
\end{align*}
Moreover, the Waldhausen exact functors
\begin{align*}
\cat{PDG}^{\cont}(U,\Lambda)\mto \cat{PDG}^{\cont}(U,\Lambda'),&\qquad \cmplx{\sheaf{F}}\mapsto \ringtransf_{\cmplx{\widetilde{P_+}}}(\cmplx{\sheaf{F}}),\\
\cat{PDG}^{\cont}(U,\Lambda)\mto \cat{PDG}^{\cont}(U,\Lambda'),&\qquad \cmplx{\sheaf{F}}\mapsto \ringtransf_{\cmplx{\widetilde{P_-}}}(\cmplx{\sheaf{F}})(1)
\end{align*}
map complexes in $\cat{PDG}^{\cont,\infty}(U,\Lambda)$ to complexes in $\cat{PDG}^{\cont,\infty}(U,\Lambda')$.

Throughout the rest of this section, we assume the validity of Conjecture~\ref{conj:vanishing of mu}.

\begin{cor}\label{cor: S torsion global CM case}
Assume that $F_\infty/F$ is any CM-admissible extension unramified over $U$. For any $\cmplx{\sheaf{F}}$ in $\cat{PDG}^{\cont,\infty}(U,\Lambda)$, the complexes
\begin{align*}
 e_{+}\RDer\Sectc(W,\RDer k_*f_!f^*\cmplx{\sheaf{F}}(1)),&& e_{-}\RDer\Sectc(W,\RDer k_*f_!f^*\cmplx{\sheaf{F}}),\\
 e_{+}\RDer\Sect(W,k_!f_!f^*\cmplx{\sheaf{F}}),&& e_{-}\RDer\Sect(W,k_!f_!f^*\cmplx{\sheaf{F}}(1))
\end{align*}
are in $\cat{PDG}^{\cont,w_H}(\Lambda[[G]])$.
\end{cor}
\begin{proof}
Without loss of generality, we may enlarge $F_\infty$ by adjoining the $p$-th roots of unity. The claim of the corollary is then an immediate consequence of Theorem~\ref{thm: S torsion global} applied to
\begin{align*}
\ringtransf_{\Lambda[[G^+]]}(\RDer\Sectc(W,\RDer k_*f_!f^*\cmplx{\sheaf{F}}(1)))&\sim \RDer\Sectc(W,\RDer k_*(f^+)_!{(f^+)}^*\cmplx{\sheaf{F}}(1)),\\
\ringtransf_{\Lambda[[G^+]]}(\RDer\Sect(W,k_!f_!f^*\cmplx{\sheaf{F}}))&\sim \RDer\Sect(W,k_!(f^+)_!(f^+)^*\cmplx{\sheaf{F}}),\\
\ringtransf_{\Lambda[[G^+]]\tensor_{\Lambda[[G]]}\Lambda(\cycchar_F)^\sharp[[G]]^\delta}(\RDer\Sectc(W,\RDer k_*f_!f^*\cmplx{\sheaf{F}}))&\sim\RDer\Sectc(W,\RDer k_*(f^+)_!(f^+)^*\cmplx{\sheaf{F}}(1)),\\
\ringtransf_{\Lambda[[G^+]]\tensor_{\Lambda[[G]]}\Lambda(\cycchar_F^{-1})^\sharp[[G]]^\delta}(\RDer\Sect(W,k_!f_!f^*\cmplx{\sheaf{F}}(1)))&\sim\RDer\Sect(W,k_!(f^+)_!(f^+)^*\cmplx{\sheaf{F}})
\end{align*}
\end{proof}

Assume that $F_\infty/F$ is CM-admissible and that $F_\infty$ contains the $p$-th roots of unity. For any $\cmplx{\sheaf{F}}$ in $\cat{PDG}^{\cont,\infty}(U,\Lambda)$, we set
\begin{align*}
\ncL^+_{F_\infty/F}(W,\RDer k_*\cmplx{\sheaf{F}}(1))&\coloneqq \ringtransf_{\Lambda[[G^+]]}(\ncL_{F^+_\infty/F}(W,\RDer k_*\cmplx{\sheaf{F}}(1))),\\
\ncL^{\Mdual,+}_{F_\infty/F}(W,k_!\cmplx{\sheaf{F}})&\coloneqq \ringtransf_{\Lambda[[G^+]]}(\ncL^{\Mdual}_{F^+_\infty/F}(W,k_!\cmplx{\sheaf{F}})),\\
\ncL^{-}_{F_\infty/F}(W,\RDer k_*\cmplx{\sheaf{F}})&\coloneqq \ringtransf_{\Lambda(\cycchar_F^{-1})^\sharp[[G]]^\delta\tensor_{\Lambda[[G]]}\Lambda[[G^+]]}(\ncL_{F^+_\infty/F}(W,\RDer k_*\cmplx{\sheaf{F}}(1))),\\
\ncL^{\Mdual,-}_{F_\infty/F}(W,k_!\cmplx{\sheaf{F}}(1))&\coloneqq  \ringtransf_{\Lambda(\cycchar_F)^\sharp[[G]]^\delta\tensor_{\Lambda[[G]]}\Lambda[[G^+]]}(\ncL^{\Mdual}_{F^+_\infty/F}(W,k_!\cmplx{\sheaf{F}}))
\end{align*}
in $\KTh_1(\Lambda[[G]]_S)$. We extend this definition to CM-admissible subextensions $F'_\infty/F$ with $F'_\infty$ not containing the $p$-th roots of unity by taking the image of the elements under
\[
 \ringtransf_{\Lambda[[\Gal(F'_\infty/F)]]}\colon \KTh_1(\Lambda[[G]]_S)\mto\KTh_1(\Lambda[[\Gal(F'_\infty/F)]]_S).
\]
Furthermore, for $\varepsilon\in\set{+,-}$, $x\in W$ and $\cmplx{\sheaf{F}}$ in $\cat{PDG}^{\cont}(U,\Lambda)$ we set
\[
\begin{aligned}
 \ncL^{\varepsilon}_{F_\infty/F}(x,\RDer k_*\cmplx{\sheaf{F}})&\coloneqq\ringtransf_{e_{\varepsilon}\Lambda[[G]]}(\ncL_{F_\infty/F}(x,\RDer k_*\cmplx{\sheaf{F}})),\\
 \ncL^{\Mdual,\varepsilon}_{F_\infty/F}(x,k_!\cmplx{\sheaf{F}})&\coloneqq\ringtransf_{e_{\varepsilon}\Lambda[[G]]}(\ncL^{\Mdual}_{F_\infty/F}(x,k_!\cmplx{\sheaf{F}})).
\end{aligned}
\]
We will write $-\varepsilon\in\set{+,-}$ for the opposite sign.

Assume that $C/\Rat_p$ is a finite field extension and $\rho\colon \Gal_F\mto \Gl_d(\IntR_C)$ is an Artin representation unramified over $U$. If $\rho(\sigma)=-\id$ for every complex conjugation $\sigma\in\Gal_F$, then $\rho$ factors through a CM-extension of $F$. In particular, $\sheaf{M}(\rho\omega_F^{-1})$ is smooth on $U$ and at $\infty$ and we may define elements
\begin{equation}\label{eqn:twisted classical L-functions}
\ncL_{F_\cyc/F}(W,\RDer k_*\sheaf{M}(\rho\omega^n_F)(-n)),\quad \ncL^{\Mdual}_{F_\cyc/F}(W,k_!\sheaf{M}(\rho\omega^n_F)(1-n))
\end{equation}
by identifying $\sheaf{M}(\rho\omega^n_F)(-n)$ with $\sheaf{M}(\rho\omega_F^{-1}\omega^{n+1}_F)(1-(n+1))$. In particular,
\begin{align*}
 \aug(\ncL_{F_\cyc/F}(W,\RDer k_*\sheaf{M}(\rho\omega_F^n)(-n)))&=L_{\Sigma,\Tau}(\rho\omega_F^n,-n),\\
 \aug(\ncL^{\Mdual}_{F_\cyc/F}(W,k_!\sheaf{M}(\rho\omega^{-n}_F)(1-n)))&=L_{\Sigma,\Tau}(\check{\rho}\omega^n_F,-n)^{-1}
\end{align*}
with $\Sigma\coloneqq X-W$, $\Tau\coloneqq W-U$ and any integer $n>0$. If $\rho$ is any Artin representation that factors through a CM-extension, then we can decompose it as in \eqref{eqn:decomposition in odd and even part} into two subrepresentations $\rho_+$ and $\rho_-$ such that
\[
 \rho_+(\sigma)=\id,\qquad \rho_-(\sigma)=-\id
\]
for all complex conjugations $\sigma\in\Gal_F$.

\begin{cor}\label{cor:transformation properties of L-functions CM case}
Let $F$ be a totally real field, $k\colon U\mto W$ be an open immersion of open dense subschemes of $X=\Spec \IntR_F$ such that $p$ is invertible on $W$, and $\Lambda$ be an adic $\Int_p$-algebra. Fix a  $\cmplx{\sheaf{F}}$ in $\cat{PDG}^{\cont,\infty}(U,\Lambda)$, and a CM-admissible extension $F_\infty/F$ unramified over $U$. If $\varepsilon=+$, we choose $n$ to be an even integer. We choose $n$ to be odd if $\varepsilon=-$.
\begin{enumerate}
\item Write $f\colon U_{F_\infty}\mto U$ for the system of coverings associated to $F_\infty/F$. Then
\begin{align*}
\bh\ncL^{\varepsilon}_{F_\infty/F}(W,\RDer k_*\cmplx{\sheaf{F}}(1+n))&=-[e_\varepsilon\RDer\Sectc(W,\RDer k_* f_!f^*\cmplx{\sheaf{F}}(1+n))],\\
\bh\ncL^{\Mdual,\varepsilon}_{F_\infty/F}(W,k_!\cmplx{\sheaf{F}}(n))&=[e_\varepsilon\RDer\Sect(W,k_! f_!f^*\cmplx{\sheaf{F}}(n))]
\end{align*}
\item If $\cmplx{\sheaf{G}}$ and $\cmplx{\sheaf{F}}$ are weakly equivalent in $\cat{PDG}^{\cont,\infty}(U,\Lambda)$, then
\begin{align*}
 \ncL^{\varepsilon}_{F_\infty/F}(W,\RDer k_*\cmplx{\sheaf{F}}(1+n))&=\ncL^{\varepsilon}_{F_\infty/F}(W,\RDer k_*\cmplx{\sheaf{G}}(1+n)),\\
 \ncL^{\Mdual,\varepsilon}_{F_\infty/F}(W,k_!\cmplx{\sheaf{F}}(n))&=\ncL^{\Mdual,\varepsilon}_{F_\infty/F}(W,k_!\cmplx{\sheaf{G}}(n)).
\end{align*}
\item If $0\mto\cmplx{\sheaf{F}'}\mto \cmplx{\sheaf{F}}\mto\cmplx{\sheaf{F}''}\mto 0$ is an exact sequence in $\cat{PDG}^{\cont,\infty}(U,\Lambda)$, then
\begin{align*}
 \ncL^{\varepsilon}_{F_\infty/F}(W,\RDer k_*\cmplx{\sheaf{F}}(1+n))&=\ncL^{\varepsilon}_{F_\infty/F}(W,\RDer k_*\cmplx{\sheaf{F}'}(1+n))\ncL^{\varepsilon}_{F_\infty/F}(W,\RDer k_*\cmplx{\sheaf{F}''}(1+n)),\\
 \ncL^{\Mdual,\varepsilon}_{F_\infty/F}(W,k_!\cmplx{\sheaf{F}}(n))&=\ncL^{\Mdual,\varepsilon}_{F_\infty/F}(W,k_!\cmplx{\sheaf{F}'}(n))\ncL^{\Mdual,\varepsilon}_{F_\infty/F}(W,k_!\cmplx{\sheaf{F}''}(n)).
\end{align*}
\item If $W'$ is an open dense subscheme of $X$ on which $p$ is invertible and $k'\colon W\mto W'$ is an open immersion, then
\begin{align*}
\ncL^{\varepsilon}_{F_\infty/F}(W',\RDer (k'k)_*\cmplx{\sheaf{F}}(1+n))=&\ncL^{\varepsilon}_{F_\infty/F}(W,\RDer k_*\cmplx{\sheaf{F}}(1+n))\\
&\prod_{x\in W'-W}\ncL^{\varepsilon}_{F_\infty/F}(x,\RDer (k'k)_*\cmplx{\sheaf{F}}(1+n)),\\
\ncL^{\Mdual,\varepsilon}_{F_\infty/F}(W',(k'k)_!\cmplx{\sheaf{F}}(n))=&\ncL^{\Mdual,\varepsilon}_{F_\infty/F}(W,k_!\cmplx{\sheaf{F}}(n))\\
&\prod_{x\in W'-W}\ncL^{\Mdual,\varepsilon}_{F_\infty/F}(x,(k'k)_!\cmplx{\sheaf{F}}(n)).
\end{align*}
\item If $i\colon x\mto U$ is a closed point, then
\begin{align*}
\ncL^{\varepsilon}_{F_\infty/F}(W,\RDer k_*i_*i^*\cmplx{\sheaf{F}}(1+n))&=\ncL_{F_\infty/F}(x,\cmplx{\sheaf{F}}(1+n)),\\
\ncL^{\Mdual,\varepsilon}_{F_\infty/F}(W,k_!i_*\RDer i^!\cmplx{\sheaf{F}}(n))&=\ncL^{\Mdual,\varepsilon}_{F_\infty/F}(x,\cmplx{\sheaf{F}}(n)).
\end{align*}
\item If $F'_\infty/F$ is a CM-admissible subextension of $F_\infty/F$, then
\begin{align*}
\ringtransf_{\Lambda[[\Gal(F'_\infty/F)]]}(\ncL^\varepsilon_{F_\infty/F}(W,\RDer k_*\cmplx{\sheaf{F}}(1+n)))&=\ncL^\varepsilon_{F'_\infty/F}(W,\RDer k_*\cmplx{\sheaf{F}}(1+n)),\\
\ringtransf_{\Lambda[[\Gal(F'_\infty/F)]]}(\ncL^{\Mdual,\varepsilon}_{F_\infty/F}(W,k_!\cmplx{\sheaf{F}}(n)))&=\ncL^{\Mdual,\varepsilon}_{F'_\infty/F}(W,k_!\cmplx{\sheaf{F}}(n)).
\end{align*}
If $F'_\infty/F$ is a really admissible subextension of $F_\infty/F$, then
\begin{align*}
\ringtransf_{\Lambda[[\Gal(F'_\infty/F)]]}(\ncL^+_{F_\infty/F}(W,\RDer k_*\cmplx{\sheaf{F}}(1+n)))&=\ncL_{F'_\infty/F}(W,\RDer k_*\cmplx{\sheaf{F}}(1+n)),\\
\ringtransf_{\Lambda[[\Gal(F'_\infty/F)]]}(\ncL^{\Mdual,+}_{F_\infty/F}(W,k_!\cmplx{\sheaf{F}}(n)))&=\ncL^{\Mdual}_{F'_\infty/F}(W,k_!\cmplx{\sheaf{F}}(n)),\\
\ringtransf_{\Lambda[[\Gal(F'_\infty/F)]]}(\ncL^-_{F_\infty/F}(W,\RDer k_*\cmplx{\sheaf{F}}(1+n)))&=1\\
\ringtransf_{\Lambda[[\Gal(F'_\infty/F)]]}(\ncL^{\Mdual,-}_{F_\infty/F}(W,k_!\cmplx{\sheaf{F}}(n)))&=1.
\end{align*}
\item If $\Lambda'$ is another adic $\Int_p$-algebra and $\cmplx{P}$ is a complex of $\Lambda'$-$\Lambda[[\Gal(F_\infty/F)]]$-bimodules, strictly perfect as complex of $\Lambda'$-modules, then
\begin{align*}
\ringtransf_{\cmplx{{P_+[[\Gal(F_\infty/F)]]^{\delta}}}}(\ncL^{\varepsilon}_{F_\infty/F}(W,\RDer k_*\cmplx{\sheaf{F}}(1+n)))&=
\ncL^{\varepsilon}_{F_\infty/F}(W,\RDer k_*\ringtransf_{\cmplx{\widetilde{P_+}}}(\cmplx{\sheaf{F}})(1+n)),\\
\ringtransf_{\cmplx{{P_-[[\Gal(F_\infty/F)]]^{\delta}}}}(\ncL^{\varepsilon}_{F_\infty/F}(W,\RDer k_*\cmplx{\sheaf{F}}(1+n)))&=
\ncL^{-\varepsilon}_{F_\infty/F}(W,\RDer k_*\ringtransf_{\cmplx{\widetilde{P_-}}}(\cmplx{\sheaf{F}})(1+n)),\\
\ringtransf_{\cmplx{{P_+[[\Gal(F_\infty/F)]]^{\delta}}}}(\ncL^{\Mdual,\varepsilon}_{F_\infty/F}(W,k_!\cmplx{\sheaf{F}}(n)))&=
\ncL^{\Mdual,\varepsilon}_{F_\infty/F}(W,k_!\ringtransf_{\cmplx{\widetilde{P_+}}}(\cmplx{\sheaf{F}}(n))),\\
\ringtransf_{\cmplx{{P_-[[\Gal(F_\infty/F)]]^{\delta}}}}(\ncL^{\Mdual,\varepsilon}_{F_\infty/F}(W,k_!\cmplx{\sheaf{F}}(n)))&=
\ncL^{\Mdual,-\varepsilon}_{F_\infty/F}(W,k_!\ringtransf_{\cmplx{\widetilde{P_-}}}(\cmplx{\sheaf{F}}(n))).
\end{align*}
\item If $F'/F$ is a finite extension inside $F_\infty$ such that $F'$ is totally real and $f_{F'}\colon U_{F'}\mto U$ is the associated covering, then
\begin{align*}
\ringtransf_{\Lambda[[\Gal(F_\infty/F)]]}(\ncL^{\varepsilon}_{F_\infty/F}(W,\RDer k_*\cmplx{\sheaf{F}}(1+n)))&=
\ncL^{\varepsilon}_{F_\infty/F'}(W_{F'},\RDer k_*f_{F'}^*\cmplx{\sheaf{F}}(1+n)),\\
\ringtransf_{\Lambda[[\Gal(F_\infty/F)]]}(\ncL^{\Mdual,\varepsilon}_{F_\infty/F}(W,k_!\cmplx{\sheaf{F}})(n))&=
\ncL^{\Mdual,\varepsilon}_{F_\infty/F'}(W_{F'},k_!f_{F'}^*\cmplx{\sheaf{F}}(n)).
\end{align*}
\item With the notation of $(8)$, if $\cmplx{\sheaf{G}}$ is in $\cat{PDG}^{\cont,\infty}(U_{F'},\Lambda)$, then
\begin{align*}
\ringtransf_{\Lambda[[\Gal(F_\infty/F)]]}(\ncL^{\varepsilon}_{F_\infty/F'}(W_{F'},\RDer k_*\cmplx{\sheaf{G}}(1+n)))&=
\ncL^{\varepsilon}_{F_\infty/F}(W,\RDer k_*{f_{F'}}_*\cmplx{\sheaf{G}}(1+n)),\\
\ringtransf_{\Lambda[[\Gal(F_\infty/F)]]}(\ncL^{\Mdual,\varepsilon}_{F_\infty/F'}(W_{F'},k_!\cmplx{\sheaf{G}}(n)))&=
\ncL^{\Mdual,\varepsilon}_{F_\infty/F}(W,k_!{f_{F'}}_*\cmplx{\sheaf{G}}(n)).
\end{align*}
\item If $\sheaf{F}$ is a smooth $\Lambda$-adic sheaf on $U$ which is smooth at $\infty$, then
\[
\ncL^{\Mdual,\varepsilon}_{F_\infty/F}(W,k_!\sheaf{F}(n))=(\ncL^{\varepsilon}_{F_\infty/F}(W,\RDer k_*\sheaf{F}^{\mdual_{\Lambda}}(1-n)))^{\Mdual}.
\]
\item If $C/\Rat_p$ is a finite field extension and $\rho\colon \Gal(F_\infty/F)\mto \Gl_d(\IntR_C)$ is an Artin representation, then
\begin{align*}
\eval_{\rho}(\ncL^{\varepsilon}_{F_\infty/F}(W,\RDer k_*(\Int_p)_U(1+n)))&=\ncL_{F_\cyc/F}(W,\RDer k_*\sheaf{M}(\rho_{\varepsilon})(1+n)),\\
\eval_{\rho}(\ncL^{\Mdual,\varepsilon}_{F_\infty/F}(W,\RDer k_*(\Int_p)_U(n)))&=\ncL_{F_\cyc/F}(W,\RDer k_*\sheaf{M}(\rho_{\varepsilon})(n)),
\end{align*}
\end{enumerate}
\end{cor}
\begin{proof}
This is an easy consequence of the preceding remarks and Corollary~\ref{cor:transformation properties for Lfunctions}.
\end{proof}

\subsection{Calculation of the cohomology}\label{sec:cohomology}

We retain the notation from the previous section. Our objective is to investigate the cohomology of the complexes
\[
\RDer\Sectc(W,\RDer k_*f_!f^*\sheaf{F}(1)),\quad \RDer\Sect(W,k_!f_!f^*\sheaf{F}),\quad \RDer\Sect(\Sigma,i^*\RDer k_*f_!f^*\sheaf{F})
\]
for a $\Lambda$-adic sheaf $\sheaf{F}$ on $U$ in order to tie the connection to classical objects in Iwasawa theory.

The following two propositions are direct consequences of Proposition~\ref{prop:S-torsion exchange property},  Theorem~\ref{thm: S torsion local}, and Corollary~\ref{cor: S torsion global CM case}.

\begin{prop}
Let $F_\infty/F$ be any admissible extension unramified over $U$. Assume that $i\colon x\mto W$ is a closed point not lying over $p$. Then
\[
\HF^s(x,i^*\RDer k_*f_!f^*\sheaf{F})\isomorph \varprojlim_{F'}\HF^{s-1}(x_{F'_\cyc},i^*\RDer k_*\sheaf{F})
\]
where $F'$ runs through the finite subextensions of $F_\infty/F$ and $x_{F'_\cyc}=x\times_W W_{F'_\cyc}$ contains the places of $F'_\cyc$ lying over $x$. In particular,
\[
\HF^s(x,i^*\RDer k_*f_!f^*\sheaf{F})=0
\]
for $s\neq 1$ if $x\in U$ and for $s\neq 1,2$ if $x\in W-U$.
\end{prop}

\begin{prop}\label{prop:calculation of plus and minus parts}
Let $F$ be totally real and $F_\infty/F$ be a CM-admissible extension unramified over $U$. Assume that $p$ is invertible on $W$ and that $\sheaf{F}$ is smooth at $\infty$. If Conjecture~\ref{conj:vanishing of mu} is valid, then
\begin{align*}
e_+\HF^s_c(W,\RDer k_*f_!f^*\sheaf{F}(1))&\isomorph\varprojlim_{F'}e_+\HF^{s-1}_c(W_{F'_\cyc},\RDer k_*\sheaf{F}(1))\\
e_-\HF^s_c(W,\RDer k_*f_!f^*\sheaf{F})&\isomorph\varprojlim_{F'}e_-\HF^{s-1}_c(W_{F'_\cyc},\RDer k_*\sheaf{F})\\
e_+\HF^s(W,k_!f_!f^*\sheaf{F})&\isomorph\varprojlim_{F'}e_+\HF^{s-1}(W_{F'_\cyc},k_!\sheaf{F})\\
e_-\HF^s(W,k_!f_!f^*\sheaf{F}(1))&\isomorph\varprojlim_{F'}e_-\HF^{s-1}(W_{F'_\cyc},k_!\sheaf{F}(1))
\end{align*}
where $F'$ runs through the finite subextensions of $F_\infty/F$. In particular, if $\sheaf{F}$ is smooth over $U$,
\begin{enumerate}
\item $e_+\HF^s_c(W,\RDer k_*f_!f^*\sheaf{F}(1))=e_-\HF^s_c(W,\RDer k_*f_!f^*\sheaf{F})=0$ for $s\neq 2$ if $U\neq W$ and for $s\neq 2,3$ if $W=U$.
\item $e_+\HF^s(W,k_!f_!f^*\sheaf{F})=e_-\HF^s(W,k_!f_!f^*\sheaf{F}(1))=0$ for $s\neq 2$ if $U\neq W$ or if $U=W$ and $F_\infty/F_\cyc$ is infinite and for $s\neq 1,2$ if $U=W$ and $F_\infty/F$ is finite.
\end{enumerate}
\end{prop}

In particular, we obtain the following corollary. We will explain in the next section in what sense this is a generalisation of \cite[Thm. 4.6]{GreitherPopescu:EquivariantMC}.

\begin{cor}\label{cor:projective dimension property}
Let $F$ be totally real and $F_\infty/F$ be a CM-admissible extension unramified over $U$. Assume that $p$ is invertible on $W\neq U$ and that $\sheaf{F}$ is smooth over $U$ and at $\infty$. If Conjecture~\ref{conj:vanishing of mu} is valid, then
\begin{align*}
e_+\HF^2_c(W,\RDer k_*f_!f^*\sheaf{F}(1)),&& e_-\HF^2_c(W,\RDer k_*f_!f^*\sheaf{F}),\\
e_+\HF^2(W,k_!f_!f^*\sheaf{F}), && e_-\HF^2(W,k_!f_!f^*\sheaf{F}(1))
\end{align*}
are finitely generated and projective as $\Lambda[[H]]$-modules. In particular, they have strictly perfect resolutions of length equal to $1$ as $\Lambda[[G]]$-modules. Hence, we may consider their classes in $\KTh_0(\Lambda[[G]],S)$ and obtain
\begin{align*}
[e_+\HF^2_c(W,\RDer k_*f_!f^*\sheaf{F}(1))]&=[e_+\RDer\Sectc(W,\RDer k_*f_!f^*\sheaf{F}(1))],\\
[e_-\HF^2_c(W,\RDer k_*f_!f^*\sheaf{F})]&=[e_-\RDer\Sectc(W,\RDer k_*f_!f^*\sheaf{F})],\\
[e_+\HF^2(W,k_!f_!f^*\sheaf{F})]&=[e_+\RDer\Sect(W,k_!f_!f^*\sheaf{F})],\\
[e_-\HF^2(W,k_!f_!f^*\sheaf{F}(1))]&=[e_-\RDer\Sect(W,k_!f_!f^*\sheaf{F}(1))].
\end{align*}
\end{cor}
\begin{proof}
We give the argument for $X\coloneqq e_+\HF^2(W,k_!f_!f^*\sheaf{F})$; the proof of the other cases is essentially the same. The $\Lambda[[G]]$-module $X$ is the only non-vanishing cohomology group of the perfect complex of $\Lambda[[G]]$-modules
\[
\cmplx{P}\coloneqq\varprojlim_{I\in\openideals_{\Lambda[[G]]}}e_+\RDer\Sect(W,k_!(f_!f^*\sheaf{F})_I).
\]
Since for any simple $\Lambda[[G]]$-module $M$,
\[
M\Ltensor_{\Lambda[[G]]}\cmplx{P}\sim
e_+\RDer\Sect(W,k_!M\tensor_{\Lambda[[G]]/\Jac(\Lambda[[G]])}(f_!f^*\sheaf{F})_{\Jac(\Lambda[[G]])})
\]
has no cohomology except in degrees $1$ and $2$, we conclude that there exists a strictly perfect complex $\cmplx{Q}$ of $\Lambda[[G]]$-modules concentrated in degrees $-1$ and $0$ and quasi-isomorphic to $X$. By Corollary~\ref{cor: S torsion global CM case}, we know that $\cmplx{Q}$ is also perfect as complex of $\Lambda[[H]]$-modules. By Lemma~\ref{lem:S-torsion complexes of length two}, we conclude that $X$ is finitely generated and projective as $\Lambda[[H]]$-module.

%
%
We then have
\[
[X]=[\cmplx{Q}]=[e_+\RDer\Sect(W,k_!f_!f^*\sheaf{F})]
\]
in $\KTh_0(\Lambda[[G]],S)$.
\end{proof}

\section{Realisations of abstract 1-motives}\label{sec:GreitherPopescu}

As a central result of this section, we want to establish the link with the theory of abstract $1$-motives considered in \cite{GreitherPopescu:EquivariantMC}. 

Assume that $F$ is any number field and let $U\subset W$ be two open dense subschemes of $X=\Spec \IntR_F$. Write $k\colon U\mto W$ for the corresponding open immersion. Fix a closed subscheme structure on the complement $\Tau$ of $U$ in $W$ and write $i\colon \Tau\mto W$ for the closed immersion.  For any scheme $S$ we write $\Gm_S$ for the \'etale sheaf of units on $S$.  Recalling that the stalk of $\Gm_W$ in a geometric point over a closed point is given by the units of the strict henselisation of the local ring in this closed point, we see that
$$
\Gm_{W}\mto i_*\Gm_\Tau. 
$$
is a surjection. We define $\Gm_{W,\Tau}$ to be its kernel. Write $\eta\colon \Spec F\mto W$ for the generic point and set
\[
\sheaf{P}_\Tau\coloneqq\ker \eta_*\Gm_F\mto i_*i^*(\eta_*\Gm_F/\Gm_{W,\Tau}).
\]
In other words, $\sheaf{P}_\Tau$ is the subsheaf of $\eta_*\Gm_F$ of elements congruent to $1$ modulo $\Tau$. For any subscheme $Z$ of $W$ we let $\Div_Z$ denote the sheaf of divisors on $W$ with support on $Z$. Hence, we obtain an exact $9$-diagram
\[
\xymatrix{
&0\ar[d]&0\ar[d]&0\ar[d]&\\
0\ar[r]&\Gm_{W,\Tau}\ar[r]\ar[d]&\Gm_{W}\ar[r]\ar[d]&i_{*}\Gm_\Tau\ar[r]\ar[d]&0\\
0\ar[r]&\sheaf{P}_\Tau\ar[r]\ar[d]^{\dv}&\eta_*\Gm_{F}\ar[r]\ar[d]&i_*i^*(\eta_*\Gm_K/\Gm_{W,\Tau})\ar[r]\ar[d]&0\\
0\ar[r]&\Div_{U}\ar[r]\ar[d]&\Div_{W}\ar[r]\ar[d]&\Div_{\Tau}\ar[r]\ar[d]&0\\
&0&0&0&}
\]
The third row is clearly also exact in the category of presheaves on $W$. Moreover,
\[
i_*i^*(\eta_*\Gm_F/\Gm_{W,\Tau})(W)=\bigoplus_{v\in \Tau} F_v^\times/\mathcal{U}_{v,n_v}
\]
where $n_v$ is the multiplicity of $v$ in $\Tau$, $F_v$ is the completion of $F$ at $v$ and $\mathcal{U}_{n,n_v}\subset F_v$ is the group of units $f$ such that the valuation of $1-f$ is larger or equal to $n_v$. From the weak approximation theorem we conclude that the second row of the $9$-diagram is also exact in the category of presheaves. The same is also true for the third column. Hilbert 90 and the Leray spectral sequence then imply that
\[
\HF^1(W,\sheaf{P}_\Tau)=\HF^1(W,\eta_*\Gm_{F})=0.
\]
We conclude that
\[
\HF^1(W,\Gm_{W,\Tau})=\coker \sheaf{P}_\Tau(W)\xrightarrow{\dv}\Div_U(W)
\]
is the ray class group of $W$ with respect to the modulus $\Tau$. If $K/F$ is a possibly infinite algebraic extension of $F$, it follows from \cite[VII, Cor. 5.8]{SGA4-2} that
\[
\HF^1(W_K,\Gm_{W_K,\Tau_K})=\coker\left(\sheaf{P}_{\Tau_K}(W_K)\mto \varinjlim_{K'\subset K}\Div_{U_{K'}}(W_{K'})\right)
\]
with
\[
\varinjlim_{K'\subset K}\Div_{U_{K'}}(W_{K'})=\bigoplus_{v}\Gamma_v,
\]
where $v$ ranges over the places of $K$ lying over the closed points of $U$ and $\Gamma_v$ denotes the value group of the associated, possibly non-discrete valuation.

Assume now that $p$ is invertible on $W$. We then obtain an exact $9$-diagram
\[
\xymatrix{
&0\ar[d]&0\ar[d]&0\ar[d]&\\
0\ar[r]&j_!\mu_{p^n}\ar[r]\ar[d]&\mu_{p^n}\ar[r]\ar[d]&i_{*}i^*\mu_{p^n}\ar[r]\ar[d]&0\\
0\ar[r]&\Gm_{W,\Tau}\ar[r]\ar[d]^{p^n}&\Gm_{W}\ar[r]\ar[d]^{p^n}&i_{*}\Gm_\Tau\ar[r]\ar[d]^{p^n}&0\\
0\ar[r]&\Gm_{W,\Tau}\ar[r]\ar[d]&\Gm_{W}\ar[r]\ar[d]&i_{*}\Gm_\Tau\ar[r]\ar[d]&0\\
&0&0&0&}
\]
and hence, an exact sequence
\[
0\mto j_!\mu_{p^n}\mto\sheaf{P}_\Tau\xrightarrow{(p^n,\dv)}\sheaf{P}_\Tau\oplus \Div_U
\xrightarrow{\left(\begin{smallmatrix}\dv \\ -p^n\end{smallmatrix}\right)}\Div_U\mto 0.
\]
We take global sections on $W$. Since $\HF^1(W,\sheaf{P}_\Tau)=0$ and since multiplication by $p^n$ is injective on $\Div_U(W)$ we obtain
\[
\HF^1(W,j_!\mu_{p^n})=\set*{f\in \sheaf{P}_\Tau(W)\given
\begin{aligned}
&\dv(f)=p^nD,\\
&D\in\Div_U(W)
\end{aligned}
}/\set{g^{p^n}\given g\in \sheaf{P}_\Tau(W)}.
\]
Note that this group does not depend on the subscheme structure of $\Tau$. So, we might as well consider it with the reduced scheme structure.

We now assume in addition that $F$ is totally real and fix a CM-admissible extension $F_\infty/F$ such that $F_\infty/F_\cyc$ is finite.  Passing to the direct limit over all finite subextensions $F'/F$ of $F_\infty/F$, we obtain
\begin{equation}\label{eqn:calculation of exceptional image of unit roots}
\HF^1(W_{F_\infty},j_!\mu_{p^n})=\set*{f\in \sheaf{P}_{\Tau_{F_\infty}}(W_{F_\infty})\given
\begin{aligned}
&\dv(f)=p^nD,\\
&D\in\Div_{U_{F_\infty}}(W_{F_\infty})
\end{aligned}
}/\set{g^{p^n}\given g\in \sheaf{P}_{\Tau_{F_\infty}}(W_{F_\infty})}.
\end{equation}

Write $\Sigma$ for the complement of $W$ in $\Spec \IntR_F[\frac{1}{p}]$.  The Iwasawa-theoretic $1$-motive associated to $(F_\infty,\Sigma_{F_\infty},\Tau_{F_\infty})$ is the complex of abelian groups
\[
\Motive^{F_\infty}_{\Sigma_{F_\infty},\Tau_{F_\infty}}\colon\qquad \Div_{\Sigma_{F_\infty}}(X_{F_\infty})\xrightarrow{\delta} \HF^1(X_{F_\infty},\Gm_{X_{F_{\infty}},\Tau_{F_\infty}})\tensor_{\Int}\Int_p
\]
sitting in degrees $0$ and $1$ \cite[\S 3.1]{GreitherPopescu:EquivariantMC}. Its group of $p^n$-torsion points is defined to be
\begin{align*}
\Motive^{F_\infty}_{\Sigma_{F_\infty},\Tau_{F_\infty}}[p^n]&\coloneqq
\set*{(D,c)\given
\begin{aligned}
&D\in\Div_{\Sigma_{F_\infty}}(X_{F_\infty}),\\
&c\in \HF^1(X_{F_\infty},\Gm_{X_{F_{\infty}},\Tau_{F_\infty}})\tensor_{\Int}\Int_p,\\
&\delta(D)=p^nc
\end{aligned}}\tensor_{\Int}\Int/(p^n)\\
&=\HF^0(\Motive^{F_\infty}_{\Sigma_{F_\infty},\Tau_{F_\infty}}\Ltensor_{\Int}\Int/(p^n))
\end{align*}
and its $p$-adic Tate module is given by
\[
\Tate_p\Motive^{F_\infty}_{\Sigma_{F_\infty},\Tau_{F_\infty}}\coloneqq \varprojlim_{n}\Motive^{F_\infty}_{\Sigma_{F_\infty},\Tau_{F_\infty}}[p^n]
\]
\cite[Def. 2.2, Def. 2.3]{GreitherPopescu:EquivariantMC}.

\begin{rem}
The complex of abelian groups $\Motive^{F_\infty}_{\Sigma_{F_\infty},\Tau_{F_\infty}}$ is an abstract $1$-motive in the sense of \cite[Def. 2.1]{GreitherPopescu:EquivariantMC} only if $\HF^1(X_{F_\infty},\Gm_{X_{F_{\infty}},\Tau_{F_\infty}})\tensor_{\Int}\Int_p$ is divisible of finite corank. The proof of \cite[Lemma 2.8]{GreitherPopescu:EquivariantMC} shows that this is true if and only if $\HF^1(X_{F_\infty},\Gm_{X_{F_{\infty}}})\tensor_{\Int}\Int_p$ is divisible of finite corank. By \cite[Thm. 11.1.8 ]{NSW:CohomNumFields} this is equivalent to the Galois group $\X_{\nr}(F_\infty)$ of the maximal abelian unramified pro-$p$-extension of $F_\infty$ being a finitely generated $\Int_p$-module. This is true if $\X_{\nr}(F_\infty(\mu_p))$ is a finitely generated $\Int_p$-module. By \cite[Thm 13.24]{Wa:ICF} the latter statement is equivalent to $e_{-}\X_{\nr}(F_\infty(\mu_p))$ being finitely generated over $\Int_p$, which is in turn equivalent to the Galois group of the maximal abelian pro-$p$-extension  unramified outside the primes over $p$ of the maximal totally real subfield $F_\infty(\mu_p)^{+}$ being finitely generated over $\Int_p$ \cite[Cor. 11.4.4]{NSW:CohomNumFields}. Hence, $\Motive^{F_\infty}_{\Sigma_{F_\infty},\Tau_{F_\infty}}$ is an abstract $1$-motive under Conjecture~\ref{conj:vanishing of mu}.
\end{rem}

\begin{prop}\label{prop:connection of 1-motives with etale cohomology}
There is a short exact sequence
\[
0\mto \HF^0(X_{F_\infty},\Gm_{X_{F_\infty},\Tau_{F_\infty}})\tensor_{\Int}\Int/(p^n)\mto \HF^1(W_{F_\infty},k_!\mu_{p^n})\mto\Motive^{F_\infty}_{\Sigma_{F_\infty},\Tau_{F_\infty}}[p^n]\mto 0.
\]
In particular, there are isomorphisms
\begin{align*}
e_{-}\HF^1(W_{F_\infty},k_!\mu_{p^n})&\isomorph e_{-}\Motive^{F_\infty}_{\Sigma_{F_\infty},\Tau_{F_\infty}}[p^n],\\
e_{-}\HF^1(W_{F_\infty},k_!(\Int_p)_{W_{F_\infty}}(1))&\isomorph e_{-}\Tate_p\Motive^{F_\infty}_{\Sigma_{F_\infty},\Tau_{F_\infty}}.
\end{align*}
\end{prop}
\begin{proof}
This follows from \eqref{eqn:calculation of exceptional image of unit roots} and \cite[Prop. 3.2, Cor. 3.4]{GreitherPopescu:EquivariantMC}. Note that the proofs of these statements do not make use of the divisibility of the group
$\HF^1(X_{F_\infty},\Gm_{X_{F_{\infty}},\Tau_{F_\infty}})\tensor_{\Int}\Int_p$.
\end{proof}

\begin{rem}
We refer to \cite{GreitherPopescu:AbstractMotivesAndTatesClass} for related work of Greither and Popescu.
\end{rem}

Writing again $f\colon U_{F_\infty}\mto U$ for the system of coverings associated to $F_\infty/F$ we conclude from Proposition~\ref{prop:calculation of plus and minus parts}:

\begin{cor}\label{cor:comparison with GreitherPopescu}
Assume that $F_\infty/F$ is unramified over $U$. Under Conjecture~\ref{conj:vanishing of mu}, there are isomorphisms
\begin{align*}
e_{-}\HF^2(W,k_!f_!f^*\mu_{p^n})&\isomorph e_{-}\Motive^{F_\infty}_{\Sigma_{F_\infty},\Tau_{F_\infty}}[p^n],\\
e_{-}\HF^2(W,k_!f_!f^*(\Int_p)_U(1))&\isomorph e_{-}\Tate_p\Motive^{F_\infty}_{\Sigma_{F_\infty},\Tau_{F_\infty}}.
\end{align*}
In particular,
\[
 \bh \ncL^{\Mdual,-}_{F_\infty/F}(W,k_!(\Int_p)_U(1))=\left[e_{-}\Tate_p\Motive^{F_\infty}_{\Sigma_{F_\infty},\Tau_{F_\infty}}\right]
\]
in $\KTh_0(\Int_p[[\Gal(F_\infty/F)]],S)$.
\end{cor}
\begin{proof}
Combine Proposition~\ref{prop:connection of 1-motives with etale cohomology} with Proposition~\ref{prop:calculation of plus and minus parts} and use Corollary~\ref{cor:transformation properties of L-functions CM case}.
\end{proof}

In particular, \cite[Thm. 4.6]{GreitherPopescu:EquivariantMC} reduces to the special case $\sheaf{F}=(\Int_p)_U$ of Corollary~\ref{cor:projective dimension property}. Moreover, if $\Gal(F_\infty/F)$ is commutative, we may identify the Fitting ideal and the characteristic ideal of $e_{-}\Tate_p\Motive^{F_\infty}_{\Sigma_{F_\infty},\Tau_{F_\infty}}$ over $\Int_p[[\Gal(F_\infty/F)]]$. The characteristic ideal may then be viewed as an element of
\[
 (\Int_p[[\Gal(F_\infty/F)]]_S)^{\times}/\Int_p[[\Gal(F_\infty/F)]]^\times\isomorph\KTh_0(\Int_p[[\Gal(F_\infty/F)]],S).
\]
Under this identification, it corresponds to the class $\left[e_{-}\Tate_p\Motive^{F_\infty}_{\Sigma_{F_\infty},\Tau_{F_\infty}}\right]^{-1}$. Furthermore, the interpolation property (11) in Corollary~\ref{cor:transformation properties of L-functions CM case} shows that the element $\ncL^{\Mdual,-}_{F_\infty/F}(W,k_!(\Int_p)_U(1))^{-1}$ agrees with the element $e_{+}+\theta_{\Sigma,\Tau}^{(\infty)}$ with $\Sigma\coloneqq X-W$, $\Tau\coloneqq W-U$ in the notation of \cite[Def.~5.16]{GreitherPopescu:EquivariantMC}. In particular, we recover the version of the equivariant main conjecture formulated in \cite[Thm.~5.6]{GreitherPopescu:EquivariantMC} as a special case of Corollary~\ref{cor:comparison with GreitherPopescu}. 

In the same way, one can also recover its non-commutative generalisation in \cite[Thm.~3.3]{Nickel:EIwTh}, which states that Nickel's non-commutative Fitting invariant of $e_{-}\Tate_\ell\Motive^{F_\infty}_{\Sigma_{F_\infty},\Tau_{F_\infty}}$ is generated by the reduced norm of $\ncL^{\Mdual,-}_{F_\infty/F,\Sigma,\Tau}(\Int_\ell(1))$. By the argument before \cite[Conj.~2.1]{Nickel:EIwTh}, this is in fact equivalent to Cor.~\ref{cor:comparison with GreitherPopescu}. However, Nickel only considers the case that $F_\infty/F$ is unramified over $W$.

With some technical effort, one can further extend Prop.~\ref{prop:connection of 1-motives with etale cohomology} and Cor.~\ref{cor:comparison with GreitherPopescu} by allowing $U$ to contain a finite number of points $x$ such that $F_\infty/F$ is ramified over $x$, but with a ramification index prime to $\ell$. The interested reader may consult \cite[Ch.~5, Ch.~6]{Witte:Habil} for a detailed exposition, which parallels the disquisition in \cite{Witte:NCIMCFuncField} on the same phenomenon in the function field case.

\section*{Appendix: Localisation in polynomial rings}

Let $R$ be any associative ring with $1$ and let $R[t]$ be the polynomial ring over $R$ in one indeterminate $t$ that commutes with the elements of $R$. Write $\cat{SP}(R[t])$ and $\cat{P}(R[t])$ for the Waldhausen categories of strictly perfect and perfect complexes of $R[t]$-modules. Consider $R$ as a $R$-$R[t]$-bimodule via the augmentation map
\[
 R[t]\mto R,\qquad t\mapsto 0.
\]
We then define full subcategories
\begin{gather*}
\cat{SP}^{w_t}(R[t])\coloneqq\set{\cmplx{P}\in \cat{SP}(R[t])\given \text{\(R\tensor_{R[t]}\cmplx{P}\) is acyclic}},\\
\cat{P}^{w_t}(R[t])\coloneqq\set{\cmplx{P}\in \cat{P}(R[t])\given \text{\(\cmplx{P}\) is  quasi-isomorphic to a complex in \(\cat{SP}^{w_t}(R[t])\)}}.
\end{gather*}
These categories are in fact Waldhausen subcategories of $\cat{SP}(R[t])$ and $\cat{P}(R[t])$, respectively, since they are closed under shifts and extensions \cite[3.1.1]{Witte:PhD}. We can then construct new Waldhausen categories $w_t\cat{SP}(R[t])$ and $w_t\cat{P}(R[t])$ with the same objects, morphisms, and cofibrations as $\cat{SP}(R[t])$ and $\cat{P}(R[t])$, but with weak equivalences being those morphisms with cone in $\cat{SP}^{w_t}(R[t])$ and $\cat{P}^{w_t}(R[t])$, respectively. By the approximation theorem \cite[1.9.1]{ThTr:HAKTS+DC}, the inclusion functor $w_t\cat{SP}(R[t])\mto w_t\cat{P}(R[t])$ induces isomorphisms
\[
 \KTh_n(w_t\cat{SP}(R[t]))\isomorph \KTh_n(w_t\cat{P}(R[t]))
\]
for all $n\geq 0$.

It might be reassuring to know that, if $R$ is noetherian, we can identify these $\KTh$-groups for $n\geq 1$ with the $\KTh$-groups of a localisation of $R[t]$: Set
\[
S_t\coloneqq\set{f(t)\in R[t] \given f(0)\in R^{\times}}
\]

\begin{propA}\label{prop:S_t denominator set}
Assume that $R$ is a noetherian. Then $S_t$ is a left (and right) denominator set in the sense of \cite[Ch. 10]{GW:NoncommNoethRings} such that the localisation $ R[t]_{S_t}$ exists and is noetherian. Its Jacobson radical $\Jac(R[t]_{S_t})$ is generated by the Jacobson radical $\Jac(R)$ of $R$ and $t$. In particular, if $R$ is semi-local, then so is $R[t]_{S_t}$.

Moreover, the category $\cat{SP}^{w_t}(R[t])$ consists precisely of those complexes $\cmplx{P}$ in $\cat{SP}(R[t])$ with $S_t$-torsion cohomology. In particular,
\[
\KTh_n(w_t\cat{SP}(R[t]))\isomorph \KTh_{n}(R[t]_{S_t})
\]
for $n\geq 1$.
\end{propA}
\begin{proof}
Clearly, the set $S_t$ consists of non-zero divisors, such that we only need to check the Ore condition:
\[
\forall s\in S_t\colon\forall a\in R[t]\colon\exists x\in R[t]\colon \exists y\in S_t\colon xs=ya.
\]
Moreover, we may assume that $s(0)=y(0)=1$. Write
\[
s=1-\sum_{i=1}^\infty s_it^i,\quad a=\sum_{i=0}^\infty a_i t^i,\quad x=\sum_{i=0}^\infty x_i t^i,\quad y=1+\sum_{i=1}^\infty y_it^i.
\]
and assume that $s_i=a_{i-1}=0$ for $i>n$. Comparing coefficients, we obtain the recurrence equation
\begin{equation*}
(*)\qquad x_i=\sum_{j=0}^{i-1}x_js_{i-j}+\sum_{j=1}^iy_ja_{i-j}+a_i=\sum_{j=1}^iy_j b_{i-j}+b_i
\end{equation*}
with
\[
b_i\coloneqq\sum_{j=0}^{i-1}b_js_{i-j}+a_i.
\]
Write $B_i\coloneqq(b_{i-n+1},\dots,b_{i})\in R^n$ with the convention that $b_i=0$ for $i<0$. Then for $i\geq n$
\[
B_i=B_{i-1}S=B_{n-1}S^{i-n+1}
\]
with
\[
S\coloneqq\begin{pmatrix}
 0      & \dots & \dots  & s_n\\
 1      & \ddots&        & s_{n-1}\\
 0      & \ddots& 0      & \vdots \\
 \vdots &       & 1      &  s_1
 \end{pmatrix}.
\]
Since $R$ was assumed to be noetherian, there exists a $m\geq n$ and $y_n,\dots, y_{m}\in R$ such that
\[
0=\sum_{j=n}^m y_j B_{m-j}+B_m=\sum_{j=n}^m y_j B_{i-j}+B_{i}
\]
for all $i\geq m$. Hence, we can find a solution $(x_i,y_i)_{i=0,1,2\dots}$ of equation~$(*)$ with $x_{i}=y_i=0$ for $i>m$ and $y_i=0$ for $i<n$. This shows that $S_t$ is indeed a left denominator set such that $R[t]_{S_t}$ exists and is noetherian \cite[Thm. 10.3, Cor. 10.16]{GW:NoncommNoethRings}.

Let $N\subset R[t]$ be the semi-prime ideal of $R[t]$ generated by $t$ and the Jacobson ideal $\Jac(R)$ of $R$. Then $S_t$ is precisely the set of elements of $\Lambda[t]$ which are units modulo $N$. In particular, the localisation $N_{S_t}$ is a semi-prime ideal of $R[t]_{S_t}$ such that
\[
R[t]_{S_t}/N_{S_t}=R[t]/N=R/\Jac(R)
\]
\cite[Thm. 10.15, 10.18]{GW:NoncommNoethRings}. We conclude $\Jac(R[t]_{S_t})\subset N_{S_t}$. For the other inclusion it suffices to note that for every $s\in S_t$ and every $n\in N$, the element $s+n$ is a unit modulo $N$.

The Nakayama lemma implies that for any noetherian ring $R$ with Jacobson radical $\Jac(R)$, a strictly perfect complex of $R$-modules $\cmplx{P}$ is acyclic if and only if $R/\Jac(R)\tensor_R\cmplx{P}$ is acyclic. Hence, if $\cmplx{P}$ is a strictly perfect complex of $R[t]$-modules, then $R\tensor_{R[t]}\cmplx{P}$ is acyclic if and only if $R[t]_{S_t}\tensor_{R}\cmplx{P}$ is acyclic. This shows that $\cat{SP}^{w_t}(R[t])$ consists precisely of those complexes $\cmplx{P}$ in $\cat{SP}(R[t])$ with $S_t$-torsion cohomology. From the localisation theorem in \cite{WeibYao:Localization} we conclude that the Waldhausen exact functor
\[
w_t\cat{SP}(R[t])\mto\cat{SP}(R[t]_{S_t}),\qquad \cmplx{P}\mapsto R[t]_{S_t}\tensor_{R[t]}\cmplx{P}
\]
induces isomorphisms
\[
\KTh_n(w_t\cat{SP}(R[t]))\isomorph
\begin{cases}
 \KTh_n(R[t]_{S_t}) & \text{if $n>0$,}\\
 \im\left(\KTh_0(R[t])\mto\KTh_0(R[t]_{S_t})\right) & \text{if $n=0$.}
\end{cases}
\]
\end{proof}

 The set $S_t$ fails to be a left denominator set if $R=\FF_p\langle\langle x,y\rangle\rangle$ is the power series ring in two non-commuting indeterminates: $a(1-xt)=by$ has no solution with $a\in R[t]$, $b\in S_t$. Note also that a commutative adic ring is always noetherian \cite[Cor. 36.35]{Warner:TopRings}. In this case, $S_t$ is the union of the complements of all maximal ideals of $\Lambda[t]$ containing $t$ and the determinant provides an isomorphism
\[
\KTh_1(w_t\cat{SP}(\Lambda[t]))\isomorph\KTh_1(\Lambda[t]_{S_t})\xrightarrow[\isomorph]{\det}\Lambda[t]_{S_t}^\times.
\]

For any adic $\Int_p$-algebra $\Lambda$ and any $\gamma\in\Gamma\isomorph\Int_p$, we have a ring homomorphism
\[
\ev_{\gamma}\colon \Lambda[t]\mapsto \Lambda[[\Gamma]],\quad f(t)\mapsto f(\gamma).
\]
inducing homomorphisms $\KTh_n(\Lambda[t])\mto\KTh_n(\Lambda[[\Gamma]])$.

\begin{propA}\label{prop:extension of evgamma}
Assume that $\gamma\neq 1$. Then the ring homomorphism $\ev_{\gamma}$ induces homomorphisms
\[
\ev_\gamma\colon\KTh_n(w_t\cat{P}(\Lambda[[t]]))\isomorph\KTh_n(w_t\cat{SP}(\Lambda[t]))\mto\KTh_n(\Lambda[[\Gamma]]_S)
\]
for all $n\geq 0$.
\end{propA}
\begin{proof}
It suffices to show that for any complex $\cmplx{P}$ in $\cat{SP}^{w_t}(\Lambda[t])$, the complex
\[
\cmplx{Q}\coloneqq\Lambda[[\Gamma]]\tensor_{\Lambda[t]}\cmplx{P}
\]
is perfect as complex of $\Lambda$-modules. We can check this after factoring out the Jacobson radical of $\Lambda$ \cite[Prop. 4.8]{Witte:MCVarFF}. Hence, we may assume that $\Lambda$ is semi-simple, i.\,e.\
\[
\Lambda=\prod_{i=1}^m M_{n_i}(k_i)
\]
where $M_{n_i}(k_i)$ is the algebra of $n_i\times n_i$-matrices over a finite field $k_i$ of characteristic $p$. By the Morita theorem, the tensor product over $\Lambda$ with the $\prod_i k_i$-$\Lambda$-bimodule
\[
\prod_{i=1}^m k_i^{n_i}
\]
induces equivalences of categories
\begin{align*}
\cat{SP}^{w_t}(\Lambda[t])&\mto\cat{SP}^{w_t}\left(\prod_{i=1}^n k_i[t]\right),\\
\cat{SP}^{w_H}(\Lambda[[\Gamma]])&\mto\cat{SP}^{w_H}(\prod_{i=1}^n k_i[[\Gamma]]),
\end{align*}
with $H\subset\Gamma$ being the trivial subgroup.
Hence, we are reduced to the case
\[
\Lambda=\prod_{i=1}^m k_i.
\]
In this case, the set $S\subset \Lambda[[\Gamma]]$ defined in \eqref{eqn:Venjakobs Ore set} consists of all non-zero divisors of $\Lambda[[\Gamma]]$, i.\,e.\ all elements with non-trivial image in each component $k_i[[\Gamma]]$. Since $\Lambda[[\Gamma]]$ is commutative, this is trivially a left denominator set. Moreover, the complex $\cmplx{Q}$ is perfect as complex of $\Lambda$-modules precisely if its cohomology groups are $S$-torsion. On the other hand, as a trivial case of Proposition~\ref{prop:S_t denominator set}, we know that $S_t$ is a left denominator set and that the cohomology groups of $\cmplx{P}$ are $S_t$-torsion. Since $f(0)$ is a unit in $\Lambda$ for each $f\in S_t$, the element $f(\gamma)$ has clearly non-trivial image in each component $k_i[[\Gamma]]$. Hence, $\ev_\gamma$ maps $S_t$ to $S$  and $\cmplx{Q}$ is indeed perfect as complex of $\Lambda$-modules.
\end{proof}

\comment{
To do: Put this in a separate paper and explain its relevance to higher dimensional knot theory.
}

\bibliographystyle{amsalpha}
\bibliography{Literature}

\newcommand{\etalchar}[1]{$^{#1}$}
\providecommand{\bysame}{\leavevmode\hbox to3em{\hrulefill}\thinspace}
\providecommand{\MR}{\relax\ifhmode\unskip\space\fi MR }
\providecommand{\MRhref}[2]{%
  \href{http://www.ams.org/mathscinet-getitem?mr=#1}{#2}
}
\providecommand{\href}[2]{#2}
\begin{thebibliography}{CFK{\etalchar{+}}05}

\bibitem[AGV72a]{SGA4-2}
M.~Artin, A.~Grothendieck, and J.L. Verdier, \emph{Th\'eorie des topos et
  cohomologie \'etale des sch\'emas ({SGA} 4-2)}, Lecture Notes in Mathematics,
  no. 270, Springer, Berlin, 1972.

\bibitem[AGV72b]{SGA4-3}
\bysame, \emph{Th\'eorie des topos et cohomologie \'etale des sch\'emas ({SGA}
  4-3)}, Lecture Notes in Mathematics, no. 305, Springer, Berlin, 1972.

\bibitem[Bau91]{Bau:CombHom}
H.~J. Baues, \emph{Combinatorial homotopy and $4$-dimensional complexes}, de
  Gruyter Expositions in Mathematics, no.~2, Walter de Gruyter \& Co., Berlin,
  1991.

\bibitem[Bur09]{Burns:AlgebraicLfunctions}
D.~Burns, \emph{Algebraic $p$-adic {$L$}-functions in non-commutative {I}wasawa
  theory}, Publ. RIMS Kyoto \textbf{45} (2009), 75--88.

\bibitem[Bur15]{Burns:MCinNCIwasawaTh+RelConj}
\bysame, \emph{On main conjectures in non-commutative {I}wasawa theory and
  related conjectures}, J. Reine Angew. Math. \textbf{698} (2015), 105--159.

\bibitem[CFK{\etalchar{+}}05]{CFKSV}
J.~Coates, T.~Fukaya, K.~Kato, R.~Sujatha, and O.~Venjakob, \emph{The {$GL_2$}
  main conjecture for elliptic curves without complex multiplication}, Publ.
  Math. Inst. Hautes Etudes Sci. (2005), no.~101, 163--208.

\bibitem[CL73]{CoatesLichtenbaum:lAdicZetaFunctions}
J.~Coates and S.~Lichtenbaum, \emph{On {$l$}-adic zeta functions}, Ann. of
  Math. (2) \textbf{98} (1973), 498--550.

\bibitem[CSSV13]{CSSV:Survey}
J.~Coates, P.~Schneider, R.~Sujatha, and O.~Venjakob (eds.),
  \emph{Noncommutative {I}wasawa main conjectures over totally real fields},
  Springer Proceedings in Mathematics \& Statistics, vol.~29, Springer,
  Heidelberg, 2013, Papers from the Workshop held at the University of
  M{\"u}nster, M{\"u}nster, April 25--30, 2011.

\bibitem[Del87]{Del:DC}
P.~Deligne, \emph{Le d{\'e}terminant de la cohomologie}, Contemporary
  Mathematics \textbf{67} (1987), 93--177.

\bibitem[FK06]{FK:CNCIT}
T.~Fukaya and K.~Kato, \emph{A formulation of conjectures on {$p$}-adic zeta
  functions in non-commutative {I}wasawa theory}, Proceedings of the St.
  Petersburg Mathematical Society (Providence, RI), vol. XII, Amer. Math. Soc.
  Transl. Ser. 2, no. 219, American Math. Soc., 2006, pp.~1--85.

\bibitem[Fu11]{LeiFu:EtaleCohomology}
L.~Fu, \emph{Etale cohomology theory}, Nankai Tracts in Mathematics, vol.~13,
  World Scientific Publishing Co. Pte. Ltd., Hackensack, NJ, 2011.

\bibitem[GP15]{GreitherPopescu:EquivariantMC}
C.~Greither and C.~Popescu, \emph{An equivariant main conjecture in {I}wasawa
  theory and applications}, J. Algebraic Geom. \textbf{24} (2015), no.~4,
  629--692.

\bibitem[GP17]{GreitherPopescu:AbstractMotivesAndTatesClass}
\bysame, \emph{Abstract $\ell$-adic $1$-motives and {T}ate's class},
  \href{https://arxiv.org/abs/1710.02596}{\texttt{arXiv:1710.02596}}, 2017.

\bibitem[Gre83]{Greenberg:ArtinLfunctions}
R.~Greenberg, \emph{On {$p$}-adic {A}rtin {$L$}-functions}, Nagoya Math. J.
  \textbf{89} (1983), 77--87. \MR{692344}

\bibitem[Gro60]{EGA1}
A.~Grothendieck, \emph{{\'E}l{\'e}ments de g{\'e}om{\'e}trie alg{\'e}brique:
  {I.} {L}e langage des sch{\'e}mas}, no.~4, Inst. Hautes \'Etudes Sci. Publ.
  Math., 1960.

\bibitem[GW04]{GW:NoncommNoethRings}
K.~R. Goodearl and R.~B. Warfield, \emph{An introduction to noncommutative
  noetherian rings}, 2 ed., London Math. Soc. Student Texts, no.~61, Cambridge
  Univ. Press, Cambridge, 2004.

\bibitem[Kak13]{Kakde2}
M.~Kakde, \emph{The main conjecture of {I}wasawa theory for totally real
  fields}, Invent. Math. \textbf{193} (2013), no.~3, 539--626.

\bibitem[Kat06]{Kato:HeisenbergType}
K.~Kato, \emph{Iwasawa theory for totally real fields for {G}alois extensions
  of {H}eisenberg type}, preprint, 2006.

\bibitem[KW01]{KiehlWeissauer:WeilConjectures}
R.~Kiehl and R.~Weissauer, \emph{Weil conjectures, perverse sheaves and
  {$l$}'adic {F}ourier transform}, Ergebnisse der Mathematik und ihrer
  Grenzgebiete. 3. Folge, vol.~42, Springer-Verlag, Berlin, 2001.

\bibitem[Mih16]{Mihailescu:MuInvariant}
P.~Mih{\u{a}}ilescu, \emph{On the vanishing of {I}wasawa's constant {$\mu$} for
  the cyclotomic {$\mathbb{Z}_p$}-extensions of {CM} number fields},
  \href{https://arxiv.org/abs/1409.3114}{\texttt{arXiv:1409.3114v2}}, February
  2016.

\bibitem[Mil80]{Milne:EtCohom}
J.~S. Milne, \emph{Etale cohomology}, Princeton Mathematical Series, no.~33,
  Princeton University Press, New Jersey, 1980.

\bibitem[Mil06]{Milne:ADT}
\bysame, \emph{Arithmetic duality theorems}, second ed., BookSurge, LLC,
  Charleston, SC, 2006. \MR{2261462 (2007e:14029)}

\bibitem[MT07]{MT:1TWKTS}
F.~Muro and A.~Tonks, \emph{The {1}-type of a {W}aldhausen {K}-theory
  spectrum}, Advances in Mathematics \textbf{216} (2007), no.~1, 178--211.

\bibitem[MT08]{MT:OnK1WaldCat}
\bysame, \emph{On {$K_1$} of a {W}aldhausen category}, K-Theory and
  Noncommutative Geometry, EMS Series of Congress Reports, 2008, pp.~91--116.

\bibitem[MTW15]{MTW:DeterminantFunctors}
Fernando Muro, Andrew Tonks, and Malte Witte, \emph{On determinant functors and
  {$K$}-theory}, Publ. Mat. \textbf{59} (2015), no.~1, 137--233.

\bibitem[Nic13]{Nickel:EIwTh}
A.~Nickel, \emph{Equivariant {I}wasawa theory and non-abelian {S}tark-type
  conjectures}, Proc. Lond. Math. Soc. (3) \textbf{106} (2013), no.~6,
  1223--1247.

\bibitem[NSW00]{NSW:CohomNumFields}
J.~Neukirch, A.~Schmidt, and K.~Wingberg, \emph{Cohomology of number fields},
  Grundlehren der mathematischen Wissenschaften, no. 323, Springer Verlag,
  Berlin Heidelberg, 2000.

\bibitem[Oli88]{Oliver:WhiteheadGroups}
R.~Oliver, \emph{Whitehead groups of finite groups}, London Mathematical
  Society lecture notes series, no. 132, Cambridge University Press, Cambridge,
  1988.

\bibitem[RW11]{RW:MainConjecture}
J.~Ritter and A.~Weiss, \emph{On the ``main conjecture'' of equivariant
  {I}wasawa theory}, J. Amer. Math. Soc. \textbf{24} (2011), no.~4, 1015--1050.

\bibitem[Sch79]{Schn:Galoiskohomologiegruppen}
P.~Schneider, \emph{{\"U}ber gewisse {G}aloiscohomologiegruppen}, Math. Z.
  \textbf{260} (1979), 181--205.

\bibitem[SV13]{SchneiderVenjakob:SK1}
P.~Schneider and O.~Venjakob, \emph{{$SK_1$} and {L}ie algebras}, Math. Ann.
  \textbf{357} (2013), no.~4, 1455--1483.

\bibitem[TT90]{ThTr:HAKTS+DC}
R.~W. Thomason and T.~Trobaugh, \emph{Higher algebraic {$K$}-theory of schemes
  and derived categories}, The Grothendieck Festschrift, vol. III, Progr.
  Math., no.~88, Birkh{\"a}user, 1990, pp.~247--435.

\bibitem[Wal85]{Wal:AlgKTheo}
F.~Waldhausen, \emph{Algebraic {$K$}-theory of spaces}, Algebraic and Geometric
  Topology (Berlin Heidelberg), Lecture Notes in Mathematics, no. 1126,
  Springer, 1985, pp.~318--419.

\bibitem[War93]{Warner:TopRings}
S.~Warner, \emph{Topological rings}, North-Holland Mathematical Studies, no.
  178, Elsevier, Amsterdam, 1993.

\bibitem[Was97]{Wa:ICF}
L.~C. Washington, \emph{Introduction to cyclotomic fields}, 2 ed., Graduate
  Texts in Mathematics, no.~83, Springer-Verlag, New York, 1997.

\bibitem[Wit]{Witte:zetaisos}
M.~Witte, \emph{On {$\zeta$}-isomorphisms for totally real fields}, in
  preparation.

\bibitem[Wit08]{Witte:PhD}
\bysame, \emph{Noncommutative {I}wasawa main conjectures for varieties over
  finite fields}, Ph.D. thesis, Universit{\"a}t Leipzig, 2008,
  \url{http://d-nb.info/995008124/34}.

\bibitem[Wit13a]{Witte:Survey}
\bysame, \emph{Noncommutative main conjectures of geometric {I}wasawa theory},
  Noncommutative Iwasawa Main Conjectures over Totally Real Fields
  (Heidelberg), PROMS, no.~29, Springer, 2013, pp.~183--206.

\bibitem[Wit13b]{Witte:Splitting}
\bysame, \emph{On a localisation sequence for the {K}-theory of skew power
  series rings}, J. K-Theory \textbf{11} (2013), no.~1, 125--154.

\bibitem[Wit13c]{Witte:NCIMCFuncField}
\bysame, \emph{On a noncommutative {I}wasawa main conjecture for function
  fields}, preprint, 2013.

\bibitem[Wit14]{Witte:MCVarFF}
\bysame, \emph{On a noncommutative {I}wasawa main conjecture for varieties over
  finite fields}, J. Eur. Math. Soc. (JEMS) \textbf{16} (2014), no.~2,
  289--325.

\bibitem[Wit16]{Witte:UnitLfunctions}
\bysame, \emph{Unit {$L$}-functions for {\'e}tale sheaves of modules over
  noncommutative rings}, Journal de th\'eorie des nombres de Bordeaux
  \textbf{28} (2016), no.~1, 89--113.

\bibitem[Wit17]{Witte:Habil}
\bysame, \emph{Non-commutative {I}wasawa theory for global fields},
  Habilitationsschrift, Ruprecht-Karls-Universit\"at Heidelberg, 2017.

\bibitem[WY92]{WeibYao:Localization}
C.~Weibel and D.~Yao, \emph{Localization for the {$K$}-theory of noncommutative
  rings}, Algebraic {$K$}-Theory, Commutative Algebra, and Algebraic Geometry,
  Contemporary Mathematics, no. 126, AMS, 1992, pp.~219--230.

\end{thebibliography}
\end{document}